\documentclass[11pt,letterpaper]{amsart}
\usepackage{amssymb}
\usepackage{mathrsfs, tikz}
\usepackage[all,cmtip]{xy}
\usepackage{pstricks}

\usepackage{ulem}
\usepackage{comment}

\usepackage{graphpap,color}

\definecolor{cof}{RGB}{219,144,71}
\definecolor{pur}{RGB}{186,146,162}
\definecolor{greeo}{RGB}{91,173,69}
\definecolor{greet}{RGB}{52,111,72}

\numberwithin{equation}{section}

\newtheorem{thm}{Theorem}[section]
\newtheorem{defn}[thm]{Definition}
\newtheorem{prop}[thm]{Proposition}

\newtheorem{lemma}[thm]{Lemmminglea}
\newtheorem{cor}[thm]{Corollary}

\newtheorem{example}[thm]{Example}

\newtheorem{rem}[thm]{Remark}

\newcommand{\PGL}{{\rm PGL}}

\newcommand{\Id}{{\rm Id}}
\newcommand{\Gr}{{\rm Gr}}

\newcommand{\lt}{{\rm lt}}

\newcommand{\CD}{\xymatrix@R=1pc@C=1pc}
\newcommand{\CDR}{\xymatrix@R=1pc}
\newcommand{\CDC}{\xymatrix@C=1pc}

 \DeclareMathOperator{\Spec}{Spec}

\def\cB{{\mathcal B}}

\def\cD{{\mathcal D}}

\def\cE{{\mathcal E}}

\def\cL{{\mathcal L}}

\def\cZ{{\mathcal Z}}

\def\sF{{\mathscr F}}

\def\sR{\mathscr{R}}
\def\sV{\mathscr{V}}
\def\tsV{\widetilde{\mathscr{V}}}

\def\fG{\mathfrak{G}}

\def\fL{\mathfrak{L}}

\def\fT{\mathfrak{T}}
\def\fV{\mathfrak{V}}

\def\fb{\mathfrak{b}}

\def\fd{\mathfrak{d}}
\def\fe{\mathfrak{e}}

\def\fl{\mathfrak{l}}
\def\fr{\mathfrak{r}}

\def\frb{{\fr\fb}}

\def\fr{\mathfrak{r}}
\def\fs{\mathfrak{s}}
\def\ft{\mathfrak{t}}

\def\FF{{\mathbb F}}
\def\NN{{\mathbb N}}
\def\PP{{\mathbb P}}

\def\GG{{\mathbb G}}
\def\GGm{{{\mathbb G}_{\rm m}}}

\def\TT{{\mathbb T}}
\def\ZZ{{\mathbb Z}}

\def\II{{\mathbb I}}

\def\AA{{\mathbb A}}

\def\QQ{{\mathbb Q}}

\def\rU{{\rm U}}
\def\Ga{{\Gamma}}
\def\tGa{{\widetilde{\Gamma}}}
\def\Om{{\Omega}}

\def\vt{{\varTheta}}

\def\vt{{\vartheta}}

\def\si{{\sigma}}

\def\bG{{\bar G}}

\def\var{{\rm Var}}
\def\Index{{\rm Index}}
\def\invlex{{\rm rlex}}
\def\rlex{{\rm rlex}}
\def\lex{{\rm lex}}

\def\zero{{=0}}
\def\one{{=1}}

\def\tX{{\widetilde X}}
\def\tY{{\widetilde Y}}
\def\tZ{{\widetilde Z}}

\def\lra{\longrightarrow}

\def\kk{{\bf k}}

\def\bp{{\bf p}}

\def\bt{{\bf t}}

\def\bU{{\bf U}}

\def\bx{{\bf x}}
\def\bz{{\bf z}}

\def\bm{{\bf m}}
\def\bn{{\bf n}}

\def\ua{{\underbar {\it a}}}
\def\ub{{\underbar {\it b}}}

\def\ue{{\underbar {\it e}}}
\def\ui{{\underbar {\it i}}}

\def\uh{{\underbar {\it h}}}
\def\uk{{\underbar {\it k}}}

\def\um{{\underbar {{\mu}}}}
\def\um{{\underbar {\it m}}}
\def\uu{{\underbar {\it u}}}
\def\uv{{\underbar {\it v}}}
\def\uw{{\underbar {\it w}}}

\def\buw{{\underbar {\bf w}}}

\def\pl{{\hbox{Pl\"ucker}}}
\def\hpl{{\hbox{$\vp\vr$-Pl\"ucker}}}

 \def\2{{\rm I\!I}}

\def\bF{{\bar F}}
\def\-{{\setminus}}

\def\ve{{\varepsilon}}

\def\vp{{\varpi}}
\def\vr{{\varrho}}
\def\hs{{\hslash}}
\def\vi{{\varphi}}
\def\vik{{\varphi_{[k]}}}
\def\cBk{{\cB_{[k]}}}

\def\cBkd1{{\cB_{[k]}^{d_\vr=1}}}
\def\cBd1{{\cB^{d_\vr=1}}}

\def\modcBkd1{{\mod \; \cB_{[k]}^{d_\vr=1}}}

\def\cBgov{{\cB^\gov}}
\def\cBngv{{\cB^\ngv}}

\def\sgn{{\rm sgn}}

\def\rk{{\rm rank \;}}
\def\ori{{\rm ori}}

\def\inc{{\rm inc}}
\def\ngv{{\rm ngv}}

\def\gov{{\rm gov}}

\def\mh{{\hbox{\scriptsize ${\rm mh}$}}}

\def\sfm{{{\sF}}}

\def\sFgr{{{\sF}^{\rm rel}_{\um, \Ga}}}
\def\sFgir{{{\sF}^{\rm irr}_{\um, \Ga}}}

\def\de{\delta}

\def\tsR{{\widetilde{\sR}}}

\def\La{\Lambda}

\def\up{{\Upsilon}}

\def\bGr{{\underline \Gr}}

\def\ud{{\underbar  d}}

\def\whwp{{\widehat\wp}}
\def\barf{{\bar f}}

\def\di{\diamond}

\makeatletter
\newcommand*\bigcdot{\mathpalette\bigcdot@{.7}}
\newcommand*\bigcdot@[2]{\mathbin{\vcenter{\hbox{\scalebox{#2}{$\m@th#1\bullet$}}}}}
\makeatother

\def\bcd{{\bigcdot}}

\makeatletter
\def\@tocline#1#2#3#4#5#6#7{\relax
  \ifnum #1>\c@tocdepth 
  \else
    \par \addpenalty\@secpenalty\addvspace{#2}%
    \begingroup \hyphenpenalty\@M
    \@ifempty{#4}{%
      \@tempdima\csname r@tocindent\number#1\endcsname\relax
    }{%
      \@tempdima#4\relax
    }%
    \parindent\z@ \leftskip#3\relax \advance\leftskip\@tempdima\relax
    \rightskip\@pnumwidth plus4em \parfillskip-\@pnumwidth
    #5\leavevmode\hskip-\@tempdima
      \ifcase #1
       \or\or \hskip 1em \or \hskip 2em \else \hskip 3em \fi%
      #6\nobreak\relax
    \hfill\hbox to\@pnumwidth{\@tocpagenum{#7}}\par
    \nobreak
    \endgroup
  \fi}
\makeatother

\begin{document}

\title{Universal Characteristic-free Resolution of Singularities, I}

\date{}
\author{Yi Hu}

\maketitle

\begin{abstract} 
We prove that for any singular integral affine variety $X$ of finite presentation over a perfect field defined over $\ZZ$,
there exists a smooth morphism from $Y$
 onto $X$ such that $Y$ admits a resolution. That is, there exists a smooth scheme $\widetilde{Y}$ and a projective birational   morphism from
$\widetilde{Y}$ onto $Y$, followed by a  smooth morphism 
from $Y$ onto $X$.

Our approach differs fundamentally from existing methods, as we neither restrict to any specific singular variety nor fix the characteristic. Instead, we design a  {\it universal} blowup process that
{\it simultaneously} resolves all possible singularities, and,
 our method is entirely characteristic-free.
 \end{abstract}

\maketitle

\tableofcontents

\section{Introduction}


\begin{thm}\label{main:intro}
{\rm (Characteristic-free Resolution of Singularity Types)}
Let $X$ be an integral affine scheme
of finite presentation  over a perfect field defined over $\ZZ$,. 
Assume further that $X$ is singular. 
Then, there exists a  smooth morphism $Y \to X$ from a scheme $Y$ onto $X$,
a smooth scheme  $\tY$ and a surjective  proper birational morphism $\tY \to Y$.  
\end{thm}

Put it in display, we have two surjective morphisms
\begin{equation}\label{reso}
\tY \lra Y \lra X
\end{equation}
such that the first is proper birational  and the second is smooth.
In such a case, we say $X$  admits a resolution of singularity type.
We prove the existence of \eqref{reso} as follow.

Our approach is fundamentally distinct from all existing methods in that we neither restrict to any specific singular variety nor fix the characteristic.  Precisely, we develop a \emph{{\it universal platform}} for all singular affine varieties, governed by \emph{{\it universal defining relations.}}
Within this framework,  every singularity as a stratum of a smooth
ambient space
is realized as an intersection of a fixed set of smooth divisors. 
We then design a \emph{{\it universal blowup process}}
 that  separates these divisors, 
thereby simultaneously resolving all possible singularities.  
This approach eliminates the need for any monotone invariants to track singularity changes, as our method inherently ensures transversal intersections of all ambient smooth divisors, universal for all possible singularities over $\ZZ$.

In summary, Theorem \ref{main:intro} provides 
{\it universal characteristic-free resolutions of all singularity types}.
Our method operates over integers, therefore, is entirely characteristic-free.
This is rigorously formulated in Theorem \ref{main2:intro}.

We now provide some details about our approach.

\subsection{Some key new insights:
 the universal relations for all singularities} $\ $

\subsubsection{}
\noindent
 {\bf Using universal relations for singularities.}
\smallskip

In inertial thinking, one needs  arbitrary relations when  treating arbitrary singularities.
Our key point  instead is that we can commence with 
a set of considerably simpler  standardized equations for singularities. 
To furnish such standard equations for arbitrary singularities, 
Mn\"ev's universality (\cite{Mnev88}) offers a platform to begin with.
By the Gelfand-MacPherson correspondence,  Lafforgue's version of Mn\"ev's universality [Theorem I.14, \cite{La03}] (cf. \cite{LV12})
 asserts that the quotient  by the maximal torus  $\GG_m^{n-1}$ of 
a matroid Schubert cell of 
the Grassmannian $\Gr^{3,E}$ with varying $n=\dim E$  
can possess singularity of any given type defined over $\ZZ$. In its rigorous term,
 Lafforgue's [Theorem I.14, \cite{La03}]
 provides the smooth morphism $Y \lra X$
as stated in Theorem \ref{main:intro}.

We still need to resolve $Y$, which has the same singularity types as $X$.

In the inertial approach, when working with the arbitrary affine variety $X$, we start with its ideal $I$, generated by finitely many arbitrary equations in an affine space. The goal is to construct a birational modification $\widetilde{X}$ of $X$, where the new ideal $\widetilde{I}$ is locally generated by a refined set of equations. Despite the apparent diversity of methods, the 
fundamental idea remains the same: one aims to replace the poorly behaved Jacobian matrix $\text{Jac}(I)$ with a well-behaved Jacobian matrix $\text{Jac}(\widetilde{I})$ (e.g., making it of maximal rank).

However, since the original equations defining $I$ are arbitrary, deriving a set of desirable new equations for $\widetilde{I}$ from these arbitrary relations seems highly non-trivial—especially in positive characteristic $p > 0$. A key obstacle to the author is that these equations may a priori contain factors like $x^p$, and it is unclear whether existing techniques, even with blowups, can reliably eliminate or make such terms irrelevant.

This is exactly where our key insight comes into play: for the equally singular variety $Y$, we can instead begin with a set of much simpler, \emph{{\it standardized universal equations}},  
thereby circumventing the aforementioned obstacle.
Perhaps more importantly, we do not fix any singular variety $Y$ and
treat it individually, instead, we put singularities altogether and treat them \emph{\it universally}.

We elaborate on the above points below.

In Proposition (p.4) of \cite{La03}, Lafforgue provides explicit equations for the variety $Y$, showing that $Y$ is a Zariski open subset of a closed subscheme $Y_\Gamma$ of 
an affine chart $\bU_\Gr \subset \mathrm{Gr}^{3,E}$. This closed 
subscheme $Y_\Gamma$ is cut out from $\bU_\Gr$ 
by explicit coordinate hyperplanes (indexed by $\Gamma$) in 
an affine chart $\bU$ of the $\pl$ projective space 
$\mathbb{P}(\wedge^3 E)$. In this article, we call this unique closed 
affine subscheme $Y_\Gamma$ a \emph{\it $\Gamma$-scheme}.

The defining equations of $Y_\Gamma$ in
the affine space $\bU$ consist precisely of:
\begin{itemize}
\item All $\pl$ relations, and
\item The equations of coordinate hyperplanes (indexed by $\Gamma$)
\end{itemize}

These equations are much less $``$arbitrary$"$ than those defining $X$: they are either linear or quadratic, square-free (thereby, of course, containing no factors like $x^p$), and equally crucially,
they are \emph{{\it exceptionally well-organized}}.

This structure has two key consequences:
\begin{enumerate}
\item The first two properties eliminate the fundamental obstruction of $x^p$ factors when working in characteristic $p>0$.
\item The last property—the highly organized nature of the equations—allows us to systematically track and manage changes in the equations throughout an extensive sequence of blowups.  This precise bookkeeping is crucial, as it allows us to:
\begin{itemize}
\item Maintain control over the equations' form throughout the extensive resolution process, and
\item Preserve and propagate certain desired properties (particularly 
 square-freeness),
\end{itemize}
both of which are essential for our approach.
\end{enumerate}

As a result, at the end of our resolution process, we can explicitly compute the Jacobian matrix $\mathrm{Jac}(\widetilde{I})$ for the transformed variety $\widetilde{Y}$, which establishes the smoothness of $\widetilde{Y}$.

\subsubsection{}
\noindent
 {\bf Resolving all singularities simultaneously:
\noindent
 universal and characteristic-free approach}
\smallskip

Moreover, our approach differs fundamentally from all existing methods in studying resolution problems in the following key aspects:
\begin{itemize}
\item Unlike traditional approaches, we neither fix nor treat individual singular varieties $X$ in isolation;
\item Nor do we restrict our method to any specific characteristic $p$. Instead, we develop a completely characteristic-free framework.
\end{itemize}

Again, according to Theorem I.14 of \cite{La03}, all singularity types $Y_\Gamma$ are embedded in $\mathrm{Gr}^{3,E}$ with varying $n = \dim E$, sharing the same-type defining relations:
\begin{itemize}
\item The $\pl$ relations
\item Equations of coordinate subspaces indexed by $\Gamma$
\end{itemize}
In other words, these are the \emph{{\it universal relations}} for all
 singularity types.

Heuristically, $\mathrm{Gr}^{3,E}$ (with varying $n = \dim E$) serves as a $``$universal space$"$ containing all possible singularity types. We develop a uniform resolution process that:
\begin{itemize}
\item Blows up the universal space $\mathrm{Gr}^{3,E}$;
\item Depends solely on the Pl\"ucker relations\footnote{Though a genius might be needed to manage arbitrary equations over characteristic zero, $\pl$ relations demand no such brilliance—they apply universally to all singularities, irrespective of characteristic.}, universal for all $Y_\Ga$;
\item Applies universally to all singularity types $Y_\Ga$.
\end{itemize}

Upon completing this process, we calculate Jacobians, 
for the transformed  universe
and all transformed $\widetilde Y_\Gamma$, and then 
find that all singularity types 
$Y_\Gamma$ are  simultaneously resolved.
During the process, we never need to check
singularities for any individual $Y_\Ga$, 
thereby eliminating  the need
 for introducing monotone invariants to track singularity changes,
For example, it is very possible that the singularities of 
some individual $Y_\Ga$ might become worse at some steps
of the transforms.

Put it differently, we developed a \emph{{\it universal blowups}} in
the ambient universe  $\Gr^{3, E}$ 
that design to resolve all singularity types simultaneously in the very end.

Since our resolution process is performed on 
the universe  $\Gr^{3, E}$ defined over
$\ZZ$ and makes no reference to any particular characteristic, the  entire
approach is fundamentally characteristic-free.

\smallskip
In the sequel, we will provide more technical 
details for the statements in this subsection.
As mentioned, we start with blowing up the universe 
$\Gr^{3,E}$, and the blowup process is designed solely 
based on the $\pl$ relations.

\subsection{The newer universe $\sV$ with relations suitable for 
designing the blowups}
$\ $

Instead of directly working with the $\pl$ relations,
we first establish a newer universe $\sV$, birational to $\Gr^{3,E}$, also
contains all singularity types.  The benefits are that
the $\pl$ relations are transformed into linear equations and all
the remaining relations are binomials.
These binomials are also neatly organized, and are especially suitable to
design desirable sequential blowups to achieve our purpose.

We now establish $\sV$.

\subsubsection{$\pl$ embedding and $\pl$ relations} $\ $

Following Lafforgue's presentation of \cite{La03}, 
suppose we have a set of vector spaces, $E_1, \cdots, E_n$ such that 
every $E_\alpha$, $1\le \alpha\le n$,  is of dimension 1 over a field $\kk$
 (or, a free module of rank 1 over $\ZZ$),
 for some positive integer $n>1$.
We let  $$E=E_1 \oplus \ldots \oplus E_n.$$ 
 Then, the Grassmannian $\Gr^{3,E}$, defined by
$$\Gr^{3, E}=\{ \hbox{linear subspaces} \;F \hookrightarrow E \mid \dim F=3\}, $$
is a projective variety defined over $\ZZ$.

We have a canonical decomposition
$$\wedge^3 E=\bigoplus_{\ui =(i_1<i_2< i_3) \in \II_{3,n}} E_{i_1}\otimes E_{i_2} \otimes E_{i_3},$$
where $\II_{3,n}$ is the set of all sequences of 3 distinct integers between 1 and $n$.

This gives rise to the $\pl$ embedding of the Grassmannian by
$$\Gr^{3, E} \hookrightarrow \PP(\wedge^3 E)=\{(p_\ui)_{\ui \in \II_{3,n}} \in \GG_m 
\backslash (\wedge^3 E \- \{0\} )\},$$
$$F \lra [\wedge^3 F],$$
where $\GG_m$ is the multiplicative group.

As a closed subscheme of $\PP(\wedge^3 E)$, the Grassmanian  $\Gr^{3, E}$ 
 is defined,  among other relations in general, by 
the $\pl$  ideal $I_\wp$, generated by all $\pl$ relations, 
whose typical member is expressed succinctly, in this article, as
\begin{equation}\label{eq1-intro}
F: \; \sum_{s \in S_F} \sgn (s) p_{\uu_s} p_{\uv_s}
\end{equation}
where $S_F$ is an index set, $\uu_s, \uv_s \in \II_{3,n}$
for any $s \in S_F$, and $\sgn (s)$ is the $\pm$ sign associated with the term $p_{\uu_s} p_{\uv_s}$
(see \eqref{pluckerEq} and \eqref{succinct-pl} for details).

\subsubsection
{$\pl$ relation $F$ and the projective space $\PP_F$} $\ $

Given the above $\pl$ relation $F$, we introduce
  the projective space $$\PP_F$$ which comes equipped with the homogeneous coordinates
$$[x_{(\uu_s,\uv_s)}]_{s\in S_F}.$$
Then, corresponding to each $\pl$ relation \eqref{eq1-intro}, there is
a linear  homogeneous equation in $\PP_F$, 
called the induced {\it  linearized $\pl$ relation}, 
\begin{equation}\label{eq2-intro}
L_F: \; \sum_{s \in S_F} \sgn (s) x_{(\uu_s,\uv_s)}
\end{equation}
 (see Definition \ref{defn:linear-pl}). We set
 $ \La_F:=\{(\uu_s,\uv_s) \mid s \in S_F\}.$
 
As  any $\Ga$-scheme is a closed subscheme of some affine chart,  we can focus on an affine chart
$\bU =(p_\um \ne 0)$ of the $\pl$ projective space $\PP(\wedge^3 E)$ for some fixed $\um \in \II_{3,n}$. 
Up to permutation, we can assume $\um=(123)$ and this is
what we do throughout this paper.
We can identify the coordinate ring of $\bU$ with the polynomial ring 
$\kk [x_\uu]_{\uu \in \II_{3,n} \- \{\um\}}$. For any $\pl$ relation $F$, we let
$\bF$ be the de-homogenization of $F$ on the chart $\bU$.
Given this chart, we then explicitly describe a set of $\pl$ relations, called
{\it $\um$-primary $\pl$ relations}, listed under a carefully chosen total order $``<_\wp"$,
\begin{equation}\label{order-sF}
\sF =\{\bF_1 <_\wp \cdots <_\wp \bF_\Upsilon\},
\end{equation}
with $\Upsilon= {n \choose 3}-1-3(n-3)$, such that together they define
the closed embedding 
$$\bU \cap \Gr^{3,E} \lra \bU.$$
Further, on the chart $\bU$,  if we set $p_\um =1$
 and set $x_\uu=p_\uu$ for any $\uu \in \II_{3,n}\-\{\um\}$,
 then any  $\um$-primary relation $\bF\in \sF$ admits the following de-homogenized expression
\begin{equation}
\bF: \;  \sgn(s_F) x_{\uu_{s_F}} +\sum_{s \in S_F \- \{s_F\}} \sgn (s) x_{\uu_s}x_{\uv_s},
\end{equation}
where $x_{\uu_{s_F}}$ is called the \emph{{\it leading variable}} of $\bF$ whose term is called the \emph{{\it leading term }}\footnote{
The notion of leading term or leading variable plays important technical
roles in several aspects: it helps to design the universal blowups,
it particularly helps to assure certain important 
square-freeness and hence helps to obtain the desirable Jacobian matrices in
the final outputs. They are special and technically important,
the reader will meet them frequently throughout this paper.} of $\bF$
and $s_F \in S_F$ is the index
for the leading term. 
Correspondingly, the term $ \sgn (s_F) x_{(\um,\uu_{s_F})}$
 is called the \emph{{\it leading term}}
 of the linearized $\pl$ relation $L_F$ of \eqref{eq2-intro}.
 See \eqref{tour: all primary pl}  for concrete presentations of all the
$\um$-primary relations; see around
\eqref{equ:localized-uu} and \eqref{the-form-LF} for precise details.

Next,   parallel to the construction used in \cite{Hu2025},
we introduce the  rational map
\begin{equation}\label{this-theta-intro}
 \xymatrix{
\Theta_{[\up],\Gr}: \bU \cap \Gr^{3, E} \ar @{^{(}->}[r]  & \bU \ar @{-->}[r]  & \prod_{\bF \in \sF} \PP_F   }
\end{equation}
$$
 [x_\uu]_{\uu \in \II_{3,n}^\star} \lra  
\prod_{\bF \in \sF}   [x_\uu x_\uv]_{(\uu,\uv) \in \La_F}
$$   
where $[x_\uu]_{\uu \in \II_{3,n}^\star} $ 
is the de-homogenized $\pl$ coordinates of a point 
of  $\bU \cap \Gr^{3, E}$.   Here, we set
$$\II_{3,n}^\star=\II_{3,n}\- \{(123)\}.$$

\subsubsection{The new model $\sV$} $\ $

We let $\sV$ be the closure of the graph of the rational map $\bar\Theta_{[\up],\Gr}$. Then, 
 we obtain the following diagram 
$$ \xymatrix{
\sV \ar[d] \ar @{^{(}->}[r]   \ar[d] \ar @{^{(}->}[r]  &
\sR:= \bU  \times \prod_{\bF \in \sF} \PP_F \ar[d] \\
\bU \cap \Gr^{3, E}  \ar @{^{(}->}[r]    & \bU.}
$$

The scheme $\sV$ is singular, in general, and is birational to $\bU \cap \Gr^{3, E}$.
In \S \ref{tour}, the reader can find a detailed exposition of the motivations behind introducing the model $\sV$.


Then, as the further necessary
 steps to achieve our ultimate goal, we are to perform some carefully designed
sequential embedded blowups for $(\sV \subset \sR)$.

For  the purpose of applying induction, employed in some  proofs, 
we also need to introduce the following intermediate rational maps. 

For any positive integer $h$, 
we set  $[h]:=\{1,\cdots,h\}.$

Then,  for any $k \in [\up]$, we have the rational map
\begin{equation}\label{this-theta[k]-intro}
 \xymatrix{
\Theta_{[k],\Gr}: \bU \cap \Gr^{3, E} \ar @{^{(}->}[r]  & \bU \ar @{-->}[r]  & 
 \prod_{i \in [k]} \PP_{F_i}   }
\end{equation}
$$
 [x_\uu]_{\uu \in \II_{3,n}^\star} \lra  
\prod_{i \in [k]}   [x_\uu x_\uv]_{(\uu,\uv) \in \La_{F_i}}
$$ 
We let $\sV_{[k]}$ be the closure of the graph of the rational map $\Theta_{[k],\Gr}$. Then, 
 we obtain the following diagram 
$$ \xymatrix{
\sV_{[k]} \ar[d] \ar @{^{(}->}[r]   \ar[d] \ar @{^{(}->}[r]  &
\sR_{[k]}:= \bU  \times  \prod_{i \in [k]} \PP_{F_i}\ar[d] \\
\bU \cap \Gr^{3, E}  \ar @{^{(}->}[r]    & \bU.
}
$$
The scheme $\sV_{[k]}$ is birational to  $\bU \cap \Gr^{3, E}$. 

Set $\sR_{\sF_{[0]}}:=\bU$.  There exists a forgetful map 
$$\sR_{\sF_{[j]}} \lra \sR_{\sF_{[j-1]}},\;\;\; \hbox{ for any $j \in [\up]$}.$$

In the above notations, we have
$$\sV=\sV_{\sF_{[\up]}}, \;\; \sR_{\sF}=\sR_{\sF_{[\up]}}.$$

\subsubsection{Establishing the new platform $(\sV \subset \sR_\sF)$} $\ $

We need defining relations for $\sV$ in the smooth ambient space $\sR_{\sF}$.

\begin{prop}  {\rm (Corollary \ref{eq-tA-for-sV})}  The scheme $\sV$, as a closed subscheme of
$\sR= \bU \times  \prod_{\bF \in \sF} \PP_F$,
is defined by the following relations, for all $\bF \in \sF$,
\begin{eqnarray} 
B_{F,(s,t)}: \;\; x_{(\uu_s, \uv_s)}x_{\uu_t}x_{ \uv_t}-x_{(\uu_t, \uv_t)}x_{\uu_s}x_{ \uv_s}, \;\; \forall \;\; 
s, t \in S_F \- \{s_F\}, \label{eq-Bres-intro}\\
B_{F, (s_F,s)}: \;\; x_{(\uu_s, \uv_s)}x_{\uu_F} - x_{(\um,\uu_F)}   x_{\uu_s} x_{\uv_s}, \;\;
\forall \;\; s \in S_F \- \{s_F\},  \label{eq-B-intro} \\ 
\cB^\frb,   \;\; \;\; \;\; \;\; \;\; \;\; \;\; \;\; \;\; \;\; \;\; \;\; \;\; \;\; \;\; \;\; \;\; \;\; \;\;
\label{eq-hq-intro}\\
L_F: \;\; \sum_{s \in S_F} \sgn (s) x_{(\uu_s,\uv_s)}, 
\;\; \;\; \;\; \;\; \;\; \;\;
\label{linear-pl-intro}
\end{eqnarray}
with $\bF$  expressed as 
$\sgn (s_F) x_{\uu_F} +\sum_{s \in S_F \- \{s_F\}} \sgn (s) x_{\uu_s}x_{\uv_s}$,
 where $\cB^\frb$  is the set of 
$\frb$-binomials  (see Definitions \ref{defn:frb}).  
\end{prop}
As promised, now, the platform
 $$(\bU \cap \Gr^{3,E} \subset \bU)$$ with the quadratic $\um$-primary $\pl$ relations
is replaced by the new platform
$$(\sV \subset \sR_\sF)$$
with linearized $\pl$ relations \eqref{linear-pl-intro} and 
some explicit binomial relations 
\eqref{eq-Bres-intro}, \eqref{eq-B-intro} and  \eqref{eq-hq-intro}.

\subsection{Designing the universal blowup algorithms 
on the new platform $(\sV \subset \sR_\sF)$} $\ $

\subsubsection
{The governing defining  relations for $(\sV \subset \sR_\sF)$} $\ $

Our construction of the desired embedded blowups on  
$\sV \subset \sR$
 is based upon the set of all binomial relations of \eqref{eq-B-intro}:
\begin{equation}\label{gov b intro}
\cB^\gov_F=\{B_{F,(s_F,s)} \mid  s \in S_F \- \{s_F\}\},\;\;
\cB^\gov=\bigsqcup_{\bF \in \sF} \cB^\gov_F,
\end{equation}
and all the linearized $\pl$ relation 
$$L_{\sF}=\{L_F \mid \bF \in \sF\}.$$
An element $B_{F,(s_F,s)}$ of  $\cB^\gov$ is called a \emph{\it governing} binomial relation (Definition \ref{gov-ngov-bi}).
The binomial relations of  $\cB^\gov$ together with
all linearized $\pl$ relations
$\{L_{\sF} \mid \bF \in \sF\}$
are called \emph{\it governing} relations.

We  also let
$$\cB^\ngv=\{ B_{F,(s,t)} \mid \bF \in \sF, \; s, t \in S_F \- \{s_F\}\}.$$
An element $B_{F,(s,t)}$ of  $\cB^\ngv$ is called a non-governing binomial relation of $\vr$-degree 1.
The non-governing binomial relations and 
all other defining binomial relations of $\cB^\frb$
play no roles in the {\it design} of the aforesaid embedded blowups.

($\di$) {\it 
The above may seem mysterious but our insight is that by calculating (estimating) 
the Jacobian of  the governing  relations, all the remaining relations
will become dependent after the process of the designed blowups.
This will be proved in \S \ref{main-statement}.}

Our process is algorithmic. To carry it out, we need various linear orders
on various finite sets.

The first such an order is the important order $(\sF,  <_\wp)$ 
 mentioned in \eqref{order-sF}.
Next, we provide a total order on the set $S_F \- \{s_F\}$ 
for every $\bF \in \sF$, and list it as
$$S_F \- \{s_F\}=\{s_1 < \cdots < s_{\ft_F}\}$$
where $(\ft_F+1)$ is the number of terms in the relation $F$, hence
$\ft_F$ only assumes value 2 or 3.
This renders us  to write 
\begin{equation}\label{Bktau intro}
\hbox{$B_{F,(s_F,s)}$ as $B_{(k\tau)}$}
\end{equation}
 where
$F=F_k$ for some $k \in [\up]$ and $s = s_\tau$ for some $ \tau \in [\ft_{F_k}]$.

We can now synopsize  the process of the embedded blowups for  $(\sV \subset \sR)$.

It is divided into two sequential blowups.

($\di$) {\it 
The first is the sequence of $\vt$-blowups: 
one for each $F_k$ with $k \in [\up]$,
based solely on the {\it leading} pairs 
$(x_{\uu_{F_k}}, x_{(\um, \uu_{F_k})})$ of
the governing binomial $B_{(k\tau)}$.}

($\di$) {\it 
The second is constructed by induction on $k \in [\up]$:
for each fixed $k \in [\up]$, it consists of a sequential $\wp$-blowups  
and then  a single $\ell$-blowup.   Then center for a $\wp$-blowup
depends solely on the two terms of
a governing binomial $B_{(k\tau)}$.
The center of a  $\ell$-blowup depends solely on a linearized
$\pl$ relation.
}

Every of the $\vt$-, $\wp$-, and $\ell$-blowups
is the blowup of a smooth ambient space along a codimension two
smooth closed subvariety that is the intersection of two smooth
divisors.

Every of the $\vt$-, $\wp$-, and $\ell$-blowups
 does its own designed purposes, which will be motivated and explained in 
some details in \S \ref{tour}.

\subsubsection{On $\vt$-sets, $\vt$-centers, and $\vt$-blowups} $\ $

For any primary $\pl$ relation $\bF_k \in \sF$, we introduce the corresponding $\vt$-set
$$\vt_{[k]}=\{x_{\uu_{F_k}}, x_{(\um,\uu_{F_k})} \}$$
 and the corresponding $\vt$-center 
$$Z_{\vt_{[k]}} = X_{\uu_{F_k}} \cap X_{(\um,\uu_{F_k})}$$ where
$$\hbox{$X_{\uu_{F_k}} = (x_{\uu_{F_k}}=0)$ and
$ X_{(\um,\uu_{F_k})} =(x_{(\um,\uu_{F_k})} =0)$.}$$
 We set $\tsR_{\vt_{[0]}}:=\sR$. Then, inductively,
 we let  $$\tsR_{\vt_{[k]}} \to \tsR_{\vt_{[k-1]}}$$ be the blowup of $\tsR_{\vt_{[k-1]}}$
 along (the proper transform of) the $\vt$-center
 $Z_{\vt_{[k]}}$ for all $k \in [\up]$. 
 This gives rise to the sequential $\vt$-blowups
\begin{equation}\label{vt-sequence-intro}
\tsR_{\vt}:=\tsR_{\vt_{[\up]}}  \lra \cdots \lra \tsR_{\vt_{[k]}} \lra \tsR_{\vt_{[k-1]}} \lra \cdots \lra \tsR_{\vt_{[0]}}.
\end{equation}
 Each morphism $\tsR_{\vt_{[k]}} \to \tsR_{\vt_{[k-1]}}$ is a smooth blowup, meaning, 
 the blowup of a smooth scheme along a smooth closed center. For any $k$, we let
 $\tsV_{\vt_{[k]}} \subset \tsR_{\vt_{[k]}} $ be the proper transform of $\sV$ in $\tsR_{\vt_{[k]}}$.
 We set $\tsV_{\vt}:=\tsV_{\vt_{[\up]}}$.

An notable benefit here is that 
after the completion of $\vt$-blowups, 
all the relations in $\cB^\ngv$ become dependent and can be discarded.

 We will describe more properties of $\vt$-blowups after we introduce all blowups.

\subsubsection{On $\wp$-sets, $\wp$-centers, and $\wp$-blowups} $\ $

All these blowups are constructed based on $\cB^\gov_{F_k}$ and $L_{F_k}$,
 inductively on $k \in [\up]$.

For any governing binomial $B_{(k\tau)} \in \cB^\gov_{F_k}$
(see \eqref{gov b intro} and \eqref{Bktau intro}), there exist a finite 
integer $\rho_{(k\tau)}$ depending on $(k\tau)$ and 
a finite integer $\si_{(k\tau)\mu}$ depending on $(k\tau)\mu$ for any 
$ \mu \in [\rho_{(k\tau)}]$.
For the initial case of $\wp$-blowups, we set $\tsR_{(\wp_{(11)}\fr_0)}=\tsR_\vt$.
Then for each 
$$\hbox{$((k\tau), \mu, h)$ with $\mu \in [\rho_{(k\tau)}]$ and
$h \in [\si_{(k\tau)\mu}]$,}$$ 
there exists a $\wp$-set $\phi_{(k\tau)\mu h}$ 
consisting of  two special divisors  on an inductively 
defined scheme $\tsR_{(\wp_{(k\tau)}\fr_{\mu -1})}$; its corresponding $\wp$-center
$Z_{\phi_{(k\tau)\mu h}}$ is the scheme-theoretic intersection of the two divisors.
Both of these two divisors are smooth. If we write $B_{(k\tau)}$ as in 
\eqref{eq-B-intro},
$$B_{F, (s_F,s)}: \;\; x_{(\uu_s, \uv_s)}x_{\uu_F} - x_{(\um,\uu_F)}   x_{\uu_s} x_{\uv_s}$$
for some $s \in  S_F \-\{s_F\}$.
Then,  one divisor is associated to
 $x_{(\uu_s, \uv_s)}x_{\uu_F}$ (i.e., determined by this term and its transformed
form in the previous blowups),
while the other is associated to $x_{(\um,\uu_F)}   x_{\uu_s} x_{\uv_s}$
(determined by it and its transformed form).
We let 
$$\cZ_{\wp_k}=\{Z_{\phi_{(k\tau)\mu h}} \mid k \in [\up], \tau \in [\ft_{F_k}], \mu \in [\rho_{(k\tau)}],
h \in [\si_{(k\tau)\mu}]\},$$
totally ordered lexicographically on the indexes $(k,\tau, \mu, h)$. Then, inductively,
 we let 
 $$\tsR_{(\wp_{(k\tau)}\fr_\mu\fs_{h})} \to \tsR_{(\wp_{(k\tau)}\fr_\mu\fs_{h-1})}$$
 be the blowup of $\tsR_{(\wp_{(k\tau)}\fr_\mu\fs_{h-1})}$ 
 along (the proper transform of) the $\wp$-center
 $Z_{\phi_{(k\tau)\mu h}}$. 
 This gives rise to the sequential  $\wp$-blowups with respect to $\cB^\gov_{F_k}$
\begin{equation}\label{wp-sequence-intro}
\tsR_{\wp_k}  \to \cdots \to
\tsR_{(\wp_{(k\tau)}\fr_\mu\fs_{h})} \to \tsR_{(\wp_{(k\tau)}\fr_\mu\fs_{h-1})} \to \cdots \to \tsR_{\ell_{k-1}},
\end{equation}
where $\tsR_{\ell_{k-1}}$ is inductively constructed from the previous $\wp$ and $\ell$-blowups, and 
$\tsR_{\wp_k} $ is the end scheme of $\wp$-blowups with respect to $\cB^\gov_{F_k}$.
  For any $(k\tau)\mu h$, we let
 $$\tsV_{(\wp_{(k\tau)}\fr_\mu\fs_{h})} \subset \tsR_{(\wp_{(k\tau)}\fr_\mu\fs_{h})} $$
  be the proper transform of $\sV$ in $\tsR_{(\wp_{(k\tau)}\fr_\mu\fs_{h})} $.
 We set $$\tsV_{\wp_k} \subset \tsR_{\wp_k}$$ be the last induced subscheme.
 Every scheme $\tsR_{(\wp_{(k\tau)}\fr_\mu\fs_{h})}$ 
  has a smooth open subset  $\tsR^\circ_{(\wp_{(k\tau)}\fr_\mu\fs_{h})}$ 
  containing  $\tsV_{(\wp_{(k\tau)}\fr_\mu\fs_{h})}$.
 We can cover $\tsR^\circ_{(\wp_{(k\tau)}\fr_\mu\fs_{h})}$,
hence also $\tsV_{(\wp_{(k\tau)}\fr_\mu\fs_{h})}$, by
smooth affine charts, called admissible affine charts.


\subsubsection{On $\ell$-sets, $\ell$-centers, and $\ell$-blowups} $\ $

Now, we let $D_{L_{F_k}}$ be the divisor of $\sR$ defined by $(L_{F_k}=0)$;
we let $E_{\vt_{k}}$ be the exceptional
divisor created by the $\vt$-blowup $\tsR_{\vt_{[k]}} \to \tsR_{\vt_{[k-1]}}$.
We let $D_{\wp_k, L_{F_k}}$ be proper transform of $D_{L_{F_k}}$ 
 and $E_{\wp_k, \vt_k}$ be the proper transform of
 $E_{\vt_{k}}$ in $\tsR_{\wp_k}$.
 We then let
\begin{equation}\label{ell-sequence-intro}
\tsR_{\ell_k} \lra \tsR_{\wp_k}
\end{equation}
 be the blowup of $\tsR_{\wp_k}$ along the intersection
 $D_{\wp_k, L_{F_k}} \cap E_{\wp_k, \vt_k}$.

We let $\tsV_{\ell_{k}}$ be the proper transform of $\tsV_{\wp_{k}}$ 
in $\tsR_{\ell_{}}$. 
The scheme $\tsR_{\ell_{k}}$ has a smooth open subset   $\tsR^\circ_{\ell_{k}}$ 
containing $\tsV_{\ell_{k}}$.  We can cover $\tsR^\circ_{\ell_{k}}$,
hence also $\tsV_{\ell_{k}}$, by
smooth affine charts, called admissible affine charts.

When $k=\up$, we obtain our \emph{\it final schemes}
$$\tsV_\ell:=\tsV_{\ell_{\up}}\subset \tsR_{\ell}:= \tsR_{\ell_{\up}}.$$

We point out here the $\ell$-blowup with respect to $F_k$ 
has to immediately follow the $\wp$-blowups
 with respect to $F_k$; the order of $\wp$-blowups  with respect to a fixed $\pl$ relation $F_k$
may be subtle and are carefully chosen. 
 
 \subsection{Properties of 
the universal $\vt$-, $\wp$-, and $\ell$-blowups} $\ $

To study the local structure of $\tsV_\ell \subset \tsR_\ell$, we approach 
it by induction via the sequential blowups \eqref{vt-sequence-intro},
\eqref{wp-sequence-intro},   and \eqref{ell-sequence-intro}.

As already mentioned, the schemes 
$$\hbox{ $\tsR_{\vt_{[k]}}$,
  $\tsR_{(\wp_{(k\tau)}\fr_\mu\fs_h)}$, and $\tsR_{\ell_{k}}$,}$$
respectively, have smooth open subsets 
 $$\hbox{   $\tsR^\circ_{\vt_{[k]}}$, $\tsR^\circ_{(\wp_{(k\tau)}\fr_\mu\fs_h)}$,
  and $\tsR^\circ_{\ell_{k}}$,}$$ containing 
$$\hbox{ $\tsV_{\vt_{[k]}}$, $\tsV_{(\wp_{(k\tau)}\fr_\mu\fs_h)}$,
 and $\tsV_{\ell_{k}}$. }$$
 We can cover  $\tsR^\circ_{\vt_{[k]}}$, $\tsR^\circ_{(\wp_{(k\tau)}\fr_\mu\fs_h)}$,
  and $\tsR^\circ_{\ell_{k}}$, hence also 
$\tsV_{\vt_{[k]}}$, $\tsV_{(\wp_{(k\tau)}\fr_\mu\fs_h)}$,
 and $\tsV_{\ell_{k}}$, respectively, by smooth affine charts, called admissible charts.

We prove the following.
\begin{itemize}
\item Proposition \ref{meaning-of-var-vtk} introduces local free variables
 for  any admissible affine chart $\fV$ of $\tsR_{\vt_{[k]}}$ and provides
 explicit geometric meaning for every of these free variables.
 
 Proposition  \ref{eq-for-sV-vtk} provides 
explicit description of the local defining equations of 
the scheme $\tsV_{\vt_{[k]}} \cap \fV$ on any admissible affine chart $\fV$
 of $\tsR_{\vt_{[k]}}$ and propagates certain square-freeness of these local equations.
 
\item Proposition \ref{meaning-of-var-wp/ell} introduces local free variables
 for  any admissible affine chart $\fV$ of $\tsR^\circ_{(\wp_{(k\tau)}\fr_\mu\fs_h)}$
 and  $\tsR^\circ_{\ell_k}$, 
 and provides explicit geometric meaning for every of these local free variables.
 
 Proposition  \ref{equas-wp/ell-kmuh}  combined with
 Proposition \ref{meaning-of-var-wp/ell} (9)
 provide explicit description of the local defining equations of the scheme  
$\tsV_{\wp_{(k\tau)}\fr_\mu\fs_h}  \cap \fV$ 
or $\tsV_{\ell_k}  \cap \fV$ on any admissible affine chart $\fV$ of 
 $\tsR^\circ_{(\wp_{(k\tau)}\fr_\mu\fs_h)}$ or $\tsR^\circ_{\ell_k}$,
and propagates certain square-freeness of these local equations.
 \end{itemize}

\subsection{The $\vt$-, $\wp$-, and $\ell$-birational transforms of $\Ga$-schemes} $\ $

To summarize the progress, we depict it in the diagram \eqref{theDiagram} below.

\begin{equation}\label{theDiagram}
 \xymatrix@C-=0.4cm{
  \tsR_{\ell} \ar[r] & \cdots  \ar[r] &  \tsR_{{\hbar}} \ar[r] &  \tsR_{{\hbar}'} \ar[r] &  \cdots  \ar[r] &   \sR_{\sF_{[j]}}  \ar[r] &  \sR_{\sF_{[j-1]}} \cdots \ar[r] &  \bU \\
   \tsR^\circ_{\ell}\ar @{^{(}->} [u]  \ar[r] & \cdots  \ar[r] &  \tsR^\circ_{\hbar}\ar @{^{(}->} [u]  \ar[r] &  \tsR^\circ_{\hbar} \ar @{^{(}->} [u] \ar[r] &  \cdots  \ar[r] &   \sR_{\sF_{[j]}} \ar @{^{(}->} [u]_{=} \ar[r] &  \sR_{\sF_{[j-1]}} \ar @{^{(}->} [u]_{=} \cdots \ar[r] &  \bU \ar @{^{(}->} [u]_{=}\\
    \tsV_{\ell} \ar @{^{(}->} [u]  \ar[r] & \cdots  \ar[r] &  \tsV_{{\hbar}}\ar @{^{(}->} [u]   \ar[r] &  \tsV_{{\hbar}'} \ar @{^{(}->} [u]  \ar[r] &  \cdots    \ar[r] &   \sV_{\sF_{[j]}} \ar @{^{(}->} [u]\ar[r] &  \sV_{\sF_{[j-1]}} \cdots \ar @{^{(}->} [u]  \ar[r] &  \bU \cap \Gr^{ d,E}   \ar @{^{(}->} [u]  \\
   \tZ_{\ell, \Ga} \ar @{^{(}->} [u]  \ar[r] & \cdots  \ar[r] &  \tZ_{{\hbar},\Ga}\ar @{^{(}->} [u]   \ar[r] &  \tZ_{{\hbar}',\Ga} \ar @{^{(}->} [u]  \ar[r] &  \cdots    \ar[r] &   Z_{\sF_{[j]},\Ga} \ar @{^{(}->} [u]\ar[r] &  Z_{\sF_{[j-1])},\Ga} \cdots \ar @{^{(}->} [u]  \ar[r] &  Z_\Ga  \ar @{^{(}->} [u]  \\
    \tZ^\dagger_{\ell, \Ga} \ar @{^{(}->} [u]  \ar[r] & \cdots  \ar[r] &  \tZ^\dagger_{{\hbar},\Ga}\ar @{^{(}->} [u]   \ar[r] &  \tZ^\dagger_{{\hbar}',\Ga} \ar @{^{(}->} [u]  \ar[r] &  \cdots    \ar[r] &   Z^\dagger_{\sF_{[j]},\Ga} \ar @{^{(}->} [u]\ar[r] &  Z^\dagger_{\sF_{[j-1])},\Ga} \cdots \ar @{^{(}->} [u]  \ar[r] &  Z_\Ga, \ar[u]_{=}       }
\end{equation}
 where all vertical uparrows are closed embeddings,
except the ones in the first row (upward vertical arrows)
where they are either identities or open embeddings.

Thus far, we have obtained the first three rows of the  diagram:  

 \begin{itemize}
\item In the first row: each morphism  $\tsR_{{\hbar}} \to  \tsR_{{\hbar}'}$ is 
$$\hbox{$\tsR_{\vt_{[k]}} \to \tsR_{\vt_{[k-1]}}$,  or 
$\tsR_{(\wp_{(k\tau)}\fr_\mu\fs_{h})} \to  \tsR_{(\wp_{(k\tau)}\fr_\mu\fs_{h-1})}$, or
  $\tsR_{\ell_{k}} \to  \tsR_{\wp_k}$,}$$
 and each is a blowup;
every  $\sR_{\sF_{[j]}}  \to  \sR_{\sF_{[j-1]}}$  is a projection, a forgetful map.
\item 
 In the third row: each morphism $\tsV_{{\hbar}} \to  \tsV_{{\hbar}'}$ is 
$$\hbox{$\tsV_{\vt_{[k]}} \to \tsV_{\vt_{[k-1]}}$, or 
$\tsV_{(\wp_{(k\tau)}\fr_\mu\fs_{h})} \to  \tsV_{(\wp_{(k\tau)}\fr_\mu\fs_{h-1})}$, or
 $\tsV_{\ell_k} \to  \tsV_{\wp_k}$, }$$
and  this morphism as well as each $\sV_{\sF_{[j]}}  \to  \sV_{\sF_{[j-1]}}$  
  is surjective, projective, and birational.
\item 
 Further, a scheme in the second row is a smooth open subset of the scheme in first row,
containing the one in the third row, correspondingly.
\end{itemize}

To explain the fourth and fifth rows of the  diagram,
we return to the fixed chart $\bU$. This is the affine space which comes equipped with
the local free variables 
$$\var_{\bU}:=\{x_\uu \mid \uu \in \II_{3,n} \- \{\um\}\}.$$
Let $\Ga$ be an \emph{\it arbitrary}  subset of $\var_{\bU}$
 and let $Z_\Ga$ be the closed affine 
subscheme of $\bU$ defined by the ideal generated by
 all the elements of $\Ga$ together with all the de-homogenized $\um$-primary $\pl$ relations $\bF$
 with $\bF \in \sF$.   
When $\Ga$ is determined by a matroid, then 
$Z_\Ga$ is the $\Ga$-scheme, having written as $Y_\Ga$,
as mentioned in the beginning of this introduction.
The precise relation between 
a given  matroid Schubert cell and its corresponding $\Ga$-scheme
is given in \eqref{ud=Ga}.

 Our goal is to resolve the $\Ga$-scheme $Z_\Ga$ when it is integral and singular.
But, it is worth to emphasize again: we do not fix any particular 
$\Ga$-scheme $Z_\Ga$ and treat it individually, instead,
we treat them altogether, \emph{{\it universally, once and for all.}}

 Let $\Ga$ be any subset of $\var_{\bU}$. Assume that 
 $Z_\Ga$ is integral. Then, starting from $Z_\Ga$, step by step,
  via induction within every of the sequential $\vt$-, $\wp$-, and
  $\ell$-blowup, we are able to construct
 the third and fourth rows in the  diagram  \eqref{theDiagram} such that 
 
\begin{itemize} 
\item every closed subscheme in the fourth row, $Z_{\sF_{[j]},\Ga}$, respectively
$\tZ_{{\hbar}}$, {\it admits explicit local defining equations}
in any admissible affine chart of a smooth open subset,
containing  $\sV_{\sF_{[j]},\Ga}$, respectively,
$\tsV_{{\hbar}}$, of the corresponding scheme in the first row;
 \item 
every closed subscheme in the fifth row $Z^\dagger_{\sF_{[j]},\Ga}$, respectively,
$\tZ^\dagger_{{\hbar}}$, is an {\it irreducible component} of its
corresponding subscheme $Z_{\sF_{[j]},\Ga}$, respectively, $\tZ_{{\hbar}}$, such that  the induced 
morphism $\hbox{$Z^\dagger_{\sF_{[j]},\Ga} \lra Z_\Ga$, respectively,
$\tZ^\dagger_{{\hbar}}\lra Z_\Ga$}$
is surjective, projective, and birational.
\item  
 the left-most   $\tZ_{\ell, \Ga} $ is smooth; so is $\tZ^\dagger_{\ell, \Ga}$,
now a connected component of $\tZ_{\ell, \Ga} $.
\end{itemize} 
 
More specifically, 
 \begin{itemize} 
\item 
The closed subscheme $Z_{\sF_{[j]},\Ga}^\dagger$, 
called an $\sF$-transform of $Z_\Ga$,
  is constructed in Lemma \ref{wp-transform-sVk-Ga};
\item 
  The closed subscheme $Z_{\vt_{[j]},\Ga}^\dagger$, called a $\vt$-transform of $Z_\Ga$,
  is constructed in Lemma \ref{vt-transform-k};
\item 
 The closed subscheme $\tZ_{(\wp_{(k\tau)}\fr_\mu\fs_{h}),\Ga}^\dagger$, 
called a $\wp$-transform of $Z_\Ga$,
 is constructed in Lemma \ref{wp/ell-transform-ktauh};
 \item  The closed subscheme $\tZ_{\ell_k,\Ga}$, called an $\ell$-transform of $Z_\Ga$,
  is also constructed in Lemma \ref{wp/ell-transform-ktauh}.
 \end{itemize}
 
($\di$) {\it  The above
echoes and reinforces our key insight stated earlier: Rather than focusing on individual singularity types, we instead perform a blow-up of the ambient 
universal space $\sV$ to obtain their birational transforms simultaneously. Crucially, during this process, there is no need to pause at any intermediate stage to check whether the singularities of any individual $Y_\Ga$
 have been resolved or improved, thus eliminating the need 
to introduce any invariants measuring singularities.
}

\subsection{Making conclusions on smoothness} $\ $

{\it
In this article, a scheme $X$ is smooth if it is a disjoint union of finitely many 
connected smooth schemes of
possibly various dimensions.}

Our main theorem on the Grassmannian is

\begin{thm}\label{main2:intro} 
{\rm 
(Universal characteristic-free resolution of singularity types)}
Let $\FF$ be either $\QQ$ or a finite field with $p$ elements where
$p$ is a prime number.
Let $\Ga$ be any subset of $\var_{\bU}$.
Assume that $Z_\Ga$ is integral. 
Let $\tZ_{\ell,\Ga}$ be  the $\ell$-transform of $Z_\Ga$ in $\tsV_{\ell}$.
Then,   $\tZ_{\ell,\Ga}$ is smooth over $\FF$.
In particular, the induced morphism $\tZ^\dagger_{\ell,\Ga} \to Z_\Ga$ is a resolution over $\FF$,
provided that $Z_\Ga$ is singular.
\end{thm}


The proof of Theorem \ref{main2:intro},
=Theorem \ref{main-thm} $+$ Theorem \ref{cor:main},
 is based upon the explicit  description of
 the governing binomials and linearized $\pl$ defining equations of $ \tZ_{\ell,\Ga}$
 given in Corollary \ref{ell-transform-up}.

First, when $\Ga=\emptyset$, we have $\tZ_{\ell,\Ga}=\tsV_\ell$.
Using  the governing binomials and linearized $\pl$ defining equations of $\tsV_\ell$,
by computing and estimating their Jacobians,
we show that $\tsV_\ell$ is smooth. This has an important implication that all
the relations in $\cB^\frb$ now become dependent and hence can be discarded.
Recall here that after the completion of $\vt$-blowups, 
the relations in $\cB^\ngv$ already become dependent and are discarded.
Then, using only the governing binomials and linearized $\pl$ relations,
we proceed to compute the Jacobian of any given $\tZ_{\ell,\Ga}$ and conclude
that it is smooth, hence so is $\tZ_{\ell,\Ga}^\dagger$, now a connected
component of $\tZ_{\ell,\Ga}$.

The detailed calculation and careful analysis on the Jacobian matrices of these equations 
occupy the entire \S \ref{main-statement}.


Theorem \ref{main:intro} is  obtained by applying Theorem \ref{main2:intro},
combining with Lafforgue's version of Mn\"ev's unversality theorem
(Theorems \ref{Mn-La} and \ref{Mn-La-Gr}), provided that $X$ is  defined over $\ZZ$.

For a singular affine or projective variety $X$ over a general perfect field $\kk$, 
we spread it out and deduce that $X/\kk$ admits a resolution as well.
The details are written in Part II.

 


We learned that Hironaka posted a preprint on resolution of singularities 
in positive characteristics \cite{Hironaka17}. 
In spite of the current article, the author is not in a position to survey the topics of
resolution of singularities, not even very briefly.
We refer to Koll\'ar's book \cite{Kollar} for an extensive list of references on resolution of singularities.
 There have been some recent progresses since the book \cite{Kollar}:
 risking inadvertently omitting some other's works, let us just mention a few recent ones
\cite{ATW},  \cite{McG}, and \cite{Temkin}.

\medskip
The construction of the $\sV$-model of 
$\Gr^{3,E}$ in this paper was parallel to that of the $\fG$-family
of Grassmannian appeared in \cite{Hu2025}.

The current article represents the first eight sections of \cite{Hu2022}.
Part II will follow.

The author thanks J\'anos Koll\'ar and Chenyang Xu
 for the suggestion to write 
a summary  section, \S \ref{tour}, to lead the reader a quick tour through the paper.

He very much thanks Laurent Lafforgue for several very kind suggestions,
including connecting him with Lean,  and for sharing a general question. 

He specially  thanks Caucher Birkar,  also Mircea Mustață,
James McKernan and Ravi Vakil for invitation to speak in workshop and
seminars, and for helpful correspondences. 

He especially thanks Caucher Birkar for  kind support and 
communications about 
some proofs in the last chapter in an earlier version, and
Bingyi Chen for spotting a mistake in an earlier version.

\bigskip 

\centerline {A List of Fixed Notations Used Throughout}
\medskip

\smallskip\noindent $[h]$: the set of all  integers from 1 to $h$, $\{1, 2 \cdots ,h \}.$

\noindent 
$\II_{3,n}$: the set of all sequences of integers $\{(1\le u_1 < \cdots < u_d\le n) \}.$

\noindent
$\PP(\wedge^3 E)$: the projective space with $\pl$ coordinates 
$p_\ui, \ui \in \II_{3,n}$. 




\noindent
$I_\wp$: the ideal of $\kk[p_\ui]_{\ui \in \II_{3,n}}$ generated by all $\pl$ relations.

\noindent
$I_\whwp$: the ideal of $\Gr^{3, E}$ in $\PP(\wedge^3 E)$.

\noindent
$\bU$: the affine chart of $\PP(\wedge^3 E)$ defined by $p_\um \ne 0$ for some fixed $\um \in \II_{3,n}$.

\noindent
$\sF$: the set of $\um$-primary $\pl$ equations.

 \noindent
$\up:= {n \choose 3} -1 - 3(n-3)$: the cardinality of $\sF$;

\noindent
$\fV$:  a standard affine chart of an ambient smooth scheme; 
 
 \noindent 
$\cB^\gov$:  the set of all governing binomial relations; 

 \noindent 
$\cB^\ngv$:  the set of all  non-governing binomial relations; 



 \noindent 
$\cB^{\fr\fb}$:  the set of all $\fr\fb$-irreducible binomial relations. 

 \noindent 
$\cB$:   $\cB^\gov \sqcup \cB^\ngv \sqcup \cB^{\fb}$;

\noindent
$L_{\sF}$: the set of all linearized $\um$-primary $\pl$ equations.



\noindent $\Ga$: a subset of $\var_\bU$.

\noindent $A \- a$: shorthand for $A\-\{a\}$ where $A$ is a finite set and $a \in A$.

\noindent $|A|$: the cardinality of a finite set $A$.

\noindent $\kk$: a fixed perfect field.

\section{A Quick Tour: the main idea and approach}\label{tour}

{\it While skippable, we recommend at least browsing this section first.}


\subsection{ A detour to $\Gr^{3,E}$ via Mn\"ev's universality}  $\;$

By Mn\"ev's universality, any (singular) affine variety 
over $\ZZ$ admits an open immersion
 in the quotient of a matroid Schubert cell of 
the Grassmannian $\Gr^{3,E}$ of three dimensional linear subspaces in a vector space $E$, up to smooth morphism.

For readers' convenience, we review
Lafforgue's  version of Mn\"ev's universality in \S \ref{universality}.

Consider the $\pl$ embedding $\Gr^{3,E} \subset \PP(\wedge^3 E)$ 
with $\pl$ coordinates $p_{ijk}$.
A  matroid Schubert cell of $\Gr^{3,E}$ is a nonempty intersection of codimension one Schubert cells
of  $\Gr^{3,E}$; it corresponds to  a matroid $\ud$ of rank 3 on the set $[n]$.
Any Schubert divisor is defined by $p_{ijk} =0$ for some $(ijk)$. Thus, 
a  matroid Schubert cell  $\Gr^{3,E}_\ud$ of the matroid $\ud$
 is an open subset of the closed subscheme $\overline{Z}_\Ga$ of
$\Gr^{3,E}$ defined by $\{p_{ijk}=0 \mid p_{ijk} \in \Ga\}$ for some subset $\Ga$ of all $\pl$ variables.
The  matroid Schubert cell  must lie in an affine chart $(p_\um \ne 0)$ for some $\um \in \II_{3,n}$.
Thus,  $\Gr^{3,E}_\ud$  is an open subset of a closed subscheme 
of  $\Gr^{3,E} \cap (p_\um \ne 0)$ of
the following form 
$$Z_\Ga=\{ p_{ijk} =0 \mid p_{ijk} \in \Ga\} \cap \Gr^{3,E} \cap (p_\um \ne 0).$$ 
This is a closed affine subscheme of the affine chart $(p_\um \ne 0)$.
We aim to resolve $Z_\Ga$, hence also the  matroid Schubert cell  
$Z_\Ga^\circ:=\Gr^{3,E}_\ud$, when both
are integral and singular.


\subsection{Minimal set of $\pl$ relations for the chart $(p_\um \ne 0)$}  $\; $
 Up to permutation, we may assume that $\um=(123)$ and the chart is
$$\bU:=(p_{123} \ne 0).$$
We write the de-homogenized coordinates of $\bU$ as
$$\{x_{abc} \mid (abc) \in \II_{3,n} \- \{(123) \}.$$

As a closed subscheme of the affine space $\bU,$
$Z_\Ga$ is defined by
$$\{ x_{ijk} =0, \;\; \bF=0 \; \mid \; x_{ijk} \in \Ga\},$$
where $\bF$ rans over all de-homogenized $\pl$ relations.
We need to pin down some explicit $\pl$ relations to form 
a minimal set of generators of the ideal of $\Gr^{3,E} \cap \bU$ in
the affine space $\bU$.

They are of the following forms:
\begin{eqnarray}\label{tour: all primary pl}
\hbox{de-homogenized primary $\pl$ relations for $\bU \cap \Gr^{3,E}$:}  \\
\bF_{(123),1uv}=x_{1uv}-x_{12u}x_{13v} + x_{13u}x_{12v}, \nonumber
\label{rk0-1uv}\\
\bF_{(123),2uv}=x_{2uv}-x_{12u}x_{23v} + x_{23u}x_{12v}, \nonumber
\label{rk0-2uv} \\
\bF_{(123),3uv}=x_{3uv}-x_{13u}x_{23v} + x_{23u}x_{13v} , \nonumber
\label{rk0-3uv}\\
\bF_{(123),abc}=x_{abc}-x_{12a}x_{3bc} + x_{13a}x_{2bc} -x_{23a}x_{1bc}, \label{rk1-abc} \nonumber
\end{eqnarray}
where $u < v \in [n]\-\{1,2,3\}$ and $a<b<c \in [n]\-\{1,2,3\}$.
Here, $[n]=\{1,\cdots,n\}$.

In a nutshell, we have the set of de-homogenized $\pl$ relations
\begin{eqnarray}\label{rk0-and-1}
\sF=\{\bF_{(123),iuv},  \; 1\le i\le 3; \;\; \bF_{(123),abc}\}
\end{eqnarray}
and they generated the ideal of $\Gr^{3,E} \cap \bU$ in $\bU$.
Every relation of $\sF$ is called $\um$-primary. Here, $\um=(123)$.



\subsection{Especially organized equations 
of $\Ga$-schemes and arbitrary singularities} $\ $

Hence, as a closed subscheme of the affine space $\bU,$
$Z_\Ga$ is defined by
\begin{eqnarray}\label{normal-form}
Z_\Ga=\{ x_\uu =0, \;\; \bF_{(123),iuv},  \; 1\le i\le 3, \;\; \bF_{(123),abc}  \mid x_\uu \in \Ga\},
\end{eqnarray}
for all $u < v \in [n]\-\{1,2,3\}$ and $a<b<c \in [n]\-\{1,2,3\}$.
The  matroid Schubert cell $Z_\Ga^\circ$ in $Z_\Ga$ is characterized by
$x_\uv \ne 0$ for any $x_\uv \notin \Ga$ (see Proposition \ref{to-Ga} and \eqref{ud=Ga}.)

Upon setting $x_\uu =0$ with $\uu \in \Ga$, we obtain 
the affine coordinate subspace $$\bU_\Ga \subset \bU$$ such that $Z_\Ga$,
 as a closed subscheme of
 the affine subspace $\bU_\Ga$, is defined by 
 \begin{eqnarray}\label{reduced-normal-form}
\{\bF_{(123),iuv}|_\Ga, \; 1\le i\le 3;  \;\;  \bF_{(123),abc}|_\Ga\},
\end{eqnarray}
where $\bF|_\Ga$ denotes the restriction of $\bF$ to the affine subspace $\bU_\Ga$.
These are in general truncated  $\pl$ equations, some of which  may be identically zero.

($\di$) {\it  One may view \eqref{normal-form} as the normal form of singularities
and \eqref{reduced-normal-form} as the reduced normal form of singularities.
These provide standardized equations of singularities over $\ZZ$, up to smooth morphisms.

In other words, while singularities themselves can be arbitrary, it is remarkable, and perhaps underappreciated, that their defining equations admit such a \emph{ strikingly organized} \emph{universal} presentation, modulo smooth morphisms.}

We  {\it do not} analyse singularities of $Z_\Ga$ nor do we even
focus on any individual $Z_\Ga$.

But, we make some heuristic remarks.
 By the normal form of singularities \eqref{normal-form},
the $\Ga$-scheme $Z_\Ga$ is the intersection
the affine chart $\bU \cap\Gr^{3,E}$ 
of the Grassmannian $\Gr^{3,E}$
with the hyperplanes $(x_\uu =0)$ for all $x_\uu \in \Ga$. 
The intersections of these
coordinate hyperplanes with the chart $\bU\cap\Gr^{3,E}$ of the Grassmannian $\Gr^{3,E}$ are arbitrary, according to Mn\"ev's universality. 
We may view $\bU$ (allowing 
$\Gr^{3,E}$ to vary) as a universe that contains arbitrary singularities.
Hence, intuitively, we need to birationally change
 the universe $``$along these intersections" 
 so that eventually in the new universe, $``$they" re-intersect properly. 




 To achieve this, it is more workable if we can
put all the singularities in a different universe $\sV$, 
birationally modified from the chart $\bU \cap \Gr^{3,E}$,
  so that in the {\it new model $\sV$},  all the terms of 
the $\pl$ relations of \eqref{tour: all primary pl}
can be \emph{\it separated}. 
(Years had been passed, or wasted in a way, before we {\it returned} to this correct approach.)

\subsection{Separating the terms of $\pl$ relations} $\ $

Motivated by a parallel construction in \cite{Hu2025}, we establish 
a local model  $\sV$, birational  to the chart $\bU \cap \Gr^{3,E}$,
such  that in a {\it specific} set of defining binomial equations of $\sV$, {\it all the terms} of 
the above $\pl$ relations are separated.\footnote{The same idea and construction
can be applied to general polynomials. In geometric terms, any affine scheme admits
a birational model that is cut-out from a toric variety by  linear hyperplanes.}

To elaborate the above, we introduce the projective space 
$$\PP_F$$ for each and
 every $\pl$ relation of \eqref{tour: all primary pl}
$$\bF=\sum_{s \in S_F} \sgn (s) x_{\uu_s}x_{\uv_s}$$ with
$$[x_{(\uu_s,\uv_s)}]_{s \in S_F}$$ 
as its homogeneous coordinates.

We then let $\sV$ be the closure of the graph of 
the rational map $\Theta_{[\up],\Gr}$ of \eqref{this-theta-intro}.
By calculating the multi-homogeneous kernel of the homomorphism
\begin{equation}\label{vi-hom} \vi: \kk[(x_\uw);(x_{(\uu,\uv)})] \lra \kk[x_\uw]
\end{equation}
$$x_{(\uu,\uv)} \to x_\uu x_\uv,$$
we determine a set of defining relations of $\sV$ as a closed subscheme
of the smooth ambient scheme $$\sR:=\bU \times \prod_{\bF \in \sF} \PP_F.$$ 
These defining relations, among many others, 
include the following binomials
\begin{eqnarray}\label{mainB-tour}
x_{1uv}x_{(12u,13v)} - x_{12u}x_{13v} x_{(123,1uv)}, \; x_{1uv}x_{(13u,12v)}- x_{13u}x_{12v}x_{(123,1uv)}; \\
x_{2uv}x_{(12u,23v)} -x_{12u}x_{23v} x_{(123,2uv)}, \;  x_{2uv}x_{(23u,12v)}-x_{23u}x_{12v} x_{(123,2uv)}; \nonumber \\
x_{3uv}x_{(13u,23v)} -x_{13u}x_{23v}x_{(123,3uv)}, \;  x_{3uv}x_{(23u,13v)} -x_{23u}x_{12v}x_{(123,3uv)}; \nonumber\\
 x_{abc}x_{(12a,3bc)}-x_{12a}x_{3bc} x_{(123,abc)},\;
 x_{abc}x_{(13a,2bc)}-x_{13a}x_{2bc} x_{(123,abc)},\nonumber \\
x_{abc}x_{(23a,1bc)} -x_{23a}x_{1bc}x_{(123,abc)}. \nonumber
\end{eqnarray}
We see that  the terms of all the $\um$-primary $\pl$ relations 
of \eqref{rk0-and-1} are separated into the two terms of the above binomials.
Also, we note here that these binomial relations, as written above,
 are in linear order.

The defining relations of $\sV$ in $\sR$ 
also include the linearized $\pl$ relations as in \eqref{eq2-intro}:
\begin{equation}\label{linPl-tour}
L_F=\sum_{s \in S_F} \sgn (s) x_{(\uu_s,\uv_s)}, \; \; \forall \; \bF \in \sF.
\end{equation}
The set of all  linearized $\pl$ relations is denoted by $L_{\sF}$.

To distinguish two kinds of variables and divisors
 in $\sR$, we call 
\begin{itemize}
\item we call $x_\uu$ (e.g., $x_{12u}$)
a $\vp$-variable and $X_\uu=(x_\uu=0)$  a $\vp$-divisor;
\item  we call $x_{(\uu,\uv)}$ (e.g., $x_{(12u,13v)}$) a $\vr$-variable and
$X_{(\uu,\uv)}=(x_{(\uu,\uv)}=0)$ a $\vr$-divisor.
\end{itemize}
In addition, for any $\bF \in \sF$, we have the divisor of $\sR$ 
$$D_{L_F} := (L_F=0),$$
called an $\fL$-divisor.

The binomial relations of \eqref{mainB-tour}
together with the linearized $\pl$ relations
\eqref{linPl-tour}
are called {\it governing} relations.
Our blowups are constructed solely in accordance with the governing relations.

There are many other {\it extra} defining relations,
but they do not play any roles in {\it designing} the blowups.

 All the $\Ga$-schemes $Z_\Ga$ admit  birational transforms in the singular model
$\sV$. We still {\it do not} analyse the singularities of these transforms. But, we 
make a quick observation: 
when all the terms of  some of the binomials in \eqref{mainB-tour}
vanish at a point of the transform of a $\Ga$-scheme, then a singularity is likely to occur.

Thus, {\it as the first steps,} we would like to eliminate
all the zero factors from all the terms of 
the binomials in \eqref{mainB-tour}. 
As it turns out, after years of trial and error, 
eliminating all the zero factors from these binomial relations 
 successfully guides us onto the correct path toward our destination.

{\it In a way, the geometric intuition behind the above sufficiency is as follows.
 The equations of \eqref{mainB-tour} alone together with $L_{\sF}$ only define
a reducible closed scheme, in general. The roles of other extra relations (to be discussed soon)
are to  pin down
its main component $\sV$. 
As the process of eliminating zero factors goes, a process of some specific blowups, 
all the boundary components are eventually blown out of existence, making  
the proper transforms of \eqref{mainB-tour} together with the linearized
$\pl$ relations  generate the ideal of the final blowup scheme $\tsV_\ell$ of $\sV$,
locally on all affine charts.}

For this reason, we call the 
binomial equations of  \eqref{mainB-tour} the {\it governing} binomials.
The set of  governing binomials is denoted $\cB^\gov$. 
The set $\cB^\gov$ is equipped with a carefully chosen total ordering 
(see \eqref{indexing-Bmn}).

The defining relations of $\sV$ in $\sR$ also include many other binomials: we classify them as {\it non-governing} binomials 
 (see Definition \ref{gov-ngov-bi}), 
 $\frb$-binomials  (see Definition \ref{defn:frb}).
The set of  non-governing binomials is denoted $\cB^\ngv$;
the set of  $\frb$-binomials is denoted $\cB^\frb$.

Together, the equations in the following sets
\begin{equation}\label{tour: all relations}
\cB^\gov, \; \cB^\ngv, \; \cB^\frb, \; L_{\sF}
\end{equation}
define the scheme $\sV$ in the smooth ambient scheme $\sR$.
See Corollary \ref{eq-tA-for-sV}.

\subsection{The process of eliminating zero factors 
of governing binomials} $\ $

To eliminate zero factors, for example, from the governing binomial 
$$x_{1uv}x_{(12u,13v)} - x_{12u}x_{13v} x_{(123,1uv)},$$
we mean to select two zero factors,  one from each term, say, 
$x_{(12u,13v)}$ and $x_{12u}$, and then blowup along the locus defined by
$$(x_{(12u,13v)}=0) \cap (x_{12u}=0).$$
But, we need to proceed under carefully chosen orders.

Thus, we need to endow a linear order on the set
 $\cB^\gov$. Note that every relation of $\cB^\gov$ as in \eqref{mainB-tour}
is associated with a $\pl$ relation of \eqref{rk0-and-1}
$$\sF=\{\bF_{(123),iuv},  \; 1\le i\le 3; \;\; \bF_{(123),abc}\}.$$


We define the linear order on $\cB^\gov$ as follows.
Take any two $B, B' \in \cB^\gov$.
\begin{itemize}
\item Suppose $B$ is associated with 
$\bF_{(123),iuv}$ for some $1\le i\le 3$ and
$B'$ is associated with $\bF_{(123),abc}$. Then, $B < B'$.
\item
Suppose $B$ is associated with $\bF_{(123),iuv}$ for some $1\le i\le 3$, 
and $B'$  is associated with $\bF_{(123),ju'v'}$ for some $1 \le j \le 3$.
Then, $B < B'$ if $(uvi)<(u'v'j)$  lexicographically.
\item  
Suppose $B$ is associated with $\bF_{(123),abc}$ for some $1\le i\le 3$, 
and $B'$  is associated with $\bF_{(123),a'b'c'}$.
Then $B <B'$ if $(abc)<(a'b'c')$  lexicographically.
\end{itemize}

Given the above order, fix and consider each governing binomial,
again for example,
$$x_{1uv}x_{(12u,13v)} - x_{12u}x_{13v} x_{(123,1uv)},$$
we find that it is important to 
select the leading factors and 
blow up along the center
$$ (x_{1uv}=0) \cap (x_{(123,1uv)}=0)$$
first. This leads us to the first sequential blowups:  $\vt$-blowups.





\subsubsection{ On $\vt$-blowups}   $\ $

 From the governing binomials of $\cB^\gov$ 
\eqref{mainB-tour}, 
we select the following closed centers
 $$Z_{iuv}: (x_{iuv}=0) \cap (x_{(123, iuv)}=0),  i \in [3];$$ 
 $$ Z_{abc}: (x_{abc}=0) \cap (x_{(123, abc)}=0), a \ne b \ne c \in [n] \- [3].$$
 We let $\cZ_\vt=\{Z_{iuv}; Z_{abc}\}$ 
be the set of these centers.

Similar to the order of $\cB^\gov$, we define the following.
\begin{itemize}
\item   $Z_{iuv} < Z_{abc}$;
\item  $Z_{iuv} < Z_{ju'v'}$ if $(uvi) < (u'v'j)$ lexicographically.
\item $Z_{abc} < Z_{a'b'c'}$ if $(abc) < (a'b'c')$ lexicographically.
\end{itemize}
  This way,  the set $\cZ_\vt$ is equipped with a total order.

We then blow up $\sR$ along (the proper transforms of) the centers in $\cZ_\vt$, in the above order.
This gives rise to the sequence \eqref{vt-sequence-intro} in the introduction
$$\tsR_{\vt}:=\tsR_{\vt_{[\up]}}  \lra \cdots \lra \tsR_{\vt_{[k]}} 
\lra \tsR_{\vt_{[k-1]}} \lra \cdots \lra \tsR_{\vt_{[0]}}.$$
Each arrow in this sequence is a smooth blowup.

For any $k \in [\up]$, we let $\tsV_{\vt_{[k]}} \subset  \tsR_{\vt_{[k]}}$ 
be the proper transform of $\sV_\um$ in $\tsR_{\vt_{[k]}}$. We then set
$\tsV_{\vt}=\tsV_{\vt_{[\up]}}$ and $ \tsR_{\vt}= \tsR_{\vt_{[\up]}}$.


By $\vt$-blowups, we achieve the following,
from the obvious to the less-trivial.

\begin{itemize}
\item
Locally on affine charts, at least one of the
 factors as displayed  in  the centers of $\cZ_\vt$
are eliminated.
\item 
 $\vt$-blowups also make the proper transforms of the non-governing binomial equations become
 dependent on the proper transforms of the governing binomial equations,
  locally on affine charts.
 Thus, upon completing $\vt$-blowups, we can discard all the non-governing binomials $\cB^\ngv$ from consideration.
\item In addition, it also leads to 
 the conclusion $\tsV_{\vt} \cap X_{\vt, (\um, \uu_k)} = \emptyset$ for all $k \in [\up]$
 where $X_{\vt, (\um, \uu_k)}$ is the proper transform of the $\vr$-divisor 
$X_{(\um, \uu_k)}=(x_{(\um, \uu_k)}=0)$. Put it differently, 
the factor $x_{(\um, \uu_k)}$, possibly zero somewhere before the $\vt$-blowup,
now that $``$ zero factor $"$ is  eliminated once and for all,  locally on all affine charts, upon completing $\vt$-blowups.
 \end{itemize}

\subsubsection{ Blocks of governing relations}\label{blocks of gov} $\ $

Here, we continue the process of eliminating zero factors
of the proper transforms of the governing binomials.

Recall that every relation of $\cB^\gov$ as in \eqref{mainB-tour}
is associated with a $\pl$ relation of 
$$\sF=\{\bF_{(123),iuv},  \; 1\le i\le 3; \;\; \bF_{(123),abc}\}.$$
A $\pl$ relation is also regarded as a governing relation.

The set $\sF$ admits a linear order as follows.
\begin{itemize}
\item   $\bF_{(123),iuv} < F_{(123), abc}$;
\item  $\bF_{(123),iuv}< \bF_{(123),ju'v'}$ if $(uvi) < (u'v'j)$ lexicographically.
\item $\bF_{(123),abc} < \bF_{(123),a'b'c'}$ if $(abc) < (a'b'c')$ lexicographically.
\end{itemize}
Given any $\pl$ relation $\bF \in \sF$, we let
$\fG_{F}$ be the set of all governing binomials associated with
$F$ together with $\bF$ itself, called the block of governing relations
with respect to $F$.
This way, if we write $\sF$ as
$$\bF_1 < \cdots < \bF_\up,$$
then, we can write
$$\fG_{\bF_1} < \cdots <\fG_{\bF_\up}.$$
The $\wp$-blowups and $\ell$-blowups will be 
performed alternatively, block by block.

For every of the relations in \eqref{mainB-tour},
its first term is called the {\it leading term} whose $\pl$ variable
(e.g., $x_{1uv}$, $x_{2uv}$, $x_{3uv}$, and $x_{abc}$)
is called its leading $\vp$-variable.

{\it Hugely importantly, } observe that we have

\noindent
($\di$) the leading $\vp$-variable {\it does not appear} in any 
relation of \eqref{mainB-tour} in any earlier block
under the  linear order mentioned above.

\noindent
($\di$) the  $\vr$-variable in the leading term
(e.g., $x_{(12u,13v)}$, $x_{(12u,23v)}$, $x_{(12a,3bc)}$, etc.) 
 {\it uniquely appear} in its own binomial relation and its corresponding 
linearized $\pl$ relation.

The above two properties assure to achieve:

\noindent
 (1) the desirable square-freeness
is preserved under any intermediate blowup; 

\noindent
(2) the desirable Jacobian matrices of the governing relations
in their proper transforms under the final blowup.

\subsubsection{On $\wp$-blowups and $\ell$-blowups in a block $(\fG_F)$} $\ $

We proceed with the first binomial in the first block $(\fG_{F_1})$.

The first governing binomial equation of \eqref{mainB-tour} is 
$$B_{145, 1}:  x_{(124,135)} x_{145} - x_{(123,145)} x_{124}x_{135} .$$
The proper transforms of all the variables of $B_{145}$ may assume zero value on $\tsV_{\vt}$ 
somewhere.

\begin{itemize}
\item For each and every term of $B_{145, 1}$, we pick a $``$zero$"$ 
 factor to form a pair. For example, $(x_{145}, x_{124})$ is such a pair
of $B_{145, 1}$.
 Such a pair is called a $\wp$-set with respect to $B_{145, 1}$. 
 The common vanishing locus of the variables 
in a $\wp$-set gives rise to  a $\wp$-center.
\item Before we can blow up these $\wp$-centers, we need to order them.
The order is somewhat subtle. But, the general rule is
 that we let $\wp$-sets having $\vr$-variables go last
and those having $\vp$-variable go as the second last.
 In other words, we first declare $\vr$-variables
are the largest, $\vp$-variables are the second largest, 
and then compare the $\wp$-sets as pairs
lexicographically. 
Then, we can  blow up $\tsR_\vt$ along (the proper transforms) of 
 these $\wp$-centers, starting from the smallest one.
\end{itemize}
 
We then move on to the next governing binomial relation in the same block
 $(\fG_{F_1})$:
$$B_{145, 2}:  x_{(125,134)} x_{145} - x_{(123,145)} x_{125}x_{134} .$$
Notice here that the variable $x_{145}$ may become an exceptional variable
or acquires one due to the previous $\wp$-blowups. Hence, $B_{145, 2}$ should have more
$\wp$-sets.
We declare these exceptional variables to be the smallest ones, and wthin them,
we order them by reversing the order of occurrence.
We can then select pairs of variables, one from each term, define $\wp$-centers, 
make an order on them,
and repeat the above.

This way, we complete our $\wp$-blowups with respect to the block of relations of
$\cB^\gov_{F_1}$ and obtain $\tsR_{\wp_1}$. 

We now turn to the last equation in the block $(\fG_{F_1}$),
the linearized $\pl$ relation
$$L_{F_1}=  x_{(123,145)}- x_{(124,135)}+x_{(125,134)}.$$
Due to the $\vt$-blowup with respect to $F_1$, there exists
the exceptional divisor $E_{\vt, (145)}$ created by that blowup.
It corresponds to a free variable $\de_{\fV, (123,145)}$ on some chart
$\fV$
whose vanishing locus is the intersection of 
the proper transform of $E_{\vt, (145)}$ 
with that chart. And, on that chart $\fV$, the proper transform of $L_{F_1}$
takes of the form
$$L_{\fV, F_1}=  \de_{\fV, (123,145)} - \cdots.$$
We can blow up $\tsR_{\wp_1}$ along 
the intersection
 $$\hbox{the proper transform of $(L_{F_1}=0)$
$\cap$ the proper transform of  $E_{\vt,145}$},$$
which is equal to, locally on the chart $\fV$,
$$(L_{\fV, F_1}=0) \cap (\de_{\fV, (123,145)}=0),$$
to obtain $$\tsR_{\ell_1} \lra \tsR_{\wp_1}.$$

This completes all the desired blowups for the block of equations
$$\fG_{F_1}=\{B_{145, 1}, \; B_{145, 2}, \; L_{F_1}\}.$$

We then move on to the next bock of relations
$$\fG_{F_2}=\{B_{245, 1}, \; B_{245, 2}, \; L_{F_2}\},$$
and repeat all the above.

And  then,  to the next block 
 $$\fG_{F_3}=\{B_{345,1}, B_{345,2}, L_{F_3}\},$$
 repeat all the above, and so on. 
 
 This gives rise to the sequential blowups \eqref{wp-sequence-intro}
 and  \eqref{ell-sequence-intro} in the introduction
 $$\tsR_{\ell_k} \lra \tsR_{\wp_k}  \to \cdots \to
\tsR_{(\wp_{(k\tau)}\fr_\mu\fs_{h})} \to \tsR_{(\wp_{(k\tau)}\fr_\mu\fs_{h-1})} \to \cdots \to \tsR_{\ell_{k-1}},$$
coming with the induced blowups
$$\tsV_{\ell_k} \lra \tsV_{\wp_k}  \to \cdots \to
\tsV_{(\wp_{(k\tau)}\fr_\mu\fs_{h})} \to \tsV_{(\wp_{(k\tau)}\fr_\mu\fs_{h-1})} \to \cdots \to \tsV_{\ell_{k-1}}.$$
Each scheme in first sequence above has a smooth open subset
$\tsR^\circ_{\ell_k}$ or $\tsR^\circ_{\wp_k}$ or  $\tsR^\circ_{(\wp_{(k\tau)}\fr_\mu\fs_{h})}$
containing $\tsV_{\ell_k}$ or  $\tsV_{\wp_k}$ or
$\tsV_{(\wp_{(k\tau)}\fr_\mu\fs_{h})}$, respectively.

\begin{itemize}
\item An intermediate blowup scheme in the above 
is denoted by $\tsR_{(\wp_{(k\tau)}\fr_\mu\fs_h)}$.
 Here $(k\tau)$ is the index of the corresponding governing binomial
 $B_{(k\tau)}$.
\item As the process of $\wp$-blowups goes on, more and more exceptional 
parameters may be acquired and appear in the proper transform of the later governing binomial $B_{(k'\tau')}>B_{(k\tau)}$, 
resulting more pairs of zero factors, hence more corresponding $\wp$-sets
and $\wp$-centers. 
The existence of the index $\fr_\mu$, called \emph{\it \it round $\mu$}, 
is due to the need to deal with the situation when 
an  exceptional parameter with exponent greater than one is accumulated
in the proper transform of the governing binomial $B_{(k\tau)}$ (such a situation does not occur for the
first few governing binomials).
\item The index $\fs_h$ of $(\wp_{(k\tau)}\fr_\mu\fs_h)$, called 
\emph{\it \it step $h$}, simply indicates the corresponding step of the blowup in that round.
 \end{itemize}

When the process of $\wp$-blowups terminates, all the governing binomials terminate, that is, 
all the variables in the proper transforms of the governing binomials are invertible along $\tsV_{\wp_k}$. This is (just) one of the crucial point for computing
the final Jacobian matrices.

Now, we have achieved
our goal to eliminate zero factors of governing binomials.

 \subsubsection{ More on $\ell$-blowups}  $\ $
  
The $\ell$-blowups are the only blowups designed on
linearized $\pl$ relations, while all $\vt$- and $\wp$-blowups are
designed according to governing binomials.

The purpose of these $\ell$-blowups is

\medskip 
$\bullet$ {\sl  eliminating zero factors of the leading terms
of linearized $\pl$ relations.}
\smallskip

  Here, we make some more comments on $\ell$-blowups.
  
For any of the $\pl$ relations of \eqref {rk0-and-1}, either 
$\bF_{(123),iuv}$ for some  $1\le i\le 3$ or  $\bF_{(123),abc}$, we express its linearized $\pl$ relation
as 
$$L_{F_k}: \;\; \sgn(s_{F_k}) x_{(\um, \uu_{F_k})}+\sum_{s \in S_{F_k} \- s_{F_k}} \sgn (s) x_{(\uu_s,\uv_s)}$$
where $\sgn(s_{F_k}) x_{(\um, \uu_{F_k})}$ is the leading term. 
It comes equipped with a divisor $$D_{L_{F_k}}=(L_{F_k} =0)$$ in $\sR.$ 
We let $D_{\wp_k,L_{F_k}}$ be the proper transform of $D_{L_{F_k}}$ in $\tsR_{\wp_k}$.

 After the process of $\vt$-blowups,
locally on some chart $\fV$,  the leading $x_{(\um, \uu_{F_k})}$, can become 
 an exceptional variable $\de_{\fV, (\um, \uu_{F_k})}$
 for the exceptional divisor $E_{\vt_{[k]}}$ created by the corresponding $\vt$-blowup with respect to $F_k$.  In other words, 
the proper transform of $L_{F_k}$ on the chart $\fV$ can take of the form
$$L_{F_k}: \;\; \sgn(s_{F_k}) \de_{\fV, (\um, \uu_{F_k})}+
\sum_{s \in S_{F_k} \- s_{F_k}} \sgn (s) x_{(\uu_s,\uv_s)}$$
We let $E_{\wp_k, \vt_k}$
 be the proper transform of $E_{\vt_{[k]}}$ in $\tsR_{\wp_k}$. 
 We can then let $$\tsR_{\ell_k} \lra \tsR_{\wp_k}$$
 be the blowup of $\tsR_{\wp_k}$ along the intersection 
$D_{\wp_k, L_{F_k}}\cap E_{\wp_k, \vt_k}$.
 Then, the blowup
 $$\tsR_{\ell_k} \lra \tsR_{\wp_k}$$ will eliminate that $``$zero factor$"$ 
$\de_{\fV, (\um, \uu_{F_k})}$
and bring up a 
 variable $y_{(\um, \uu_{F_k})}$ invertible along $\tsV_{\ell_k}$.
 Geometrically, 
this process {\it separates the two divisors}
 $D_{\wp_k, L_{F_k}}$
 and $E_{\wp_k, \vt_k}$, while $\tsV_{\ell_k}$ is contained in 
the proper transform of $D_{\wp_k, L_{F_k}}$.
  The induced blowup
 $\tsV_{\ell_k} \lra \tsV_{\wp_k}$ is an isomorphism as it is a blowup along
a Cartier divisor. {\it  However, had this blowup  
not performed \emph{\it here}, then 
 that $``$zero factor$"$ would remain around $\tsV_{\wp_k}$ and 
affect or would be affected by the later blowups with unknown 
consequences.}
 
We point out here the $\ell$-blowup with respect to $F$ has to immediately follow the $\wp$-blowups
 with respect to $F$; the order of $\wp$-blowups  with respect to a fixed $\pl$ relation $F$,
 as already mentioned,  may be subtle and are carefully chosen. 
 
{\it  In fact, the birational model $\sV$ of the chart $\bU \cap \Gr^{3,E}$, 
a new universe for singularities, 
also as a (singular) blowup of $\bU \cap \Gr^{3,E}$, 
has to be constructed first, as experience has shown. 
That is to say,  the method of our approach is highly sensitive to (some of)
 the orders used above.}

\smallskip

In the above, the constructions of $\wp$-, and $\ell$-blowups 
are discussed in terms of local free variables
 of the proper transforms of the governing binomials or lineaized $\pl$ relations on local charts.
 In the main text, the constructions of all these blowups, like $\vt$-blowups, are done globally via induction.


Every of the $\vt$-, $\wp$-, and $\ell$-blowups
is the blowup of a smooth ambient space along a codimension two
smooth closed subvariety that is the intersection of two smooth
divisors. Hence, we resolve singularites by separating poor intersections
of divisors pairwise, one pair at a time. This partly explains that
this process is highly inefficient.

 \smallskip
Finally, a final remark for this subsection.
 From the previous discussions, one sees that the process of $\wp$-blowups
 is highly inefficient. This is not a surprise as we treat {\it all singularities} 
 all together, {\it once and for all.}
To provide a concrete example for the whole process,  $\Gr(2,n)$ would miss some main points;
$\Gr(3,6)$ would be too long to include, and also, perhaps not too helpful as far as 
showing (a resolution of) a singularity is concerned.


\subsection{$\Ga$-schemes and their 
$\vt$-, $\wp$-, $\ell$-transforms} $\ $
 

Fix any $\Ga$-scheme $Z_\Ga$, 
considered as a closed subscheme of $\bU \cap \Gr^{3,E}$.
Our goal is to resolve $Z_\Ga$ when it is integral and singular,
even though we do not particularly focus on any of them.

 As in the introduction,
  we have the  instrumental diagram \eqref{theDiagram2}.
\begin{equation}\label{theDiagram2}
 \xymatrix@C-=0.4cm{
  \tsR_{\ell} \ar[r] & \cdots  \ar[r] &  \tsR_{\hbar} \ar[r] &  \tsR_{\hbar} \ar[r] &  \cdots  \ar[r] &   \sR_{\sF_{[j]}}  \ar[r] &  \sR_{\sF_{[j-1]}} \cdots \ar[r] &  \bU \\
  \tsR^\circ_{\ell}\ar @{^{(}->} [u]  \ar[r] & \cdots  \ar[r] &  \tsR^\circ_{\hbar}\ar @{^{(}->} [u]  \ar[r] &  \tsR^\circ_{\hbar} \ar @{^{(}->} [u] \ar[r] &  \cdots  \ar[r] &   \sR_{\sF_{[j]}} \ar @{^{(}->} [u]_{=} \ar[r] &  \sR_{\sF_{[j-1]}} \ar @{^{(}->} [u]_{=} \cdots \ar[r] &  \bU \ar @{^{(}->} [u]_{=}\\
    \tsV_{\ell} \ar @{^{(}->} [u]  \ar[r] & \cdots  \ar[r] &  \tsV_{\hbar}\ar @{^{(}->} [u]   \ar[r] &  \tsV_{\hbar'} \ar @{^{(}->} [u]  \ar[r] &  \cdots    \ar[r] &   \sV_{\sF_{[j]}} \ar @{^{(}->} [u]\ar[r] &  \sV_{\sF_{[j-1]}} \cdots \ar @{^{(}->} [u]  \ar[r] &  \bU \cap \Gr^{3,E}   \ar @{^{(}->} [u]  \\
   \tZ_{\ell, \Ga} \ar @{^{(}->} [u]  \ar[r] & \cdots  \ar[r] &  \tZ_{\hbar,\Ga}\ar @{^{(}->} [u]   \ar[r] &  \tZ_{\hbar',\Ga} \ar @{^{(}->} [u]  \ar[r] &  \cdots    \ar[r] &   Z_{\sF_{[j]},\Ga} \ar @{^{(}->} [u]\ar[r] &  Z_{\sF_{[j-1])},\Ga} \cdots \ar @{^{(}->} [u]  \ar[r] &  Z_\Ga  \ar @{^{(}->} [u]  \\
    \tZ^\dagger_{\ell, \Ga} \ar @{^{(}->} [u]  \ar[r] & \cdots  \ar[r] &  \tZ^\dagger_{\hbar,\Ga}\ar @{^{(}->} [u]   \ar[r] &  \tZ^\dagger_{\hbar',\Ga} \ar @{^{(}->} [u]  \ar[r] &  \cdots    \ar[r] &   Z^\dagger_{\sF_{[j]},\Ga} \ar @{^{(}->} [u]\ar[r] &  Z^\dagger_{\sF_{[j-1])},\Ga} \cdots \ar @{^{(}->} [u]  \ar[r] &  Z_\Ga.  \ar[u]_{=}       }
\end{equation}

The first three rows follow from the above discussion;
we  only need to explain the  fourth and fifth rows.

\begin{itemize}
\item When  $Z_{\sF_{[j-1])}}$ (resp.  $\tZ_{{\hbar}',\Ga}$)
is not contained in the corresponding
  blowup center $Z_{\sF_{[j])}}$ (resp. $\tZ_{{\hbar},\Ga}$) 
is  obtained from  the proper transform
  of $Z_{\sF_{[j-1])}}$ (resp. $\tZ_{{\hbar}',\Ga}$). 
\item
  When  $Z_{\sF_{[j-1])}}$ (resp. $\tZ_{{\hbar}',\Ga}$)
 is contained in the corresponding
  blowup center, then $Z_{\sF_{[j])}}$ (resp. $\tZ_{{\hbar},\Ga}$)
  is obtained from  a rational slice
  of  the total  transform of $Z_{\sF_{[j-1])}}$  
(resp.  $\tZ_{{\hbar}',\Ga}$) under the morphism 
   $ \sV_{\sF_{[j]}} \to  \sV_{\sF_{[j-1]}}$
 (resp. $\tsV_{\hbar} \to \tsV_{{\hbar}'}$). 
\item   Moreover, every $Z_{\sF_{[j])}}$  (resp.  $\tZ_{{\hbar},\Ga}$) admits explicit defining equations 
   over any admissible affine chart of the corresponding smooth  scheme in the second row.
\item Furthermore, in every case, $Z_{\sF_{[j])}}$  (resp.  $\tZ_{{\hbar},\Ga}$)  contains an irreducible
component  $Z^\dagger_{\sF_{[j]},\Ga}$  (resp. 
   $\tZ^\dagger_{{\hbar},\Ga}$) such that it maps onto $Z_\Ga$ 
   projectively and birationally.
 \end{itemize}

\subsection{Smoothness by Jacobian of governing
 relations} $\ $

We are now ready to explain the smoothness of 
 $\tZ_{\ell, \Ga}$ when $Z_\Ga$ is integral.
  We first investigate the smoothness of $\tsV_\ell$ which is a special case 
 of  $\tZ_{\ell, \Ga}$ when 
 $\Ga=\emptyset$.

The question is local. So we focus on any fixed
admissible smooth affine chart of $\fV$ of $\tsR_\ell$.
Corollary \ref{ell-transform-up} provides  
the local defining equations for $\tZ_{\ell, \Ga}$.

As envisioned, we confirm that the scheme $\tsV_\ell$
 is smooth on the chart $\fV$ by  explicitly computing
and carefully analysing 
$${\rm Jac}(\cB^\gov_\fV, L_{\sF, \fV}),$$
 the Jacobian of
{\it the governing binomial relations of $\cB^\gov_\fV$ and linearized $\pl$ relations of $L_{\sF, \fV}$.}
This implies that on the chart $\fV$, 
 the governing binomial relations of $\cB^\gov_\fV$ and the linearized $\pl$ relations 
 of $L_{\fV,\sF}$
together generate the ideal of $\tsV_\ell \cap \fV$. Thus,  as a consequence, 
the $\frb$-binomials of   $\cB^\frb_\fV$ can be discarded from consideration, as well.  Recall here that after the completion of $\vt$-blowups, all relations
in $\cB^\ngv$ become dependent and are already discarded.

Then, a parallel calculation and analysis on the Jacobian of the induced governing binomial relations of
$\cB^\gov_\fV$ and the induced linearized $\pl$ relations of $L_{\fV,\sF}$ for 
$\tZ_{\ell,\Ga}$ imply that $\tZ_{\ell,\Ga}$  is smooth as well, on all charts $\fV$.
In particular, $\tZ^\dagger_{\ell,\Ga}$, now a connected component of $\tZ_{\ell,\Ga}$,  is  smooth, too.

This implies that $\tZ^\dagger_{\ell,\Ga} \lra Z_\Ga$ is a resolution, if $Z_\Ga$ is singular.

The above are done in \S \ref{main-statement}.

\medskip

We end this tour with the following to start the main text.

{ \it  Let $p$ be an arbitrarily fixed prime number.
 Let $\mathbb F$ be either $\QQ$ or a finite field with $p$ elements.
From Section \ref{localization} to Section \ref{main-statement}, 
every scheme considered is defined over $\ZZ$, consequently,
 is defined over $\mathbb F$, and is considered as a scheme over 
 the perfect field $\mathbb F$.}

\section{Primary $\pl$ relations and de-homogenized  $\pl$-ideals}\label{localization}

{\it The purpose of this section is to 
describe a minimal set of $\pl$ relations so that they
generate the $\pl$ ideal  for a given chart. The entire section is 
 elementary, hence could have been placed in the appendix,
 but then it would have interrupted the smooth
flow of reading as the entire approach depends on the materials set up in this section.

Not particularly recommended, however,  if the reader has read \S \ref{tour}, he may skip this section for the first read and only refers back when called upon.

}


Since all the results of this section are elementary,
 many must have already been known. 
 Nonetheless,  since the development in the current section is  instrumental for our approach, some good details are necessary.

All being said, in \S \ref{Gr3n},
the  part of this section directly relevant for the remainder of this paper,  we encapsulate the notations and terms 
to be used throughout.

We make a convention. Let $A$ be a finite set and $a \in A$. Then, we write
$$A \- a := A\-\{a\}.$$
Also, we use $|A|$ to denote the cardinality of the set A.

\subsection{$\pl$ relations} $\ $

Fix any positive integer $n>3$. 
We denote the set $\{1,\cdots, n\}$ by $[n]$.
We let $\II_{3,n}$ be the set of all sequences of distinct integers
 $$\{1\le u_1 < u_2 < u_3\le n \}.$$
An element of $\II_{3,n}$ is frequently written as $\uu=(u_1u_2u_3)$.
We also regard an element of $\II_{3,n}$ as a subset of 3 distinct integers in $[n]$.
For instance, for any $\uu, \um \in \II_{3,n}$, $\uu \- \um$ takes its set-theoretic meaning.
Also, $v \in [n] \- \uu$  if and only if $v \ne u_i$ for all $1\le i\le 3$.

As in the introduction, suppose we have a set of vector spaces, 
$E_1, \cdots, E_n$ such that 
every $E_\alpha$, $1\le \alpha\le n$,  is of dimension 1 over $\kk$ (or, a free module of rank 1 over $\ZZ$), 
and, we let 
$$E:=E_1 \oplus \ldots \oplus E_n.$$ 

For any fixed  integer $1\le d <n$, the Grassmannian, defined by
$$\Gr^{3, E}=\{ F \hookrightarrow E \mid \dim F=3\}, $$
is a projective variety defined over $\ZZ$.

We have the canonical decomposition
$$\wedge^3 E=\bigoplus_{\ui =(i_1i_2 i_3)\in \II_{3,n}} E_{i_1}\otimes E_{i_2} \otimes E_{i_3}.$$
This gives rise to the $\pl$ embedding of the Grassmannian:
$$\Gr^{3, E} \hookrightarrow \PP(\wedge^3 E)=\{(p_\ui)_{\ui \in \II_{3,n}} \in \GG_m 
\backslash (\wedge^3 E \- \{0\} )\},$$
$$F \lra [\wedge^3 F],$$
 where $\GGm$ is the multiplicative group.

The group 
$$\TT=(\GGm)^n/\GG_m,$$ where $\GGm$
 is embedded in $(\GGm)^n$ as the diagonal, acts on $\PP(\wedge^3 E)$ by
 $${\bf t} \cdot p_{\ui} = t_{i_1} t_{i_2} t_{i_3} p_{\ui}$$
where ${\bf t} = (t_1, \cdots , t_n)$ is (a representative of) an element of $\TT$
and $\ui=(i_1 i_2  i_3)$. This action leaves $\Gr^{3, E}$ invariant.
If we have an algebraic group $G$ (over $\kk$) acting on a scheme $X$,
we say that the action is quasi-free if every isotropy subgroup of
$G$ on any point $x \in X$ is a connected group. Then, the above
 $\TT$-action on $\PP(\wedge^3 E)$, hence on $\Gr^{3, E}$,
is quasi-free.

The Grassmannian $\Gr^{3, E}$ as a closed subscheme of $\PP(\wedge^3 E)$ is
defined by, among others,  a set of specific quadratic relations, called $\pl$ relations. We describe them below.

For narrative convenience, we will assume that
$p_{u_1u_2 u_3}$ is defined for any sequence of
$d$ distinct integers between 1 and $n$,  not necessarily listed in 
the sequential order of natural numbers,
subject to the relation
\begin{equation}\label{signConvention}
p_{\si(u_1) \si(u_2) \si (u_3)}=\sgn(\si) p_{u_1 u_2u_3}
\end{equation}
for any permutation $\si$ on the set $[n]$, 
where $\sgn(\si)$ denotes the sign of the permutation.
Furthermore, also for convenience, we set 
\begin{equation}\label{zeroConvention}
 p_{\uu} := 0,
\end{equation}
for any $\uu=(u_1u_2 u_3)$ of a set of $3$  integers  in $[n]$ if
$u_i=u_j$ for some $1\le i \ne j\le 3$.

Now, for any pair $(\uh, \uk) \in \II_{d-1,n} \times \II_{d+1,n}$ with
$$\uh=\{h_1, h_2\}  \;\;
 \hbox{and} \;\; \uk=\{k_1, k_2,k_3, k_4\} ,$$
we have the  relation:
\begin{equation} \label{pluckerEq}
F_{\uh,\uk}= p_{h_1 h_2 k_1 } p_{k_2 k_3k_4}
-p_{h_1 h_2 k_2 } p_{k_1 k_3k_4} +p_{h_1 h_2 k_3 } p_{k_1k_2 k_4}
-p_{h_1 h_2 k_4 } p_{k_1k_2 k_3}.
\end{equation}
It is called a $\pl$ relation if it is not identically zero.

{\it For the sake of  presentation, we frequently  
express a  general $\pl$ relation as
\begin{equation}\label{succinct-pl}
F= \sum_{s \in S_F} \sgn(s) p_{\uu_s}p_{\uv_s},
\end{equation}
 for some index set $S_F$, with $\uu_s, \uv_s \in \II_{3,n}$, where
$ \sgn(s) $ is the $\pm$ sign associated with the quadratic  monomial term $p_{\uu_s}p_{\uv_s}$.
We note here that $\sgn(s)$ depends on 
how every of ${\uu_s}$ and ${\uv_s}$ is presented, per the convention \eqref{signConvention}.
}

\begin{defn}\label{ftF}
Consider any $\pl$ relation $F=F_{\uh,\uk}$ for some pair
 $(\uh, \uk) \in \II_{2,n} \times \II_{4,n}$.
We let $\ft_{F}+1$ be the number of terms in $F$
(which assumes values 3 or 4). We then define the rank of $F$ to be $\ft_{F}-2$. We denote this number by $\rk (F)$.
\end{defn}
The integer $\ft_{F}$, as defined above, will  be frequently used throughout. It assumes values 2 or 3. Hence,
 the integer $\rk (F)$ assumes values 0 or 1.

\begin{example}\label{exam:(3,6)}
Consider the Grassmannian $\Gr(3,6)$. Then, the $\pl$ relation
$$F_{(16), (3456)}:  p_{163}p_{456} - p_{164}p_{356} + p_{165}p_{346}$$
is of rank zero; the $\pl$ relation
$$F_{(12), (3456)}: p_{123}p_{456} - p_{124}p_{356} + p_{125}p_{346}- p_{126}p_{345}$$
is of rank one. 
\end{example}

Let $\ZZ[p_\ui]_{\ui \in \II_{3,n}}$ be the homogeneous coordinate ring
 of the $\pl$ projective space $\PP(\wedge^3 E)$ and $I_\wp \subset \ZZ[p_\ui]_{\ui \in \II_{3,n}}$ 
be the homogeneous ideal generated by all the $\pl$ relations  \eqref{pluckerEq} or
\eqref{succinct-pl}.
We let $I_{\widehat\wp}$ by the homogeneous ideal of  $\Gr^{3, E}$ in $\PP(\wedge^3 E)$.
Then $I_{\widehat\wp} \supset I_\wp$, but not equal over $\ZZ$,
or a field of positive characteristic, in general: there are other additional
relations (multivariate $\pl$ relations) for the Grassmannian $\Gr^{3, E}$, thanks to Matt Baker for pointing this out to the author. In characteristic zero, $I_{\widehat\wp} = I_\wp$.


\subsection{Primary $\pl$ equations with respect to a fixed affine chart} $\ $

In this subsection, we focus on a fixed affine chart 
($p_\um \ne 0$) of the $\pl$ projective space $\PP(\wedge^3 E)$
for some $\um \in \II_{3,n}$. Up to permutation, it suffices to assume
$\um=(123)$. More precisely, in $\PP(\wedge^3 E)$,
we let $$\bU:=(p_\um \equiv 1)$$  stand for the open chart
 defined by $p_\um \ne 0$.  Then,
the affine space $\bU$  comes equipped with the 
local free variables $x_\uu=p_\uu/p_\um$
 for all $\uu \in \II_{3,n} \- \um$.
In practical calculations, we will simply set $p_\um =1$, whence the notation
$(p_\um \equiv 1)$ for the chart. We let
$$\bU_\Gr = \bU \cap \Gr^{3, E}$$
be the corresponding induced open chart  of $\Gr^{3, E}$.

The chart $\bU_\Gr$ is canonically an affine space.
Below, we explicitly describe $$\up:={n \choose 3} -1- 3(n-3) $$
many specific  $\pl$ relations with respect to the chart $\bU$, called 
the $\um$-primary 
$\pl$ relations, such that their restrictions to the chart $\bU$ define
$\bU_\Gr$ as a closed subscheme of the affine space $\bU$.

To this end, for  $\um=(123)$, we set 
\begin{equation}\label{nbc}
\II_{3,n}^\lt=\{\uu \in \II_{3,n} \mid |\uu \- \um| \ge 2\} \subset \II_{3,n}
\end{equation}
where $\uu$ and  $\um$ are also regarded as subsets of integers, and $|\uu \- \um|$ denotes the cardinality of $\uu \- \um$. 
In  words,  $\uu \in \II^\um_{3,n}$ if and only if
$\uu=(u_1u_2 u_3)$ contains at least two elements distinct from elements in  
$\um=(123)$. Here, lt stands for leading term,
 the choice of which will become
clear soon. It is helpful to write explicitly the set $\II_{3,n}^\lt$
$$
 \II_{3,n}^\lt=
\{(1uv),(2uv),(3uv),(abc) \mid 3<u<v \le n,\; 3<a<b<c\le n\}.
$$
Hence,
$$\II_{3,n} \- \II_{3,n}^\lt =
\{(123), (12u), (13u), (23u) \mid \hbox{for all $u \in [n]\- [3]$}\}.$$
Then, one calculates and finds
$$|\II_{3,n}^\lt|=\up={n \choose 3} -1- 3(n-3) ,$$
where $|\II_{3,n}^\lt|$ denotes the cardinality of $\II_{3,n}^\lt$.

Further, let $\ua=(a_1\cdots a_k)$ be a list of some elements of $[n]$, not necessarily mutually distinct,
 for some $k<n$.
We will write $$v \ua=v(a_1\cdots a_k)=(v a_1\cdots a_k)
\;\; \hbox{and} \;\;  \ua  v=(a_1\cdots a_k)v=(a_1\cdots a_k v),$$
each is considered as a list of some elements of $[n]$, 
for any $v \in [n] \- \ua$.

Now, take any element $\uu =(u_1u_2u_3) \in \II_{3,n}^\lt$.
We then set 
$$ \uh=(u_2u_3) \;\;\hbox{and} \;\; 
\uk=(u123).$$
 This gives rise
to the $\pl$ relation $F_{\uh,\uk}$, taking of the following form
\begin{equation} \label{keyTrick}
F_{\uh,\uk}=p_{u_2u_3 u_1} p_{123} -  p_{u_2u_3 1} p_{u_1 23}
+ p_{u_2u_3 2} p_{u_1 13} - p_{u_2u_3 3} p_{u_2 12},
\end{equation}
where $\um \setminus m_i = \um \setminus \{m_i\}$ and $\um$ is regarded as a set of integers,
for all $i \in [d]$.

Alternatively, we give a new notation for this particular equation: 
\begin{equation} \label{keyTrick2}
F_{\um, \uu}=p_{(\uu \setminus u_1)u_1} p_\um + \sum_{i=1}^3 (-1)^i  
p_{(\uu \setminus u_1) i} p_{u_1 (\um \setminus \{i\} )},
\end{equation} because it only depends on $\um$ and $\uu \in \II_{3,n}^\lt$. 
To simplify the notation, we introduce
$$\uu^r= \uu \setminus u_1, \;\; \widehat{\um_i} = \um\setminus \{i\}, \;\; \hbox{for all $i \in [3]$.}$$
Then, \eqref{keyTrick2} becomes
\begin{equation} \label{keyTrick4}
F_{\um, \uu}=p_{\um}p_{\uu} +
\sum_{i=1}^3 (-1)^i p_{ \uu^r i } p_{u_1 \widehat{\um_i} }.
\end{equation}
More concretely, all of the equations in \eqref{keyTrick4}
are of the following forms:
\begin{eqnarray}
F_{(123),1uv}=p_{123}p_{1uv}-p_{12u}p_{13v} + p_{13u}p_{12v}, 
\nonumber \label{rk0-1uv}\\
F_{(123),2uv}=p_{123}p_{2uv}-p_{12u}p_{23v} + p_{23u}p_{12v}, 
\nonumber \label{rk0-2uv} \\
F_{(123),3uv}=p_{123}p_{3uv}-p_{13u}p_{23v} + p_{23u}p_{13v} ,
\nonumber \label{rk0-3uv}\\
F_{(123),abc}=p_{123}p_{abc}-p_{12a}p_{3bc} + p_{13a}p_{2bc} -p_{23a}p_{1bc}, 
\nonumber \label{rk1-abc}
\end{eqnarray}
where $u < v \in [n]\-\{1,2,3\}$ and $a<b<c \in [n]\-\{1,2,3\}$.
Here, $[n]=\{1,\cdots,n\}$.

\begin{defn}\label{lt-ltvar}
We call the $\pl$ equation $F_{\um, \uu}$ in
 \eqref{keyTrick4} a primary $\pl$ equation
for the chart $\bU=(p_\um \equiv 1)$. 
We also say $F_{\um,\uu}$ is $\um$-primary.
The term $p_\um p_{\uu}$ is called the {\it leading} term of $F_{\um, \uu}$.
\end{defn}

One should not confuse  $F_{\um, \uu}$ with 
the expression of a general $\pl$ equation
$F_{\uh,\uk}$: we have $(\um, \uu)\in \II_{3,n}^2$ for the former and 
$(\uh,\uk) \in \II_{2,n} \times \II_{4,n}$ for the latter.

\subsection{De-homogenized $\pl$ ideal with respect to a fixed affine chart} $\ $

Following the previous subsection, we continue to fix the element 
$\um=(123) \in \II_{3,n}$ 
and  focus on the chart $\bU$ of $\PP(\wedge^3 E)$.
For convenience, 
we may write  $\II_{3,n} \- \um$ for  $\II_{3,n} \- \{\um\}$.

Given any $\uu \in \II_{3,n}^\lt$, by \eqref{keyTrick4}, 
it gives rise to the $\um$-primary equation 
$$ F_{\um, \uu}=p_{\um}p_{\uu} +
\sum_{i=1}^3 (-1)^i p_{\uu^r i } p_{ u_1 \widehat{\um_i} }.$$
 If we set $p_\um =1$ and let $x_\uw=p_\uw$, 
for all $\uw \in \II_{3,n} \- \um$,
 then it  becomes 
\begin{equation}\label{equ:localized-uu}
\bF_{\um, \uu}=x_{\uu} +
\sum_{i=1}^3 (-1)^i  x_{\uu^r i} x_{ u_1 \widehat{\um_i} }.
\end{equation}

\begin{defn}\label{localized-primary} 
We call the relation \eqref{equ:localized-uu} 
 the de-homogenized (or the localized)
$\um$-primary $\pl$ relation corresponding to $\uu \in \II_{3,n}^\lt$.
We call the unique distinguished variable, $x_{\uu}$, the 
{\it leading} variable
of the  de-homogenized $\pl$ relation $\bF_{\um, \uu}$. 
\end{defn}

We let $\sF_\um$
 be the set of all of the equations in \eqref{equ:localized-uu}.
Then, we have the correspondence 
$$\II_{3,n}^\lt \lra \sF_\um, \;\; \uu \to F_{\um, \uu}.$$
Obviously, this is a bijection.

In concrete terms, all of the equations in \eqref{keyTrick4}
are of the following forms:
\begin{eqnarray}
F_{(123),1uv}=x_{1uv}-x_{12u}x_{13v} + x_{13u}x_{12v}, 
 \label{-rk0-1uv}\\
F_{(123),2uv}=x_{2uv}-x_{12u}x_{23v} + x_{23u}x_{12v}, 
 \label{-rk0-2uv} \\
F_{(123),3uv}=x_{3uv}-x_{13u}x_{23v} + x_{23u}x_{13v} ,
 \label{-rk0-3uv}\\
F_{(123),abc}=x_{abc}-x_{12a}x_{3bc} + x_{13a}x_{2bc}
 -x_{23a}x_{1bc}, 
 \label{-rk1-abc}
\end{eqnarray}
where $u < v \in [n]\-\{1,2,3\}$ and $a<b<c \in [n]\-\{1,2,3\}$.
Here, $[n]=\{1,\cdots,n\}$.

Throughout  this paper, 
we often express an $\um$-primary $\pl$ equation $F$ as
\begin{equation}\label{the-form-F}
F=\sum_{s \in S_F} \sgn (s) p_{\uu_s} p_{\uv_s}=
 \sgn (s_F) p_\um p_{\uu_{s_F}}+
\sum_{s \in S_F\- s_F} \sgn (s) p_{\uu_s} p_{\uv_s}
\end{equation}
where $s_F$ is the index for the leading term of $F$, and $S_F\- s_F:=S_F\-\{s_F\}$.
Then, upon setting $p_\um=1$ and letting $x_\uw=p_\uw$ for all $\uw \in \II_{3,n}\- \um$,
 we can write the corresponding de-homogenized  $\um$-primary $\pl$ equation $\bF$ as
\begin{equation}\label{the-form-LF}
\bF=\sum_{s \in S_F} \sgn (s) x_{\uu_s} x_{\uv_s}= \sgn (s_F) x_{\uu_F}+
\sum_{s \in S_F\- s_F} \sgn (s) x_{\uu_s} x_{\uv_s}
\end{equation}
where  $x_{\uu_F}:=x_{\uu_{s_F}}$ is the {\it leading variable}\footnote{Leading variables  crucially help to propagate and derive
certain square-freeness property such as in Corollary \ref{linear-or-vanish}.}
 of $\bF$.

\begin{defn}\label{ft-bF}
Let $F$ be an $\um$-primary $\pl$ relation and $\bF$ its 
de-homogenization with respect to the chart $\bU$.
We set $\ft_{\bF}=\ft_F$  and $\rk (\bF) =\rk (F)$.
\end{defn}

For any $\uu \in \II_{3,n} \- \um$, 
 we let $x_\uu=p_\uu/p_\um$ for all $\uu \in \II_{3,n} \- \um$.
Then, we can identify the coordinate ring of $\bU$ 
with $\kk [x_\uu]_{\uu \in \II_{3,n} \- \um}$.
We let 
$I_{\wp,\um}$ be the ideal of $\kk [x_\uu]_{\uu \in \II_{3,n} \- \um}$
obtained from the ideal $I_\wp$ be setting $p_\um=1$ and letting 
$x_\uu=p_\uu$ for all $\uu \in \II_{3,n} \- \um$.
The ideal $I_{\wp,\um}$ is the de-homogenization of the 
homogeneous $\pl$ ideal $I_\wp$ on the chart $\bU$.
Recall that we always set $\um=(123)$.

\begin{prop}\label{primary-generate} 
The affine subspace $\bU_\Gr=\bU \cap \Gr^{3, E}$ embedded in
the affine space $\bU$ is defined by the relations in
$$\sF=\{ \bF_{(123), \uu} \mid \uu \in \II_{3,n}^{(123)} \}.$$
 Consequently, the chart $\bU_\Gr=\bU \cap \Gr^{3, E}$ comes equipped with
 the set of free variables
 $$\var_{\bU}:=\{x_\uu \mid \uu \in \II_{3,n} \-\{ (123)\} \- \II_{3,n}^{(123)}\}$$ and is 
 canonically isomorphic to the affine space 
 with  the above variables as its local free variables.
\end{prop}
\begin{proof}
(This proposition is elementary; 
it serves as the initial check of an induction for  some later proposition;
we provide sufficient details for completeness.)

We first prove the claim that for any  $\uu \in \II_{3,n}^{(123)}$, 
its corresponding de-homogenized $\pl$ 
primary $\pl$ equation $\bF_{(123), \uu}$ is equivalent to an expression of
the leading variable $x_\uu$ as a polynomial in  the free variables of
$\var_{\bU}$. 

Fix any  $\uu \in \II_{3,n}^{(123)}$.
Suppose $\rk \bF_{(123), \uu}= 0$.  Then,
$\bF_{(123), \uu}$ takes one of the forms in
\eqref{-rk0-1uv}, \eqref{-rk0-2uv}, and \eqref{-rk0-3uv}.
Then, by inspection, the claim obviously holds in such a case.

Now suppose that Suppose $\rk \bF_{(123), \uu}= 1$.
Then by \eqref{-rk1-abc}, $\bF_{(123), \uu)}$ takes the form
$$F_{(123),abc}=x_{abc}-x_{12a}x_{3bc} + x_{13a}x_{2bc}
 -x_{23a}x_{1bc}$$
with $\uu=(abc)$ for some $a<b<c \in [n]\-\{1,2,3\}$.
By the previous case $\rk \bF_{(123), \uu}= 0$, 
we have that using the relations $\bF_{(123), (3bc)},
\bF_{(123), (2bc)}, \bF_{(123), (1bc)}$, respectively, the variables
$x_{3bc}, x_{2bc}, x_{1bc}$ can be  expressed 
as  polynomials in  the free variables of
$\var_{\bU}$. 
 Therefore,  the relation $F_{(123),abc}$, 
   is also equivalent to an expression of $x_{(abc)}$ as a polynomial in 
 the variables of $\var_{\bU}$.

Let $J$ be the ideal of $\kk [x_\uu]_{\uu \in \II_{3,n}\- (123)}$ 
generated by
 $\{ \bF_{(123), \uu} \mid \uu \in \II_{3,n}^{(123)} \}$ and let
 $V(J)$ the subscheme of $\bU$ defined by $J$.
By the above discussion,  $V(J)$ is canonically isomorphic to
 the affine space of dimension $3(n-3)$ with  
  the variables of $\var_{\bU}$ as its local free variables.
Since $\bU_\Gr \subset V(J)$, we conclude $\bU_\Gr=V(J)$.
\end{proof}

\begin{defn}\label{basic} We call the variables in
$$\var_{\bU}:=\{x_\uu \mid \uu \in \II_{3,n} \- \um \- \II_{3,n}^\lt\}$$
the $\um$-basic $\pl$ variables, or simply basic variables.
\end{defn}
Only non-basic $\pl$ variables correspond to $\um$-primary $\pl$ equations.

For $r=0$ or 1, we let
$$\sF^r= 
\{\bF_{\um,\uu}  \mid {\rm rank} (F_{\um,\uu}) =r,\; \uu \in \II_{3,n}^\lt\}.$$
Note that for any $\pl$ relations $F$, we have  ${\rm rank} (F)=0$ or 1.
Hence $$\sfm=\sF^0 \sqcup \sF^1.$$

Then, one observes the following easy but useful fact.

\begin{prop}\label{leadingTerm} Fix  any $r=0$ or 1 and 
 any $\uu \in \II_{3,n}^\lt$ with ${\rm rank}_\um (F_{\um, \uu})=r$.
 Then,  the leading variable 
$x_\uu$ of $\bF_{\um, \uu}$ does not appear in any relation in
$$\sF^{r-1} \cup (\sF^r_\um \setminus \bF_{\um, \uu})$$
where we set $\sF^{-1} =\emptyset$ for the case of $r=0$.
\end{prop}

To close this subsection, we raise  a concrete question.
 Fix the chart $(p_\um \equiv 1)$.
 In  $\kk[x_\uu]_{\uu \in \II_{3,n}\- \um}$,
according to  Proposition \ref{primary-generate}, 
 every de-homogenized $\pl$ equation
$\bF_{\uh,\uk}$ on the chart $\bU$ can be expressed
as a polynomial in the  de-homogenized primary $\pl$ relations $\bF_{\um, \uu}$ 
with $\uu \in \II_{3,n}^\lt$. 
 It may be useful  in practice to find such an expression explicitly for an arbitrary $F_{\uh,\uk}$.
For example, for the case of $\Gr(2,5)$, this can be done as follows.

\begin{example}\label{pl2-5} For $\Gr(2,5)$, we have five $\pl$ relations:
$$F_1= p_{12}p_{34}-p_{13}p_{24} + p_{14}p_{23},\;
F_2= p_{12}p_{35}-p_{13}p_{25} + p_{15}p_{23}, \; $$
$$F_3= p_{12}p_{45}-p_{14}p_{25} + p_{15}p_{24},\;
F_4= p_{13}p_{45}-p_{14}p_{35} + p_{15}p_{34}, \;$$
$$F_5= p_{23}p_{45}-p_{24}p_{35} + p_{25}p_{34}. $$
On the chart $(p_{45} \equiv 1)$,
$F_3, F_4,$ and $ F_5$ are primary. 
One calculates and finds
$$p_{45} F_1 = p_{34} F_3 -p_{24} F_4 + p_{14} F_5, $$ 
$$p_{45} F_2 = p_{35} F_3 -p_{25} F_4 + p_{15} F_5.$$
In addition,  the  Jacobian of the de-homogenized $\pl$ equations of $\bF_3, \bF_4, \bF_5$ with respect to 
all the variables,
$x_{12}, x_{14}, x_{15} ,x_{13}, x_{35} , x_{34},x_{23} ,x_{24} ,x_{25}, $
is given by
$$
\left(
\begin{array}{cccccccccc}
1 & x_{25} & x_{24} & 0  & 0 & 0 & 0 & 0 & 0 \\
0 & 0 & 0 & 1  & x_{14} & x_{15}& 0 & 0 & 0 \\
0& 0 & 0 & 0  & 0 & 0 & 1 & x_{35} & x_{34} \\
\end{array}
\right).
$$
There, one sees visibly  a $(3 \times 3)$  identity minor. 
\end{example}

\begin{rem}\label{many-pri} Fix $\um$ and $\uu \in \II_{3,n}^\lt$. There can be many $\pl$ equations
containing the term $p_\um p_\uu$. For example, consider $\Gr^{3,E}$ with $\dim E=n$. 
 
 Let $\um =(123)$ and
$\uu=(abc)$ with $a< b< c \notin \{1,2,3\}$. Then, we have the following 
$$F_{\um, (abc,a)}= p_\um p_{abc}-p_{12a}p_{3bc}+p_{13a}p_{2bc}-p_{23a}p_{1bc},$$
$$F_{\um, (abc,b)}= p_\um p_{abc}-p_{12b}p_{3ac}+p_{13b}p_{2ac}-p_{23b}p_{1ac},$$
$$F_{\um, (abc,c)}= p_\um p_{abc}-p_{12c}p_{3ab}+p_{13c}p_{2ab}-p_{23c}p_{1ab}.$$
But, we also have
$$F_{(12), (3abc)}=p_\um p_{abc}-p_{12a}p_{3bc}+p_{12b}p_{3ac}-p_{12c}p_{3ab},$$
$$F_{(13), (2abc)}=p_{(132)}p_{abc}-p_{13a}p_{2bc}+p_{13b}p_{2ac}-p_{13c}p_{2ab},$$
$$F_{(23), (1abc)}=p_\um p_{abc}-p_{23a}p_{1bc}+p_{23b}p_{1ac}-p_{23c}p_{1ab}.$$

The one we used in this paper is $F_{\um, (abc,a)}$, although
any other among the above list might be chosen
and fixed to take the role of $F_{\um, (abc,a)}$.
\end{rem}

\subsection{Ordering the $\pl$ variables and the primary $\pl$ relations}\label{orders}

\begin{defn}\label{gen-order} Let $K$ be any fixed totally ordered finite set, 
with its order denoted by $<$.
Consider any two subsets $\eta \subset K$ and ${ \zeta} \subset K$ with the cardinality $n$
for some positive integer $n$.
We write $\eta=(\eta_1,\cdots,\eta_n)$ 
and  ${\zeta}=(\zeta_1,\cdots,\zeta_n)$ as arrays according to the ordering of $K$.
We say $\eta <_{\lex} { \zeta}$ if the left most nonzero number in the vector $\eta-{\zeta}$ is negative,
or more explicitly,  if we can express
$$\eta=\{t_{1}< \cdots <t_{r-1} <s_r< \cdots \}$$
$${\zeta}=\{t_{1}< \cdots <t_{r-1}<t_r  < \cdots \}$$
such that $s_r< t_r$ for some integer $r \ge 1$.  We call $<_\lex$ the lexicographic order
induced by $(K, <)$.

Likewise, we say $\eta <_{\invlex} { \zeta}$
 if the right most non-zero number in the vector $\eta-{ \zeta}$ is negative,
or more explicitly, if we can express
$$\eta=\{\cdots <s_r< t_{r+1}< \cdots <t_n\}$$
$${\zeta}=\{\cdots <t_r  < t_{r+1}< \cdots <t_n\}$$
such that $s_r< t_r$ for some integer $r \ge 1$. 
We call $<_\invlex$ the reverse lexicographic order
induced by $(K, <)$. 
\end{defn}

This definition can be applied to the set
$$\II_{3,n} =\{ (i_1i_2 i_3) \; \mid \; 1 \le  i_1 < i_2 < i_3 \le n\}$$
This way, we have equipped the set $\II_{3,n}$ 
with both the  lexicographic ordering $``<_\lex "$ and
the reverse lexicographic ordering $``<_\invlex "$. 

{\it We point out here that neither is the order we used for the set of $\pl$ variables
$$\{x_\uu \mid \uu \in \II_{3,n} \- \um\},$$
even though by the obvious bijection
between $\II_{3,n} \- \um$ and $\{x_\uu \mid \uu \in \II_{3,n} \- \um\}$, each  of $``<_\lex "$ and $``<_\invlex "$
provides a total ordering on $\{x_\uu \mid \uu \in \II_{3,n} \- \um\}$. }

\begin{defn} For any  $\uu \in \II_{3,n}^\lt$, we define 
the $\um$-rank of $\uu$ (resp.  $x_\uu$) to be the rank of its corresponding primary $\pl$ equation $F_{\um,\uu}$.  
If $\uu \in (\II_{3,n} \- \um) \- \II_{3,n}^\lt$,
then we set $\rk (\uu)=-1$ (resp.  $\rk (x_\uu)=-1$).
\end{defn}

\begin{defn}\label{invlex}
Consider any $\uu, \uv \in \II_{3,n} \- \um$. 
We say $$\uu <_\wp \uv$$ if one of the following three holds:
\begin{itemize}
\item ${\rm rank}_\um \; \uu < {\rm rank}_\um \; \uv$;
\item ${\rm rank}_\um \; \uu = {\rm rank}_\um \; \uv$, $\uu \- \um<_\lex \uv \- \um$;
\item ${\rm rank}_\um \; \uu ={\rm rank}_\um \; \uv$, $\uu \- \um = \uv\- \um$, and
$\um \cap \uu  <_\lex \um \cap \uv$.
\end{itemize}
\end{defn}

\begin{defn}\label{cFi-partial-order} Consider any two $\pl$ variables $x_\uu$ and $x_\uv$.
We say 
$$\hbox{$x_\uu <_\wp x_\uv$ if $\uu<_\wp \uv$.}$$

Consider any two distinct primary equations, 
$\bF_{\um,\uu}, \bF_{\um,\uv} \in \sfm$ 
 We say 
$$\bF_{\um,\uu} <_\wp \bF_{\um,\uv}\; \; \hbox{if} \;\;  \uu <_\wp \uv.$$
\end{defn} 
This order coincides with the order on $\sF$
 described in \S \ref{blocks of gov}. The reader may consult with 
it since it is more concrete.
Under the above order, we can write
$$\sfm=\{\bF_1 <_\wp \cdots <_\wp \bF_\up\}.$$

In the sequel, when comparing two $\pl$ variables $x_\uu$ and $x_\uv$
or  two $\um$-primary $\pl$ equations, we exclusively use $<_\wp$.
Thus, throughout, for simplicity, we will simply write $<$ for $<_\wp$.
A confusion is unlikely.

{\it We point out here that $x_\uu <_\wp x_\uv$ is neither lexicographic
nor reverse-lexicographic on the indexes. Indeed,
every non-lexicographic or   non-reverse-lexicographic order,
introduced in this article, is important for our method. Some orders we shall use may be subtle.}


For later use, we introduce 

\begin{defn}\label{p-t}
Let $\fT_i$ be a finite set for all $i \in [h]$ for some positive integer $h$. Then,
the order $$\fT_1 < \cdots < \fT_h$$ naturally induces a partial order on the disjoint union
$\sqcup_{i \in [h]} \fT_i $ as follows. Take any $i < j \in [h]$,
$a_i \in \fT_i$, and $a_j \in \fT_j$. Then, we say $a_i < a_j$.
\end{defn}

If every $\fT_i$ comes equipped with a total order for all $i \in [h]$. Then,
in the situation of Definition \ref{p-t}, the disjoint union
$\sqcup_{i \in [h]} \fT_i $ is totally ordered.

\subsection{Summary:  primary $\pl$ relations,  total order, and notation }\label{Gr3n} $\ $

Throughout the remainder of this paper, we will focus on $\Gr^{3,E} \subset \PP(\wedge^3 E)$ and the chart
$$\bU=(p_{123} \equiv 1) \subset \PP(\wedge^3 E).$$
We let
$$\{ x_{abc}:= p_{abc}/p_{123} \mid (abc) \in \II_{3,n} \- (123)\}$$
be the affine coordinates of the chart $\bU$.

By Proposition \ref{primary-generate}, $\bU_\Gr:=\bU \cap \Gr^{3,E}$ is defined by 
the set of de-homogenized 
primary $\pl$ relations with respect to the chart $\bU$.
Specifically, they are in the following forms.
\begin{eqnarray}\label{all primaries}
\hbox{de-homogenized primary $\pl$ relations for $\bU_\Gr=\bU \cap \Gr^{3,E}$:}  \\  \nonumber 
\bF_{(123),1uv}=x_{1uv}-x_{12u}x_{13v} + x_{13u}x_{12v}, \label{rk0-1uv}\\ \nonumber
\bF_{(123),2uv}=x_{2uv}-x_{12u}x_{23v} + x_{23u}x_{12v}, \label{rk0-2uv} \\ \nonumber
\bF_{(123),3uv}=x_{3uv}-x_{13u}x_{23v} + x_{23u}x_{13v} ,\label{rk0-3uv}\\ \nonumber
\bF_{(123),abc}=x_{abc}-x_{12a}x_{3bc} + x_{13a}x_{2bc} -x_{23a}x_{1bc}, \label{rk1-abc} \nonumber
\end{eqnarray}
for all $3 <u<v$ and $3<a<b<c$. The first three are of rank 0, the last is of rank 1.

The above (de-homogenized) $\pl$ relations are ordered as follows.
\begin{itemize}
\item any relation of rank 0 is smaller than any relation of rank 1;
\item $\bF_{(123),1uv}<\bF_{(123),2uv}<\bF_{(123),3uv}$;
\item  $\bF_{(123),iuv} <\bF_{(123),ju'v'} $ if $(uv)<_\lex (u'v')$, for any $i, j \in [3]$.
\item $\bF_{(123),abc}<\bF_{(123),a'b'c'}$ if $(abc)<_\lex (a'b'c')$.
\end{itemize}

We let $\sF=\sF_{(123)}$ denote the set of  all primary $\pl$ relations with respect to the chart $\bU$.
Then, we list all the above primary $\pl$ relations as
\begin{equation}\label{list sF}
\bF_1< \bF_2 <\cdots < \bF_\up.
\end{equation}

As in  \eqref{nbc},  we have
\begin{equation}\label{ld-index}
 \II_{3,n}^\lt=\{(1uv),(2uv),(3uv),(abc) \mid 3<u<v \le n,\; 3<a<b<c\le n\}.
\end{equation}
This is 
 the index set for the
 leading terms of primary $\pl$ relations. It contains
\begin{equation}\label{ld-index-0}
 \II_{3,n}^{\lt, 0}=\{(1uv),(2uv),(3uv)\mid 3<u<v \le n\},
\end{equation}
  the index set for the leading terms of primary $\pl$ relations of rank-0,
as a  subset.

\section{A Singular Local Birational Model $\sV$ for $\Gr^{3,E}$}\label{singular-model}

{\it The purpose of this section is to establish a local model $\sV$, 
birational to $\Gr^{3,E}$,  such that  all
terms of all $\um$-primary $\pl$ equations can be separated in  the
defining governing binomial relations of $\sV$ in a smooth ambient scheme $\sR$.

We view $\sV$ as a new paradigm for all possible singularities.
And, we will design the universal blowups based on the
governing defining relations of $\sV$.
}

\subsection{The construction  of $\sV$ in a smooth model $\sR$} $\ $

For any $\bF \in \sF$ with
 $$\bF=\sum_{s \in S_F} \sgn (s) x_{\uu_s}x_{\uv_s}$$
we let $\PP_F$ be the projective space with 
homogeneous coordinates written as $[x_{(\uu_s, \uv_s)}]_{s \in S_F}$.
For convenience, we make a convention: 
\begin{equation}\label{uv=vu}
x_{(\uu_s, \uv_s)}=x_{(\uv_s, \uu_s)}, \; \forall \; s \in S_{F}.
\end{equation}
(If we write $(\uu_s, \uv_s)$ in the lexicographical order, i.e., we insist $\uu_s <_\lex \uv_s$,
then the ambiguity is automatically avoided. However, the convention allows flexibility of writing.)

\begin{defn}
We call $x_{(\uu_s, \uv_s)}$ a $\vr$-variable of $\PP_F$, or simply a $\vr$-variable.
To distinguish, we call a $\pl$ variable, $x_\uu$ with $\uu \in 
\II_{3,n} \- \um$, a $\vp$-variable. By convention, $x_\um=1$.
\end{defn}

{\it In what follows, we will ultimately only be interested in the case when $k=\up$. But, as it will be
useful for applying induction in later proofs, we will do it for all $k \in [\up]$.
}

\smallskip

Fix $k \in [\up]$. We introduce the natural rational map
\begin{equation}\label{theta-k}
 \xymatrix{ \Theta_{[k]}: \bU 
 \ar @{-->}[r]  & \prod_{i \in [k]} \PP_{F_i}  } \end{equation} $$
 [x_\uu]_{\uu \in \II_{3,n}} \lra  
\prod_{i \in [k]}  [x_\uu x_\uv]_{(\uu,\uv) \in \La_{F_i}} $$   
where $[x_\uu]_{\uu \in \II_{3,n}}$ is the de-homogenized  $\pl$
 coordinates of a point in $\bU \subset \PP(\wedge^3 E)$. 
It restricts to give rise to 
\begin{equation}\label{theta-k-gr}
 \xymatrix{ \Theta_{[k],\Gr}: \bU_\Gr
 \ar @{-->}[r]  & \prod_{i \in [k]} \PP_{F_i}  } \end{equation} 
where recall that $\bU_\Gr:=\bU \cap \Gr^{3,E}$

 We let
 \begin{equation}\label{embed sV}
 \xymatrix{
\sV_{[k]} \ar @{^{(}->}[r]  & \sR_{[k]}:= \bU \times  \prod_{i \in [k]} \PP_{F_i}  
 }
\end{equation}
be the closure of the graph of the rational map $\Theta_{[k], \Gr}$.
The case when $k=\up$ is what we are mainly interested.  We let
$$\sV:=  \sV_{\sF_{[\up]}},\;\;
\sR=   \sR_{\sF_{[\up]}}.$$
We ultimately are only interested in the model $\sV$;
the intermediate models $\sV_{[k]}$ are introduced so that
induction can be applied in some proofs.

As an intermediate toroidal ambient space of $\sV_{[k]}$,
we let
 \begin{equation}\label{embed sU}
 \xymatrix{
\bU_{[k]} \ar @{^{(}->}[r]  & \sR_{[k]}:= \bU \times  \prod_{i \in [k]} \PP_{F_i}  
 }
\end{equation}
be the closure of the graph of the rational map $\Theta_{[k]}$.


\begin{lemma}\label{lift-action-to-Um[k]}
 The torus $\TT$ acts on 
$\sR_{[k]}= \bU \times  \prod_{i \in [k]} \PP_{F_i}$ 
with the given action on $\bU$
and the trivial action of $\TT$ on $\prod_{i \in [k]} \PP_{F_i}$.
Then, the rational map $ \Theta_{[k]}$ of \eqref{theta-k} is
$\TT$-equivarint.

Moreover, the quasi-free action of the torus $\TT$ on $\bU$ lifts to a
quasi-free action on $\bU_{[k]}$ so that the birational morphism
$$\bU_{[k]} \lra \bU$$ is $\TT$-equivariant. 
Restricting the above to the invariant subset $\sV_{[k]}\subset \bU_{[k]}$,
 we obtain that the birational morphism
$$\sV_{[k]} \lra \bU_\Gr=\bU \cap \Gr^{3,E}$$
is $\TT$-equivariant.
\end{lemma}
\begin{proof} All statements are immediate.
\end{proof}
We can summarize the above lemma by the following diagram
\begin{equation}\label{action-1}
\xymatrix{
\sV_{[k]} \ar [d] \ar @{^{(}->}[r]  & \bU_{[k]} \ar [d] \ar @{^{(}->}[r] & \sR_{[k]} \ar [d]\\
\bU_\Gr \ar @{^{(}->}[r]  &  \bU \ar[r]^{=} & \bU
}
\end{equation} 
where all the arrows are   $\TT$-equivariant and the 
first two vertical arrows are birational.


We  need $(\sV_{[k]} \subset \bU_{[k]})$ 
only for inductive purpose. We are mainly interested in
$$\xymatrix{
\sV_{[k]}  \ar @{^{(}->}[r] & \sR_{[k]}
}
$$
and ultimately interested in the final case of the above when $k=\up$,
$$\xymatrix{
\sV  \ar @{^{(}->}[r]  & \sR .
}
$$

We need to obtain the defining relations for $\sV$ in $\sR$.
For this purpose, we  first introduce and consider the following algebras.

\begin{defn}
Fix any $ k \in [\up]$. We let 
$$ R_0=\kk[x_\uu]_{\uu \in \II_{3,n} \- \um} \; ,\;\;\;
 R_{[k]}=R_0[ x_{(\uv_s, \uu_s)}]_{s \in S_{F_i},  i \in [k]}.$$
A  polynomial $f \in R_{[k]}$  is called multi-homogeneous if it is homogenous 
in $[x_{(\uv_s, \uu_s)}]_{s \in S_{F_i}}$, for every $i \in [k]$.
 A multi-homogeneous polynomial $f \in  R_{[k]}$ 
 is $\vr$-linear if it is linear 
 in $[x_{(\uv_s, \uu_s)}]_{s \in S_{F_i}}$, 
 whenever it contains some $\vr$-variables of $\PP_{F_i}$,
 for  any $i \in [k]$.
 \end{defn}

Then, corresponding to the embedding \eqref{embed sU}, 
we have the degree two homomorphism
\begin{equation}\label{vik}
\vi_{[k]}: \; R_{[k]} \lra R_0   
\end{equation} 
$$\vi_{[k]}|_{R_0} =\Id_{R_0}; \;\;\; x_{(\uu_s,\uv_s)} \to x_{\uu_s} x_{\uv_s}$$
for all  $s \in S_{F_i}, \; i \in [k]$. 

As we are ultimately interested in the case when $k=\up$, we set 
$$R:=R_{[\up]},\;\; \vi:=\vi_{[\up]}.$$

We let $\ker^\mh \vi_{[k]}$ 
denote the set of all  multi-homogeneous polynomials
 in $\ker  \vi_{[k]}$. 


\begin{lemma}\label{defined-by-ker} 
The scheme 
$\rU_{[k]}$, as a closed subscheme of 
 $\sR_{[k]}= \bU \times  \prod_{i \in [k]} \PP_{F_i}$,
 is defined by the ideal $\ker^\mh \vi_{[k]}$. 
\end{lemma}
\begin{proof} This is immediate.
\end{proof}

\subsection{The toroidal part of the  defining ideal of $\sV$ in $\sR$: 
$\ker^\mh \vi$} $\;$

We need to investigate $\ker^\mh \vi_{[k]}$.

For any $f \in R_{[k]}$ ($\subset  R$), note that $\vi_{[k]} (f) =\vi(f)$,
thus knowing $f \in R_{[k]}$, 
$$f \in \ker \vi_{[k]} \iff f \in \ker \vi.$$
Hence, we frequently just use $\vi (f)$ even if we intend to prove a statement
within $R_{[k]}$. 

Consider any $f \in \ker^\mh \vi_{[k]}$.
 We express it as the sum of its  monomials
$$f= \sum \bm_i.$$
We have $\vi_{[k]} (f)=\sum \vi_{[k]} (\bm_i)=0$ in $R_0$. Thus, the set of the monomials $\{\bm_i\}$ 
can be grouped into minimal groups to form partial sums
of $f$ so that {\it the 
$\vi_{[k]}$-images of elements of each group are 
 identical} and the image of the partial sum of each minimal group equals 0 in $R_0$.
When ch.$\kk=0$, 
this means each minimal group consists of a pair 
$\{\bm_i, \bm_j\}$ with the identical image under $\vi_{[k]}$
and its partial sum
is the difference $\bm_i -\bm_j$.
When ch.$\kk=p>0$ for some prime number $p$, this means each minimal group 
 consists of either  (1): a pair $\{\bm_i, \bm_j\}$ with the identical image under $\vi_{[k]}$
 and $\bm_i -\bm_j$ is a partial sum of $f$;
or (2):   exactly $p$ elements $\{\bm_{i_1}, \cdots,\bm_{i_p}\}$
with the identical image under $\vi_{[k]}$
and $\bm_{i_1}+ \cdots + \bm_{i_p}$ is a partial sum of $f$.
But, the relation $\bm_{i_1}+ \cdots + \bm_{i_p}$, expressed as
$$\bm_{i_1}+ \cdots + \bm_{i_p}-p \; \bm_{i_1}=
\sum_{i=2}^p  (\bm_{i_a}- \bm_{i_1}),$$ is always generated by 
the relations $\bm_{i_a} -\bm_{i_1}$, for all $2 \le a \le p$.
Thus, regardless of the characteristic of the field $\kk$, it suffices to consider binomials
$\bm -\bm' \in \ker^\mh \vi_{[k]}$.

\subsubsection{$\ker^\mh \vi$: $\wp$-binomials and
$\wp$-reducibility}

\begin{lemma} \label{trivialB}
Fix any $i \in [k]$. We have
\begin{equation}\label{trivialBk} 
x_{\uu'}x_{\uv'}x_{(\uu,\uv)} - x_\uu x_\uv x_{(\uu',\uv')} \in \ker^\mh \bar\vi_{[k]}.
\end{equation}
where $x_{(\uu,\uv)}, x_{(\uu',\uv')}$ are any two distinct $\vr$-variables of $\PP_{F_i}$.
\end{lemma}
\begin{proof} This is trivial.
\end{proof}

\begin{defn}\label{wp-bino}
Any nonzero binomial  as in \eqref{trivialBk} 
 is called a $\wp$-binomial of $R_{[k]}$.
We let $\cB_{[k]}^\wp$ denote the set of all $\wp$-binomials of 
$R_{[k]}$.
\end{defn}

\begin{defn}\label{no common wp factor}
Consider any nonzero binomial 
$f=\bm -\bm' \in \ker^\mh \vi_{[k]}$.
We say that $f$ is $\wp$-reducible if there exists 
$x_{\uu'} x_{\uv'}   x_{(\uu,\uv)} -  x_{\uu}x_{\uv}  x_{(\uu',\uv')} \in \cB^\wp_{[k]}$ 
such that either
$$\hbox{ $x_{\uu'} x_{\uv'}   x_{(\uu,\uv)}  \mid \bm$
 and $x_{(\uu',\uv')} \mid \bm'$}$$
or $$\hbox{    $x_{(\uu,\uv)}  \mid \bm$ and
$x_{\uu}x_{\uv}  x_{(\uu',\uv')} \mid \bm'$.}$$

We say $f$ is $\wp$-irreducible if 
it is not $\wp$-reducible.
\end{defn}

Suppose $f$ is $\wp$-reducible as in the definition, w.l.o.g., say,
$$\hbox{$x_{(\uu,\uv)}  \mid \bm$ \; and \;
$  x_{\uu}x_{\uv}  x_{(\uu',\uv')} \mid \bm'$.}$$
 Then, we can write
$$\bm =x_{(\uu,\uv)}   \bn \;\; \hbox{and} \;\;
\bm'= x_{\uu}x_{\uv}  x_{(\uu',\uv')} \bn'$$ for some $\bn, \bn'$.
Then,  using, 
$x_{\uu'} x_{\uv'}   x_{(\uu,\uv)} -  x_{\uu}x_{\uv}  x_{(\uu',\uv')} \in \cB^\wp_{[k]}$,
 we obtain
\begin{equation}\label{fully recover}
\bm- \bm'=x_{(\uu,\uv)}   \bn -x_{\uu}x_{\uv}  x_{(\uu',\uv')} \bn'
\equiv x_{(\uu,\uv)} (  \bn -x_{\uu'}x_{\uv'} \bn'), \; \mod \cB^\wp_{[k]}.
\end{equation}
Furthermore, knowing that the relation 
$x_{\uu'} x_{\uv'}   x_{(\uu,\uv)} -  x_{\uu}x_{\uv}  x_{(\uu',\uv')}$
is applied in the above, we can fully recover the original form of $\bm-\bm'$.

As usual, 
we say $f$ has no common factor
 if $\bm$ and $\bm'$ do not have any non-constant common factor.

\begin{lemma}\label{abc}
 Let $F_{(123,abc)} \in \sF$ be of rank one. 
Consider a nonzero 
 binomial $\bm -\bm' \in \ker^\mh \vi_{[k]}$,
having no common factor. 
If $x_{(123,abc)}$, respectively, $x_{abc}$ divides $\bm$, then $x_{abc}$, respectively,  $x_{(123,abc)}$ divides $\bm'$.
Consequently, we have that  $\bm -\bm'$  is $\wp$-reducible  and
$$\bm-\bm' 
\equiv x_{(123,abc)} (\bn-\bn'),
\mod \cB_{[k]}^\wp.$$
\end{lemma}
\begin{proof}
Because $f$ admits no common factor,
if  $x_{(123,abc)} \mid \bm$, then $x_{(123,abc)} \nmid \bm'$;
if $x_{abc} \mid \bm$, then $x_{abc} \nmid \bm'$. 
Suppose $x_{(123,abc)}$, respectively,  $x_{abc}$ divides $\bm$, because $x_{(123,abc)}$ is the unique $\vr$-variable such that
its $\vi$-image is dividable by $x_{abc}$, then under the assumption, 
we must have that  $x_{abc}$, respectively,  $x_{(123,abc)}$ divides $\bm'$.
By symmetry, it suffices to consider the case $x_{(123,abc)} \mid \bm$ and   $x_{abc} \mid \bm'$.
Then, we can express
$$\bm -\bm'=x_{(123,abc)}\bn - x_{abc} x_{(\uu, \uv)} \bar\bn' $$
such that $x_{(123,abc)}, x_{(\uu, \uv)}$ are homogeneous coordinates of $\PP_{F_{(123,abc)}}$.
This implies that  $\bm -\bm'$  is $\wp$-reducible  and
$$\bm -\bm'=x_{(123,abc)}( \bn - x_{\uu} x_{\uv} \bar\bn'), \mod \cB_{[k]}^\wp,$$
hence the lemma holds.
\end{proof}

Modulo $\cB^\wp_{[k]}$, 
this allows us to reduce generators of $\ker^\mh \vi_{[k]}$
to $\bn-\bn'$ such that $\vi(\bn)$ $(=\vi(\bn'))$ is not divisible by any rank-1 variable $x_{abc}$.

\begin{lemma}\label{uu uv}
Consider a nonzero binomial $\bm -\bm' \in \ker^\mh \vi_{[k]}$.
Suppose $x_{(\uu,\uv)} \mid \bm$ (resp. $\bm'$) and $x_\uu x_\uv \mid  \bm'$ (resp. $\bm$).
Then, $\bm -\bm'$  is $\wp$-reducible  and
$$\bm-\bm' \equiv x_{(\uu,\uv)} (\bn-\bn'),
\mod \cB_{[k]}^\wp.$$
\end{lemma}
\begin{proof}
It suffices to prove the case when  $x_{(\uu,\uv)} \mid \bm$  and $x_\uu x_\uv \mid  \bm'$. In this case, by the homogeneity of $\bm -\bm'$,
there exists $x_{(\uu',\uv')} $ such that $\{x_{(\uu,\uv)},x_{(\uu',\uv')} \}$
are homogeneous coordinates of $\PP_F$ for some $\bF \in \sF$, and
we can write
$$\bm-\bm'= \bn x_{(\uu,\uv)}  - \bn'x_\uu x_\uv x_{(\uu',\uv')}. $$
This implies that $\bm -\bm'$  is $\wp$-reducible  and
$$\bm-\bm' \equiv \bn x_{(\uu,\uv)}  - \bn' x_{\uu' }x_{\uv'} x_{(\uu,\uv)}
= x_{(\uu,\uv)} (\bn - \bn' x_{\uu' }x_{\uv'}), \mod \cB_{[k]}^\wp.$$
\end{proof}

\begin{cor}\label{wp-irr implies not abc no iuv}
Consider a nonzero binomial $\bm -\bm' \in \ker^\mh \vi_{[k]}$,
having no common factor. 
Suppose $\bm-\bm'$ is $\wp$-irreducible. Then
 $\vi(\bm)$ $(=\vi(\bm'))$ is not divisible by any rank-1 variable $x_{abc}$. Moreover, it is impossible to have
$x_{(123,iuv)} \mid \bm$ and
$x_{iuv}  \mid  \bm'$ nor to have
$x_{(123,iuv)} \mid \bm'$ and
$x_{iuv}  \mid \bm $,
for any rank-0 variable $x_{iuv}$.
\end{cor}
\begin{proof}
It follows by combining Lemmas \ref{abc} and \ref{uu uv}.
\end{proof}

\subsubsection{Descendants and partial descendants of 
multi-homogeneous polynomials}

\begin{defn} 
Let $f=\sum_i x_{(\uu_i, \uv_i)}\bn_i \in R_{[k]} \; (\subset R)$ be a multi-homogeneous 
expression/polynomial,
written as above,  which is allowed to be zero,
such that  $\{x_{(\uu_i, \uv_i)}\}$ are a subset of homogeneous coordinates of $\PP_F$
for some $\bF \in \sF $. Set
$$\barf = \sum_i x_{\uu_i}x_{ \uv_i}\bn_i \in R_{[k]}.$$
Then, we call $f$ a parent of $\barf$ and $\barf$ a descendant of $f$.

Further, if $f$ is a parent of $g$, and $g$ is parent of $h$, then we also say
$f$ is a parent of $h$ and $h$ is a descendant of $f$.
 
Moreover, if a multi-homogeneous polynomial 
$f \in  R_{[k]}$ does not admit any parent other than itself,
then we say $f$ is root polynomial.
In particular, $f$ is a root parent of $g$ if it is a root polynomial
and a parent of $g$.
\end{defn}

One sees that if $f$ belongs to $\ker^\mh \vi_{[k]}$ if and only if  any of its descendant does.

To deal with non-homogeneous decompositions of binomials
of $\ker^\mh \vi_{[k]}$ (cf. \eqref{non-hom decom}),
we also need the following.

\begin{defn} 
Let $f=\sum_{i \in I} x_{(\uu_i, \uv_i)}\bn_i \in R_{[k]} \; (\subset R)$ be a multi-homogeneous expression/polynomial,
written as above,  which is allowed to be zero,
such that  $\{x_{(\uu_i, \uv_i)}\}$
 are a subset of homogeneous coordinates of $\PP_F$
for some $\bF \in \sF$. 
Let $I= I_1 \sqcup I_2$ be a disjoint union.
Set
$$\barf = \sum_{i \in I_1} p_{\uu_i}p_{ \uv_i}\bn_i
+ \sum_{i \in I_2} x_{(\uu_i, \uv_i)}\bn_i \in R_{[k]}
 \in R_{[k]}.$$
Note that $\barf$ is non-homogeneous if $I_1, I_2 \ne \emptyset$.
Then, we call $f$ a parent of $\barf$ and $\barf$ 
a \emph{partial} descendant of $f$.

Further, if $f$ is a parent of $g$, and $g$ is parent of $h$, then we also say
$f$ is a parent of $h$ and $h$ is a partial descendant of $f$.
\end{defn}

One sees that if $f$ belongs to $\ker \vi_{[k]}$ if and only if  any of its 
partial descendant does.

\begin{example} 
$$ x_{(\ua,\ub)} x_{(\uu,\uv)} - x_{\ua'}x_{\ub'} x_{(\uu',\uv')}$$
is a partial descendant of 
$x_{(\ua,\ub)} x_{(\uu,\uv)} - x_{(\ua',\ub')} x_{(\uu,\uv)}$.
\end{example}

\begin{example}\label{exam-expression}
The following homogeneous {\bf expression}, written as 
\begin{equation}\label{expression}
x_{(\uu', \uv')}x_{(\uu,\uv)} - x_{(\uu, \uv)} x_{(\uu',\uv')} 
\end{equation}
(which is zero as a binomial),
descends to the $\wp$-binomial
$$x_{\uu'}x_{\uv'}x_{(\uu,\uv)} - x_\uu x_\uv x_{(\uu',\uv')} .$$
It is not hard to see that a  $\wp$-binomial does not admit 
any non-zero  root parent.

The following homogeneous binomial of $\ker^\mh \vi$
\begin{eqnarray}\nonumber
\;\; \; x_{(12a,13a')}x_{(13a,2bc)}x_{(12a',3\bar b \bar c)}x_{(12 \bar a,3bc)} x_{(13 \bar a,2 \bar b \bar c)} \;
\nonumber \\
- x_{(13a,12a')} x_{(12a,3bc)}x_{(13a',2\bar b \bar c)}x_{(13 \bar a,2bc)} x_{(12 \bar a, 3 \bar b \bar c)}. \nonumber
\end{eqnarray}  
descends to 
\begin{eqnarray}\nonumber
\;\; \; 
x_{12a}x_{13a'}x_{(13a,2bc)}x_{(12a',3\bar b \bar c)}x_{(12 \bar a,3bc)} x_{(13 \bar a,2 \bar b \bar c)} \;
\nonumber \\
- x_{13a}x_{12a'} x_{(12a,3bc)}x_{(13a',2\bar b \bar c)}x_{(13 \bar a,2bc)} x_{(12 \bar a, 3 \bar b \bar c)}.  \nonumber
\end{eqnarray}  
Also, the first binomial as a  root parent is uniquely determined by the second.
\end{example}

\subsubsection{$\fb$-reducibility of binomials of $\ker^\mh \vi_{[k]}$}

\begin{defn}\label{fb-irr} Let 
$f=\bm-\bm' \in \ker^\mh \vi_{[k]}$. 
 We say $f$ is  ${\fb}$-reducible if
 there exists a decomposition
\begin{equation}\label{fb reducible decom}
\bm=\bm_1\bm_2, \;\; \bm'=\bm_1'\bm_2'
\end{equation}
such that all the monomials in the expression are not constant and 
$$\bm_1-\bm_1', \;\; \bm_2 -\bm_2' \in \ker \vi_{[k]}$$
and every of $\bm_1-\bm_1'$ and  $\bm_2 -\bm_2'$
is a descendant or a partial descendant of a multi-homogenous binomial.
Here, when one of the two  
above binomials in the display is zero, say, $\bm_2 -\bm_2'=0$, then we have
 $f=\bm_2(\bm_1-\bm_1')$, as a special case.

We call \eqref{fb reducible decom} a $\fb$-reducible decomposition of
$f$.

We say $f$ is $\fb$-irreducible if it is not $\fb$-reducible.
\end{defn}

By definition, $\bm-\bm'$ is $\fb$-irreducible,
then $\bm$ and $\bm'$ do not have
any (non-constant) common factor.

Here, in \eqref{fb reducible decom}, we emphasize that $\bm_1-\bm_1'$ 
and $\bm_2 -\bm_2'$ are required to
belong to $\ker \vi$ but are not required to be multi-homogeneous.
This can occur in the following situation. Suppose we have
two multi-homogenous binomials
$$ X\bn_1 -  X'\bn_1', \;\; Y \bn_2 -  Y'\bn_2' \in \ker^\mh \vi_{[k]}$$
such that $X, X', Y, Y'$ are homogeneous coordinates of
$\PP_F$ for some $\bF \in \sF$.
Then, the multi-homogenous binomial
$$f= (X\bn_1) (Y\bn_2) -  (X'\bn_1') (Y'\bn_2') \in \ker^\mh \vi_{[k]} $$
can have a \emph{\it multi-homogeneous} descendant,
\begin{equation}\label{non-hom decom}
\barf= (\vi(X)\bn_1) (Y\bn_2) - 
 (X'\bn_1') (\vi(Y')\bn_2') \in \ker^\mh \vi_{[k]} \end{equation}
such that  each of $\vi(X)\bn_1 - X'\bn_1'$
and $Y\bn_2- \vi(Y')\bn_2'$ still belongs to $\ker \vi_{[k]} $
but is not multi-homogeneous.

\begin{lemma}\label{auto homo} Let $f$ 
be as in Definition \ref{fb-irr}. Assume that $f$ is a  root binomial.
Then, in the 
$\fb$-reducible decomposition 
\eqref{fb reducible decom} of $f$, 
\begin{equation}\label{fb reducible decom 2}\nonumber
f= \bm_1\bm_2- \bm_1'\bm_2',
\end{equation}
we automatically have
$$\bm_1-\bm_1', \;\; \bm_2 -\bm_2' \in \ker^\mh \vi_{[k]}.$$
\end{lemma}
\begin{proof}
Suppose not. W.l.o.g., assume 
$\bm_1-\bm_1' \in \ker \vi_{[k]} \- \ker^\mh \vi_{[k]}.$
Then, w.l.o.g.,  we can write
$$\bm_1 -\bm_1'= \vi (X) \bn_1 -X' \bn_1'$$
as a partial descendant of  $X \bn_1 -X' \bn_1'$,
where $X$ and $X'$ are homogeneous coordinates of $\PP_F$ for
some $\bF \in \sF$.
Meanwhile, by definition,
$\bm_2 -\bm_2'$ is also the partial descendant 
$$\bm_2 -\bm_2'=Y \bn_2 -\vi(Y') \bn_2'$$ of a multi-homogeneous binomial
$Y \bn_1 -Y' \bn_1'$,
where $Y$ and $Y'$ are homogeneous coordinates of $\PP_{F'}$ for
some $F' \in \sF$. Because $f$ is multi-homogeneous, we see that 
we must have
$$F' =F.$$
 Therefore, 
$$f= (\vi (X) \bn_1) (Y \bn_2) - (X' \bn_1') (\vi (Y') \bn_2'),$$
 implying that $f$ is the (partial) descendant of
$$\hat f= (X \bn_1) (Y \bn_2) - (X' \bn_1') (Y' \bn_2'),$$
a contradiction to that $f$ is a root binomial.

This establishes the lemma.
\end{proof}

There are many  binomials of
 $\ker^\mh \vi$ that are $\fb$-irreducible and also
$\wp$-irreducible.

\begin{example}\label{exam:Bq} 
The following binomials are $\fb$-irreducible. Since
they are all binomials in $R_\vr$,
they cannot be $\wp$-reducible.

Fix $a,b,c \in [n]$, all being distinct:
\begin{eqnarray} x_{(12b,13c)}  x_{(23b,12c)} x_{(13b,23c)}  \; \nonumber \\
 - x_{(13b,12c)} x_{(12b,23c)} x_{(23b,13c)}.  \label{rk0-0-3'} 
 \end{eqnarray}
 \begin{eqnarray}
x_{(12a,13b)}x_{(13a,12c)}x_{(12b,13c)} \; \nonumber \\
-x_{(13a,12b)}x_{(12a,13c)}x_{(13b,12c)}. \label{rk0-0-3} 
\end{eqnarray}
\begin{eqnarray}
x_{(12a,13b)}x_{(13a,12c)}x_{(12b,23c)} x_{(23b,13c)} \;\nonumber \\
-x_{(13a,12b)}x_{(12a,13c)}x_{(23b,12c)} x_{(13b,23c)}. \label{rk0-0-4}
\end{eqnarray}
\begin{eqnarray}
x_{(123,3bc)}x_{(13a,2bc)}x_{(12b,23c)} x_{(12a,13b)} \;\nonumber \\
-x_{(13b,23c)}x_{(12a,3bc)}x_{(123,2bc)} x_{(13a,12b)}. \label{rk0-1-4}
\end{eqnarray}
Fix $a,b,c, \bar a, \bar b, \bar c \in [n]$, all being distinct:
\begin{eqnarray}
\;\; \;x_{(12a,3bc)}x_{(13a,2\bar b \bar c)}x_{(13 \bar a,2bc)} x_{(12 \bar a,3 \bar b \bar c)}\; \nonumber \\
-x_{(13a,2bc)}x_{(12a,3\bar b \bar c)}x_{(12 \bar a,3bc)} x_{(13 \bar a,2 \bar b \bar c)}. \label{rk1-1}
\end{eqnarray}

Fix $a,b,c, a', \bar a, \bar b, \bar c \in [n]$, all being distinct:
\begin{eqnarray}
\;\; \; x_{(12a,13a')}x_{(13a,2bc)}x_{(12a',3\bar b \bar c)}x_{(12 \bar a,3bc)} x_{(13 \bar a,2 \bar b \bar c)} \;
\nonumber \\
- x_{(13a,12a')} x_{(12a,3bc)}x_{(13a',2\bar b \bar c)}x_{(13 \bar a,2bc)} x_{(12 \bar a, 3 \bar b \bar c)}.\label{rk0-1}
\end{eqnarray} 
These binomials are arranged so that one sees visibly the matching for multi-homogeneity.
\end{example}

More examples of  binomials of
 $\ker^\mh \vi$ that are $\fb$-irreducible but are not
$\wp$-reducible, including ones not belonging to $R_\vr$,
can be found in Examples \ref{exam:frb deg =2} ,
\ref{exam:frb deg >3} and \ref{exam:frb deg >3 not vr}.

\subsubsection{$\vr$-linearity and $\vi$-square-freeness}

For any $\bm -\bm' \in \ker^\mh \vi_{[k]}$, we define 
$$\deg_{\vr} (\bm -\bm')$$
 to be the total degree of $\bm$ (equivalently, $\bm'$)
in $\vr$-variables of $R_{[k]}$.

Observe here that  for any nonzero binomial 
$\bm -\bm' \in \ker^\mh \vi_{[k]}$,
we automatically have $ \deg_\vr (\bm -\bm') > 0$, since $\vi_{[k]}$ restricts
to the identity on $R_0$.

\begin{lemma}\label{rho=1}
Consider any nonzero binomial $\bm -\bm' \in \ker^\mh \vi_{[k]}$ with $ \deg_\vr (\bm -\bm') = 1$.
Then $\bm -\bm'$ is a multiple of $\wp$-binomial.
\end{lemma}
\begin{proof}
Because $\deg_{\vr}(\bm-\bm') =1$, we can write $$\bm -\bm' = f x_{(\uu,\uv)} -g  x_{(\uu',\uv')} $$ 
for some $f, g \in R_0$,
and two $\vr$-variables  
$x_{(\uu,\uv)}$ and $x_{(\uu',\uv')}$ of $\PP_F$
for some $\bf \in \sF$.  
If $x_{(\uu,\uv)} =x_{(\uu',\uv')}$, 
then one sees that $f=g$ and $\bm -\bm'=0$.
Hence, we assume that  $x_{(\uu,\uv)} \ne x_{(\uu',\uv')}$. Then, we have
$$f x_{\uu} x_{\uv}=g  x_{\uu'} x_{\uv'}.$$
Because  $x_{(\uu,\uv)}$ and $ x_{(\uu',\uv')}$
are two distinct $\vr$-variables of $\PP_{\buw}$,  one checks from the definition that the two sets
$$\{ x_{\uu}, x_{\uv} \}, \; \{x_{\uu'}, x_{\uv'} \}$$ are disjoint. 
Consequently, 
$$  x_{\uu}x_{\uv} \mid g, \;\; x_{\uu'} x_{\uv'}  \mid f.$$
Write $$ g=g_1 x_{\uu}x_{\uv} , \;\;  f=f_1 x_{\uu'} x_{\uv'} .$$
Then we have 
$$ x_{\uu}x_{\uv} x_{\uu'} x_{\uv'}  (f_1-g_1)=0 \in R_0.$$
Hence, $f_1=g_1$. Then, we have
$$\bm -\bm' =h ( x_{\uu'} x_{\uv'}   x_{(\uu,\uv)} -  x_{\uu}x_{\uv}  x_{(\uu',\uv')}) $$ 
where $h:=f_1=g_1$. This implies the statement.
\end{proof}

Let $\AA^l$ (resp. $\PP^l$) be the affine (resp. projective)
space of dimension $l$ for some positive integer $l$ with
coordinate variables $(x_1,\cdots, x_l)$ (resp. with
homogeneous coordinates $[x_1,\cdots, x_l]$).
A monomial $\bf m$ is {\it square-free} if 
$x^2$ does not divide $\bf m$ for every coordinate variable $x$ in the affine space.  A polynomial is square-free if all of its monomials are
square-free. Similarly, suppose $\fV$ is a smooth affine variety with local
free variables $(x_1,\cdots, x_l)$, then a monomial $\bf m$ 
in  $(x_1,\cdots, x_l)$
is square-free if 
$x^2$ does not divide $\bf m$ for every variable $x$ in  
$(x_1,\cdots, x_l)$.  A polynomial in  $(x_1,\cdots, x_l)$ is square-free if all of its monomials are
square-free.

\begin{defn}\label{vi sq free}
A polynomial $f$ of $R$ is called $\vi$-square-free if
for any monomial summand $\bm$ of $f$, $\vi(\bm)$ is square-free.
\end{defn}

Recall that a polynomial $f \in R_{[k]} \ (\subset R)$
 is $\vr$-linear if it is linear
in the homogeneous coordinates of $\PP_F$ for any $\bF \in \sF$ whenever 
some homogeneous coordinate of $\PP_F$ divides a monomial summand of $f$.

\begin{example}\label{exam:frb deg =2} 
The following is a $\fb$-irreducible root binomial of $\ker^\mh \vi$.
$$ x_{1bc} x_{2b'c'} x_{(13a,2bc)} x_{(23a,1b'c')}-$$
$$x_{2bc}  x_{1b'c'} x_{(23a,1bc)} x_{(13a,2b'c')}$$
for suitable choices of $a,b,c, b',c'$  not belonging to $\{1,2,3\}$.

The following example
$$ x_{23a} x_{13b}x_{(13a,2ef)} x_{(23b,1ef)} -$$
$$x_{13a} x_{23b} x_{(23a,1ef)} x_{(13b,2ef)}$$
is a $\fb$-irreducible binomial of $\ker^\mh \vi$ but
has a root-parent
$$ x_{(23a,13b)} x_{(13a,2ef)} x_{(23b,1ef)} -$$
$$x_{(13a,23b)}  x_{(23a,1ef)} x_{(13b,2ef)}.$$

Note that both examples are $\vr$-linear and $\vi$-square-free.
\end{example}

We will show, in Lemma \ref{rho=2} and
Corollar \ref{cor:not used}, that any 
$\fb$-irreducible  binomial of $\ker^\mh \vi$
of $\deg_\vr=2$ takes form like one of the above.

To be used below, we introduce the following.
Let $\uu=(u_1u_2u_3), \uv=(v_1v_2v_3) \in \II_{3,n}$.
We then let
$$\uu \vee \uv :=\{u_1,u_2,u_3, v_1,v_2,v_3 \}$$
as a set of six integers, allowing repetition. Thus, if
$\uu \vee \uv = \ua \vee \uv$, then $\uu=\ua$.
Observe that if $\{ x_{(\uu_1,\uv_1)},  x_{(\uu_2,\uv_2)} \}$
 are $\vr$-variables of $\PP_F$ for some $\bF \in \sF$, then
$\uu_1 \vee \uv_1=\uu_2 \vee \uv_2$.

Consider any nonzero binomial 
$\bm -\bm' \in \ker^\mh \vi_{[k]}$ with $ \deg_\vr (\bm -\bm') = 2$.
Because $\deg_{\vr}(\bm-\bm') =2$, we can write 
\begin{equation}\label{rho=2 form}
\bm -\bm' = f x_{(\ua_1,\ub_1)}x_{(\uu_1,\uv_1)} -g x_{(\ua_2,\ub_2)} x_{(\uu_2,\uv_2)} 
\end{equation}
for some $f, g \in R_0$ and 
$\{x_{(\ua_1,\ub_1)},x_{(\ua_2,\ub_2)}\}$
($\{ x_{(\uu_1,\uv_1)},  x_{(\uu_2,\uv_2)} \}$) are $\vr$-variables of $\PP_F$ for some $\bF \in \sF$ ($\bF' \in \sF$).

\begin{lemma}\label{rho=2}
Consider any nonzero binomial 
$\bm -\bm' \in \ker^\mh \vi_{[k]}$ with $ \deg_\vr (\bm -\bm') = 2$.
We write $\bm -\bm'$ as in \eqref{rho=2 form}.
Assume  
$\bm$ and $\bm'$ do not have a common factor,
and $\bm-\bm'$ is $\wp$-irreducible.
Then, w.l.o.g.,  more precisely, 
up to writing $x_{(a_i, b_i)}$ and/or  $x_{(u_i, v_i)}$ 
as $x_{(b_i, a_i)}$ and/or  $x_{(v_i, u_i)}$, 
$\bm -\bm' $ takes of the form
$$\bm -\bm' =x_{\uv_2} x_{\ub_2} x_{(\ua_1,\ub_1)}x_{(\ua_2,\uv_1)} 
-x_{\uv_1}x_{\ub_1} x_{(\ua_2,\ub_2)} x_{(\ua_1,\uv_2)} $$ 
such that $\uv_2 \vee \ub_2, \ua_1 \vee \ub_1, \ua_2 \vee \uv_1$
are pairwise distinct. In particular, $\bm -\bm' $ is $\vr$-linear.
Moreover, $\bm -\bm' $ is also  $\vi$-square-free.
\end{lemma}
\begin{proof}
Consider 
$$ f x_{\ua_1}x_{\ub_1}x_{(\uu_1,\uv_1)} -g x_{\ua_2} x_{\ub_2} x_{(\uu_2,\uv_2)} .$$ 
Because $\bm-\bm'$ has no common factor, we must have
$x_{(\ua_1,\ub_1)} \ne x_{(\ua_2,\ub_2)}$, equivalently, 
$$\{ x_{\ua_1}, x_{\ub_1} \} \cap \{ x_{\ua_2}, x_{\ub_2} \}=\emptyset,$$
 for otherwise $\bm$ and $\bm'$ have a common factor,
contradicting to the assumption;
$ x_{\ua_1}x_{\ub_1} \nmid g$, for otherwise 
$ x_{\ua_1} x_{\ub_1}x_{(\ua_2,\ub_2)} \mid \bm'$,
contradicting to that $\bm-\bm'$ is $\wp$-irreducible.
This leaves only one possibility: 
$\{ x_{\ua_1}, x_{\ub_1} \} \cap \{ x_{\uu_2}, x_{\uv_2} \}$ is a singleton,
w.l.o.g., say, 
$$\hbox{$ x_{\ua_1}=x_{\uu_2}$, and then $x_{\ub_1} \mid g$.}$$
By symmetry, the similar arguments when applied to 
$x_{\ua_2} x_{\ub_2}$ will give us, 
w.l.o.g., say, 
$$\hbox{$ x_{\ua_2}=x_{\uu_1}$, and then $x_{\ub_2} \mid f$.}$$
Hence,
$$\bm -\bm' = {\bar f}x_{\ub_2} x_{(\ua_1,\ub_1)}x_{(\ua_2,\uv_1)} 
-{\bar g}x_{\ub_1} x_{(\ua_2,\ub_2)} x_{(\ua_1,\uv_2)} .$$ 
Because $\bm -\bm' \in \ker^\mh \vi_{[k]}$, we obtain
$${\bar f} x_{\uv_1}={\bar g} x_{\uv_2}.$$
Because $\bm -\bm' $  has no common factor, we obtain
$${\bar f}=x_{\uv_2}, \; {\bar g} =x_{\uv_1}.$$
That is
$$\bm -\bm' =x_{\uv_2} x_{\ub_2} x_{(\ua_1,\ub_1)}x_{(\ua_2,\uv_1)} 
-x_{\uv_1}x_{\ub_1} x_{(\ua_2,\ub_2)} x_{(\ua_1,\uv_2)} .$$ 
Because $\ua_1 \vee \ub_1 = \ua_2 \vee \ub_2$ and $\ua_2 \vee \uv_1 =\ua_1 \vee \uv_2$,
we obtain 
$$\ua_1 \vee \ub_1 \vee \ua_2 \vee \uv_1 = \ua_2 \vee \ub_2 \vee \ua_1 \vee \uv_2,$$
hence $$ \ub_1  \vee \uv_1 =  \ub_2 \vee \uv_2.$$
This implies $\bm -\bm' $ is a descendant of 
$$x_{(\uv_2, \ub_2)} x_{(\ua_1,\ub_1)}x_{(\ua_2,\uv_1)} 
-x_{(\uv_1, \ub_1)} x_{(\ua_2,\ub_2)} x_{(\ua_1,\uv_2)}, $$
provided that $x_{(\uv_2, \ub_2)}$ and
$x_{(\uv_1, \ub_1)}$ are homogeneous coordinates of $\PP_F$ for some
$\bF \in \sF$.

We now prove that $\uv_2 \vee \ub_2 \ne \ua_1 \vee \ub_1.$
Assume on the contrary
that $\uv_2 \vee \ub_2 = \ua_1 \vee \ub_1$. Then 
$\uv_2 \vee \ub_2 = \ua_1 \vee \ub_1 =\ua_2 \vee \ub_2$,
hence $\uv_2= \ua_2$. But then $\ua_1 = \uv_1$
since $\ua_2 \vee \uv_1 =  \ua_1 \vee \uv_2$,
implying a common factor $x_{(\uv_1, \uv_2)}$,  equal to both
$x_{(\ua_2,\uv_1)}$ and $x_{(\ua_1,\uv_2)}$.
Similarly, $\uv_2 \vee \ub_2 \ne \ua_2 \vee \uv_1.$
Finally, suppose $\ua_1 \vee \ub_1=\ua_2 \vee \uv_1.$
Then, $\ua_1 \vee \ub_1=\ua_1 \vee \uv_2$, hence
$\ub_1=\uv_2$, implying a common factor 
$x_{(\ua_1,\ub_1)}=x_{(\ua_1,\uv_2)}$ of $\bm-\bm'$,
a contradiction. Consequently, $\bm -\bm'$ is $\vr$-linear.

Assume on the contrary  that 
$\bm -\bm'$ is not $\vi$-square-free. For example, 
Assume 
$$\{\uv_2, \ub_2 \} \cap \{\ua_1,\ub_1\} \ne \emptyset.$$
Because of $x_{(\ua_1,\uv_2)}$,  we see that $\uv_2 \ne \ua_1$.
If $\uv_2=\ub_1$, then $\uv_2 \vee \ub_2=\uv_1\vee \ub_1$
implies $\ub_2=\uv_1$. Then $x_{(\ub_1, \ub_2)}$ is a common factor 
of $\bm-\bm'$, a contradiction. Because 
$\ua_1 \vee \ub_1=\ua_2 \vee \ub_2$, we see that 
$\ub_2 \ne \ua_1,\ub_1$ because either case would imply 
a common factor $x_{(\ua_1,\ub_1)}$. Hence, we conclude
$\{\uv_2, \ub_2 \} \cap \{\ua_1,\ub_1\} = \emptyset.$
All the remaining cases follow from similar routine checks.

This establishes the lemma.
\end{proof}

\begin{cor}\label{cor:not used} Let the notation and assumption
be as in Lemma \ref{rho=2}.
Assume in addition that $x_{(\uv_2, \ub_2)}$ and
$x_{(\uv_1, \ub_1)}$ are homogeneous coordinates of $\PP_F$ for some
$\bF \in \sF$.
Then, we have $$\bm -\bm'=\vi(X_1)X_2X_3-\vi(X_1')X_2'X_3'$$
where $\{X_i, X_i' \}$ are $\vr$-variables of $\PP_{F_i}$ 
for some $\bF_i \in \sF$ 
such that $\bF_1, \bF_2$ and $\bF_3$ are mutually distinct.
 In particular, 
 the root parent $X_1X_2X_3 - X_1'X_2'X_3'$  of $\bm -\bm'$ is $\vr$-linear and $\vi$-square-free.
\end{cor}

\begin{example}\label{p-reducible}
Consider the following binomial $f$:
$$x_{13a}x_{2bc}x_{(13b,23c)}x_{(12a,3bc)}-$$
$$x_{12a}x_{13b}x_{23c}x_{(123,3bc)}x_{(13a,2bc)}.$$
This binomial belongs to $\ker^\mh (\vi)$, is
$\fb$-irreducible, but $\wp$-reducible.
Modulo $\cB^\wp$,
using $x_{13a}x_{2bc}x_{(12a,3bc)} -x_{12a}x_{3bc} x_{(13a,2bc)}$,
 $f$ can be reduced to
$$x_{12a}x_{(13a,2bc)}(x_{3bc}x_{(13b,23c)}
-x_{13b}x_{23c}x_{(123,3bc)}).$$
Or, modulo $\cB^\wp$, using 
$x_{13b}x_{23c}x_{(123,3bc)} - x_{3bc}x_{(13b,23c)}$,
 it can also be reduced to
$$x_{(13b,23c)}(x_{13a}x_{2bc}x_{(12a,3bc)}
-x_{12a}x_{3bc}x_{(13a,2bc)}).$$
In addition, observe that the binomial $f$ is a root binomial.
\end{example}

Indeed, below, we will prove that any 
$\fb$-irreducible but $\wp$-reducible 
binomial $\bm-\bm' \in \ker^\mh  \vi_{[k]}$ of
$\deg_\vr =2$ takes a form like the above.

\begin{lemma}\label{rho=2 wp-reducible}
Suppose $\bm-\bm' \in \ker^\mh  \vi_{[k]}$ with
$\deg_\vr (\bm-\bm')=2$ is $\fb$-irreducible and $\wp$-reducible.
Then, it takes of the following form
$$x_{\uu'}x_{\uv'} x_{\ua'} x_{(\ua,\ub)} x_{(\uu, \uv)}-
x_{\ua}x_{\ub} x_\uu x_{(\ua',\uv)} x_{(\uu', \uv')}$$
where $x_{(\ua,\ub)}, x_{(\ua',\uv)}$ are homogeneous coordinates
of $\PP_F$ for some $\bF \in \sF$ and 
$x_{(\uu,\uv)}, x_{(\uu',\uv')}$ are homogeneous coordinates
of $\PP_{F'}$ for some $\bF' \ne \bF \in \sF$. In particular,
$\bm-\bm'$ is $\vr$-linear. Moreover, $\bm-\bm'$ is also
$\vi$-square-free.
\end{lemma}
\begin{proof} Because $\bm-\bm'$ is $\wp$-reducible, w.l.o.g., we can assume
$$\bm-\bm'= f x_{(\ua,\ub)} x_{(\uu, \uv)}-
g x_{\ua}x_{\ub}  x_{(\ua',\ub')} x_{(\uu', \uv')}$$
for some $f, g \in R_0$.
Then, we have 
$$f  x_{(\uu, \uv)}- g x_{\ua'}x_{\ub'}   x_{(\uu', \uv')} \in \ker^\mh \vi_{[k]}.$$
By Lemma \ref{rho=1}, the above must be equal to 
$$hx_{\uu'}x_{\uv'} x_{(\uu, \uv)}-  hx_{\uu}x_{\uv}   x_{(\uu', \uv')}$$
for some $h \in R_0$. Hence, by comparing the coefficients, 
we obtain
$$f=hx_{\uu'}x_{\uv'} ,\; g x_{\ua'}x_{\ub'}=hx_{\uu}x_{\uv}.$$
Because $h \mid f$, using that $\bm-\bm'$ is $\fb$-irreduicble 
(hence having no common factor),
we have  that $h$ and $g$ are co-prime,  
therefore $h \mid x_{\ua'}x_{\ub'}$. Observe that
$h$ can not equal to $x_{\ua'}x_{\ub'}$, for otherwise, $x_{\ua'}x_{\ub'} \mid f$, and we obtain 
$$\hbox{$x_{\ua'}x_{\ub'}  x_{(\ua,\ub)} \mid \bm$ and
$x_{\ua}x_{\ub}  x_{(\ua',\ub')} \mid \bm'$},$$
then this would imply that $\bm -\bm'$ is $\fb$-reducible, 
a contradiction.
W.l.o.g., we assume $h =x_{\ua'}$. Then $g x_{\ub'} = x_{\uu}x_{\uv}$.
W.l.o.g., we can assume $x_{\ub'}=x_\uv$. 
Then, we obtain $$f=x_{\ua'}x_{\uu'}x_{\uv'} ,\; g=x_{\uu},
\; \ub'=\uv.$$
This implies 
$$\bm -\bm'=x_{\uu'}x_{\uv'} x_{\ua'} x_{(\ua,\ub)} x_{(\uu, \uv)}-
x_{\ua}x_{\ub} x_\uu x_{(\ua',\uv)} x_{(\uu', \uv')}$$

Suppose $\bF =\bF'$ as in the statement.
Then, we have $\uu \vee \uv=\ua' \vee \uv$, hence $\uu=\ua'$,
implying a common factor $ x_{(\uu, \uv)}=x_{(\ua',\uv)}$,
contradicting to that $\bm -\bm'$ is $\fb$-irreducible.
This implies that $\bm-\bm'$ is $\vr$-linear.

To show that $\bm-\bm'$ is $\vi$-square-free, 
we offer two different proofs.

(1).  Suppose that $\bm-\bm'$ is not $\vi$-square-free, 
w.l.o.g., say, $\ua=\uu$. Then, we have
$$\ua \vee \ub = \ua' \vee \uv \;\; \hbox{and} \;\;
  \ua \vee \uv=\uu' \vee \uv'.$$
Hence, 
$$\ua \vee \ub \vee  \ua \vee \uv= \ua' \vee \uv \vee \uu' \vee \uv'.  $$
Thus, $$\ua \vee \ub \vee  \ua=\ua' \vee \uu' \vee \uv'.  $$
Note that $$\ua' \subset \ua \vee \ub=\ua' \vee \uv.$$
Therefore, as sets of  integers (allowing repetition), we have
\begin{equation}\label{wrong id}
\{\ua \vee \ub \vee  \ua\} \; \- \{\ua'\}= \uu' \vee \uv'
\end{equation}
Observe also that as a set of  integers (allowing repetition), 
$\ua'$ is made of exactly half of
$\ua \vee \ub=\ua' \vee \uv$, hence as a set of 
\emph{\it distinct} integers, we see that 
$$\ua \vee \ub \vee  \ua \- \{\ua'\}$$
contains at most 3 distinct integers. 
But, by the definition,
$$\uu' \vee \uv' \supset \{1,2,3\}$$
contains at least $3+2$ distinct integers, a contradiction to the identity
\eqref{wrong id}. The remaining cases such as 
$\ua'=\uu'$, etc., are totally parallel.

(2). For an alternative proof,
we begin with analysing
$x_{(\uu, \uv)}$ in  $\bm$ and $x_{(\ua',\uv)}$ in $\bm'$.
First, observe that $\uv$ (or equivalently $x_\uv$) can not be of rank-1,
for otherwise, both $x_{(\uu, \uv)}$ and $x_{(\ua',\uv)}$ have to equal
to $x_{(123, \uv)}$, implying a common factor, a contradiction.
Therefore, 
up to permutations within $\{1,2,3\}$ for sub-indexes of some variables, 
w.l.o.g., we can assume
$$x_{(\uu, \uv)}=x_{(12a, 3bc)} \;\; \hbox{and} \;\; 
x_{(\ua',\uv)}=x_{(12a', 3bc)}.$$
Here, we allow $a$ or $a'$ equal to 3, and allow $b$ equals to 1 or 2,
thus exhausting all possibilities, up to permutations within $\{1,2,3\}$
for sub-indexes of some variables.\footnote{One could consider
the cases $$x_{(\uu, \uv)}=x_{(13a, 2bc)} \;\; \hbox{and} \;\; 
x_{(\ua',\uv)}=x_{(13a', 3bc)},$$
$$x_{(\uu, \uv)}=x_{(23a, 1bc)} \;\; \hbox{and} \;\; 
x_{(\ua',\uv)}=x_{(23a', 1bc)},$$
independently. But, up to permutations within $\{1,2,3\}$,
they are all parallel.}
Hence, up to permutations within $\{1,2,3\}$ for sub-indexes of some variables, w.l.o.g., we obtain
$\bm-\bm'$ takes the form of 
\begin{equation}\label{wp-red 1}
x_{13a}x_{2bc} x_{12a'} x_{(23a',1bc)} x_{(12a, 3bc)}-
x_{12a}x_{23a'} x_{1bc} x_{(12a', 3bc)} x_{(13a, 2bc)}
\end{equation}
when $a' \ne 3$; or else, when $a'=3$,
\begin{equation}\label{wp-red 2}
x_{13a}x_{2bc}  x_{(13b,23c)} x_{(12a, 3bc)}-
x_{12a}x_{13b} x_{23c} x_{(123, 3bc)} x_{(13a, 2bc)}
\end{equation}
From here, one checks that it is impossible to get two identical 
sub-indexes from $\bm$ (or respectively, from $\bm'$) without yielding a contradiction.

This establishes the lemma.
\end{proof}

Note that the binomial \eqref{wp-red 2}
is the same as the one in Example \ref{p-reducible}

\begin{lemma}\label{rho=2 fb reducible}
Suppose $\bm-\bm' \in \ker^\mh  \vi_{[k]}$ with
$\deg_\vr (\bm-\bm')=2$ has no common factor 
and is $\fb$-reducible.
Then, we have
$$\bm-\bm'=\bm_1\bm_2 - \bm_1' \bm_2'$$
such that $\bm_i-\bm'_i \in \ker^\mh \vi$, $i=1,2$.
Consequently, we have $\deg_\vr (\bm_i-\bm_i')=1$, $i=1,2$, and
by Lemma \ref{rho=1}, each of 
$\bm_i-\bm'_i \in \ker^\mh \vi$, $i=1,2$, is a $\wp$-binomial.
\end{lemma}
\begin{proof}
Assume $\bm-\bm'$ is $\fb$-reducible, and is written as
$$\bm-\bm'=\bm_1\bm_2 - \bm_1' \bm_2'.$$
A priori, there are the following four possibilities.
\begin{enumerate}
\item $\bm_i-\bm'_i \in \ker^\mh \vi$, $i=1,2$.
Then, necessarily we have $\deg_\vr (\bm_i-\bm'_i) =1$, $i=1,2$.
Then by Lemma \ref{rho=1}, each of 
$\bm_i-\bm'_i \in \ker^\mh \vi$, $i=1,2$, is a $\wp$-binomial.
\item both $\bm_i-\bm'_i$, $i=1,2$, are partial descendants of
binomials of Lemma \ref{rho=2}.
\item both $\bm_i-\bm'_i$, $i=1,2$, are partial descendants of
binomials of Lemma \ref{rho=2 wp-reducible}.
\item one of $\bm_i-\bm'_i$, $i=1,2$, is a partial descendant of
binomials of Lemma \ref{rho=2}, the other is a partial descendant of
binomials of Lemma \ref{rho=2 wp-reducible}.
\end{enumerate}

But, one quickly checks from Lemma \ref{rho=2} and 
Lemma \ref{rho=2 wp-reducible} that any partial descendant
of a binomial from these two lemmas will give rise to a common
factor for the descended binomial, this contradicts that
$\bm-\bm'$ admits no common factor. Hence, this excludes 
(2) - (4) and leaves (1) as the only possible case, which is the statement of the lemma.
\end{proof}

\begin{lemma}\label{vi X stays} 
Consider any binomial $f \in \ker^\mh \vi_{[k]}$.
Assume $f= \vi(X) \bm - \vi(X') \bm'$ 
is a descendant of $X \bm - X' \bm'$ where
$X$ and $X'$ are two \sout{distinct}
homogeneous coordinates of $\PP_F$ for some 
$\bF \in \sF$. Suppose 
$$f = \bm_1 \bm_2 - \bm_1' \bm_2'$$ is a
$\fb$-reducible decomposition of $f$. Then, $\vi(X) \mid \bm_1$
or $\vi(X) \mid \bm_2$. Similarly, $\vi(X') \mid \bm_1'$
or $\vi(X') \mid \bm_2'$.
\end{lemma} 
\begin{proof}
We apply induction on $\deg_\vr(\bm -\bm' )$.

By  applying Lemma \ref{rho=1} (for $\deg_\vr=1$), 
Lemmas \ref{rho=2}, 
 \ref{rho=2 wp-reducible}, and  \ref{rho=2 fb reducible}
(for $\deg_\vr=2$), we see that the lemma
holds when $\deg_\vr(\bm -\bm' ) \le 2$.

Assume that $\deg_\vr(\bm -\bm' )=e$
and the lemma holds when $\deg_\vr =e-1$.
Then, we can write $f$ as
$$f= \vi(X) Y \bn - \vi(X') Y' \bn'$$ such that
$Y$ and $Y'$ are homogeneous coordinates of $\PP_{F'}$ for some
$\bF' \in \sF$.
In the decomposition $f = \bm_1 \bm_2 - \bm_1' \bm_2'$,
we have
\begin{enumerate}
\item either $Y$ and $Y'$ appear in exactly one of $\bm_1  - \bm_1' $
and  $\bm_2 - \bm_2'$, w.l.o.g., say, in $\bm_1  - \bm_1' $,
\item or else, w.l.o.g., say, $Y$ appears in $\bm_1$ and $Y'$ appears in $\bm_2'$.
\end{enumerate}
In the former case, we write $\bm_1 = Y \bn_1$ and $\bm_1'=Y'\bn_1'$,
Then, we have
$$f =(Y \bn_1) \bm_2 - (Y' \bn_1') \bm_2'.$$
Note that $\vi(X) \mid \bn_1 \bm_2$ and $\vi(X') \mid \bn_1' \bm_2'.$
Consider $$\barf=(\vi(Y) \bn_1) \bm_2 - (\vi(Y') \bn_1') \bm_2'.$$ 
Because $\deg_\vr \barf =e-1$, by induction, we have
$\vi(X) \mid \bn_1 \mid \bm_1$ or $\vi(X) \mid  \bm_2$. Likewise,
$\vi(X') \mid \bn_1' \mid \bm_1$ or $\vi(X') \mid  \bm_2'$. This implies
the statement.

In the latter case, we write $\bm_1 = Y \bn_1$ and $\bm_2'=Y'\bn_2'$,
Then, we have
$$f =(Y \bn_1) \bm_2 - \bm_1' (Y' \bn_2') $$
Note that $\vi(X) \mid \bn_1 \bm_2$ and $\vi(X') \mid \bm_1' \bn_2'.$
Consider $$\barf=(\vi(Y) \bn_1) \bm_2 -  \bm_1' (\vi(Y')\bn_2').$$ 
Because $\deg_\vr \barf =e-1$, by induction, we have
$\vi(X) \mid \bn_1 \mid \bm_1$ or $\vi(X) \mid  \bm_2$. Likewise,
$\vi(X') \mid \bm_1' $ or $\vi(X') \mid \bn_2' \mid  \bm_2'$. This implies
the statement in this case.

All in all, the lemma is proved.
\end{proof}

\begin{lemma}\label{linear and free} 
Consider any binomial $f \in \ker^\mh \vi_{[k]}$.
Assume $f$ is either  not $\vr$-linear or not $\vi$-square-free.
Then $f$ is $\fb$-reducible.
\end{lemma}
\begin{proof}
We apply induction on $\deg_\vr(\bm -\bm' )$.

Let $f \in \ker^\mh \vi_{[k]}$ be a binomial.
Assume $f$ is either not $\vr$-linear or not $\vi$-square-free.

By  applying Lemmas \ref{rho=1}, \ref{rho=2}, 
 \ref{rho=2 wp-reducible}, and  \ref{rho=2 fb reducible},
we see that the lemma
holds when $\deg_\vr(\bm -\bm' ) \le 2$.

Suppose now $\deg_\vr f  =e > 2$.
We assume that the lemma holds when $\deg_\vr   =e-1$.

We can express
$f$ as
$$f= XY \bm - \hbox{(the other term)}$$ such that
either $X$ and $Y$ are homogeneous coordinates of $\PP_F$
for  some $\bF \in \sF$  when $f$ is not $\vr$-linear,
or else, $X$ and $Y$ are two variables  such that
$\vi(XY)$ is not square-free when $f$ is not $\vi$-square-free.
Because $\deg_\vr f  > 2$, there are $Z$ and $Z'$
such that $Z$ and $Z'$ are homogeneous coordinates of $\PP_{F'}$ 
for  some $\bF' \in \sF$ (which may be equal to the $\bF$ above), and we can write 
$$f= XY Z\bn - Z' \bn'$$
for some $\bn$ and $\bn'$.
Then, we consider
$$\barf= XY \vi(Z) \bn -  \vi(Z') \bn' .$$
Clearly, $\barf$ is either  
not $\vr$-linear, or else, not $\vi$-square-free. 
Because $\deg_\vr \barf  =e  -1$, by induction, we
have a $\fb$-reducible decomposition
$$\barf=  (X \bn_1) (Y\bn_2) - (\bn_1') (\bn_2')  $$
such that $X \bn_1 -\bn_1'$ and $Y\bn_2 -  \bn_2'$, 
 are (partial) descendant of a multi-homogenous binomial, 
and belong to $\ker\vi_{[k]}$.
By Lemma \ref{vi X stays}, applied to the $\fb$-reducible decomposition 
$\barf = (X \bn_1) (Y\bn_2) - (\bn_1') (\bn_2') $,
we obtain that $\vi(Z)$ and $\vi(Z')$ appears in exactly one of
$ X \bn_1 -\bn_1'$ and  $Y\bn_2 - \bn_1'$, w.l.o.g., say. in 
$ X \bn_1 -\bn_1'$, or else, w.l.o.g., say, 
$\vi (Z)$ appears in $ X \bn_1$ and $\vi(Z')$ appears in $\bn_2'$.

In the former case, we write  $\bn_1 = \vi(Z) \bn_{11}$ and
$\bn_1'= \vi(Z') \bn_{11}'$. Then, we have
$$\barf = (X \vi(Z) \bn_{11}) (Y\bn_2) - (\vi(Z')\bn_{11}') (\bn_2') .$$
Hence, 
$$f = (X Z \bn_{11}) (Y\bn_2) - ( Z'\bn_{11}') (\bn_2') $$
which implies the desired statement.

In the latter case, we write  $\bn_1 = \vi(Z) \bn_{11}$ and
$\bn_2'= \vi(Z') \bn_{21}'$. Then, we have
$$\barf = (X \vi(Z) \bn_{11}) (Y\bn_2) - (\bn_1') (\vi(Z')\bn_{21}') .$$
Hence, 
$$f = (X Z \bn_{11}) (Y\bn_2) - (\bn_1') (Z'\bn_{21}') .$$
which also implies the desired statement.

All in all, the lemma is proved.
\end{proof}

\begin{cor}\label{cor:linear and free}
Suppose a binomial $f \in \ker^\mh \vi_{[k]}$  is $\fb$-irreducible.
Then it is both $\vr$-linear and $\vi$-square-free.
\end{cor}

\begin{rem}
We say $\bm-\bm \in \ker^\mh \vi_{[k]}$ is homogeneously
$\vi$-square-free if for any element $\uu \in \II_{3,n}$, including
$\uu=(123)$, $\uu$ as a sub-index of a $\vp$-variable
$x_\uu$ or a $\vr$-variable $x_{(\uu, \uv)}$, can only appear in
$\bm$, and respectively in $\bm'$, at most once.
For instance, $x_{(123, 1bc)}x_{(123, 3uv)} \mid \bm$
or $\mid \bm'$ is not
homogeneous $\vi$-square-free.

Corollary \ref{cor:linear and free} admits 
the following stronger version:
Suppose a binomial $f \in \ker^\mh \vi_{[k]}$  is $\fb$-irreducible.
Then it is both $\vr$-linear and 
homogeneously $\vi$-square-free.
By  applying Lemmas \ref{rho=1}, \ref{rho=2}, 
 \ref{rho=2 wp-reducible}, and  \ref{rho=2 fb reducible},
one sees that the above holds when $\deg_\ve \le 2$.
Then the remaining proof of Lemma \ref{linear and free},
carries over verbatim in this case.

But we do not need the above in this article.
\end{rem}

\begin{rem}\label{vr-linear} 
Let $f=\bm-\bm' \in \ker^\mh \vi_{[k]}$. 
Assume $f$ is  $\wp$-irreducible but not $\vr$-linear.
Then $f$ is not square-free in $\vr$-variables.
Consequently, being $\wp$-irreducible and 
square-free implies $\vr$-linear.

By Lemma \ref{abc},  since $f$ is  $\wp$-irreducible
there does not exists $3<a<b<c$ 
such that $x_{(123, abc)} \mid \bm\bm'$. 
Hence, for any fixed $\bF \in \sF$, at most three homogeneous coordinates of $\PP_F$ can appear in the monomials $\bm, \bm'$. 
Suppose $f$ is not linear in the homogeneous coordinates of $\PP_F$ for some $\bF \in \sF$. Then,
having (counting multiplicities) at least two  homogeneous coordinates of $\PP_F$ in  $\bm$ and at least two in $\bm'$,
we see that $f$ is not square-free in the $\vr$-variables of $\PP_F$.
\end{rem}

\subsection{$\ker^\mh \vi$:
$\wp$- and $\frb$-binomials,  and their properties} $\ $

In this subsection, we introduce $\frb$-binomials:
it will be shown in Lemma \ref{equas-for-sVk}
 that $\ker^\mh \vi$ can be reduced to
 $\frb$-binomials, modulo $\cB^\wp$.

\begin{cor}\label{reduce to rb}
Let $f=\bm-\bm'$ be any  root binomial in $\ker^\mh \vi_{[k]}$.
Then it is generated by $\fb$-irreducible  root binomials 
in $\ker^\mh \vi_{[k]}$.

Moreover, if $f$ is $\wp$-irreducible, then it is generated by 
$\fb$-irreducible  root binomials that are also $\wp$-irreducible.
\end{cor}
\begin{proof} 
By  applying Lemmas \ref{rho=1}, \ref{rho=2}, 
 \ref{rho=2 wp-reducible}, and  \ref{rho=2 fb reducible},
we see that the corollary
holds when $\deg_\vr(\bm -\bm' ) \le 2$.

Take any binomial
$(\bm -\bm') \in \ker^\mh \vi_{[k]}$
 with $\deg_\vr (\bm -\bm') > 2$. 
By applying Lemma \ref{auto homo}, we can express
$$\bm -\bm' =\prod_{i=1}^\ell \bm_i -  \prod_{i=1}^\ell \bm_i',$$ 
such that $\bm_i -\bm'_i \in \ker^\mh \vi_{[k]}$ and
are $\fb$-irreducible root binomials,   for all $i \in [\ell]$.
Then, we have
$$\bm -\bm'=\prod_{i=1}^{\ell} \bm_i - \bm_\ell'  \prod_{i=1}^{\ell-1} \bm_i
+ \bm_\ell'  \prod_{i=1}^{\ell-1} \bm_i -\prod_{i=1}^{\ell} \bm_i',
$$ 
$$=(\bm_\ell - \bm_\ell')  \prod_{i=1}^{\ell-1} \bm_i
+ \bm_\ell'  (\prod_{i=1}^{\ell-1} \bm_i -\prod_{i=1}^{\ell-1} \bm_i').$$ 
Thus, by a simple induction on the integer $\ell$, 
we conclude that $f$ is generated by 
$\bm_i -\bm_i' \in \ker^\mh \vi_{[k]}$, $i \in [\ell]$, 
which are $\fb$-irreducible.

Moreover, 
since $\bm -\bm' =\prod_{i=1}^\ell \bm_i - 
 \prod_{i=1}^\ell \bm_i'$,
it follows that if any $\bm_i -\bm_i'$ is $\wp$-reducible, so is $\bm-\bm'$.
That is, if $\bm -\bm'$ is $\wp$-irreducible, 
so are all $\bm_i -\bm_i'$, $i \in [\ell]$.

This establish the corollary.
\end{proof}

\begin{defn} \label{defn:frb}
A $\fb$-irreducible root binomial $\bm-\bm'$ that is also $\wp$-irreducible
is called an $\frb$-binomial. We let 
 $\cB^\frb_{[k]}$ be the set of all $\frb$-binomials of $\ker^\mh \vi_{[k]}$.  
\end{defn}

We summarize the properties of 
$\frb$-binomials below.

\begin{prop}\label{strong sq free}
Let $f=\bm-\bm' \in \ker^\mh \vi$
be an $\frb$-binomial, then we have 
\begin{enumerate}
\item  $f$ is $\vr$-linear and $\vi$-square-free. 
\item $x_{abc} \nmid \bm \bm', \;\; x_{(123,abc)} \nmid \bm \bm'$
for all $3<a<b<c$.
\item $x_{(123, iuv)} x_{iuv} \nmid \bm \bm'$
for all $i \in [3]$.
\end{enumerate}
\end{prop}
\begin{proof} 
(1) follows from  Corollary \ref{cor:linear and free}
because $f$ is $\fb$-irreducible.
 (2)  follows from Lemma \ref{abc} because $f$ is $\wp$-irreducible.

For (3), we assume $x_{(123, iuv)} x_{iuv} \mid \bm \bm'$ for the contrary.
If $x_{(123, iuv)} x_{iuv} \mid \bm$ or  $\bm'$, then because of (1), it contradicts that $\bm -\bm'$ is $\vi$-square-free.
Thus, by assumption, we can assume
$$x_{(123, iuv)} \mid  \bm, \;\; x_{iuv} \mid \bm'
\;\; \hbox{or} \;\; x_{(123, iuv)} \mid  \bm', \;\; x_{iuv} \mid \bm.
$$
Then by Lemma \ref{uu uv}, $f$ is $\wp$-reducible,
a contradiction.  
\end{proof}


Examples of $\frb$-binomials  of $\deg_vr =2$ can be found in
Example \ref{exam:frb deg =2}.   
Example of $\frb$-binomials  with $\deg_vr >2$ are given below.

\begin{example}\label{exam:frb deg >3} 
Examples of $\frb$-binomials  of
 $\deg_\vr \ge 3$ belonging to $R_\vr$.

Fix $a,b,c \in [k]$, all being distinct:
\begin{eqnarray}
x_{(12a,13b)}x_{(13a,12c)}x_{(12b,13c)} \; \nonumber \\
-x_{(13a,12b)}x_{(12a,13c)}x_{(13b,12c)}. \label{rk0-0} 
\end{eqnarray}
\begin{eqnarray}
x_{(12a,13b)}x_{(13a,12c)}x_{(12b,23c)} x_{(23b,13c)} \; \nonumber \\
-x_{(13a,12b)}x_{(12a,13c)}x_{(23b,12c)} x_{(13b,23c)}. \label{rk0-0-4}
\end{eqnarray}

Fix $a,b,c, \bar a, \bar b, \bar c \in [k]$, all being distinct, satisfying $a, \bar a < b<c$ and $a, \bar a < \bar b<\bar c$:
\begin{eqnarray}
\;\; \;x_{(12a,3bc)}x_{(13a,2\bar b \bar c)}x_{(13 \bar a,2bc)} x_{(12 \bar a,3 \bar b \bar c)} \; \nonumber \\
-x_{(13a,2bc)}x_{(12a,3\bar b \bar c)}x_{(12 \bar a,3bc)} x_{(13 \bar a,2 \bar b \bar c)}.  \label{rk1-1}
\end{eqnarray}
Fix $a,b,c, a', \bar a, \bar b, \bar c \in [k]$, all being distinct, satisfying $a, \bar a < b<c$ and $a', \bar a < \bar b<\bar c$:
\begin{eqnarray}
\;\; \; x_{(12a,13a')}x_{(13a,2bc)}x_{(12a',3\bar b \bar c)}x_{(12 \bar a,3bc)} x_{(13 \bar a,2 \bar b \bar c)} 
\; \nonumber \\
- x_{(13a,12a')} x_{(12a,3bc)}x_{(13a',2\bar b \bar c)}x_{(13 \bar a,2bc)} x_{(12 \bar a, 3 \bar b \bar c)}.\label{rk0-1}
\end{eqnarray} 
These binomials are arranged so that one sees visibly the matching for multi-homogeneity.
\end{example}

\begin{example}\label{exam:frb deg >3 not vr} 
Examples of $\frb$-binomials of
 $\deg_\vr \ge 3$ but not belonging to $R_\vr$.

Fix $a,b,c, \bar a, \bar b, \bar c \in [k]$, all being distinct, satisfying $\bar a <a< b<c$ and $\bar a ,b< \bar b<\bar c$:
\begin{eqnarray}
\;\; \;x_{12b} x_{3ac}x_{(13b,2\bar b \bar c)} x_{(12 \bar a,3 \bar b \bar c)} x_{(13 \bar a,2ac)}  \; \nonumber \\
-x_{13b} x_{2ac} x_{(12b,3\bar b \bar c)}  x_{(13 \bar a,2 \bar b \bar c)}x_{(12 \bar a,3ac)} .  \label{rk1-1}
\end{eqnarray}
Here, the only obstruction for the above binomial to admit
 a root-parent in $R_\vr$  is the inequality $a<b$.

Fix $a,b,c, a', \bar a, \bar b, \bar c \in [k]$, all being distinct, satisfying $\bar a <a < b<c$ and $a', \bar a< \bar b<\bar c$:
\begin{eqnarray}
 x_{13b}  x_{2ac} x_{(12b, 13a')} x_{(12a', 3\bar b \bar c)} x_{(13 \bar a, 2 \bar b \bar c)} x_{(12 \bar a,3ac)}   \; \nonumber \\
- x_{12b} x_{3ac}  x_{(13b, 12a')} x_{(13a', 2\bar b \bar c)} x_{(12 \bar a, 3 \bar b \bar c)} x_{(13 \bar a,2ac)} .\label{rk0-1}
\end{eqnarray} 
Here, the only obstruction for the above binomial to admit
 a root-parent in $R_\vr$  is the inequality $a<b$.
\end{example}

\subsection{The non-toroidal part of the defining ideal   
of $\sV$ in $\sR$: $\ker^\mh \vi_\Gr$ } $\;$

 We now investigate the  defining ideal of $\sV$ in $\sR$ in this subsection.

Corresponding to the embedding $\sV_{[k]} \subset \sR_{[k]}$ of \eqref{embed sV}, 
 we have the homomorphism
   \begin{equation}\label{bar-vi-k-gr}
\vi_{[k],\Gr}: \; R_{[k]}  \lra R_0/\bar I_\wp, \;\;\;
\vi_{[k],\Gr}|_{R_0}=\Id_{R_0},
\end{equation} 
$$  x_{(\uu_s,\uv_s)} \to x_{\uu_s} x_{\uv_s} +\bar I_\wp$$
for all  $s \in S_{F_i}, i \in [k]$, where $\bar I_\wp$ is the de-homogenization of the $\pl$ ideal $I_\wp$
with respect to the chart $(p_{123} \equiv 1)$, i,e., $\bar I_\wp$ is the ideal of $R_0$ generated by the relations in $\sF$.

  \begin{lemma}\label{defined-by-ker-Gr} 
The scheme $\sV_{[k]}$, as a closed subscheme of 
 $\sR_{[k]}= \bU \times  \prod_{i \in [k]} \PP_{F_i}$,
 is defined by $\ker^\mh \vi_{[k], \Gr}$.
\end{lemma}
\begin{proof} This is immediate.
\end{proof}

\subsubsection{Linearized $\pl$ relations}

We need to investigate   
$\ker^\mh \vi_{[k], \Gr}$.

 We let $f \in R_{[k]}$ be any multi-homogenous polynomial.
  Then, by \eqref{bar-vi-k-gr}, $\vi_{[k], \Gr}( f)=0$
 if and only if $\vi_{[k]} (f) \in \bar I_\wp$.
As $\ker^\mh \vi_{[k]}$ has been studied, we can assume
$f \notin \ker^\mh \vi_{[k]}$.
 Thus, we can express $f=\sum_{\bF \in \sF} f_F$ such that 
 $\vi_{[k]} (f_F)$ is a multiple of $\bF$ for all $\bF \in \sF$. It suffices to consider an
 arbitrarily fixed $\bF \in \sF$. Hence, we may assume that $f=f_F$ for some arbitrarily fixed $\bF \in \sF$.
 That is, in such a case,  $\vi_{[k], \Gr} ( f)=0$ if and only if
  $$\hbox{$ \vi_{[k]} (f) =h \bF$  for some $h \in R_{[k]}$.}$$

The definition below introduces a class of obvious but important members of 
$\ker^\mh \vi_{[k], \Gr}$. 

\begin{defn} \label{defn:linear-pl} Given any $\bF \in \sF$, written as 
$\bF =\sum_{s \in S_F} \sgn (s) x_{\uu_s}x_{\uv_s}$, we introduce
$$L_F: \;\; \sum_{s \in S_F} \sgn (s) x_{(\uu_s,\uv_s)} .$$ 
This is called the linearized $\pl$ relation with respect to $\bF$ (or $F$). It is a
canonical linear relation on $\PP_F$.
\end{defn}

We let
\begin{equation}\label{LsF} 
L_\sF=\{L_F \mid \bF \in \sF \} \end{equation}
be the set of all linearized $\pl$ relations.

In more concrete terms, the following are all the  linearized $\pl$ relations
\begin{eqnarray}\nonumber 
L_{(123),1uv}=x_{(123,1uv)}-x_{(12u,13v)} + x_{(13u,12v)}, \label{rk0-1uv}\\ \nonumber
L_{(123),2uv}=x_{(123,2uv)}-x_{(12u,23v)} + x_{(23u,12v)}, \label{rk0-2uv} \\ \nonumber
L_{(123),3uv}=x_{(123,3uv)}-x_{(13u,23v)} + x_{(23u,13v)} ,\label{rk0-3uv}\\ \nonumber
L_{(123),abc}=x_{(123,abc)}-x_{(12a,3bc)} + x_{(13a,2bc)} -x_{(23a,1bc)}. \label{rk1-abc} \nonumber
\end{eqnarray}
  
  Observe here that among all linearized $\pl$ relations, only  
  $L_{F_i}$ with $i \in [k]$ belong to $\bar R_{[k]}$. Clearly, $L_{F_i} \in 
  \ker^\mh \vi_{[k], \Gr}$.

  
  Fix any $\bF \in \sF$.  Let $f \in  R_{[k]}$ be any multi-homogenous polynomial such that 
$\vi_{[k]} (f) =h \bF$  for some $h \in R_0$.
 We then write $\bF=\sum_{s\in S_F} \sgn (s) x_{\uu_s}x_{\uv_s}$. Accordingly, we express
  \begin{equation}\label{initial f}
f=\sum_{s \in S_F} \sgn (s) f_s
\end{equation} 
such that 
 \begin{equation}\label{initial f 2} \hbox{$\vi_{[k]} (f_s) =h x_{\uu_s}x_{\uv_s}$, for all $s \in S_F$.}
\end{equation}

\begin{lemma} \label{f-to-F} 
 Suppose $\vi_{[k]} (f)=h\bF$ for some  $h \in R_{[k]}$ and
$$\hbox{ $x_{\uu_s}x_{\uv_s} \mid f_s$ for some $s \in S_F$.}$$
 Then, $$\hbox{ $f \equiv \bn \bF, \mod (\ker^\mh \vi_{[k]})$, for some $\bn \in R_{[k]}$.}$$
\end{lemma}
\begin{proof}
We prove the  statement by applying induction on $\deg_\vr (f)$.

When $\deg_\vr (f)=0$, then by a direct inspection, we must have
$f=h \bF$. Hence, the statement holds.

Assume $\deg_\vr (f)=e>0$ and the statement holds for $\deg_\vr (f)=e-1$.

Then, we  can write
 $$f= X_s T_s' \sgn(s) x_{\uu_s}x_{\uv_s} + \sum_{t \in S_F \- s} \sgn(t) X_t T_t $$
 for some $T_s' \in R$ and $T_t \in R$ for all $t \in S_F \-s$ such that
 $\{X_t \mid t \in S_F\}$ is a subset of the homogenous coordinates  of $\PP_G$
 for some $G\in \sF$.
 
 Consider 
 $$\bar  f= \vi_{[k]}(X_s) T_s'  \sgn(s) x_{\uu_s}x_{\uv_s} + \sum_{t \in S_F \- s} \sgn(t) \vi_{[k]}(X_t) T_t.$$
Then, we have  $\vi_{[k]} (\bar f)= \vi_{[k]} (f)=h\bF$.
 By the inductive assumption, applied to $\bar f$, we have 
 $$\hbox{ $\bar f \equiv \bar\bn \bF, \mod (\ker^\mh \vi_{[k]})$,
 for some $\bar\bn \in  \bar R_{[k]}$.}$$
 This implies that 
 $$ \vi_{[k]}(X_s) T_s'  \equiv \bar\bn, \; \vi_{[k]}(X_t) T_t \equiv \bar\bn x_{\uu_t}x_{\uv_t},
 \mod (\ker^\mh \vi_{[k]}), \; \forall \; t \in S_F \- s.$$
 Substituting $\bar\bn$ by $\vi_{[k]}(X_s) T_s' $ in the second equality above, we have
  $$\vi_{[k]}(X_t) T_t \equiv \vi_{[k]}(X_s) T_s' x_{\uu_t}x_{\uv_t}, \mod (\ker^\mh \vi_{[k]}) , \; \forall \; t \in S_F \- s.$$
  Consequently,
   $$X_t T_t \equiv X_s T_s' x_{\uu_t}x_{\uv_t}, \mod (\ker^\mh \vi_{[k]}), \; \forall \; t \in S_F \- s.$$
 Hence, we obtain
 $$f 
  \equiv X_s T_s' \sgn (s) x_{\uu_s}x_{\uv_s} + \sum_{t \in S_F \- s} X_s T_s' \sgn (t) x_{\uu_t}x_{\uv_t}
  = X_s T_s' \bF, \mod (\ker^\mh \vi_{[k]}).$$
  Thus, the lemma follows by induction.
\end{proof}

If $\bF=\sum_{t \in S_F} \sgn (t) x_{\uu_t}x_{\uv_t}$, then we have
$L_F=\sum_{t \in S_F} \sgn (t) x_{(\uu_t,\uv_t)}$.

\begin{lemma} \label{f-to-LF} 
 Suppose $\vi_{[k]} (f)=h\bF$ for some  $h \in R_{[k]}$ and
$$\hbox{ $x_{(\uu_s, \uv_s)} \mid f_s$ \; for some $s \in S_F$.}$$
 Then, $$\hbox{ $f \equiv \bn L_F, \mod (\ker^\mh \vi_{[k]})$, 
for some $\bn \in \bar R_{[k]}$.}$$
\end{lemma}
\begin{proof}
By assumption, we  can write
 $$f=    \sgn(s) x_{(\uu_s, \uv_s)} T_s+ \sum_{t \in S_F \- s} \sgn(t) X_t T_t $$
 for some $T_t \in R$ for all $t \in S_F$ such that
 $\{x_{(\uu_s, \uv_s)} , X_t \mid t \in S_F \- s\}$ 
 is a subset of the homogenous coordinates  of $\PP_G$
 for some $G \in \sF$.
 
 Consider 
 $$\bar  f= \sgn(s)  x_{\uu_s}x_{\uv_s}  T_s+ \sum_{t \in S_F \- s} \sgn(t) \vi_{[k]}(X_t) T_t. $$
 By Lemma \ref{f-to-F}, we have 
 $$\bar f \equiv \bn F, \mod (\ker^\mh \vi_{[k]}), \; \hbox{for some $\bn \in R_{[k]}$}.$$
  This implies that 
 $$ T_s \equiv \bn, \; \vi_{[k]}(X_t) T_t \equiv x_{\uu_t}x_{\uv_t} \bn,
 \mod (\ker^\mh \vi_{[k]}) , \; \forall \; t \in S_F \- s.$$
 Observe that $X_t$ and $x_{(\uu_t, \uv_t)}$ are both homogenous coordinate of $\PP_G$.
 Consequently, 
 $$X_t T_t \equiv  x_{(\uu_t, \uv_t)} \bn,
 \mod (\ker^\mh \vi_{[k]}) , \; \forall \; t \in S_F \- s.$$
 Therefore, 
  $$f  =    \sgn(s) x_{(\uu_s, \uv_s)} T_s+ \sum_{t \in S_F \- s} \sgn(t) X_t T_t $$
 $$ \equiv \sgn(s) x_{(\uu_s, \uv_s)} \bn+ \sum_{t \in S_F \- s} \sgn(t) x_{(\uu_t, \uv_t)} \bn
= \bn L_F,  \mod (\ker^\mh \vi_{[k]}).$$
By induction, the lemma holds.
\end{proof}

\begin{lemma}\label{f-abc} Suppose $\vi(f) = h F$ for some $F=F_{(123,abc)}$ of rank-1. Let $f_1$ be the term of $f$ giving rise
to $h x_{abc}$.
Then either $x_{abc} \mid f_1$ or $x_{(123,abc)} \mid f_1$.
Correspondingly, either $f \equiv \bn F$ or $f \equiv \bn L_F,
\mod (\ker^\mh \vi_{[k]})$.
\end{lemma}
\begin{proof}
The first statement is immediate 
because  $x_{(123,abc)} $ is the unique $\vr$-variable whose $\vi$-image contains the factor $x_{abc}$.
The last statement follows from  Lemma \ref{f-to-F} and Lemma \ref{f-to-LF}, correspondingly.
\end{proof}

\subsubsection{$\vp\vr$-$\pl$ relations}

\begin{example}\label{this ex}
The following exmaple shows that for $F$ of rank-0, e.g., 
$F_{(123, 1bc)}$, there exists $f \in R$ such that $\vi(f) = h F$ for some $h$ but
$$\hbox{neither $f \equiv \bn F$ nor $f \equiv \bn L_F,
\mod (\ker^\mh \vi)$.}$$
\begin{equation}\label{that ex}
x_{13a}x_{23b}x_{23c}x_{(23a,1bc)}x_{(123,2bc)}-
x_{13c}x_{23a}x_{23b}x_{(13a,2bc)}x_{(12b,23c)}+
x_{13b}x_{23a}x_{23c}x_{(13a,2bc)}x_{(23b,12c)}.
\end{equation}
\end{example}

\begin{defn}\label{defn:h-pl}
Let  $f \in R_{[k]}$. Assume $f \notin L_\sF$ with $\deg_\vr f>0$ be as in \eqref{initial f} such that
 $\vi (f) =h F$ for some $\bF \in \sF$ and all the terms of $f$ as
in the expression of \eqref{initial f} are co-prime, 
modulo $\ker^\mh \vi_{[k]}$,
then we say $f$ is a $\hpl$ relation.

We denote the set of all $\hpl$  relations  in  $R_{[k]}$
by  $\sF^{\vp\vr}_{[k]}$.
\end{defn}

 \begin{cor}\label{cB-LF-generate}
The ideal  $ \ker^\mh \vi_{[k],\Gr}$ is generated by 
$\ker^\mh \vi_{[k]}$,  and all the relations in $\sF$,
$\{L_{F_i} \mid i \in [k]\}$, and
$\sF^{\vp\vr}_{[k]}$.
\end{cor}  
\begin{proof}
This follows directly from the above discussions.
\end{proof}

\begin{prop}\label{h-pl}
 Let $f \in  R_{[k]}$ be any $\vp\vr$-relation
such that $\vi(f)=hF$ for some $h$.
 Then, following the notation above, we have
$$ 
\hbox{$x_{(\uu_t,\uv_t)} f_s - x_{(\uu_s,\uv_s)} f_t \in 
\ker^\mh \vi_{[k]}$} $$
for all $s, t \in S_F$.
\end{prop}
\begin{proof}
This follows directly from \eqref{initial f 2}.
\end{proof}

\subsubsection{Reducing $\ker^\mh \vi_\Gr$ 
 to $\cB^\wp, \cB^\frb$, and $L_\sF$} $\ $
More precisely, below, we will
\begin{itemize}
\item reduce $\ker^\mh \vi$ to $\cB^\wp, \cB^\frb$;
\item reduce  $\sF$ and $\sF^{\vp\vr}$ to $L_\sF$ modulo 
   $\ker^\mh \vi$.
\end{itemize}

As in \eqref{list sF}, we can list all relations of $\sF$ as
$$\bF_1< \bF_2 <\cdots < \bF_\up.$$

\begin{defn}\label{fv-k=0} Fix $k \in [\up]$.
For every $\bF_i \in \sF$ with $i \in [k]$, choose and fix an arbitrary element 
$s_{F_i, o}\in S_{F_i}$.
Then, the scheme $\sR_{[k]}$ is covered by the affine open charts
of the form  $$\bU \times \prod_{i \in [k]}
(x_{(\uu_{s_{F_i, o}}, \uv_{s_{F_i, o}})} \equiv 1)
\subset \sR_{[k]}= \bU \times  \prod_{i \in [k]} \PP_{F_i} .$$ 
We call such an affine open subset a standard chart of
$\sR_{[k]}$, often denoted by $\fV$.
\end{defn}

Given any standard chart $\fV$ as in the definition above, we let 
$$\fV'=\bU \times \prod_{i \in [k-1]}
(x_{(\uu_{s_{F_i, o}}, \uv_{s_{F_i, o}})} \equiv 1)
\subset \sR_{\sF_{[k-1]}}= \bU \times  \prod_{i \in [k-1]} \PP_{F_i} .$$ 
Then, this is a standard chart of $\sR_{\sF_{[k-1]}}$, uniquely determined by $\fV$.
We say $\fV$ lies over $\fV'$. In general, 
suppose $\fV''$ is a standard chart of $\sR_{\sF_{[j]}}$ with $j <k-1$. Via induction,
we say $\fV$ lies over $\fV''$ if $\fV$ lies over some $\fV'$ as above and
$\fV'$ lies over $\fV''$.

Note that the standard chart $\fV$ of
$\sR_{[k]}$ in the above definition is uniquely indexed 
by the set 
\begin{equation}\label{index-sR}
 \La_{ [k]}^o=\{(\uu_{s_{F_i,o}},\uv_{s_{F_i,o}}) \in \La_{F_i} \mid i \in [k] \}
 \end{equation}
where 
\begin{equation}\label{LaFi}\nonumber
 \La_{F_i}=\{(\uu_s, \uv_s) \mid s \in S_{F_i}\}
\end{equation}
 is  the index set for all the homogeneous coordinates in $\PP_{F_i}$.
Given $ \La_{[k]}^o$, we let 
$$ \La_{[k]}^\star=(\bigcup_{i \in [k]}\La_{F_i}) \- \La_{[k]}^o.$$
We set 
$$\La_\sF^o:=\La_{\sF_{[\up]}}^o \;\;\hbox{and} \;\;\La_\sF^\star:=\La_{\sF_{[\up]}}^\star.$$

Recall from
Definition \ref{defn:frb}, $\cB_{[k]}^\frb$ is the set of
root binomials that are both $\fb$-irreducible and $\wp$-irreducible.

\begin{lemma}\label{equas-for-sVk}
 The scheme $\sV_{[k]}$, as a closed subscheme of
$\sR_{[k]}= \bU \times  \prod_{i=1}^k \PP_{F_i} $,
is defined by the following relations in
$$\cB^\wp_{[k]}, \cB^\frb_{[k]} \; 
\{L_{F_i} ,  i \in [k]\}, \; \{\bF_j,  k < j\le \up\}. $$
\end{lemma}
\begin{proof}  
By Corollary \ref{cB-LF-generate},
it suffices to show the following
\begin{enumerate}
\item
$\ker^\mh \vi_{[k]}$ can be reduced to $\cB_{[k]}^\frb$,
modulo $\cB^\wp_{[k]}$;
\item 
$\sF^{\vp\vr}_{[k]}$ 
can be reduced to $\{L_{F_i} ,  i \in [k]\}$,
\item  $\bF_i$  can also be reduced to $L_{F_i}$ for all $i \in [k]$.
\end{enumerate}

(1). We first reduce $\ker^\mh \vi_{[k]}$ to $\cB_{[k]}^\frb$,
modulo $\cB^\wp_{[k]}$;

\noindent
{\bf Claim 1.} Suppose
a binomial $f =\bm-\bm' \in \ker^\mh \vi_{[k]}$
is not a root binomial. Then $f$
can be reduced to a root binomial in
 $\ker^\mh \vi_{[k]}$.

By definition \ref{fv-k=0}, 
$\sR_{[k]}=\bU \times  \prod_{i=1}^k \PP_{F_i} $
is covered by standard affine charts.
Fix any standard chart $\sR_{[k]}$ as in Definition \ref{fv-k=0}. 
We want to reduce   $f|_\fV$  to a root binomial in $\ker^\mh \vi_{[k]}$, restricted to the chart $\fV$.

We prove it by induction on $\deg_\vp (f)$, 
where $\deg_\vp f$ denote the degree of $f$ considered as
a polynomial in $\vp$-variables only

When  $\deg_\vp (f) =0$, the statement holds trivially.

Assume that statement holds for $\deg_\vp< e$ for some $e>0$.

Consider $\deg_\vp (f) = e$.

 By definition, since $f$ is not a root binomial,    we can write 
$$f= x_{\uu_s} x_{\uv_s} \bn_s - x_{\uu_t} x_{\uv_t} \bn_t$$
with  $x_{(\uu_s,\uv_s)}$ and $x_{(\uu_t,\uv_t)}$
being  homogeneous coordinates of $\PP_{F_i}$ for some 
$i \in [k]$, and for some $\bn_s, \bn_t \in R_{[k]}$.

Consider the following two relations in $\cB_{[k]}^\wp$
$$x_{\uu_s} x_{\uv_s} 
x_{(\uu_{s_{F_i, o}}, \uv_{s_{F_i, o}})}-
x_{\uu_{s_{F_i, o}}}x_{\uv_{s_{F_i, o}}} x_{(\uu_s, \uv_s)}, $$
$$
x_{\uu_t} x_{\uv_t} x_{(\uu_{s_{F_i, o}}, \uv_{s_{F_i, o}})}- x_{\uu_{s_{F_i, o}}}x_{\uv_{s_{F_i, o}}}x_{(\uu_t, \uv_t)}.$$
This implies that
$$ x_{(\uu_{s_{F_i, o}}, \uv_{s_{F_i, o}})}f= 
x_{\uu_{s_{F_i, o}}}x_{\uv_{s_{F_i, o}}}(x_{(\uu_s, \uv_s)} \bn_s - x_{(\uu_t, \uv_t)} \bn_t),
\; \mod  \cB_{[k]}^\wp.$$
Note that on the chart $\fV$,  
$x_{(\uu_{s_{F_i, o}}, \uv_{s_{F_i, o}})}$ 
is invertible (see Definition \ref{fv-k=0}).
Also, observe that 
$(x_{(\uu_s, \uv_s)} \bn_s - 
x_{(\uu_t, \uv_t)} \bn_t) \in  \ker^\mh \vi_{[k]}$
because $f \in  \ker^\mh \vi_{[k]}$.
Since  
$$\deg_\vp (x_{(\uu_s, \uv_s)} \bn_s - x_{(\uu_t, \uv_t)} \bn_t) < e,$$
the statement then follows from the inductive assumption.

{\bf Claim 2.} Any $\wp$-reducible 
root binomial $f =\bm-\bm' \in \ker^\mh \vi_{[k]}$
 can be reduced to $\wp$-irreducible root binomials,
modulo $\cB^\wp$.

Since $f$ is $\wp$-reducible, by using the notation immediately above
\eqref{fully recover}, we obtain
\begin{equation}\label{fully recover 2}
f=\bm- \bm'=x_{(\uu,\uv)}   \bn -x_{\uu}x_{\uv}  x_{(\uu',\uv')} \bn'
\equiv x_{(\uu,\uv)} (  \bn -x_{\uu'}x_{\uv'} \bn'), \; \mod \cB^\wp.
\end{equation}
Because, $f$ is a root binomial, one sees that 
$g=\bn -x_{\uu'}x_{\uv'} \bn'$ must also be a root binomial.
Now if $g$  is $\wp$-irreducible, then we are
done. Otherwise, we  repeat and apply the same procedure to
$g$. As $\deg_\vr g < \deg_\vr f$, the process must terminate.
Thus, the claim holds.

Then, by applying  Corollary \ref{reduce to rb},
  we obtain that 

{\bf Claim 3.} Any $\wp$-irreducible
 root binomial $f =\bm-\bm' \in \ker^\mh \vi_{[k]}$
can be reduced to $\fb$-irreducible
 root binomials in $\ker^\mh \vi_{[k]}$
that are also $\wp$-irreducible, that is,
 reduced to $\frb$-binomials of $\ker^\mh \vi_{[k]}$.

Combining Claims 1, 2, and 3,
we obtain that the defining equations of $\ker^\mh \vi_{[k]}$
can be reduced to $\cB^\wp_{[k]} \sqcup \cB^{\frb}_{[k]}$.
This proves (1).

(2).
Next, we reduce 
$\sF^{\vp\vr}_{[k]}$ to $\{L_{F_i} ,  i \in [k]\}$.
Take any $f \in \sF^{\vp\vr}_{[k]}$  as in Lemma \ref{h-pl}.
Fix any $s \in S_F$.
By  Lemma \ref{h-pl},  for any $t \in S_F$, we have
$$ \hbox{$x_{(\uu_t,\uv_t)} f_s - x_{(\uu_s,\uv_s)} f_t \in 
\ker^\mh \vi_{[k]}$}. $$
Thus, from $f=\sum_{t \in S_F} \sgn (t) f_t$, we obtain
$$x_{(\uu_s,\uv_s)} f = \sum_{t \in S_F} \sgn (t) x_{(\uu_s,\uv_s)} f_t
\equiv \sum_{t \in S_F} \sgn (t) x_{(\uu_t,\uv_t)} f_s = f_s L_F, 
\; \mod \ker^\mh \vi_{[k]} .$$
Since $\PP_{F_i}$ can be covered by affine open charts 
 $(x_{(\uu_s,\uv_s)} \ne 0)$, $s \in S_F$, 
we conclude that  $f$ depends on $L_F$,
modulo $\ker^\mh \vi_{[k]}$, or, now equivalently,
modulo $\cB_{[k]}^\wp  \sqcup \cB_{[k]}^\frb$.

(3). 
We now reduce $\bF_i$  to $L_{F_i}$ for all $i \in [k]$.

Fix any $i \in [k]$. Take and fix any 
$(\uu, \uv) \in \{(\uu_s, \uv_s) \mid s \in S_{F_i}\}$. 
Consider the binomial relations of $\cB_{[k]}^\wp$
  \begin{equation}\label{Buv-s-1st}
 x_{(\uu,\uv)} x_{\uu_s} x_{\uv_s} - 
x_{\uu} x_{\uv}x_{(\uu_s, \uv_s)}  , \;\; \hbox{for all $s \in S_{F_i}$.}
  \end{equation}
  By multiplying $\sgn (s)$ to  \eqref{Buv-s-1st}
and adding them all together,  
we obtain, \begin{equation}\label{Fi=Li-1st}
   x_{(\uu,\uv)} \bF_i =x_{\uu} x_{\uv} L_{F_i}, \mod (\cB_{[k]}^\wp),
 \end{equation}
 Since $\PP_{F_i}$ can be covered by affine open charts 
 $(x_{(\uu,\uv)} \ne 0)$, $(\uu, \uv) \in \La_{F_i}$, we conclude that  $\bF_i$ depends on $L_{F_i}$, modulo $\cB_{[k]}^\wp$,
for all $i \in [k]$.

All in all, the lemma is proved.
\end{proof}


As before, we let
\begin{equation}\label{LsF} \nonumber
L_\sF=\{L_F \mid \bF \in \sF \} \end{equation}
be the set of all linearized $\pl$ relations.

By the case of Lemma \ref{equas-for-sVk} when $k=\up$, we have

\begin{cor}\label{eq-tA-for-sV}  
The scheme $\sV$, as a closed subscheme of
$\sR= \bU \times  \prod_{\bF \in \sF} \PP_F$,
is defined by the following relations in 
$$\cB^\wp , \cB^\frb , \; L_\sF. $$
\end{cor}

\subsection{ Governing generating relations} $\;$

\subsubsection{We begin with setting and recalling some notations.}

For any $\bF \in \sF$,  we let 
\begin{equation}\label{LaF}
 \La_F=\{(\uu_s, \uv_s) \mid s \in S_F\}.
\end{equation}
This is  an index set for all the homogeneous coordinates in $\PP_F$.
For  later use, 
 we also set  $$ \La_{\sF} =\sqcup_{\bF \in \sF} \La_F.$$
For any $\bF \in \sF$, we let
\begin{equation}\label{cB1F}
\cB_F^\wp=\{x_{\uu'} x_{\uv'}  x_{(\uu,\uv)} - x_{\uu} x_{\uv} x_{(\uu',\uv')} \mid (\uu,\uv), 
(\uu',\uv') \in  \La_F\}
\end{equation}
Thus, we have
$$ \cB^\wp =\bigcup_{\bF \in  \sF} \cB_F^\wp.$$
 Further,  we set
\begin{equation}\label{cBk all}
 \cBk =\cB_{[k]}^\wp \sqcup \cB_{[k]}^\frb.
\end{equation}
When $k=\up$, we write
$$ \cB =\cB_{[\up]}^\wp \sqcup \cB_{[\up]}^\frb.$$

\subsubsection{Governing binomial relations}

Given $\bF \in \sF$, recall from  \eqref{the-form-LF},
 $\uu_F$ is the index for the leading variable $x_{\uu_F}$ of the primary
$\pl$ relation $\bF$.
\begin{defn}
We call $x_{((123), \uu_F)}$ the leading $\vr$-variable for any $\bF \in \sF$.
\end{defn}

\begin{defn}\label{gov-b}
A binomial relations of $\cB^\wp$ is said to be governing if it contains a leading $\vr$-variable, equivalently it contains a leading $\vp$-variable.
\end{defn}

\begin{defn}\label{gov-ngov-bi}
Fix any $\bF \in \sF$. 
We let  $\cB^\gov_F$ be the set of governing binomials of $\cB_F^\wp$ 
that contains the leading $\vr$-variable $x_{((123), \uu_F)}$. These are called governing binomials 
with respect to $\bF$. We  set $\cB_F^\ngv= \cB_F  \- \cB_F^\gov$.
The binomials of  $\cB_F^\ngv$ are called non-governing binomials with respect to $\bF$.
\end{defn}

Typically in a proof we write governing binomials of $\cB^\gov_F$ as
\begin{equation}\label{gv-bi}
B_{F, s}: \;
x_{(\uu_s, \uv_s)}x_{\uu_F} - x_{(\um, \uu_F)}x_{\uu_s}x_{\uv_s},
\; s \in S_F \- s_F
\end{equation}
and we write non-governing binomials of $\cB^\ngv_F$ as
\begin{equation}\label{ngv-bi}
B_{F, (s,t)}: \;
x_{(\uu_s, \uv_s)}x_{\uu_t}x_{\uv_t} - x_{(\uu_t, \uv_t)}x_{\uu_t}x_{\uu_s}x_{\uv_s},
\; s, t \in S_F \- s_F
\end{equation}

We let 
\begin{equation}\label{cBk-gov}
\cB_{[k]}^\gov =\sqcup_{i \in [k]} \cB^\gov_{F_i}
\end{equation}
 and then set
\begin{equation}\label{cBk-ngv}\cB_{[k]}^\ngv=\cB_{[k]}^\wp  \- \cB_{[k]}^\gov.
\end{equation}
 We let
$$\cBgov=\cB_{[\up]}^\gov,\;\; \cBngv=\cB_{[\up]}^\ngv, \;\; \cB^\frb=\cB^\frb_{[\up]}.$$
Then, we have
\begin{equation}\label{cB all}
\cB=\cB^\gov \sqcup \cB^\ngv \sqcup \cB^\frb.
\end{equation}
Although a binomial of $\cB^\frb$ is also non-governing, we usually refers it as
$\frb$-binomial (it is of $\vr$-degree 2 or higher), and by a non-governing binomial, we mean a binomial of $\cB^\ngv$ (it is always of
$\vr$-degree 1).

In concrete terms, 
all the governing binomials are classified as follows.
\begin{eqnarray}\label{all gov bi}
\hbox{all the governing binomials:} \;\;\;\;\;\;\;\;\;\;\; \;\;\;\;\;\;\;\;\;\;\; \;\;\;\;\;\;\;\; \\
x_{1uv}x_{(12u,13v)} - x_{12u}x_{13v} x_{(123,1uv)}, \; x_{1uv}x_{(13u,12v)}- x_{13u}x_{12v}x_{(123,1uv)}, \nonumber \\
x_{2uv}x_{(12u,23v)} -x_{12u}x_{23v} x_{(123,2uv)}, \;  x_{2uv}x_{(23u,12v)}-x_{23u}x_{12v} x_{(123,2uv)}, \nonumber \\
x_{3uv}x_{(13u,23v)} -x_{13u}x_{23v}x_{(123,3uv)}, \;  x_{3uv}x_{(23u,13v)} -x_{23u}x_{12v}x_{(123,3uv)},\nonumber\\
 x_{abc}x_{(12a,3bc)}-x_{12a}x_{3bc} x_{(123,abc)},\;
 x_{abc}x_{(13a,2bc)}-x_{13a}x_{2bc} x_{(123,abc)},\nonumber \\
x_{abc}x_{(23a,1bc)} -x_{23a}x_{1bc}x_{(123,abc)} \nonumber
\end{eqnarray}
for all $u<v \in [n]\-[3]$ and $a<b<c \in [n]\- [3]$.
We see that  the terms of all the primary $\pl$ relations 
 are separated into the two terms of the above binomials.

\begin{defn}
We also call the linearized $\pl$ relation $L_F$ a governing relation
for any $\bF \in \sF$.
\end{defn}

\begin{defn}\label{defn:block}
Given any $\bF \in \sF$, we let $\fG_F=\cB^\gov_F \sqcup \{L_F\}$. We call it the block of  governing  relations with respect to $F$. 
\end{defn}
We let
$$\fG=\bigsqcup_{\bF \in \sF} \fG_F.$$
Note that the set $\fG$ is totally ordered:
$$\fG_{F_1} < \cdots  < \fG_{F_\up}.$$

We observe here that
\begin{equation}\label{dim}
\dim (\prod_{\bF \in \sF} \PP_F) = \sum_{\bF \in \sF} |S_F\- s_F| 
= \sum_{\bF \in \sF} |\cB_F^\gov|=|\cB^\gov|,
\end{equation}
where $|K|$ denotes the cardinality of a finite set $K$.
{\it This is an important identity — not accidental, even if it may seem so}

\subsection{Some initial setups for induction: $\vp$-, $\vr$-, and 
$\fL$-divisors of $\sR$} $\ $

\begin{defn} \label{-divisor}
Consider the scheme $\sR =\bU \times  \prod_{\bF \in \sF} \PP_F$.

Recall that the affine chart $\bU$ comes equipped with the local free variables 
$\{x_\uu\}_{\uu \in \II_{3,n}\- (123)}$.
For any $\uu \in \II_{3,n}\- (123)$, we set
$$X_\uu:=(x_\uu =0) \subset \sR.$$
We call $X_\uu$ the $\pl$ divisor, in short, the $\vp$-divisor,  of $\sR$ associated with $\uu$.
We let $\cD_\vp$ be the set of all $\vp$-divisors on the scheme $\sR$.

In addition to the $\vp$-divisors,  the scheme $\sR$  
comes equipped with the divisors
$$X_{(\uu, \uv)}:=(x_{(\uu, \uv)}=0)$$
for all $(\uu,\uv) \in \La_{{ \sF}}$.
We call $X_{(\uu, \uv)}$ the $\vr$-divisor corresponding to $(\uu, \uv)$.  
We let $\cD_{\vr}$ be the set  of all $\vr$-divisors of $\sR$. 

Further, the scheme $\sR$ also 
comes equipped with the divisors 
$$ D_{L_F} := (L_F=0)$$
for all $\bF \in \sF$. 
We call $D_{L_F}$ the $\fL$-divisor corresponding to $F$.
We let $\cD_{\fL}$ be the set of all $\fL$-divisors of $\sR$.
\end{defn}


To be cited as the initial cases of certain inductions later on,
we need the following two propositions.

\begin{prop}\label{meaning-of-var-p-k=0} Consider any standard
 chart (Definition \ref{fv-k=0})
$$\fV=\bU \times \prod_{i \in [k]}
(x_{(\uu_{s_{F_i, o}}, \uv_{s_{F_i, o}})} \equiv 1)$$ of $\sR_{[k]}$, 
indexed by  $\La_{[k]}^o$ as in \eqref{index-sR}.
It  comes equipped with the set of free variables
$$\var_\fV=\{x_{\fV, \uw}, \; x_{\fV, (\uu,\uv)} \mid \uw \in \II_{3,n} \- (123), \; 
(\uu,\uv) \in \La_{[k]}^\star
\}$$ 
and the de-homogenized linearized $\pl$ relations $L_{\fV, F}$ for all $\bF \in \sF$
such that on the standard chart $\fV$, we have
\begin{enumerate}
\item the divisor  $X_{ \uw}\cap \fV$ is defined by $(x_{\fV,\uw}=0)$ for every 
$\uw \in \II_{3,n} \setminus (123)$;
\item the divisor  $X_{(\uu,\uv)}\cap \fV$ is defined by $(x_{\fV,(\uu,\uv)}=0)$ for every 
$(\uu,\uv) \in \La_{[k]}^\star.$     
\item the divisor  $D_{L_F}\cap \fV$ is defined by $(L_{\fV,F}=0)$ for every 
$\bF \in \sF$.
\end{enumerate}
In particular, all the above divisors are smooth.
\end{prop}
\begin{proof} 
Recall  that $\bU=(p_{(123)} \equiv 1)$.
Then, we let $x_{\fV, \uw}=x_\uw$ for all  $\uw \in \II_{3,n} \- (123)$.
Now consider every 
$i \in [k]$.
Upon setting $x_{(\uu_{s_{F_i, 0}}, \uv_{s_{F_i, 0}})} \equiv 1$, we let
 $x_{\fV, (\uu_s,\uv_s)} = x_{(\uu_s,\uv_s)}$ be the de-homogenization of $x_{(\uu_s,\uv_s)}$
 for all  $s  \in S_{F_i} \- s_{F_i,o}$.
From here,  the statement is straightforward to check.
\end{proof}

\begin{prop}\label{equas-fV[k]}  Let the notation be as in
Propsotion \ref{meaning-of-var-p-k=0}. Then,
the scheme $\sV_{[k]} \cap \fV$, as a closed subscheme of $\fV$
is defined by the following relations
$$\cB_{\fV,[k]}, \; \{L_{\fV, F_i} ,  i \in [k]\}, \; \{\bF_j,  k < j\le \up\}. $$
where  the equations of $\cB_{\fV, [k]}$ and $ \{L_{\fV, F_i} ,  i \in [k]\}$
are the de-homogenizations of the equations
of $\cB_{[k]}$ (see \eqref{cBk all})
 and $\{L_{F_i} ,  i \in [k]\}$ with respect to the chart $\fV$.
\end{prop}
\begin{proof} This follows directly from Lemma \ref{equas-for-sVk},
using the notations of Proposition \ref{meaning-of-var-p-k=0} and its proofs.
\end{proof}

As a special case of the above proposition, we have

\begin{prop}\label{equas-p-k=0} 
Let the notation be as in Proposition \ref{equas-fV[k]} for $k=\up$.  
Then, the scheme $\sV\cap \fV$, as a closed subscheme of
the chart $\fV$ of $\sR$,  is defined by 
$$\cB_{\fV} \sqcup  \{L_{\fV, F_i} ,  i \in [\up]\}$$
where $\cB=\cB^\gov \sqcup \cB^\ngv \sqcup \cB^\frb.$
\end{prop}

Here, it is worth to mention that the  governing relations in the proposition are
\begin{eqnarray}  
B_{ \fV,(s_F,s)}: \;\; x_{\fV, (\uu_s, \uv_s)}x_{\fV, \uu_F} - x_{\fV, (\um,\uu_F)}   x_{\fV, \uu_s} x_{\fV, \uv_s}, \;\;
\forall \;\; s \in S_F \- s_F, \label{eq-B-k=0} \\
\;\;\; L_{\fV, F}: \;\; \;\;\;\; \sum_{s \in S_F} \sgn (s) x_{\fV,(\uu_s,\uv_s)} 
\label{linear-pl-k=0} \;\;\;\;\;\;\;\;\;\;\;\;\;\; \;\;\;\;\;\;\;\;\;\;\;\;\;\;\;\;\;\;\;\;\;\;
\end{eqnarray}
for all $\bF \in \sF$.
The  non-governing binomial relations of $\vr$-degree 1 in the proposition are
\begin{equation} \label{eq-Bres-pk=0}
B_{ \fV,(s,t)}: \;\; x_{\fV,(\uu_s, \uv_s)}x_{\fV,\uu_t}x_{\fV, \uv_t}
-x_{\fV,(\uu_t, \uv_t)}x_{\fV,\uu_s}x_{ \fV,\uv_s}, \;\; \forall \;\; s, t \in S_F \- s_F
\end{equation}
for all $\bF \in \sF$.

\section{The Universal $\vt$-Blowups}\label{vt-blowups} 

{\it  In this section, we
 begin the process of eliminating zero factors of  governing binomials
by  sequential blowups. It is divided into two subsequences.
The first are $\vt$-blowups. The second are $\wp$-blowups, which importantly mingle with
$\ell$-blowups. The $\wp$- and $\ell$-blowups will be the topics of the next section.}
 

To start, it is useful to fix some terminology, used throughout.

\subsection{Some conventions on blowups} \label{blowupConventions} $\ $

Let $X$ be a 
scheme over the base field $\kk$.
When we blow up the scheme $X$ along the ideal (the homogeneous ideal, respectively)
$I=\langle f_0, \cdots, f_m \rangle$,  generated by some elements $f_0, \cdots, f_m$,
we will realize the blowup scheme $\widetilde X$ as the graph of the closure
of the rational map $$f: X \dashrightarrow \PP^m,$$
$$ x \to [f_0(x), \cdots, f_m(x)].$$
Then, upon fixing the generators  $f_0, \cdots, f_m$, we have a natural embedding
\begin{equation}\label{general-blowup} \xymatrix{
\widetilde{X}  \ar @{^{(}->}[r]  & X \times \PP^m.
 }
\end{equation}
We let
\begin{equation}
\pi: \widetilde{X} \lra X 
\end{equation}
be the induced blowup morphism.

We will refer to the projective space $\PP^m$ as the  {\it factor
projective space}  of the blowup corresponding to the generators  $f_0, \cdots, f_m$.
We let $[\xi_0, \cdots, \xi_m]$ be the homogeneous coordinates of
the factor projective space $\PP^m$,  corresponding  to $(f_0, \cdots, f_m)$. 

When $X$ is smooth and the center of the blowup is also smooth, then, the scheme 
$\widetilde X$, as a closed subscheme of $X \times \PP^m$, is defined by the relations
\begin{equation}
f_i \xi_j - f_j \xi_i, \;\; \hbox{for all $0\le i \ne j\le m$}.
\end{equation}

In addition, if we have an algebraic group $G$ (over $\kk$) acting on a scheme $X$,
we say that the action is quasi-free if every isotropy subgroup of
$G$ on any point $x \in X$ is a connected group.
This in particular implies that when $G$ is reductive 
(such as the split torus $(\GG_m)^r$ for any positive integer),
then it acts freely on any stable locus $X^s$ for any
$G$-ample linearization over $X$. In particular,
any geometric GIT quotient $X^s/G$ is smooth,  provided
$X$ is smooth and the action is quasi-free.
Note that the maximal torus $\TT=(\GG_m)^n/\GG_m$ 
acts on the $\pl$ projective space $\PP(\wedge^3 E)$ quasi-freely,
hence on the Grassmannian $\Gr^{3, E}$ and the open subset $\bU$
 quasi-freely, as well. 
In what follows, we will show that
the $\TT$-action on $\bU \subset \Gr^{3, E}$
lifts to a quasi-free action on the blowup scheme in every step.

\begin{defn}\label{general-standard-chart}
Suppose that the scheme $X$ is covered by a set $\{\fV' \}$ of open subsets, called
(standard) charts. 

Fix any $0\le i\le m$. We let
\begin{equation} 
\fV=  (\fV' \times (\xi_i \ne 0)) \cap \widetilde{X} .
\end{equation}
We also often express this chart as
 $$\fV= (\fV' \times (\xi_i \equiv 1)) \cap \widetilde{X}.$$
It is an open subset of  $\widetilde{X}$, and will be called a standard chart of $\widetilde{X}$
lying over the (standard) chart $\fV'$ of $X$. Note that every standard chart of $\widetilde{X}$
lies over a unique (standard) chart $\fV'$ of $X$.
Clearly, $\widetilde{X}$ is covered by the finitely many  standard charts.

 In general, we let
 $$\widetilde{X}_k \lra \widetilde{X}_{k-1} \lra  \cdots  \lra \widetilde X_0:=X$$
 be a sequence of blowups such that every blowup $\widetilde{X_j} \to \widetilde{X}_{j-1}$ is
 as in \eqref{general-blowup}, $j \in [k]$.  
 
 Consider any $0\le j < k$.  Let $\fV$ (resp. $\fV''$) be a standard chart of $\widetilde{X}_k$
 (resp. of $\widetilde{X}_j$). Let $\fV'$ be the unique standard chart $\fV'$ of $\widetilde X_{k-1}$
 such that $\fV$ lies over $\fV'$.
 Via induction, we say $\fV$ lies over $\fV''$ if $\fV'$ equals to (when $j=k-1$) or lies over $\fV''$
 (when $j < k-1$).
\end{defn}

 We keep the notation as above. Let $\widetilde{X}  \to X$ be a blowup as
in \eqref{general-blowup}; we let $\fV$ be a standard chart of $\widetilde{X}$, lying over
a unique (standard) chart $\fV'$ of $X$; let
$\pi_{\fV, \fV'}: \fV \lra \fV'$ be the induced projection.

\begin{defn}\label{general-proper-transform-of-variable}  Assume that
the open chart $\fV$ (resp. $\fV'$) 
comes equipped with
a set of local free variables  in $\var_\fV$ (resp. $\var_{\fV'}$).
Let $y \in \var_\fV$ (resp. $y' \in \var_{\fV'}$) be a local free variable of $\fV$ (resp. $\fV'$).
We say the local free variable $y$ is a proper transform of
the local free variable $y'$ if the divisor  $(y=0)$ on the chart $\fV$ is the proper transform of  
the divisor $(y'=0)$ on the chart $\fV'$.
\end{defn}

Keep the notation and assumption as in Definition \ref{general-proper-transform-of-variable}.

We assume in addition that the induced blowup morphism 
$$\pi^{-1} (\fV') \lra \fV'$$
corresponds to  the blowup of $\fV'$ along
the coordinate subspace of $\fV'$ defined by
 $$Z=\{y'_0= \cdots =y'_m =0\}$$
 with $\{y'_0, \cdots, y'_m\} \subset \var_{\fV'}$.
As earlier, we let $\PP^m$ be the corresponding factor projective space
with homogeneous coordinates $[\xi_0, \cdots, \xi_m]$, corresponding to $(y'_0, \cdots, y'_m)$.
 
 Without loss of generality, we assume that 
  the standard chart $\fV$ corresponds to $(\xi_0 \equiv 1)$, that is,
 $$\fV = (\fV' \times (\xi_0 \equiv 1)) \cap \widetilde{X}.$$
 Then, we have that $\fV$, as a closed subscheme of $\fV' \times (\xi_0 \equiv 1)$,
 is defined 
 \begin{equation}\label{general-blowup-formulas}
y'_i  - y'_0 \xi_i, \;\; \hbox{for all $i \in [m]$}.
\end{equation}

The following proposition is standard and will be applied throughout. 

\begin{prop}\label{generalmeaning-of-variables} Keep the notation and assumption as above.
In addition, we let $E$ be the exceptional divisor of the blowup $\widetilde{X}  \to X$.

Then, the standard chart $\fV$ comes equipped with a set of free variables 
$$\var_\fV=\{ \zeta, y_1, \cdots, y_m;  y:=y'  \mid y' \in \var_{\fV'} \- \{y_0' , \cdots, y_m'\} \}$$
where $\zeta:=y_0', y_i :=\xi_i, i \in [m]$
such that
\begin{enumerate}
\item $E \cap \fV =(\zeta =0)$; we call $\zeta$ the exceptional variable/parameter of $E$ on $\fV$;
\item $y_i \in \var_\fV$ is a proper transform of  $y'_i \in \var_{\fV'}$ for all  $i \in [m]$;
\item $y \in \var_\fV$ is a proper transform of  
$y' \in \var_\fV$ for all  $y' \in \var_{\fV'} \- \{y_0' , \cdots, y_m'\}$.
\end{enumerate}
\end{prop}
\begin{proof}
It is straightforward from \eqref{general-blowup-formulas}.
\end{proof}
 
Let  $\bf m$ be a monomial in $\var_{\fV}$. Then, for every variable $x \in \var_{\fV}$,
we let $\deg_x {\bf m}$ be the degree of $x$ in $\bf m$. 
 
 \begin{defn}\label{general-proper-transforms} 
 Keep the notation and assumption as in Proposition \ref{generalmeaning-of-variables}.
In addition, we let 
$$\phi=\{y'_0, \cdots, y'_m\} \subset \var_{\fV'}.$$

 Let $B_{\fV'}=T^0_{\fV'}- T^1_{\fV'}$ be 
 a binomial with variables in $\var_{\fV'}$.
We let $$m_{\phi, T^i_{\fV'}} = \sum_{j=0}^m \deg_{y'_j} (T^i_{\fV'}), \;\; i =0, 1, $$ 
$$l_{\phi, B_{\fV'}} = \min \{m_{\phi, T^0_{\fV'}}, m_{\phi, T^1_{\fV'}}\}.$$ 

Applying \eqref{general-blowup-formulas}, we substitute $y'_i$ by $y'_0 \xi_i$, for all $i \in [m]$,
into $B_{\fV'}$ and switch $y_0'$ by $\zeta$ and $\xi_i$ by $y_i$ with $i \in [m]$ to obtain
the pullback $\pi_{\fV,\fV'}^* B_{\fV'}$
where $\pi_{\fV,\fV'}: \fV \lra \fV'$ is the induced projection.
We then let 
\begin{equation}\label{define-proper-t}
B_\fV = (\pi_{\fV,\fV'}^* B_{\fV'}) / \zeta^{l_{\phi, B_{\fV'}}}.
\end{equation}
We call $B_\fV$, a binomial in $\var_\fV$, the proper transform of $B_{\fV'}$.

In general, for any polynomial $f_{\fV'}$ in $\var_{\fV'}$ such that 
$f_{\fV'}$ does not vanish identically along 
$Z= (y'_0= \cdots= y'_m=0)$, we let
$f_\fV = \pi_{\fV,\fV'}^* f_{\fV'}$. This is the pullback, but for convenience, we also call
$f_\fV$ the proper transform of $f_{\fV'}$. 

Moreover, suppose $\zeta$ appears in $B_\fV =(\pi_{\fV,\fV'}^* B_{\fV'}) / \zeta^{l_{\psi, B_{\fV'}}}$ or 
in $f_\fV= \pi_{\fV,\fV'}^* f_{\fV'}$,  and is obtained through the substitution $y'_i$ by $y'_0 \xi_i$
(note here that $\zeta:=y'_0$ and $i$ needs not to be unique), 
then we say that the exceptional parameter $\zeta$
is acquired by $y_i'$. In general, for  sequential blowups, if $\zeta$ is acquired by $y'$ and $y'$ is 
acquired by $y''$, then we also say $\zeta$ is acquired by $y''$.
\end{defn}

\begin{lemma}\label{same-degree}
We keep the same assumption and notation as in Definition \ref{general-proper-transforms}.

We let $T_{\fV', B}$ (resp. $T_{\fV, B}$) be any fixed term of $B_{\fV'}$ (resp.
$B_\fV$).  Consider any $y \in \var_\fV \- \zeta$ and  let 
$y' \in \var_{\fV'}$ be such that $y$ is the proper transform of $y'$. 
Then,  $y^b \mid  T_{\fV, B}$ if and only if 
$y'^b \mid  T_{\fV', B}$ for all integers $b \ge 0$.
\end{lemma}
\begin{proof}
This is clear from \eqref{define-proper-t}.
\end{proof}

\begin{defn}\label{general-termination} 
We keep the same assumption and notation as in Definition \ref{general-proper-transforms}.

Consider an arbitrary binomial  $B_{\fV'}$
(resp. $B_\fV$) with variables in $\var_{\fV'}$ (resp. $\var_\fV$).
Let $\bz' \in \fV'$  (resp. $\bz \in \fV$)  be any fixed closed point of the chart. 
We say $B_{\fV'}$ (resp. $B_\fV$)  terminates at $\bz'$ (resp. $\bz$)
 if (at least) one of the monomial terms of $B_{\fV'}$
(resp. $B_\fV$), say, $T_{\fV', B}$
(resp. $T_{\fV,B}$),
does not vanish at  $\bz'$ (resp. $\bz$).  In such a case, we also say 
$T_{\fV', B}$ (resp. $T_{\fV,B}$) terminates  at $\bz'$ (resp. $\bz$).
 \end{defn}

\subsection{Governing binomial equations: revisited} $\ $


Recall that we have  chosen and fix the total order $``<"$ on $\sF$ and we have listed it as
$$\sF=\{\bF_1 < \cdots < \bF_\Upsilon\}.$$

Fix  and consider $F_k$ for any $k \in [\up]$.
We express $F_k=\sum_{s \in S_{F_k}} \sgn (s) p_{\uu_s} p_{\uv_s}$.
 Its corresponding linearized $\pl$ equation can
be expressed as $\sum_{s \in S_{F_k}} \sgn(s) x_{(\uu_s,\uv_s)}$, denoted by $L_{F_k}$. 
We let $s_{F_k} \in S_{F_k}$ be the index for the leading term of $F_k$, written as 
$ \sgn (s_{F_k}) p_\um p_{\uu_{F_k}}$.
Correspondingly, the leading term of the de-homogenization $\bF_k$ of $F_k$,
and  the leading term of the linearized $\pl$ equation $L_{F_k}$, 
 are defined to be $\sgn (s_{F_k})x_{\uu_{F_k}}$, and  $\sgn (s_{F_k}) x_{(\um,\uu_{F_k})}$,
 respectively. Here, as before, $\um=(123)$.

We are to provide a total ordering on the set $\cB^\gov$.

Throughout the remaining part of this article, for a given totally ordered
set $(K,<)$ ,
when we use $<_\lex$ (resp. $<_\rlex$), we mean
the lexicographical (resp. reverse lexicographical) order induced on 
the subsets of the power set $2^K$ consisting of elements
of equal cardinalities (see Definition \ref{gen-order}).
 
 Recall  from Definition \ref{cFi-partial-order} 
that the set of all $\pl$-variables is also totally ordered, compatible with that
on $\sF$.
(It is neither lexicographical nor reverse-lexicographical.)

 \begin{defn}\label{ordering-cBF} First,
 for any fixed $\bF \in \sF$, we  provide a total ordering on the set $S_F \- s_F$ 
 by induction on $\rk F$ as follows.

 $\bullet$ Suppose $\rk F=0$.  Then, $S_F \- s_F$ consists of two elements $\{s, t\}$.
 We say $s<t$ if $(\uu_s, \uv_s) <_\lex (\uu_t, \uv_t)$ where each pair is listed
 lexicographically according to the order on the set of all $\pl$ variables.

 $\bullet$ Suppose $\rk F=1$. Then, for any $s\in  S_F \- s_F$,
 one of $p_{\uu_s}$ and $p_{\uv_s}$ is a basic variable, the other is of rank equal to
 $0$. Without loss of generality, we suppose $\rk p_{\uu_s}=0$ and let
 $F_{\uu_s}$ be its corresponding $\pl$ relation. Then, we say $s <t$  if $\uu_s < \uu_t$, that is,
 $F_{\uu_s} < F_{\uu_t}$.
 
Next, we furnish a total order on the set $\cB^\gov=\sqcup_{\bF \in \sF} \cB^\gov_F$ as follows.

For any $B' \in \cB^\gov_{F'}$ and $B \in \cB^\gov_{F}$ with $F' \ne F$,
 we say $B'<B$ if $F' < F$.

Recall that
$\cB^\gov_F=\{B_{F,(s_F,s)} \mid s \in S_F \setminus s_F  \}.$
Then, for any two distinct $s, t \in  S_F \- s_F$, we say $B_{F,(s_F,s)}<B_{F,(s_F,t)}$
if $s<t$.
 \end{defn} 
 This ordering is important for our purpose. But, we mention that for a fixed $\bF \in \sF$,
the order within $\cB^\gov_F$ is unimportant: we just need one to proceed.

We shall map $(\cB^\gov, <)$ into $(\ZZ \times \ZZ, <_\lex)$, preserving the orders.

We let $(\ft_{F_k}+1)$ be the number of terms in ${F_k}$ for any $k \in [\up]$. 
($\ft_{F_k}$ only assumes values 2 or 3.  See also Definition \ref{ftF}.)
Then, we can  list $S_{F_k}$ as
$$S_{F_k}=\{s_{F_k}; \; s_1 < \cdots < s_{\ft_{F_k}}\}.$$
Then, we can write 
$$\cB^\gov_{F_k}=\{B_{(k\tau)} \mid  \tau \in [\ft_{F_k}]  \}$$
with 
\begin{equation} \label{eq-B-ktau'}
B_{(k\tau)}: \; x_{(\uu_{s_\tau}, \uv_{s_\tau})}x_{\uu_{F_k}} - x_{(\um,\uu_{F_k})}   x_{\uu_{s_\tau}} x_{\uv_{s_\tau}}, \;
\forall \;\; {s_\tau} \in S_{F_k} \- s_{F_k}. \end{equation}
As
 $$\cB^\gov=\bigsqcup_{k\in [\up]} \cB^\gov_{F_k}=\{B_{(k\tau)} \mid k \in [\up], \;  \tau \in [\ft_{F_k}]  \},$$ 
we let
\begin{equation}\label{indexing-Bmn}
\Index_{\cB^\gov}=\{ (k\tau) \mid k \in [\up], \; \tau \in [\ft_{F_k}] \} \subset \ZZ \times \ZZ
\end{equation}
be the index set of $\cB^\gov$.

Now, observe that the order $`` < "$ on the set $\cB^\gov$  coincides with
the lexicographic order on $\Index_{\cB^\gov}$, that is, 
$$B_{(k\tau)} < B_{(k'\tau')} \iff (k,\tau) <_\lex (k',\tau').$$

\begin{defn}\label{pm-term} Given any governing binomial equation
$B_{(k\tau)}$ as in \eqref{eq-B-ktau'},
we let 
$$T^+_{(k\tau)}= \; x_{(\uu_{s_\tau}, \uv_{s_\tau})}x_{\uu_{F_k}},$$ called the plus-term of $B_{(k\tau)}$,
and $$T^-_{(k\tau)}= x_{(\um,\uu_{F_k})}   x_{\uu_{s_\tau}} x_{\uv_{s_\tau}},$$
called the minus-term of $B_{(k\tau)}$.
Then, we have $$B_{(k\tau)}=T^+_{(k\tau)}-T^-_{(k\tau)}.$$
\end{defn}

We do not name any term of a binomial of $\cB^\ngv \cup \cB^\frb$ a plus-term or a minus-term since
the two terms of such a  binomial  are rather indistinguishable.

 Recall that $$\fG=\bigsqcup_{\bF \in \sF} \fG_F, \;\; \hbox{where $\fG_F=\{\cB^\gov_F, L_F\}$.}$$
 We endow a total order on $\fG$ as follows. We say
 $\fG_{F'} < \fG_F$ if $F'< F$. 
This order is compatible with the order on $\sF$ as well as with the one on $\cB^\gov$.

Before turning to a new subsection, it is worth to remind the reader here that by 
Corollary \ref{eq-tA-for-sV}, the scheme $\sV$, as a closed subscheme of
$\sR= \bU \times  \prod_{\bF \in  \sF} \PP_F $,
is defined by the following relations
\begin{eqnarray}
   B_{(k\tau)}, \;\; L_\sF, \;\;  \cB^\ngv, \;\;  \cB^\frb 
\end{eqnarray}
where $L_\sF$ is the set of all linearized $\pl$ relation as in \eqref{LsF}.

\subsection{$\vt$-centers and $\vt$-blowups}\label{vr-centers} $\ $

{\it Besides serving as a part of   the process of eliminating  zero factors of
 the governing binomial relations, upon performing $\vt$-blowups,
 the proper transforms of all non-governing binomial relations 
in $\cB^\ngv$ will
become dependent on the proper transforms of the governing binomial relations,  therefore, these relations will be discarded after  $\vt$-blowups.

Every of the $\vt$-blowups
is the blowup of a smooth ambient space along a codimension two
smooth closed subvariety that is the intersection of two smooth
divisors. 
}

Recall that the scheme $\tsR_{\vt_{[0]}}:=\sR$ comes equipped with two kinds of divisors:
$\vp$-divisors $X_\uw$ for all $\uw \in \II_{3,n}\- (123)$
and $\vr$-divisors $X_{(\uu,\uv)}$ for all $(\uu,\uv) \in \La_\sF$. 

As in \eqref{ld-index}, we have
$$\II_{3,n}^{\lt}=\{(1uv),(2uv),(3uv),(abc) \mid 3<u<v \le n,\; 3<a<b<c\le n\}.$$
This is the index set for nonbasic variables, or equivalently, the index set for the leading variables of
primary $\pl$ relations.

\begin{defn}\label{defn:vr-centers}
Fix any $\uu \in \II_{3,n}^{\lt}$. 
We let
$$\vt_\uu=(X_\uu, X_{((123),\uu)}).$$
We call it the  $\vt$-set with respect to $\uu$.
We then call the scheme-theoretic intersection
 $$Z_{\vt_\uu}=X_\uu \cap X_{((123),\uu)}$$
the $\vt$-center with respect to $\uu$.
\end{defn}

We let
$$\Theta=\{\vt_\uu \mid \uu \in \II_{3,n}^{\lt}\},\;\;\;
\cZ_\Theta=\{Z_{\vt_\uu} \mid \uu \in \II_{3,n}^{\lt}\}.$$
We let $\Theta$, respectively,  $\cZ_\Theta$,
inherit the total order from $\II_{3,n}^{\lt}$. 
Thus,  if we write
$$\II_{3,n}^{\lt}=\{\uu_1 < \cdots <\uu_{\up}\}$$
and also write $\vt_{\uu_k}=\vt_{[k]}$, $Z_{\vt_{\uu_k}}=Z_{\vt_{[k]}}$, then, we can express
$$\cZ_\Theta=\{Z_{\vt_{[1]}} < \cdots < Z_{\vt_{[\up]}} \}.$$

We then blow up $\sR$ along 
$Z_{\vt_{[k]}},\; k \in [\up]$, in the above order.  More precisely,
we start by setting $\tsR_{\vt_{[0]}}:=\sR$.  Suppose 
$\tsR_{\vt_{[k-1]}}$ has been constructed for some $k \in [\up]$. We then let
$$\tsR_{\vt_{[k]}} \lra \tsR_{\vt_{[k-1]}}$$
be the blowup of $\tsR_{\vt_{[k-1]}}$ along the proper transform of $Z_{\vt_{[k]}}$,
and we call it  the $\vt$-blowup in ($\vt_{[k]}$).

 The above gives rise to the following sequential $\vt$-blowups
\begin{equation}\label{vt-sequence}
\tsR_{\vt_{[\up]}} \to \cdots \to \tsR_{\vt_{[1]}} \to \tsR_{\vt_{[0]}}:=\sR,
\end{equation}

Every blowup $\tsR_{\vt_{[j]}} \lra \tsR_{\vt_{[j-1]}}$
comes equipped with an exceptional divisor, denoted by $E_{\vt_{[j]}}$.
Fix $k \in [\up]$. For any $j < k$, we let $E_{\vt_{[k]},j }$ be the proper transform
of $E_{\vt_{[j]}}$ in $\tsR_{\vt_{[k]}}$.  For notational consistency, we set
$E_{\vt_{[k]}}=E_{\vt_{[k]},k }$. We call the divisors $E_{\vt_{[k]},j }$, $j\le k$, the exceptional divisors
on $\tsR_{\vt_{[k]}}$.

\begin{prop}\label{Taction-vt}
The quasi-free $\TT$-action on $\tsR_{\vt_{[k]}}$ lifts to 
a quasi-free $\TT$-action on $\tsR_{\vt_{[k]}}$.
\end{prop}
\begin{proof}
By Lemma \ref{lift-action-to-Um[k]}, we have a
quasi-free $\TT$-action on $\tsR_{\vt_{[0]}}=\sR$.

Suppose the
quasi-free $\TT$-action on $\tsR_{\vt_{[k-1]}}$ has been constructed.
Consider the embedding $$\tsR_{\vt_{[k]}} \lra  \tsR_{\vt_{[k-1]}} \times \PP_{\vt_{[k]}}$$
where $\PP_{\vt_{[k]}}$ is the factor projective space with homogeneous coordinates $[\xi_0,\xi_1]$
corresponding to $(X_{\uu_k}, X_{(\um,\uu_k)})$.
Writing $\uu_k=(u_1u_2u_3)$, we let $\TT$ act on $\PP_{\vt_{[k]}}$
by
$$\bt \cdot [\xi_0, \xi_1] = [\frac{t_{u_1}t_{u_2}t_{u_3}}{t_1t_2t_3}
\xi_0,\xi_1]$$
where $\bt=(t_1,\cdots, t_n)$ is a representative of an arbitrary element
of $\TT=\GG_m^n/\GG_m$. One checks that this is well-defined and
lifts the
quasi-free $\TT$-action on $\tsR_{\vt_{[k-1]}}$ to 
a quasi-free $\TT$-action on $\tsR_{\vt_{[k]}}$.
\end{proof}

For every 
$\uw \in \II_{3,n} \setminus (123)$, we let
$X_{\vt_{[k]}, \uw}$ be the proper transform of $X_\uw$ in $\tsR_{\vt_{[k]}}$,
still called $\vp$-divisor;
for every $(\uu,\uv) \in \La_\sF$,
we let  $X_{\vt_{[k]}, (\uu,\uv)}\cap \fV$ be 
 the proper transform of $X_{(\uu,\uv)}$ in $\tsR_{\vt_{[k]}}$,
still called $\vr$-divisor.
for every $\bF \in \sF$,
we let  $D_{\vt_{[k]}, L_F}$ be the proper transform of $D_{L_F}$ in $\tsR_{\vt_{[k]}}$,
still called the $\fL$-divisor.


\subsection{Properties of $\vt$-blowups}\label{prop-vt-blowups} $\ $

{\it
As in the rest of this paper,  we will use $\um$ and $(123)$ interchangeably, as
using $\um$ in place of $(123)$ may save some space in display.}

By Definition \ref{general-standard-chart}, the scheme $\tsR_{\vt_{[k]}}$
is covered by a set of standard charts.

\begin{prop}\label{meaning-of-var-vtk}
Consider any standard chart $\fV$ of $\tsR_{\vt_{[k]}}$, 
 lying over a unique chart $ \fV_{[0]}$ of $\tsR_{\vt_{[0]}}=\sR$.
 We suppose that the chart $ \fV_{[0]}$ is indexed by
$\La_\sF^o=\{(\uu_{s_{F,o}},\uv_{s_{F,o}}) \mid \bF \in \sF \}$
(cf. \eqref{index-sR}).
We let $\II_{3,n}^\star=\II_{3,n}\- (123)$ and $\La_\sF^\star=\La_\sF \- \La_\sF^o$.

Then, the standard chart $\fV$ comes equipped with 
$$\hbox{a subset}\;\; \fe_\fV  \subset \II_{3,n}^\star \;\;
 \hbox{and a subset} \;\; \fd_\fV  \subset \La_{\sF}^\star$$
such that every exceptional divisor  $E_{\vt_{[k]}, j} , j \in[k]$
of $\tsR_{\vt_{[k]}}$
with $E_{\vt_{[k]}, j} \cap \fV \ne \emptyset$ is 
either labeled by a unique element $\uw \in \fe_\fV$
or labeled by a unique element $(\uu,\uv) \in \fd_\fV$. 
We let $E_{\vt_{[k]}, \uw}$ be the unique exceptional divisor 
on the chart $\fV$ labeled by $\uw \in \fe_\fV$; we call it an $\vp$-exceptional divisor.
We let $E_{\vt_{[k]}, (\uu,\uv)}$ be the unique exceptional divisor 
on the chart $\fV$ labeled by $(\uu,\uv) \in \fd_\fV$;  we call it an $\vr$-exceptional divisor.
(We note here that being $\vp$-exceptional or $\vr$-exceptional is strictly relative to the given
standard chart.)

Further, the standard chart $\fV$  comes equipped with the set of free variables
\begin{equation}\label{variables-vtk} 
\var_{\fV}:=\left\{ \begin{array}{ccccccc}
\ve_{\fV, \uw} , \;\; \de_{\fV, (\uu,\uv) }\\
x_{\fV, \uw} , \;\; x_{\fV, (\uu,\uv)}
\end{array}
  \; \Bigg| \;
\begin{array}{ccccc}
 \uw \in  \fe_\fV,  \;\; (\uu,\uv)  \in \fd_\fV  \\ 
\uw \in  \II_{3,n}^\star \- \fe_\fV,  \;\; (\uu, \uv) \in \La_\sF^\star \-  \fd_\fV  \\
\end{array} \right \}
\end{equation} such that 
on the standard chart $\fV$, we have
\begin{enumerate}
\item the divisor  $X_{\vt_{[k]}, \uw}\cap \fV$ 
 is defined by $(x_{\fV,\uw}=0)$ for every 
$\uw \in \II_{3,n}^\star \- \fe_\fV$;
\item the divisor  $X_{\vt_{[k]}, (\uu,\uv)}\cap \fV$ is defined by $(x_{\fV,(\uu,\uv)}=0)$ for every 
$(\uu,\uv) \in \La^\star_\sF\- \fd_\fV$;
\item the divisor  $D_{\vt_{[k]}, L}\cap \fV$ is defined by $(L_{\fV,F}=0)$ for every
$\bF \in \sF$ where $L_{\fV, F}$ is the proper transform of $L_F$; 
\item the divisor  $X_{\vt_{[k]}, \uw}$ does not intersect the chart for all $\uw \in \fe_\fV$;
\item the divisor  $X_{\vt_{[k]}, (\uu, \uv)}$ does not intersect the chart for all $(\uu, \uv) \in \fd_\fV$;
\item the $\vp$-exceptional divisor 
$E_{\vt_{[k]}, \uw} \;\! \cap  \fV$  labeled by an element $\uw \in \fe_\fV$
is define by  $(\ve_{\fV,  \uw}=0)$ for all $ \uw \in \fe_\fV$;
\item the $\vr$-exceptional divisor 
$E_{\vt_{[k]},  (\uu, \uv)}\cap \fV$ labeled by  an element $(\uu, \uv) \in \fd_\fV$
is define by  $(\de_{\fV,  (\uu, \uv)}=0)$ for all $ (\uu, \uv) \in \fd_\fV$;
\item  any of the remaining exceptional divisor of $\tsR_{\vt_{[k]}}$
other than those that are labelled by some  $\uw \in \fe_\fV$ or $(\uu,\uv) \in \fd_\fV$ 
 does not intersect the chart.
\end{enumerate}
\end{prop}
\begin{proof}
When $k=0$, we have $\tsR_{\vt_{[0]}}=\sR$. 
In this case, we set $$ \fe_\fV = \fd_\fV =\emptyset.$$
Then, the statement follows from 
Proposition \ref{meaning-of-var-p-k=0} with $k=\up$.

We now suppose that the statement holds for $\tsR_{\vt_{[k-1]}}$
for some $k \in [\up]$. 

We consider $\tsR_{\vt_{[k]}}$.

As in the statement, we let $\fV$ be a standard chart 
of $\tsR_{\vt_{[k]}}$, lying over a (necessarily unique) 
standard chart  $\fV'$ of $\tsR_{\vt_{[k-1]}}$.

If $(\um,\uu_k) \in  \La_{[k]}^o$
(cf. \eqref{index-sR}), 
then $\fV'$
does not intersect the proper transform of the blowup center $Z_{\vt_k}$ and
$\fV \to \fV'$ is an isomorphism. In this case, we let $\var_\fV=\var_{\fV'}$,
$ \fe_{\fV'}=\fe_\fV, $ and $\fd_{\fV} =\fd_{\fV'}.$ Then, the statements on $\fV'$
carry over to $\fV$.

In what follows, we assume $(\um,\uu_k) \notin  \La_{[k]}^o$.

Consider the embedding $$\tsR_{\vt_{[k]}} \lra  \tsR_{\vt_{[k-1]}} \times \PP_{\vt_{[k]}}$$
where $\PP_{\vt_{[k]}}$ is the factor projective space with homogeneous coordinates $[\xi_0,\xi_1]$
corresponding to $(X_{\uu_k}, X_{(\um,\uu_k)})$.
We let $E_{\vt_{[k]}}$ be the exceptional divisor created by 
the blowup $\tsR_{\vt_{[k]}} \to  \tsR_{\vt_{[k-1]}}$.

First, we consider the case when 
$$\fV = \tsR_{\vt_{[k]}} \cap (\fV' \times (\xi_0 \equiv 1).$$
We let $Z'_{\vt_k}$ be the proper transform of the $\vt$-center $Z_{\vt_k}$
in  $\tsR_{\vt_{[k-1]}}$. Then,
in this case, on the chart $\fV'$, we have
 $$Z'_{\vt_k} \cap \fV' = \{ x_{\fV', \uu_k} = x_{\fV', (\um,\uu_k)}=0\}$$
 where   $x_{\fV', \uu_k}$ (resp. $x_{\fV', (\um,\uu_k)}$) is the proper transform of $x_\uu$ (resp.
 $x_{ (\um,\uu_k)}$)
 on the chart $\fV'$.
 Then, $\fV$ as a closed subset of  $\fV' \times (\xi_0 \equiv 1)$ is defined by
 $$x_{\fV', (\um,\uu_k)} = x_{\fV', \uu_k} \xi_1.$$
 We let 
 $$\fe_\fV= \uu_k \sqcup \fe_{\fV'},\; \fd_\fV=  \fd_{\fV'},\;\; \hbox{and}$$
 $$\ve_{\fV, \uu_k}=x_{\fV', \uu_k}, \;  x_{\fV, (\um,\uu_k)}=\xi_1; \;
 y_\fV = y_{\fV'}, \; \forall \; y_{\fV'} \in \var_{\fV'}\- \{x_{\fV', \uu_k}, x_{\fV', (\um,\uu_k)}\}.$$
 Observe that  
 $E_{\vt_{[k]}} \cap \fV = (\ve_{\fV, \uu_k}=0)$ and
 $x_{\fV, (\um,\uu_k)}=\xi_1$ is the proper transform of  $x_{\fV', (\um,\uu_k)}$.
 By the inductive assumption on the chart $\fV'$, one verifies directly that
 (1) - (7) of the proposition hold (cf. Proposition \ref{generalmeaning-of-variables}).
 
Next, we consider the case when 
$$\fV = \tsR_{\vt_{[k]}} \cap (\fV' \times (\xi_1 \equiv 1).$$
 Then, $\fV$ as a closed subset of  $\fV' \times (\xi_1 \equiv 1)$ is defined by
 $$ x_{\fV', \uu_k} =x_{\fV', (\um,\uu_k)} \xi_0.$$
 We let 
 $$\fe_\fV=  \fe_{\fV'},\; \fd_\fV=  \{(\um,\uu_k)\} \sqcup \fd_{\fV'},\;\; \hbox{and}$$
 $$\de_{\fV,  (\um,\uu_k)}=x_{\fV',  (\um,\uu_k)}, \;  x_{\fV, \uu_k}=\xi_0; \;
 y_\fV = y_{\fV'}, \; \forall \; y_{\fV'} \in \var_{\fV'}\- \{x_{\fV', \uu_k}, x_{\fV', (\um,\uu_k)}\}.$$
 Observe that  
 $E_{\vt_{[k]}} \cap \fV = (\de_{\fV,  (\um,\uu_k)}=0)$ and
 $x_{\fV, \uu_k}=\xi_0$ is the proper transform of  $x_{\fV', \uu_k}$.
 By the inductive assumption on the chart $\fV'$, like in the above case,
 one checks directly that
 (1) - (7) of the proposition hold.

 This proves the proposition.
 \end{proof}

Observe here that $x_{\fV, \uu}$ with $\uu \in \fe_{\fV}$ and
 $x_{\fV, (\uu,\uv)}$ with  $(\uu,\uv) \in \de_{\fV}$ are not variables in $\var_\fV$.
For notational convenience, to be used throughout, we make a convention:  
\begin{equation}\label{conv:=1}
\hbox{$\bullet$ $x_{\fV, \uu} = 1$ if  $\uu \in \fe_{\fV}$; 
\;\; $\bullet$ $x_{\fV, (\uu,\uv)} = 1$ if  $(\uu,\uv) \in \fd_{\fV}$.}
\end{equation}

\smallskip
For any $k \in [\up]$, the $\vt$-blowup in ($\vt_{[k]}$)
gives rise to \begin{equation}\label{tsV-vt-k} \xymatrix{
\tsV_{\vt_{[k]}} \ar[d] \ar @{^{(}->}[r]  &\tsR_{\vt_{[k]}} \ar[d] \\
\sV \ar @{^{(}->}[r]  & \sR,
}
\end{equation}
where $\tsV_{\vt_{[k]}}$ is the proper transform of $\sV$ 
in  $\tsR_{\vt_{[k]}}$. 

Alternatively, we can set $\tsV_{\vt_{[0]}}:=\sV_\sF$.  Suppose 
$\tsV_{\vt_{[k-1]}}$ has been constructed for some $k \in [\up]$. We then let
 $\tsV_{\vt_{[k]}} \subset \tsR_{\vt_{[k]}}$ be the proper transform of $\tsV_{\vt_{[k-1]}}$.

\begin{prop}\label{Taction-vt-v}
The quasi-free $\TT$-action on $\tsV_{\vt_{[k]}}$ lifts to 
a quasi-free $\TT$-action on $\tsV_{\vt_{[k]}}$.
\end{prop}
\begin{proof}
This follows from Lemma \ref{Taction-vt}.
\end{proof}

\begin{defn} Fix any standard chart $\fV$ of $\tsR_{\vt_{[k]}}$ lying over
a unique standard chart $\fV'$ of $\tsR_{\vt_{[k-1]}}$ for any $k \in [\up]$. 
When $k=0$, we let $B_\fV$ and $L_{\fV, F}$ be as in 
 Proposition \ref{equas-p-k=0} for any
$B \in \cB^\gov \cup \cB^\ngv  \cup \cB^\frb$ and $\bF \in \sF$. Consider any fixed general $k \in [\up]$.
Suppose $B_{\fV'}$ and $L_{\fV', F}$ have been constructed over $\fV'$.
Applying Definition \ref{general-proper-transforms}, we obtain their proper transforms on the chart $\fV$
$$B_{\fV}, \;\; \forall \; B \in \cB^\gov \cup \cB^\ngv  \cup \cB^\frb; \;\; L_{\fV, F}, \;\;
\forall \; \bF \in \sF.$$ 
\end{defn}

We need the following notations.

\begin{defn}\label{<,>}
Fix any $k \in [\up]$.
We let $\cB^\gov_{< k}$, respectively $\cB^\ngv_{< k}$ or $L_{ <k}$,
be the set of all governing, respectively non-governing binomials of $\vr$-degree 1 or linear $\pl$ relations corresponding to
$F<F_k$. Similarly, we let $\cB^\gov_{ > k}$ (resp. $\cB^\ngv_{ > k}$, $\cL_{ >k}$)
be the set of all main (resp. non-governing binomials of  $\vr$-degree 1 or linear $\pl$) relations corresponding to
$F>F_k$.  Likewise, replacing $<$ by $\le$ or $>$ by $\ge$,
we can introduce 
$\cB^\gov_{\le k}$, $\cB^\ngv_{\le k}$, and $L_{\le k}$ or
$\cB^\gov_{ \ge k}$,  $\cB^\ngv_{ \ge k}$, and $\cL_{\ge k}$.
Then, upon restricting the above to a fixed standard chart $\fV$, we obtain
$\cB^\gov_{\fV, < k}$, $\cB^\ngv_{\fV, < k}$, $L_{\fV, <k}$, etc..
\end{defn}

Recall from the above proof, we have 
$$\tsR_{\vt_{[k]}} \subset \tsR_{\vt_{[k-1]}} \times \PP_{\vt_k}$$
where $\PP_{\vt_k}$ be the factor projective space of the blowup
$\tsR_{\vt_{[k]}} \lra \tsR_{\vt_{[k-1]}}$. 
We write $\PP_{\vt_k}=\PP_{[\xi_0,\xi_1]}$ such that 
$[\xi_0,\xi_1]$ corresponds to $(X_{\uu_k}, X_{(\um,\uu_k)})$.

\begin{defn} \label{vp-vr-chart}
Let $\fV'$ be any standard chart on $\tsR_{\vt_{[k-1]}}$. Then, 
we call $$\fV=\tsR_{\vt_{[k]}} \cap (\fV' \times (\xi_0 \equiv 1))$$
a $\vp$-standard chart of $\tsR_{\vt_{[k]}}$; we call $$\fV=\tsR_{\vt_{[k]}} \cap (\fV' \times (\xi_1 \equiv 1))$$
a $\vr$-standard chart of $\tsR_{\vt_{[k]}}$.
\end{defn}

As in  the proof of 
Proposition \ref{meaning-of-var-vtk}, we let $\fV$ be a standard chart 
of $\tsR_{\vt_{[k]}}$ lying over a (necessarily unique) 
standard chart  $\fV'$ of $\tsR_{\vt_{[k-1]}}$.
Also, $\fV$ lies over a unique standard chart $\fV_{[0]}$ of 
of $\tsR_{\vt_{[0]}}$. We let 
\begin{equation}\label{fV-fV0}
\pi_{\fV, \fV_{[0]}}: \fV \to \fV_{[0]}
\end{equation}
be the induced projection.

\begin{prop}\label{eq-for-sV-vtk}
We keep the notation and assumptions
 in Proposition \ref{meaning-of-var-vtk}. 

Suppose $(\um,\uu_k) \in \La_{[k]}^o$ or
 $\fV$ is a $\vr$-standard chart.  Then, we have
that the scheme $\tsV_{\vt_{[k]}}  \cap \fV$, as a closed subscheme of
the chart $\fV$ of $\tsR_{\vt_{[k]}} $,  is defined by 
\begin{eqnarray} 
\cB^\frb_\fV,  \;\; \cB^\gov_{\fV, < k} ,\;\; \cL_{\fV, <k} , \;\;\;\;\;\;\;\;\;\;  \\
B_{ \fV,(s_{F_k},s)}: \;\;\;\;\;\; x_{\fV, (\uu_s, \uv_s)}  x_{\fV, \uu_k} 
  - \tilde{x}_{\fV, \uu_s} \tilde{x}_{\fV, \uv_s}, \;\;
\forall \;\; s \in S_{F_k} \- s_{F_k},  \label{eq-B-vt-lek=0} \\ 
L_{\fV, F_k}: \;\; \sgn (s_F) \de_{\fV, (\um,\uu_k)} +
\sum_{s \in S_F \- s_F} \sgn (s) x_{\fV,(\uu_s,\uv_s)},   \label{linear-pl-vtk=0} \\
\cB^\gov_{\fV, > k},\;\; \cB^\ngv_{\fV, > k}, \;\; \cL_{\fV, >k},\;\;\; \;\;\; \;\; \;\; \;\; \;\; \;\; \;\; \;\; \;\; \;\; \;\; \;\; \;\; \label{eq-hq-vtk=0}
\end{eqnarray}
where $\tilde{x}_{\fV, \uu_s}=
\pi_{\fV, \fV_{[0]}}^* x_{\fV_{[0]}, \uu_s}$ and 
$ \tilde{x}_{\fV, \uv_s}= \pi_{\fV, \fV_{[0]}}^* x_{\fV_{[0]}, \uv_s}$
 are monomials in $\var_\fV$, 
and $\pi_{\fV, \fV_{[0]}}$ is as in \eqref{fV-fV0}.

 Suppose  $(\um,\uu_k) \notin \La_{[k]}^o$ and $\fV$ is a $\vp$-standard chart.
Then, we have
that the scheme $\tsV_{\vt_{[k]}}  \cap \fV$, as a closed subscheme of
the chart $\fV$ of $\tsR_{\vt_{[k]}} $,  is defined by 
\begin{eqnarray} 
\cB_\fV^\frb,  \;\; \cB^\gov_{\fV, < k} ,\;\; \cL_{\fV, <k}, \;\;\;\;\;\;\;\;\;\;  \\
B_{ \fV,(s_{F_k},s)}: \;\;\;\;\;\; x_{\fV, (\uu_s, \uv_s)} - x_{\fV, (\um,\uu_k)}  
 \tilde{x}_{\fV, \uu_s} \tilde{x}_{\fV, \uv_s}, \;\;
\forall \;\; s \in S_{F_k} \- s_{F_k},  \label{eq-B-vt-lek=0-00} \\ 
L_{\fV, F_k}: \;\; \sgn (s_F) \ve_{\fV, \uu_k} x_{\fV,(\um,\uu_k)}+
\sum_{s \in S_F \- s_F} \sgn (s) x_{\fV,(\uu_s,\uv_s)},   \label{linear-pl-vtk=0-00} \\
\cB^\gov_{\fV, > k},\;\; \cB^\ngv_{\fV, > k}, \;\; \cL_{\fV, >k},\;\;\; \;\;\; \;\; \;\; \;\; \;\; \;\; \;\; \;\; \;\; \;\; \;\; \;\; \;\; \label{eq-hq-vtk=0-00}
\end{eqnarray}
where $\tilde{x}_{\fV, \uu_s}=
\pi_{\fV, \fV_{[0]}}^* x_{\fV_{[0]}, \uu_s}$ and 
$ \tilde{x}_{\fV, \uv_s}= \pi_{\fV, \fV_{[0]}}^* x_{\fV_{[0]}, \uv_s}$
 are  monomials in $\var_\fV$,


Moreover, for any binomial 
$B \in \cB^\gov \sqcup \cB^\ngv_{>k}$, $B_\fV$ is $\vr$-linear and square-free.

Furthermore, consider  an arbitrary binomial 
$B \in \cB^\frb$ and its proper transform $B_{\fV}$ on the chart 
$\fV$. Let $T_{\fV, B}$ be any fixed term of $B_\fV$.
Then, $T_{\fV, B}$  is $\vr$-linear and admits at most one
(counting multiplicity) $\vt$-exceptional parameter in the form of 
$$\hbox{$\de_{\fV,(\um, \uu)}$ for some $(\um, \uu) \in \fd_\fV$ 
or $\ve_{\fV,\uu} x_{\fV, (\um, \uu)}$ for some $\uu \in \fe_\fV$} $$
if $T_{\fV, B}$ contains $x_{\fV, (\um, \uu)}$ 
(hence $\uu \in \II_{3,n}^{\lt,0}$); or 
$$\hbox{  $\ve_{\fV,\uu}$ for some $\uu \in \fe_\fV$ or
$\de_{\fV,(\um, \uu)} x_{\fV,\uu}$ for some $(\um, \uu) \in \fd_\fV$}$$
if $T_{\fV, B}$ contains $x_{\fV, \uu}$ (hence $\uu \in \II_{3,n}^{\lt,0}$).
In particular, $B_\fV$ is square-free.
\end{prop}
\begin{proof} We follow the notation as in the proof of 
Proposition \ref{meaning-of-var-vtk}.

When $k=0$,   we have $(\tsV_{\vt_{[0]}} \subset \tsR_{\vt_{[0]}})=  (\sV_\sF \subset \sR)$. 
Then, the statement follows from
Proposition \ref{equas-p-k=0}.

Suppose that the statement holds for $(\tsV_{\vt_{[k-1]}} \subset \tsR_{\vt_{[k-1]}})$
for some $k \in [\up]$. 

We now consider $(\tsV_{\vt_{[k]}} \subset \tsR_{\vt_{[k]}})$.

As in  the proof of 
Proposition \ref{meaning-of-var-vtk}, we let $\fV$ be a standard chart 
of $\tsR_{\vt_{[k]}}$ lying over a (necessarily unique) 
standard chart  $\fV'$ of $\tsR_{\vt_{[k-1]}}$.
Also, $\fV$ lies over a unique standard chart $\fV_{[0]}$ of 
of $\tsR_{\vt_{[0]}}$. We let $\pi_{\fV, \fV_{[0]}}: \fV \to \fV_{[0]}$ be the induced projection.

We divide the proof into two parts according to different types of the chart.
 
\noindent 
$\bullet$ {\bf (A).}
{\it First, we suppose 
$(\um,\uu_k) \in \La_{[k]}^o$ on the chart $\fV$
or $\fV$ is a $\vr$-standard chart.}
In the former  case, 
by Definition \ref{fv-k=0} (cf. also \eqref{index-sR}),
$x_{\fV_{[0]}, (\um,\uu_k)}\equiv 1$,
hence $x_{\fV, (\um,\uu_k)}\equiv 1$.
In the latter case, as $\fV$ is a $\vr$-standard chart,
 $(\um,\uu_k) \in \fd_\fV$, hence
  $x_{\fV, (\um,\uu_k)}= 1$  by \eqref{conv:=1}.

{\bf (A): Defining relations}.
 We first prove the statement about the defining equations of 
$\tsV_{\vt_{[k]}} \cap \fV$ in $\fV$. 
 By applying the inductive assumption to  $\fV'$,
 it suffices to prove that on the chart $\fV$,
the proper transform of any non-governing binomial  with respect to
 $F_k$ depends on the governing binomials with respect to $F_k$
(see Definition \ref{gov-ngov-bi}).
To this end,  we take any two $s, t \in S_{F_k} \- s_{F_k}$ and consider
 the non-governing binomial  $B_{F_k,(s,t)}$ (see \eqref{ngv-bi}).
Then, in any of the two cases under the situation {\bf A}, one calculates and finds that we have the following two governing binomials 
 $$B_{ \fV,(s_{F_k},s)}: \;\;\;\;\;\; x_{\fV, (\uu_s, \uv_s)}  x_{\fV, \uu_k}   - 
 \tilde x_{\fV, \uu_s} \tilde x_{\fV, \uv_s},$$
 $$B_{ \fV,(s_{F_k},t)}: \;\;\;\;\;\; x_{\fV, (\uu_t, \uv_t)}  x_{\fV, \uu_k}   - \tilde x_{\fV, \uu_t} \tilde x_{\fV, \uv_t},$$
 where  $ \tilde x_{\fV, \uw}=\pi_{\fV, \fV_{[0]}}^* x_{\fV_{[0]},\uw}$ denoted the pullback
 for any $\uw \in \II_{3,n} \- \um$.
  Similarly, one calculates and finds  that 
 we have 
$$ B_{ \fV,(s,t)}: \;\; x_{\fV,(\uu_s, \uv_s)}\tilde x_{\fV,\uu_t} \tilde x_{\fV, \uv_t}-
x_{\fV,(\uu_t, \uv_t)} \tilde x_{\fV,\uu_s} \tilde x_{ \fV,\uv_s}.$$ 
Then, one verifies directly that we have
$$ B_{ \fV,(s,t)}=x_{\fV, (\uu_s, \uv_s)} B_{ \fV,(s_{F_k},t)} -x_{\fV, (\uu_t, \uv_t)} B_{ \fV,(s_{F_k},s)}.$$
This proves  the statement about the defining equations of $\tsV_{\vt_{[k]}} \cap \fV$ in $\fV$.

The above allows us to discard the non-governing relations 
in $\cB^\ngv_{\le k}$ on the chart $\fV$.

 {\bf (A): $\vr$-linear and square-free}.
First, consider any $B \in \cB^\gov$ with respect to $F_j$.
Observe that  $x_{(\um, \uu_k)}$ uniquely appears in the governing binomials 
with respect to $F_k$;  $x_{\uu_k}$ only appears in the governing binomials 
with respect to $F_k$ and   the minus terms of certain governing binomials 
of $F_{j}$ with $j>k$. It follows that $B_\fV$ is $\vr$-linear and square-free.
 
 Likewise, consider any $B \in \cB^\ngv_{>k}$
 with respect to $F_j$ with $j >k$.
It is of the form
$$ B_{(s,t)}: \;\; x_{(\uu_s, \uv_s)}x_{\uu_t}  x_{\uv_t}-
x_{(\uu_t, \uv_t)} x_{\uu_s}  x_{ \uv_s}$$ 
for some $s \ne t \in S_{F_j}$.
Observe here that $B=B_{(s,t)}$ does not contain any $\vr$-variable of the form $x_{(\um, \uu)}$ and
the $\vp$-variables in $B$ are identical to
those of the minus terms of the corresponding governing binomials.
Hence, the same line of the proof above for  governing binomials
 implies that $B_\fV$ is $\vr$-linear and square-free.

Finally, consider any $B \in \cB^\frb$.  


Note that the proper transform of the $\vt$-center 
$\vt_{[k]}$ on the chart $\fV'$ equals to 
$$(x_{\fV', \uu_k}, x_{\fV', (\um, \uu_k)}).$$ Thus,
from the chart $\fV'$ to the $\vr$-standard chart $\fV$, 
we have that  $x_{\fV',(\um, \uu_k)}$ becomes 
$\de_{\fV,(\um, \uu_k)}$ and
 $x_{\fV',\uu_k}$ turns into $\de_{\fV,(\um, \uu_k)} x_{\fV, \uu_k}$ 
 in $B_\fV$.

If $B_{\fV'}$ does not contain $x_{\fV',\uu_k}$
nor $x_{\fV',(\um, \uu_k)}$, or $(\um,\uu_k) \in \La_{[k]}^o$,
then the form of $B_{\fV'}$ remains unchanged
 (except for the meanings of its variables). 
In this case, all the statements about $B_\fV$ holds by the inductive assumption.

Suppose next that $B_{\fV'}$  contains  $ x_{\fV',\uu_k}$ or
$x_{\fV',(\um, \uu_k)}$, and $(\um,\uu_k) \notin \La_{[k]}^o$.
Then it contains exactly one of $ x_{\fV',\uu_k}$ and
$x_{\fV',(\um, \uu_k)}$ but not both 
according to   Proposition \ref{strong sq free} (4) (also cf. (3)).
Therefore, $B_\fV$ admits
the $\vt$-exceptional parameter in the form
 $$\de_{\fV,(\um, \uu_k)}$$
without $x_{\fV, \uu_k}$
if one of the terms of $B$ contains $x_{(\um, \uu_k)}$, or, in the form
 $$\de_{\fV,(\um, \uu_k)}x_{\fV, \uu_k}$$
if one of the terms of $B$ contains $x_{\uu_k}$.
Thus, by the inductive assumption on $\fV'$, 
the final statements on  $B_\fV$ hold, in the situation {\bf A}. 

Thus,  this proves all the  statement of the proposition when 
$(\um,\uu_k) \in  \La_{[k]}^o$ or when $\fV$ is a $\vr$-standard chart.

\noindent 
$\bullet$ {\bf (B).}
{\it Next, we consider the case when $(\um,\uu_k) \notin \La_{[k]}^o$
and  $\fV$ is a $\vp$-standard chart. }
 
{\bf (B): Defining relations}. 
 Again, to prove the statement about the defining equations of $\tsV_{\vt_{[k]}} \cap \fV$ in $\fV$, 
  it suffices to prove that the proper transform
 of any non-governing binomial of $F_k$ depends on the governing binomials on the chart $\fV$.
To show this, we again take any two $s, t \in S_{F_k} \- s_{F_k}$.
 On the chart $\fV$, we have the following two the governing binomials 
 $$B_{ \fV,(s_{F_k},s)}: \;\;\;\;\;\; x_{\fV, (\uu_s, \uv_s)}    - x_{\fV, (\um,\uu_k)}
 \tilde x_{\fV, \uu_s} \tilde x_{\fV, \uv_s},$$
 $$B_{ \fV,(s_{F_k},t)}: \;\;\;\;\;\; x_{\fV, (\uu_t, \uv_t)}    - x_{\fV, (\um,\uu_k)}
 \tilde x_{\fV, \uu_t} \tilde x_{\fV, \uv_t}.$$
We also have  the following non-governing binomial
$$ B_{ \fV,(s,t)}: \;\; x_{\fV,(\uu_s, \uv_s)}\tilde x_{\fV,\uu_t} \tilde x_{\fV, \uv_t}-
x_{\fV,(\uu_t, \uv_t)} \tilde x_{\fV,\uu_s} \tilde x_{ \fV,\uv_s}.$$ 
 Then, we have
 $$B_{ \fV,(s,t)}=  \tilde x_{\fV, \uu_t} \tilde x_{\fV, \uv_t} B_{ \fV,(s_{F_k},s)}
 -\tilde x_{\fV, \uu_s} \tilde x_{\fV, \uv_s} B_{ \fV,(s_{F_k},t)}.$$
 Thus, the  statement of the proposition about the equations of  $\tsV_{\vt_{[k]}}  \cap \fV$ follows.
 
 {\bf (B): $\vr$-linear and square-free}.
 Next, consider any $B \in \cB^\gov \sqcup \cB^\ngv_{>k}$. The fact
 that $B_\fV$ is $\vr$-linear and 
 square-free follows from the same line of proof in the previous case.
 
Finally, consider any $B \in \cB^\frb$.
Again, the proper transform of the $\vt$-center 
$\vt_{[k]}$ on the chart $\fV'$ equals to $(x_{\fV', \uu_k}, x_{\fV', (\um, \uu_k)})$. Thus, 
 if $B_{\fV'}$ does not contain $ x_{\fV',\uu_k}$
nor $x_{\fV',(\um, \uu_k)}$,
then the form of $B_{\fV'}$ remains unchanged.
In this case, all the statements about $B_\fV$ holds by the inductive assumption.
Suppose next that $B_{\fV'}$  contains  $x_{\fV',\uu_k}$ or
$x_{\fV',(\um, \uu_k)}$.
Then it contains exactly one of $x_{\fV',\uu_k}$ and
$x_{\fV',(\um, \uu_k)}$ but not both 
according to   Proposition \ref{strong sq free} (3) and (4).
 Hence, from the chart $\fV'$ to the $\vp$-standard chart $\fV$, 
we have that 
$ x_{\fV',\uu_k}$ turns into $\ve_{\fV,\uu_k}$ and
 $x_{\fV',(\um, \uu_k)}$ turns into $\ve_{\fV, \uu_k} x_{\fV,(\um, \uu_k)}$  in $B_\fV$.
  Therefore, by the inductive induction on $B_{\fV'}$, we see that 
the final statements on  $B_\fV$ hold.

This completes the proof of the proposition.
 \end{proof}

We need the final case of $\vt$-blowups. So, we set 
$$\tsR_{\vt}:=\tsR_{\vt_{[\up]}}, \; \tsV_{\vt}:=\tsV_{\vt_{[\up]}}.$$

\begin{cor}\label{eq-for-sV-vr} 
Let the notation be as in Proposition \ref{equas-fV[k]} for $k=\up$.  
Then, the scheme $\tsV_{\vt} \cap \fV$, as a closed subscheme of
the chart $\fV$ of $\tsR_{\vt}=\tsR_{\vt_{[\up]}}$,  is defined by 
\begin{eqnarray} 
\cB^\gov_{\fV} ,\;\; L_{\sF, \fV}, \;\; \cB^\frb_\fV,
\end{eqnarray}
where the first two sets make of all governing relations.
Further, for any binomial $B_\fV \in \cB^\gov_\fV \cup \cB^\frb_\fV$, it is  $\vr$-linear and square-free.
\end{cor}

\begin{cor}\label{no-(um,uu)} Fix any $k \in [\up]$. 
Let $X_{\vt, (\um,\uu_k)}$ be the proper transform of
$X_{(\um,\uu_k)}$ in $\tsR_{\vt}$. Then 
$$\tsV_{\vt} \cap X_{\vt, (\um,\uu_k)} =\emptyset.$$ 
Put it differently, on any chart $\fV$ of  $\tsR_{\vt}$, 
$x_{\fV, (\um,\uu_k)}$ is invertible as long as it is a variable
in $\var_\fV$. In particular, 
the scheme $\tsV_{\vt}$ is
 covered by the standard charts that either lie over  the
chart $(x_{(\um,\uu_k)} \equiv 1)$ of $\sR$ or lie over
the $\vr$-standard chart of $\tsR_{\vt_{[k]}}$. 
\end{cor}
\begin{proof} Fix any standard chart $\fV$. 

If $\fV$ lies over the chart $(x_{(\um,\uu_k)} \equiv 1)$ of $\sR$, 
that is, $(\um,\uu_k) \in  \La_{[k]}^o$,
then the first statement $\tsV_{\vt} \cap X_{\vt, (\um,\uu_k)} =\emptyset$
follows from the definition. 

If $\fV$ lies over a $\vr$-standard chart of $\tsR_{\vt_{[k]}}$, 
then the fact that $\tsV_{\vt} \cap X_{\vt, (\um,\uu_k)} =\emptyset$ follows from Proposition \ref{meaning-of-var-vtk}  (4)
(cf. also $x_{\fV, (\um,\uu_k)} = 1$ by the convention \eqref{conv:=1}).

Suppose $\fV$ lies over a $\vp$-standard chart 
$\fV_{\vt_{[k]}}$ of $\tsR_{\vt_{[k]}}$. Then, in this case,
we have the following governing binomial relation on the chart $\fV_{\vt_{[k]}}$
\begin{equation}\label{invertible-umuu}
 B_{ \fV_{\vt_{[k]}},(s_{F_k}, s_{F, o})}: \;\;\;\; 1   - 
x_{\fV_{\vt_{[k]}}, (\um,\uu_k)}x_{\fV_{\vt_{[k]}}, \uu_{s_{F, o}}} x_{\fV_{\vt_{[k]}}, \uv_{s_{F, o}}}
\end{equation}
because $x_{\fV_{\vt_{[k]}}, (\uu_{s_{F,o}},\uu_{s_{F,o}})} \equiv 1$ 
with $(\uu_{s_{F,o}},\uu_{s_{F,o}}) \in  \La_{[k]}^o$
and $x_{\fV_{\vt_{[k]}}, \uu_k} = 1$ by \eqref{conv:=1}.  
Thus $x_{\fV_{\vt_{[k]}}, (\um,\uu_k)}$ is nowhere vanishing along
  $\tsV_{\vt_{[k]}} \cap \fV$.  This implies the statement.
  \end{proof}

\begin{defn}\label{preferred-chart-vt}
We call any standard chart of $\tsR_\vt$ as described by Corollary \ref{no-(um,uu)} 
 a preferred standard chart. 
\end{defn}

\begin{rem} On $\vp$-standard charts, $x_{\fV_{\vt_{[k]}}, (\um,\uu_k)}$ is an invertible
variable. Thus,  all $\vp$-standard charts  miss the divisor $X_{\vt, (\um,\uu_k)}$. 
On the other hand, on $\vr$-standard charts, $x_{\fV_{\vt_{[k]}}, (\um,\uu_k)}$ is not a variable, instead, $\de_{\fV_{\vt_{[k]}}, (\um,\uu_k)}$ is. 
We choose to work with the latter.
\end{rem}

\section{The Universal $\wp$- and $\ell$-Blowups}\label{sect:wp/ell-blowups} 

\subsection{ The initial setup: $\wp_0$-blowups and $\ell_0$-blowup} $\ $

Our initial scheme is  $\tsR_{\ell_{-1}}:=\tsR_\vt$.

We let 
$$\tsR_{\ell_0} \lra \tsR_{\wp_0} \lra \tsR_{\ell_{-1}}:= \tsR_\vt$$
be the trivial blowups along the empty set. 
In the sequel, a blowup is called trivial if it is a blowup along the emptyset.
We make
the identifications
$$\tsR_{\ell_0} = \tsR_{\wp_0} =\tsR_{\ell_{-1}}= \tsR_\vt.$$

For any $k \in [\up]$, we are to construct 
$$\tsR_{\ell_k} \lra \tsR_{\wp_k} \lra \tsR_{\ell_{k-1}}.$$ 
The morphism $\tsR_{\wp_k} \to \tsR_{\ell_{k-1}}$ is decomposed as
a sequential  blowups based on all relations in $\cB^\gov_{F_k}$,
and, $\tsR_{\ell_k} \to \tsR_{\wp_k} $ is a single   blowup
based on the relation $L_{F_k}$.

We do it by induction on the set $[\up]$.

As a part of the initial data on
the scheme $\tsR_{\ell_{-1}}=\tsR_{\vt}$, we have equipped it with the following sets of smooth divisors,

 $\bullet$  The set  $\cD_{\vt, \vp}$ of
  the proper transforms $X_{\vt, \uw}$ of $\vp$-divisors $X_\uw$ for all $\uw \in \II_{3,n}\- \um$.
  These are still called $\vp$-divisors. 

$\bullet$  The set  $\cD_{\vt, \vr}$ of the proper transforms $X_{\vt, (\uu,\uv)}$
of $\vr$-divisors $X_{(\uu,\uv)}$ for all $(\uu,\uv) \in \La_\sF$.
These are still called $\vr$-divisors. 

 $\bullet$    The set $\cE_{\vt}$ of the proper transforms
 $E_{\vt, k} \subset \tsR_{\vt}$ of the $\vt$-exceptional divisors 
 $E_{\vt_{[k]}} \subset \tsR_{\vt_{[k]}}$ for all $k \in [\up]$.


$\bullet$    The set  $\cD_{\vt, \fL}$ of the proper transforms $D_{\vt, L_F}$
of $\fL$-divisors $D_{L_F}$ for all $\bF \in \sF$.
These are still called $\fL$-divisors. 

Thus, on the intial scheme $\tsR_{\vt}$, we have the set of  smooth divisors
$$\cD_\vt=\cD_{\vt, \vr}\sqcup \cD_{\vt, \vp} \sqcup \cE_{\vt}.$$
In addition, we have $\cD_{\vt, \fL}$. The set $\cD_{\vt, \fL}$ will 
not be used until the $\ell$-blowup.

As the further initial data, we need to introduce the instrumental notion: $``$association$"$
with multiplicity  as follows. 

\begin{defn}
 Consider any governing binomial relation $B \in \cB^\gov$ written as
$$
 B=B_{(s_{F_k},s)}=T^+_B-T^-_B: 
\;\;\;\;\;\; x_{(\uu_s, \uv_s)}  x_{\uu_k}   - 
x_{(\um, \uu_k)} x_{\uu_s} x_{uv\uv_s}$$
for some $k \in [\up]$ and $s \in S_{F_k}$, where $\uu_k=\uu_{F_k}$.
Consider any $\vp$-divisor, $\vr$-divisor, or  exceptional divisor $Y$ on $\tsR_{\vt}$.

Let  $Y=X_{\vt,\uu}$ be any $\vp$ divisor for some $\uu \in \II_{3,n}$. 
We set
$$m_{Y, T^+_B}=\left\{
\begin{array}{rcl}
1, & \;\; \hbox{if  $\uu=\uu_k$} \\
0 ,& \;\; \hbox{otherwise.}\\ 
\end{array} \right. $$ 
$$m_{Y, T^-_B}=\left\{
\begin{array}{rcl}
1, & \;\; \hbox{if  $\uu=\uu_s$ or $\uu=\uv_s$,} \\
0 ,& \;\; \hbox{otherwise.} \\ 
\end{array} \right. $$ 

Let  $Y=X_{\vt, (\uu,\uv)}$ be any $\vr$ divisor. We set
$$m_{Y, T^+_B}=\left\{
\begin{array}{rcl}
1, & \;\; \hbox{if  $(\uu,\uv)=(\uu_s,\uv_s)$} \\
0 ,& \;\; \hbox{otherwise.}\\ 
\end{array} \right. $$

Due to Corollary \ref{no-(um,uu)}, we do not  associate
$X_{(\um,\uu_k)}$ with $T^-_B$.
Hence, we set
$$m_{Y, T^-_B}=0.$$


Let  $Y=E_{\vt, j}$ be any exceptional-divisor for some $j \in [\up]$. If $k=j$,
we set  $$m_{Y, T^+_B}=m_{Y, T^-_B}=0.$$
Suppose now $k \ne j$. We  set
$$m_{Y, T^+_B}=0, $$ 
$$m_{Y, T^-_B}=m_{X_{\vt, \uu_j}, T^-_B}. $$ 

We call the number $m_{Y, T^\pm_B}$ the multiplicity of $Y$ associated with the term $T^\pm_B$.
We say $Y$ is associated with $T^\pm_B$ if $m_{Y, T^\pm_B}$ is positive. 
We do not say $Y$ is associated with $T^\pm_B$
if the multiplicity $m_{Y, T^\pm_B}$ is zero.
\end{defn}

\begin{example}
  Let $B=T_B^+ - T_B^-= x_{2uv}x_{(12u,23v)} -x_{12u}x_{23v} x_{(123,2uv)}$.
Then, we have 
$$m_{Y, T^-_B}=\left\{
\begin{array}{rcl}
1, & \;\; \hbox{if \; $Y=X_{12u}, X_{23v}$ or $X_{(123,2uv)}$;} \\
0 ,& \;\; \hbox{otherwise.}\\ 
\end{array} \right. $$
\end{example}

\begin{defn} Consider any linearized $\pl$ relation
$$L_F=\sum_{s \in S_F} \sgn (s) x_{(\uu_s, \uv_s)}.$$
for some $F\in \sF$.  Fix any $s \in S_F$.
Consider any $\vp$-divisor, $\vr$-divisor, or  exceptional-divisor $Y$ on $\tsR_{\vt}$.

Let  $Y=X_{\vt,\uw}$ be any $\vp$ divisor for some $\uw \in \II_{3,n}$. We set
$m_{Y, s}=0.$

Let  $Y=X_{\vt, (\uu,\uv)}$ be any $\vr$ divisor. We set
$$m_{Y, s}=\left\{
\begin{array}{rcl}
1, & \;\; \hbox{if  $(\uu,\uv)=(\uu_s,\uv_s)$} \\
0 ,& \;\; \hbox{otherwise.}\\ 
\end{array} \right. $$

Let  $Y=E_{\vt, j}$ be any exceptional-divisor for some $k \in [\up]$. We let
$$m_{Y, s}=m_{X_{\vt, (\um, \uu_k)}, s} \;. $$

We call the number $m_{Y, s}$ the multiplicity of $Y$ associated with $s \in S_F$.
We say $Y$ is associated with $s$ if $m_{Y,s}$ is positive. We do not say $Y$ is associated with $s$
if the multiplicity $m_{Y,s}$ is zero.
\end{defn}

Now, take any $k \in [\up]$. 

We  suppose that all the blowups
$$\tsR_{\ell_{j}} \lra \tsR_{\wp_{j}} \lra \tsR_{\ell_{j-1}}$$  
have been constructed for all the blocks in 
$$\fG_{ [k-1]}=\bigsqcup_{j \in [k-1]} \fG_j$$
 such that for all $j \in [k-1]$:

$\bullet$ $\tsR_{\wp_{j}}\lra\tsR_{\ell_{j-1}}$ is a sequential $\wp$-blowups
with respect to  the governing binomial relations of the block $\cB^\gov_{F_{j}}$.

$\bullet$ $\tsR_{\ell_{j}}\lra\tsR_{\wp_{j}}$ is a single $\ell$-blowup
with respect to  $L_{F_{j}}$.

We  are to construct 
$$\tsR_{\ell_{k}} \lra \tsR_{\wp_{k}} \lra \tsR_{\ell_{k-1}}$$  
in the next subsection.

\subsection{$\wp$-blowups and $\ell$-blowup in
the block $(\fG_k)$}\label{wp/ell-centers} $\ $

\subsubsection{$\wp$-blowups in the block $(\fG_k)$}


We proceed by applying  induction on 
the set $$\{(k\tau)\mu \mid k \in [\up], \tau \in [\ft_{F_k}],
\mu \in [\rho_{(k\tau)}]\},$$ ordered lexicographically on $(k,\tau, \mu)$, 
where $\rho_{(k\tau)}$ is a  to-be-defined 
finite positive integer depending on $(k\tau) \in \Index_{\cB^\gov}$
(cf. \eqref{indexing-Bmn}).


The initial case for $\wp$-blowups with respect to
the block $\fG_k$  is $\wp_{(k1)}\fr_0$ and the initial scheme is 
$\sR_{({\wp}_{(k1)}\fr_{0})}:=\tsR_{\ell_{k-1}}$. 
When $k=0$, we get $\sR_{({\wp}_{(11)}\fr_{0})}:=\tsR_{\ell_{-1}}=\tsR_\vt$. 

We suppose that the scheme $\tsR_{({\wp}_{(k\tau)}\fr_{\mu-1})}$ has been constructed and
the following package in $({\wp}_{(k\tau)}\fr_{\mu-1})$ has been introduced
for some  integer $\mu \in [\rho_{(k\tau)}]$, where $1\le \rho_{(k\tau)}\le \infty$ is 
 an integer depending on $(k\tau) \in \Index_{\cB^\gov}$. (It will be proved to be finite.)
 Here, to reconcile notations, we make the convention:
$$({\wp}_{(k\tau)}\fr_0):=({\wp}_{(k(\tau -1))}\fr_{\rho_{(k(\tau-1))}}), \; \forall \;\; 1\le k\le \up, \;
2\le \tau\le \ft_{F_k},$$
$$({\wp}_{(k1)}\fr_0):=({\wp}_{((k-1)\ft_{F_{k-1}})}\fr_{\rho_{((k-1)\ft_{F_{k-1}})}}), \; \forall \;\; 2\le k\le \up, $$
provided that $\rho_{(k(\tau-1))}$ and $\rho_{((k-1)\ft_{F_{k-1}})}$
are (proved to be) finite.

\medskip\noindent
$\bullet$  {\sl The inductive assumption package.} 
{\it The scheme $\tsR_{({\wp}_{(k\tau)}\fr_{\mu-1})}$ has been constructed.
It comes equipped with the following.

\smallskip
{\bf ($\di$) The sets of  divisors required to define blowup centers.}

The set of $\vp$-divisors, 
$$\cD_{({\wp}_{(k\tau)}\fr_{\mu-1}),\vp}: \;\; X_{({\wp}_{(k\tau)}\fr_{\mu-1}), \uw}, \;\; \uw \in \II_{3,n} \- \um.$$

The set of $\vr$-divisors
$$\cD_{({\wp}_{(k\tau)}\fr_{\mu-1}), \vr}: \;\; X_{({\wp}_{(k\tau)}\fr_{\mu-1}), (\uu,\uv)}, \;\; (\uu, \uv) \in \La_\sF.$$


The set of $ \fL$-divisors 
$$\cD_{({\wp}_{(k\tau)}\fr_{\mu-1}),  \fL}: 
{ D_{({\wp}_{(k\tau)}\fr_{\mu-1}), L_{F_j},} \;\; j \in [k-1].}$$

 The set of the exceptional divisors 
 $$\cE_{({\wp}_{(k\tau)}\fr_{\mu-1})}: \;\; E_{({\wp}_{(k\tau)}\fr_{\mu-1}), (k'\tau') \mu' h'},\;\; 
 \hbox{${ (11)0}\le (k'\tau') \mu' \le (k\tau)(\mu -1), \; h' \in [\si_{(k'\tau') \mu'}]$} $$
for some finite positive integer $\si_{(k'\tau')\mu'}$ depending on $(k'\tau')\mu'$.
We set $\si_{(11)0}=\up$. This counts the number of exceptional divisors on 
$\tsR_{(\wp_1\fr_1\fs_0)}=\tsR_\vt$.


We let $$\cD_{({\wp}_{(k\tau)}\fr_{\mu-1})}=\cD_{({\wp}_{(k\tau)}\fr_{\mu-1}),\vr}
 \sqcup \cD_{({\wp}_{(k\tau)}\fr_{\mu-1}),\vp}
\sqcup \cE_{({\wp}_{(k\tau)}\fr_{\mu-1})}$$ be the set of all  the aforelisted divisors. 
The set $\cD_{({\wp}_{(k\tau)}\fr_{\mu-1}),  \fL}$ will not be used until the $\ell$-blowup.

\smallskip
{\bf ($\di$) Associate divisors with terms of defining relations.}

Fix  any $Y \in \cD_{({\wp}_{(k\tau)}\fr_{\mu-1})}$.
 Consider any $B \in \cB^\frb \cup \cB^\gov$ and let $T_B$ be any fixed  term of $B$.
Then, we have that $Y$ is associated with $T_B$ with the multiplicity  $m_{Y,T_B}$, 
a nonnegative integer.  
In the sequel, we say $Y$ is associated with $T_B$  if $m_{Y,T_B}>0$; 
we do not say $Y$ is associated with $T_B$  if $m_{Y,T_B}=0$.

Likewise, for any term of $T_s=\sgn (s) x_{(\uu_s,\uv_s)}$ of $L_F
=\sum_{s \in S_F} \sgn (s) x_{(\uu_s,\uv_s)}$, $Y$ 
is associated with $T_s$ with the multiplicity  $m_{Y,s}$, a nonnegative integer. 
We say $Y$ is associated with $T_s$  if $m_{Y,s}>0$; 
we do not say $Y$ is associated with $T_s$  if $m_{Y,s}=0$.}

\smallskip
We are now  to construct the scheme $\tsR_{({\wp}_{(k\tau)}\fr_\mu)}$.
 The process consists of a finite steps of blowups; the scheme $\tsR_{({\wp}_{(k\tau)}\fr_\mu)}$
is the one obtained in the final step. 




As before, fix any $k \in [\up]$,  we write $\cB^\gov_{F_k}=\{B_{(k\tau)} \mid \tau \in [\ft_{F_k}]\}.$
For every $B_{(k\tau)} \in \cB_{F_k}^\gov$, we have the expression  
$$B_{(k\tau)}=T_{(k\tau)}^+ -T_{(k\tau)}^- =x_{(\uu_s,\uv_s)}x_{\uu_k} -
x_{\uu_s}x_{\uv_s}x_{(\um,\uu_k)} $$  where $ s\in S_{F_k} \- s_{F_k}$ corresponds to $\tau$
and
$x_{\uu_k}$ is the leading variable of $\bF_k$ for some $\uu_k \in \II_{3,n}^\lt$.
We can write $s=(k\tau)$, and, use $B_s$ and $B_{(k\tau)}$ interchangeblly.

\smallskip
{\bf ($\di$)  We now introduce the blowup centers.}

\begin{defn}\label{wp-sets-kmu} 
A pre-${\wp}$-set $\phi$ in $({\wp}_{(k\tau)}\fr_\mu)$, written as
$$\phi=\{Y^+, \; Y^- \},$$
 consists of exactly two divisors 
 of the scheme $\tsR_{({\wp}_{(k\tau)}\fr_{\mu-1})}$ such that  
$Y^\pm$ is  associated with $T_{(k\tau)}^\pm$. 
  
  Given the above pre-${\wp}$-set $\phi$, we let 
$$Z_\phi = Y^+ \cap Y^-$$
 be the scheme-theoretic intersection.
The pre-$\wp$-set $\phi$ (resp. $Z_\phi $) is called a $\wp$-set (resp. $\wp$-center) 
in $({\wp}_{(k\tau)}\fr_\mu)$ if $$Z_\phi \cap \tsV_{({\wp}_{(k\tau)}\fr_{\mu-1})} \ne \emptyset.$$
In such a case, we also call $\phi$ (resp. $Z_\phi $) a $\wp_k$-set (resp. $\wp_k$-center). 
\end{defn}
Recall that due to Corollary \ref{no-(um,uu)}, we do not  associate
$X_{(\um,\uu_k)}$ with $T_{(k\tau)}^-$.
Hence $Y^- \ne X_{({\wp}_{(k\tau)}\fr_{\mu-1}), (\um,\uu_k)}$.
Had we associated $X_{({\wp}_{(k\tau)}\fr_{\mu-1}),(\um,\uu_k)}$ with $T_{(k\tau)}^-$,  the condition
$Z_\phi \cap \tsV_{({\wp}_{(k\tau)}\fr_{\mu-1})} \ne \emptyset$ would also exclude it.



As there are only finitely many $\vp$-, $\vr$-, and exceptional 
divisors on the scheme $\tsR_{({\wp}_{(k\tau)}\fr_{\mu-1})}$, that is,
the set $\cD_{({\wp}_{(k\tau)}\fr_{\mu-1})}$ is finite,
one sees that there are only finitely many ${\wp}$-sets in $({\wp}_{(k\tau)}\fr_\mu)$. 
We let $\Phi_{{\wp}_{(k\tau)}\fr_\mu}$ be the finite set of all ${\wp}$-sets
in $({\wp}_{(k\tau)}\fr_\mu)$; we let $\cZ_{{\wp}_{(k\tau)}\fr_\mu}$ be the finite set of all 
corresponding ${\wp}$-centers in $({\wp}_{(k\tau)}\fr_\mu)$.
We need  a total ordering on the set
$\Phi_{{\wp}_{(k\tau)}\fr_\mu}$, hence on the set $\cZ_{{\wp}_{(k\tau)}\fr_\mu}$,
to produce a canonical sequential blowups.

\smallskip
{\bf ($\di$) Ordering the set of blowup centers.}

\begin{defn}\label{order-phi} Let $\cD^\pm_{({\wp}_{(k\tau)}\fr_{\mu-1})}$ be the set of all 
divisors associated with $T_{(k\tau)}^\pm$. 

We order the set  $\cD^+_{({\wp}_{(k\tau)}\fr_{\mu-1})}$  as follows.
We let $X_{({\wp}_{(k\tau)}\fr_{\mu-1}),(\uu_s,\uv_s)}$ be the largest and
$X_{({\wp}_{(k\tau)}\fr_{\mu-1}),\uu_k}$ be the second largest. The rest are
exceptional divisors. We order them by reversing the order of occurrence
of the exceptional divisors. 
By Definition \ref{p-t}, $\cD^+_{({\wp}_{(k\tau)}\fr_{\mu-1})}$ is totally ordered.

We order the set  $\cD^-_{({\wp}_{(k\tau)}\fr_{\mu-1})}$  as follows.
We let $\cD^-_{({\wp}_{(k\tau)}\fr_{\mu-1}),\vp}$ be the subset 
of $\vp$-divisors with the order 
induced from that  on the set of all $\pl$-variables.
We let $\cE^-_{({\wp}_{(k\tau)}\fr_{\mu-1})}$  be the subset of exceptional divisors
by reversing the order of occurrence.
We then declare
$$
\cE^-_{({\wp}_{(k\tau)}\fr_{\mu-1})}<  \cD^-_{({\wp}_{(k\tau)}\fr_{\mu-1}),\vp}.$$
By Definition \ref{p-t}, $\cD^-_{({\wp}_{(k\tau)}\fr_{\mu-1})}$ is totally ordered.

Now, let $\phi_1, \phi_2 \in \Phi_{{\wp}_{(k\tau)}\fr_\mu}$ be any two distinct elements.
Write $\phi=\{Y_i^+, Y_i^-\}, i=1, 2$.
We $\phi_1 < \phi_2$ if $Y_1^+<Y_2^+$ or $Y_1^+=Y_2^+$ and $Y_1^-<Y_2^-$. 

This endows $\Phi_{{\wp}_{(k\tau)}\fr_\mu}$ a total order $`` < "$.
\end{defn}

Thus, we can list $\Phi_{{\wp}_{(k\tau)}\fr_\mu}$ as
$$\Phi_{ {\wp}_{(k\tau)}\fr_\mu}=\{\phi_{(k\tau)\mu 1} < \cdots < \phi_{(k\tau)\mu \si_{(k\tau)\mu}}\}$$
for some finite positive integer $\si_{(k\tau)\mu}$ depending on $(k\tau)\mu$. We then let the set 
$\cZ_{\wp_{(k\tau)}\fr_\mu}$ of the corresponding $\wp$-centers
inherit the total order from that of $\Phi_{{\wp}_{(k\tau)}\fr_\mu}$. 
Then, we can express
$$\cZ_{\wp_{(k\tau)}\fr_\mu}=\{Z_{\phi_{(k\tau)\mu1}} < \cdots < Z_{\phi_{(k\tau)\mu\si_{(k\tau)\mu}}}\}.$$


We let $\tsR_{({\wp}_{(k\tau)}\fr_\mu \fs_1)} \lra \tsR_{({\wp}_{(k\tau)}\fr_{\mu-1})}$ 
be the blowup of $\tsR_{({\wp}_{(k\tau)}\fr_{\mu-1})}$
along the {$\wp$-}center $Z_{\phi_{(k\tau)\mu1}}$. 
Inductively, we assume that 
$\tsR_{({\wp}_{(k\tau)}\fr_\mu\fs_{(h-1)})}$ has been constructed for some 
$ h \in [\si_{(k\tau)\mu}]$.
We then let  $$\tsR_{({\wp}_{(k\tau)}\fr_\mu\fs_h)} \lra \tsR_{({\wp}_{(k\tau)}\fr_\mu\fs_{h-1})}$$ be the blowup of
$ \tsR_{({\wp}_{(k\tau)}\fr_\mu\fs_{h-1})}$ along (the proper transform of) the  $\wp$-center 
$Z_{\phi_{(k\tau)\mu h}}$. We call it a $\wp$-blowup or a $\wp_k$-blowup.

Here, to reconcile notation, we set 
$$\tsR_{({\wp}_{(k\tau)}\fr_\mu\fs_{0})}:=\tsR_{({\wp}_{(k\tau)}\fr_{\mu-1})}
:=\tsR_{({\wp}_{(k\tau)}\fr_{\mu-1}\fs_{\si_{(k\tau)(\mu-1)}})}.$$

All of these can be summarized as the sequence 
$$\tsR_{({\wp}_{(k\tau)}\fr_\mu)}:= \tsR_{({\wp}_{(k\tau)}\fr_\mu\fs_{\si_{(k\tau)\mu}})}
 \lra \cdots \lra \tsR_{({\wp}_{(k\tau)}\fr_\mu\fs_1)} \lra \tsR_{({\wp}_{(k\tau)}\fr_{\mu-1})}.$$

\smallskip
{\bf ($\di$) The new  divisors required
 to continue the blowing up process.}

 Given $h \in [\si_{(k\tau)\mu}]$,
consider the induced morphism
$\tsR_{({\wp}_{(k\tau)}\fr_\mu\fs_h)} \lra 
\tsR_{({\wp}_{(k\tau)}\fr_{\mu-1})}$. 

$\bullet$ We let $X_{({\wp}_{(k\tau)}\fr_\mu\fs_h), \uw}$ be the proper transform
of $X_{(\wp_{(k\tau)}\fr_{\mu-1}), \uw}$ in $\tsR_{({\wp}_{(k\tau)}\fr_\mu\fs_h)}$,
for all $\uw \in \II_{3,n} \- \um$. These are still called $\vp$-divisors.
We denote the set of all $\vp$-divisors
 on  $\tsR_{({\wp}_{(k\tau)}\fr_\mu\fs_h)}$ by  $\cD_{({\wp}_{(k\tau)}\fr_\mu\fs_h),\vp}$.

$\bullet$ We let $X_{({\wp}_{(k\tau)}\fr_\mu\fs_h), (\uu, \uv)}$ be the proper transform
  of the $\vr$-divisor $X_{({\wp}_{(k\tau)}\fr_{\mu-1}), (\uu, \uv)}$ in $\tsR_{({\wp}_{(k\tau)}\fr_\mu\fs_h)}$,
  for all $(\uu, \uv) \in \La_\sF$. These are still called $\vr$-divisors.
We denote the set of all $\vr$-divisors
 on  $\tsR_{({\wp}_{(k\tau)}\fr_\mu\fs_h)}$ by  $\cD_{({\wp}_{(k\tau)}\fr_\mu\fs_h),\vr}$.
 
 
 $\bullet$ We let $D_{({\wp}_{(k\tau)}\fr_\mu\fs_h), L_F}$ be the proper transform
  of the $\fL$-divisor $D_{({\wp}_{(k\tau)}\fr_{\mu-1}), L_F}$ in $\tsR_{({\wp}_{(k\tau)}\fr_\mu\fs_h)}$,
  for all $\bF \in \sF$. These are still called $\fL$-divisors.
We denote the set of all $\fL$-divisors
 on  $\tsR_{({\wp}_{(k\tau)}\fr_\mu\fs_h)}$ by  $\cD_{({\wp}_{(k\tau)}\fr_\mu\fs_h),\fL}$.

 $\bullet$ We let
 $E_{({\wp}_{(k\tau)}\fr_\mu\fs_h),  (k'\tau')\mu' h'}$ be the proper transform 
 of $E_{({\wp}_{(k\tau)}\fr_{\mu-1}),  (k'\tau')\mu' h'}$ in $\tsR_{({\wp}_{(k\tau)}\fr_\mu \fs_h)}$,
 for all   $(11)0 \le (k'\tau')\mu'\le (k\tau)(\mu-1)$
 with $h' \in [\si_{(k'\tau')\mu'}]$.
We denote the set of these exceptional divisors on  
  $\tsR_{({\wp}_{(k\tau)}\fr_\mu\fs_h)}$ by $\cE_{({\wp}_{(k\tau)}\fr_\mu \fs_h),{\rm old}}$.
 
  We let  $$\bar\cD_{({\wp}_{(k\tau)}\fr_\mu\fs_h)}= \cD_{({\wp}_{(k\tau)}\fr_\mu\fs_h),\vp}
  \sqcup \cD_{({\wp}_{(k\tau)}\fr_\mu\fs_h),\vr} \sqcup \cE_{({\wp}_{(k\tau)}\fr_\mu\fs_h),{\rm old}}$$ 
  be the set of  all of the aforementioned divisors on $\tsR_{({\wp}_{(k\tau)}\fr_\mu \fs_h)}$.
 The set $\cD_{({\wp}_{(k\tau)}\fr_\mu\fs_h),\fL}$ will not be used until the $\ell$-blowup.

 In addition to the proper transforms of the divisors
  from $\tsR_{({\wp}_{(k\tau)}\fr_{\mu-1})}$, there are
the following {\it new} exceptional divisors.

For any $ h \in [\si_{(k\tau)\mu}]$, we let $E_{({\wp}_{(k\tau)}\fr_\mu\fs_h)}$ be the exceptional divisor of 
the blowup $\tsR_{({\wp}_{(k\tau)}\fr_\mu\fs_h)} \lra \tsR_{({\wp}_{(k\tau)}\fr_\mu\fs_{h-1})}$.
 For any $1\le h'< h\le \si_{(k\tau)\mu}$,
we let $E_{({\wp}_{(k\tau)}\fr_\mu\fs_h), (k\tau)\mu h'}$ 
be the proper transform in $\tsR_{({\wp}_{(k\tau)}\fr_\mu \fs_h)} $
of the exceptional divisor $E_{({\wp}_{(k\tau)}\fr_\mu \fs_{h'})}$. 
To reconcile notation, we also set 
$E_{({\wp}_{(k\tau)}\fr_\mu\fs_h), (k\tau)\mu h}:=E_{({\wp}_{(k\tau)}\fr_\mu \fs_h)}$.
We  set $$\cE_{({\wp}_{(k\tau)}\fr_\mu \fs_h),{\rm new}}=\{ E_{({\wp}_{(k\tau)}\fr_\mu\fs_h), 
(k\tau)\mu h'} \mid 1\le h'\le h\le \si_{(k\tau)\mu} \} .$$
We then order the exceptional divisors of $ \cE_{({\wp}_{(k\tau)}\fr_\mu\fs_h),{\rm new}}$
by  reversing the order of occurrence, that is, 
$E_{({\wp}_{(k\tau)}\fr_\mu\fs_h), (k\tau)\mu h''}\le E_{({\wp}_{(k\tau)}\fr_\mu\fs_h), (k\tau)\mu h'}$
 if $h'' \ge h'$.
 
 We then let 
 $$\cD_{({\wp}_{(k\tau)}\fr_\mu \fs_h)} =\bar\cD_{({\wp}_{(k\tau)}\fr_\mu \fs_h)} 
\sqcup  \cE_{({\wp}_{(k\tau)}\fr_\mu \fs_h),{\rm new}}.$$

Finally, we set
 $\tsR_{({\wp}_{(k\tau)})\fr_\mu}:= \tsR_{({\wp}_{(k\tau)}\fr_\mu\fs_{\si_{(k\tau)\mu}})}$, and let
 $$\cD_{({\wp}_{(k\tau)}\fr_\mu),\vp}=\cD_{({\wp}_{(k\tau)}\fr_\mu\fs_{\si_{(k\tau)\mu}}),\vp}, \;
 \cD_{({\wp}_{(k\tau)}\fr_\mu),\vr}=\cD_{({\wp}_{(k\tau)}\fr_\mu\fs_{\si_{(k\tau)\mu}}),\vr},$$
 $$\cE_{{\wp}_{(k\tau)}\fr_\mu}=\cE_{({\wp}_{(k\tau)}\fr_\mu\fs_{\si_{(k\tau)\mu}}), {\rm old}}  
 \sqcup \cE_{({\wp}_{(k\tau)}\fr_\mu\fs_{\si_{(k\tau)\mu}}), {\rm new}} .$$
 This can be  summarized as 
$$\cD_{({\wp}_{(k\tau)}\fr_\mu)}:=\cD_{({\wp}_{(k\tau)}\fr_\mu),\vp}  
\sqcup \cD_{({\wp}_{(k\tau)}\fr_\mu),\vr}
\sqcup \cE_{({\wp}_{(k\tau)}\fr_\mu)}.$$
 This way,  we have equipped 
  the scheme  $\tsR_{({\wp}_{(k\tau)}\fr_\mu)}$ 
  with the set $\cD_{({\wp}_{(k\tau)}\fr_\mu),\vp}$ of $\vp$-divisors, 
 the set $\cD_{({\wp}_{(k\tau)}\fr_\mu),\vr}$ of $\vr$-divisors, 
 the set  $\cE_{{\wp}_{(k\tau)}\fr_\mu}$ of exceptional divisors,
 together with the set $\cD_{({\wp}_{(k\tau)}\fr_\mu),\fL}$ of $\fL$-divisors
 (which will not be used until the $\ell$-blowup).

\smallskip
{\bf ($\di$) Associate divisors with  defining relations in 
the round $({\wp}_{(k\tau)}\fr_\mu)$,
as required to carry on the process of induction. }

We do it inductively  on  the set $ [ \si_{(k\tau)\mu}]$. 



\begin{defn} \label{association-vs} 
Fix any $B \in \cB^\gov \sqcup \cB^\frb $. We let $T_B$ be any fixed term of the binomial $B$.
Meanwhile, we also consider any $\bF \in \sF$ and let $T_s$ be the term of $L_F$ corresponding to
any fixed $s \in S_F$.

We assume that the notion of $``$association$"$ in  
$({\wp}_{(k\tau)}\fr_\mu\fs_{h-1})$ has been introduced. That is, for every divisor 
$Y' \in \cD_{({\wp}_{(k\tau)}\fr_\mu\fs_{h-1})}$, the multiplicities
$m_{Y', T_B}$ and $m_{Y',s}$ have been defined.

Consider an arbitrary divisor $Y \in \cD_{({\wp}_{(k\tau)}\fr_\mu\fs_h)}$.



First, suppose $Y \ne E_{({\wp}_{(k\tau)}\fr_\mu\fs_h)}$.
Then, it is the proper transform of a (unique) divisor $Y' \in \cD_{({\wp}_{(k\tau)}\fr_\mu\fs_{h-1})}$.
In this case, we set 
$$m_{Y, T_B}=m_{Y', T_B},\;\;\; m_{Y,s}=m_{Y',s}.$$

Next, we consider the exceptional $Y=E_{({\wp}_{(k\tau)}\fr_\mu\fs_h)}$.

We let $\phi= \phi_{(k\tau)\mu h}$. 
We have that 
$$\phi=\{ Y^+,  Y^-  \} \subset \cD_{(\wp_{(k\tau)}\fr_{\mu-1})}.$$

For any $B \in \cB^\gov \cup \cB^\frb$, we write $B=T_B^0-T_B^1$. We let
$$m_{\phi, T_B^i}=m_{Y^+, T_B^i}+ m_{Y^-, T_B^i}, \; i=0,1,$$ 
 $$l_{\phi, B}=\min \{m_{\phi, T_B^0}, m_{\phi, T_B^1}\}.$$ 
 (For instance, by definition, $l_{\phi, B}>0$ when $B=B_{(k\tau)}$. In general, it can be zero.)
 Then, we let
  $$m_{E_{({\wp}_{(k\tau)}\fr_\mu\fs_h)},T_B^i}= m_{\phi, T_B^i}-  l_{\phi, T_B^i}.$$

Likewise, 
for $s \in S_F$ with $F\in \sF$,
we let
$$m_{E_{(\wp_{(k\tau)}\fr_\mu\fs_h)},s}=m_{Y^+, s} +  m_{Y^-, s}.$$

We say $Y$ is associated with $T_B$ (resp. $T_s$)  if its multiplicity $m_{Y, T_B}$ 
(resp. $m_{Y, s}$) is positive.
We do not say $Y$ is associated with $T_B$ (resp. $T_s$)   
if its multiplicity $m_{Y,T_B}$ (resp. $m_{Y, s}$) equals to zero.
\end{defn}

When $h = \si_{(k\tau)\mu}$, we obtain all the desired data on 
$\tsR_{({\wp}_{(k\tau)}\fr_\mu)}=\tsR_{({\wp}_{(k\tau)}\fr_\mu)\fs_{\si_{(k\tau)\mu}}}$.

\smallskip
{\bf ($\di$) Recapitulating  the inductive procedure of $\wp$-blowups
in ($\fG_k$).}

Now, with all the aforedescribed data equipped for $\tsR_{({\wp}_{(k\tau)}\fr_\mu)}$, 
we obtain our inductive package in 
 $({\wp}_{(k\tau)}\fr_\mu)$. This allows us to introduce the set 
 $\Phi_{\wp_{(k\tau)}\fr_{\mu+1}}$ of ${\wp}$-sets 
 and the set  $\cZ_{\wp_{(k\tau)}\fr_{\mu+1}}$ of ${\wp}$-centers
 in  $({\wp}_{(k\tau)}\fr_{\mu+1})$ as in Definition \ref{wp-sets-kmu}, 
 endow a total order on $\Phi_{\wp_{(k\tau)}\fr_{\mu+1}}$
 and  $\cZ_{\wp_{(k\tau)}\fr_{\mu+1}}$ as in 
 Definition \ref{order-phi}, 
 and then  advance to the next 
round of  the $\wp$-blowups.  Here, 
 to reconcile notations, we set
 $$({\wp}_{(k\tau)}\fr_{\rho_{(k\tau)+1}}):=({\wp}_{((k(\tau +1))}\fr_{1}), \;\; 1\le \tau <\ft_{F_k};$$
$$({\wp}_{(k\ft_{F_k})}\fr_{\rho_{(k\ft_{F_k})+1}}):=({\wp}_{((k+1)1)}\fr_{1}), \;\;  1\le k < \up,$$
provided that $\rho_{(k\tau)}$ and $\rho_{(k\ft_{F_k})}$ are (proved to be) finite.

Given any  $({\wp}_{(k\tau)}\fr_\mu \fs_h)$, the ${\wp}$-blowup in (${\wp}_{(k\tau)}\fr_\mu\fs_h$)
gives rise to \begin{equation}\label{tsv-ktauh} \xymatrix{
\tsV_{({\wp}_{(k\tau)}\fr_\mu\fs_h)} \ar[d] \ar @{^{(}->}[r]  & \tsR_{({\wp}_{(k\tau)}\fr_\mu\fs_h)} \ar[d] \\
\sV \ar @{^{(}->}[r]  & \sR,
}
\end{equation}
where $\tsV_{({\wp}_{(k\tau)}\fr_\mu\fs_h)} $ is the proper transform of $\sV$ 
in  $\tsR_{({\wp}_{(k\tau)}\fr_\mu\fs_h)}$. 

We let $\tsV_{({\wp}_{(k\tau)}\fr_\mu)}=\tsV_{({\wp}_{(k\tau)}\fr_\mu\fs_{\si_{(k\tau)\mu}})}$.

\begin{defn}\label{defn:rhoktau}
Fix any $k \in [\up], \tau \in [\ft_{F_k}]$. Suppose there exists a finite integer $\mu$ such that
for any pre-$\wp$-set $\phi$  
in $({\wp}_{(k\tau)}\fr_{\mu +1})$ (cf. Definition \ref{wp-sets-kmu}), 
we have 
$$Z_\phi  \cap \tsV_{({\wp}_{(k\tau)}\fr_\mu)} = \emptyset.$$
Then, we let $\rho_{(k\tau)}$ be the smallest integer such that the above holds.
Otherwise, we let $\rho_{(k\tau)}=\infty$.
\end{defn}
{\it It will be shown soon 
that $\rho_{(k\tau)}$ is finite for all  $k \in [\up]$,
$\tau \in [\ft_{F_k}]$.}

For later use, granting the finiteness of $\rho_{(k\tau)}$, we let
\begin{equation}\label{indexing-Phi}
\Phi_k=\{ \phi_{(k\tau)\mu h} \mid                    
\tau \in [\ft_{F_k}],
\mu \in [\rho_{(k\tau)}],
 h \in [\si_{(k\tau)\mu}]\},  \end{equation}
$$\Index_{\Phi_k}=\{ (k\tau)\mu h \mid     
 \tau \in [\ft_{F_k}],  \mu \in [\rho_{(k\tau)}],
 h \in [\si_{(k\tau)\mu}]\}. $$
 Then, the order of $\wp_k$-blowups coincides with the lexicographical  order 
on $\Index_{\Phi_k}$

Upon proving that $\rho_{(k\tau)}$ is finite  for all $\tau \in [\ft_{F_k}]$,
 we can summarize the process
of ${\wp_k}$-blowups as a single sequence of blowup morphisms:
\begin{equation}\label{grand-sequence-wp}
\tsR_{\wp_k} \lra \cdots \lra 
\tsR_{({\wp}_{(k\tau)}\fr_\mu\fs_h)} \lra \tsR_{({\wp}_{(k\tau)}\fr_\mu\fs_{h-1})}
\lra \cdots \lra \tsR_{({\wp}_{(k1)}\fr_0)}:=\tsR_{\ell_{k-1}},\end{equation}
where $\tsR_{\wp_k}:=\tsR_{({\wp}_{(k \ft_{F_k})}\fr_{\rho_{k \ft_{F_k}}})} 
:=\tsR_{({\wp}_{(k \ft_{F_k})}\fr_{\rho_{k \ft_{F_k}}}\fs_{\si_{(k \ft_{F_k})\rho_{k \ft_{F_k}}}})} $
is the  blowup scheme reached in  the final step 
$({\wp}_{(k \ft_{F_k})}\fr_{\rho_{k \ft_{F_k}}}
\fs_{\si_{(k \ft_{F_k})\rho_{k \ft_{F_k}}}})$ of all $\wp$-blowups in $(\fG_k)$.

Further, the end of all ${\wp}$-blowups in $(\fG_k)$ gives rise to the following induced diagram 
\begin{equation}\label{tsv-final} \xymatrix{
\tsV_{\wp_k} \ar[d] \ar @{^{(}->}[r]  & \tsR_{\wp_k} \ar[d] \\
\sV \ar @{^{(}->}[r]  & \sR,
}
\end{equation}
where $\tsV_{\wp_k}$ is the proper transform of $\sV$  in  $\tsR_{\wp_k}$.

\subsubsection{The $\ell$-blowup in  the block $(\fG_k)$} 


\begin{defn}\label{ell-set-k} 
Let $D_{\wp_k, L_{F_k}}$ be proper transform of the $\fL$-divisor defined by
$L_{F_k}$ and $E_{\wp_k, \vt_k}$ be  the proper transform
of the exceptional divisor $E_{\vt, k}$ (of $\tsR_\vt$)
 in $\tsR_{\wp_k}$.

We call the set of the two divisors 
$$\chi_k=\{D_{\wp_k, L_{F_k}}, E_{\wp_k, \vt_k} \}$$
the pre-$\ell$-set with respect to $L_{F_k}$ or just pre-$\ell_k$-set.
We let $Z_{\chi_k}$ be the scheme-theoretic intersection
$$Z_{\chi_k} = D_{ \wp_k, L_{F_k}} \cap  E_{ \wp_k, \vt_k}.$$
The pre-$\ell$-set $\chi_k$ (resp. $Z_{\chi_k} $) is called a $\ell$-set
 (resp. $\ell$-center)  with respect to $L_{F_k}$
 if $$Z_{\chi_k} \cap \tsR_{\wp_k} \ne \emptyset.$$
\end{defn}
{\it Indeed, we believe that the intersection 
$Z_{\chi_k} \cap \tsR_{\wp_k}$ is always non-empty. 
But, the non-empty assumption possesses no harm for what to follow.}

We then let
$$\tsR_{\ell_k} \lra \tsR_{\wp_k}$$
be the blowup of $\tsR_{\wp_k}$ along $Z_{\chi_k}$.
This is called the $\ell$-blowup with respect to $F_k$ or in $(\fG_k)$,
or simply $\ell_k$-blowup.

\smallskip
{\bf ($\di$) The divisors in  $(\ell_k)$,
 required to carry on the process of induction.}

Recall that $\tsR_{\wp_k}$ comes equipped with the sets of
$\vr$-, $\vp$-, $\fL$-, and exceptional divisors.
For any such a divisor, we take its proper transform in $\tsR_{\ell_k}$ and let
it inherit its original name.

There is only one new divisor on $\tsR_{\ell_k}$:
we let $E_{\ell_k}$ be the exceptional divisor of the blowup 
$\tsR_{\ell_k} \lra \tsR_{\wp_k}$,  and call it the $\ell_k$-exceptional divisor.
(We comment here that this is not to be confused with the $\fL$-divisor
$D_{\ell_k, L_k}$.)

\smallskip
{\bf ($\di$) The notion of $``$association$"$ in  $(\ell_k)$,
 required for carrying on induction. }

\begin{defn} \label{ell-association} 
Fix any binomial $B \in \cB^\gov_G \sqcup 
 B^\frb$ with $G>F_k$, written as $B=T_B^+ -T_B^-$.
For any $\bF \in \sF$, let $T_s$ be the term of $L_F$ corresponding to
some fixed $s \in S_F$.

As assumed, the notion of $``$association$"$ in  
$ (\wp_k)$ has been introduced. That is, for any divisor 
$Y'$ on $\tsR_{\wp_k}$ , the multiplicities
 $m_{Y', T^\pm_B}$, $m_{Y',s}$ have been defined.

Suppose a divisor $Y$ of  $\tsR_{\ell_k}$ is the proper transform 
of $Y'$ of  $\tsR_{\wp_k}$.
We define
$$ m_{Y, T^\pm_{B}} =  m_{Y', T^\pm_{B}}, \;\; m_{Y,s}= m_{Y', s}.$$
We define 
$$  m_{E_{\ell_k}, T^\pm_{B}} =  m_{E_{\wp_k, \vt_k}, T^\pm_{B}}, \;\;
 m_{E_{\ell_k},s}= m_{E_{\wp_k, \vt_k}, s}.$$

We say $Y$ is associated with  $T_B^\pm$ or $T_s$ if its multiplicity  
$m_{Y, T^\pm_B}$ or $m_{Y,s}$  is positive.
\end{defn}

Finally, we have the following diagram \begin{equation}\label{ell_k-schemes} \xymatrix{
\tsV_{\ell_k} \ar[d] \ar @{^{(}->}[r]  & \tsR_{\ell_k} \ar[d] \\
\tsV_{\wp_k} \ar @{^{(}->}[r]  & \sR_{\wp_k},  }
\end{equation}
where $\tsV_{\ell_k} $ is the proper transform of $\sV$ 
in  $\tsR_{\ell_k} $.

When $k=\up$, we write $\tsR_{\ell}:= \tsR_{\ell_\up}$
and $\tsV_{\ell}:=\tsV_{\ell_\up}$. We obtain
\begin{equation}\label{final-schemes} \xymatrix{
\tsV_{\ell} \ar[d] \ar @{^{(}->}[r]  & \tsR_{\ell} \ar[d] \\
\sV \ar @{^{(}->}[r]  & \sR.  }
\end{equation}
The schemes $(\tsV_{\ell} \subset \tsR_{\ell})$ are the ones we aim to construct.

\medskip

For the use in induction in the next subsection, we set
$\tsR_{({\wp}_{(k1)}\fr_1\fs_{0})}:=\tsR_{\ell_{k-1}}$.
When $k=0$, we have 
$\tsR_{({\wp}_{(11)}\fr_1\fs_{0})}:=\tsR_{\ell_{-1}}:=\tsR_{\vt}$.

In the next two sections, we will state and prove  certain properties about 
$$\tsV_{\wp_{(k\tau)}\fr_\mu\fs_h} \subset \tsR_{\wp_{(k\tau)}\fr_\mu\fs_h},$$
 using induction on the indexes $(k\tau)\mu h$.
To include $(\tsV_{\ell_k} \subset \tsR_{\ell_k})$ in the statements and proofs,
 we let ${\ell_k}$ correspond to one more step after the last step of $\wp_k$-blowups,
namely, $\si_{(k \ft_{F_k})\rho_{k \ft_{F_k}}}+1$. In full writing, we set
$$\tsR_{({\wp}_{(k \ft_{F_k})}\fr_{\rho_{k \ft_{F_k}}}\fs_{(\si_{(k \ft_{F_k})\rho_{k \ft_{F_k}}}+1)})}
:=\tsR_{\ell_k}, \;\; 
\tsV_{({\wp}_{(k \ft_{F_k})}\fr_{\rho_{k \ft_{F_k}}}\fs_{(\si_{(k \ft_{F_k})\rho_{k \ft_{F_k}}}+1)})}
:=\tsV_{\ell_k}, $$
We also write 
$$\ell_k:= (k \ft_{F_k})\rho_{k \ft_{F_k}}(\si_{(k \ft_{F_k})\rho_{k \ft_{F_k}}}+1),$$
or
$$\ell_k:=({\wp}_{(k \ft_{F_k})}\fr_{\rho_{k \ft_{F_k}}}
\fs_{(\si_{(k \ft_{F_k})\rho_{k \ft_{F_k}}}+1)}).$$
Both mean we consider the $\ell_k$-blowup, and the notations are
only used for convenience in induction.

We will prove various statements by induction on the following totally ordered set
\begin{equation}\label{the set}
\Om= \{(11)10 \}\sqcup \bigsqcup_{k=1}^\up 
(\Index_{\Phi_k} \sqcup \{\ell_k\}) 
\end{equation}
whose general element is denoted by  $(k\tau)\mu h$,
where  we use 
$\ell_k:= (k \ft_{F_k})\rho_{k \ft_{F_k}}(\si_{(k \ft_{F_k})\rho_{k \ft_{F_k}}}+1)$
so that $\ell_k$ is also in the form of $(k\tau)\mu h$.
We have the following linear order
$$\{(11)10 \}< \Index_{\Phi_k} <  \{\ell_k\} < \Index_{\Phi_{k+1}}$$
for all $k \in [\up-1]$.
Now, using the lexicographic order on  the set 
$\Index_{\Phi_k}$ as in \eqref{indexing-Phi}, 
by the comment immediately after Definition \ref{p-t}, the set
$\Om$ of \eqref{the set}
is totally ordered. Furthermore, this order also coincides with
the lexicographical order on  $\{(k\tau)\mu h\}$.

\subsection{Properties of $\wp$-blowups and $\ell$-blowup in 
the block $(\fG_k)$}\label{subs:prop-vs-blowups} $\ $

\begin{defn}
Fix  any $(k\tau)\mu h \in \Om$. Consider the inclusion
$$\tsV_{({\wp}_{(k\tau)}\fr_\mu\fs_h)}\subset \tsR_{({\wp}_{(k\tau)}\fr_\mu\fs_h)}.$$
An affine chart $\fV$ 
of $\tsR_{({\wp}_{(k\tau)}\fr_\mu\fs_h)}$ is called admissible if
 $\fV \cap \tsV_{({\wp}_{(k\tau)}\fr_\mu\fs_h)} \ne \emptyset$.
\end{defn}
Thus,  $\tsV_{({\wp}_{(k\tau)}\fr_\mu\fs_h)}$ can be covered by
admissible charts. Again, recall that give a blowup morphsim
$Y \lra X$, we say an affine chart $\fV$ of $Y$ lies over
an affine chart $\fV'$ $X$ if  $\fV$ is mapped into $\fV'$.

Fix  any $(k\tau)\mu h \in \Om$.
In the following proposition,  we will introduce 
various terms  about $\tsR_{({\wp}_{(k\tau)}\fr_\mu\fs_h)}$
and prove their  properties by using induction on $(k\tau)\mu h \in \Om$.
This will set up a platform required
to reach our ultimate goal.

\begin{prop}\label{meaning-of-var-wp/ell} We apply induction to 
$(k\tau)\mu h \in \Om$.
Suppose that inductively the scheme 
$$\tsR_{({\wp}_{(k\tau)}\fr_\mu\fs_h)}/\tsR_{({\wp}_{(k\tau)}\fr_\mu\fs_{h-1})}$$
 has been constructed, and admits 
a finite set of admissible affine open charts covering 
$\tsV_{({\wp}_{(k\tau)}\fr_\mu\fs_h)}/\tsV_{({\wp}_{(k\tau)}\fr_\mu\fs_{h-1})}$.
These charts enjoy the following properties.

Fix and consider any admissible
 affine chart $\fV$ of $\tsR_{({\wp}_{(k\tau)}\fr_\mu\fs_h)}$, 
 lying over a unique admissible affine chart 
$\fV'$ of $\tsR_{({\wp}_{(k\tau)}\fr_\mu\fs_{h-1})}$,
 and both lying over a unique affine  chart $ \fV_{[0]}$ of $\sR$
(cf. Definition \ref{general-standard-chart}).
 We suppose that the chart $ \fV_{[0]}$ is indexed by
$\La_\sF^o=\{(\uv_{s_{F,o}},\uv_{s_{F,o}}) \mid \bF \in \sF \}.$  
We set 
$$\hbox{$\II_{3,n}^\star=\II_{3,n}\- \um$, $\La_\sF^\star=\La_\sF \- \La_\sF^o$,
and $L_{\sF, [k]}=\{L_{F_j} \mid j \in [k]\}$.}$$


\medskip
{\bf ($\di$) Smoothness 
of $\tsR_{({\wp}_{(k\tau)}\fr_\mu\fs_h)}$ along $\tsV_{({\wp}_{(k\tau)}\fr_\mu\fs_h)}$ }
\smallskip

By shrinking the admissible affine  chart $\fV$  if necessary, we have that
$\fV$ is smooth. More precisely, $\fV' \cap Z'_{\phi_{(k\tau) \mu h}}$
is smooth, where $Z'_{\phi_{(k\tau) \mu h}}$ is the proper transform
of the $\wp$- or $\ell$-center $Z_{\phi_{(k\tau) \mu h}}$ in  
$\tsR_{({\wp}_{(k\tau)}\fr_\mu\fs_{h-1})}$.

\medskip
{\bf ($\di$) Various exceptional divisors on the chart and their labels}
\smallskip

The smooth admissible 
affine chart $\fV$ of $ \tsR_{({\wp}_{(k\tau)}\fr_\mu\fs_h)}$ comes equipped with 
$$\hbox{a subset}\;\; \fe_\fV  \subset \II_{3,n}^\star , \;\;
 \hbox{a subset} \;\; \fd_\fV  \subset \La_{\sF}^\star, \;\;
 \hbox{and a subset} \;\;  \fl_\fV \subset  {L_{\sF,[k]}}
 $$ and also
 $$\hbox{the subset} \;\; \fd_\fV^\lt=\{(\um,\uu_F) \mid L_F \in \fl_\fV\} \subset \fd_\fV$$ 
 such that every exceptional divisor $E$ 
(i.e., not a $\vp$- nor a $\vr$-divisor nor a $\fL$-divisor)
  of $\tsR_{{\wp}_{(k\tau)}}$
with $E \cap \fV \ne \emptyset$ is 
either labeled by a unique element $\uw \in \fe_\fV$
 or labeled by a unique element $(\uu,\uv) \in \fd_\fV$ or
 labeled by a unique element $ L \in \fl_\fV$. 
We let $E_{ \ell_k, \uw}$ be the unique exceptional divisor 
on the chart $\fV$ labeled by $\uw \in \fe_\fV$; we call it an $\vp$-exceptional divisor.
We let $E_{\ell_k , (\uu,\uv)}$ be the unique exceptional divisor 
on the chart $\fV$ labeled by $(\uu,\uv) \in \fd_\fV$;  we call it an $\vr$-exceptional divisor.
 We let $E_{\ell_k, L}$ be the unique exceptional divisor 
on the chart $\fV$ labeled by $L \in \fl_\fV$; we call it an $\fl$-exceptional divisor.
(We note here that being $\vp$-exceptional or $\vr$-exceptional or $\fl$-exceptional is strictly relative to the given admissible affine chart. Again, $E_{\ell_k, L}$ is not to be confused with
$D_{\ell_k, L}$, the $\fL$-divisor of $\tsR_{({\wp}_{(k\tau)}\fr_\mu\fs_h)}$.)

\medskip
{\bf ($\di$)  Various free variables of  on the chart and their labels}
\smallskip

Further, the chart $\fV$ admits 
 a set of free variables
\begin{equation}\label{variables-wp/ell} 
\var_{\fV}:=\left\{ \begin{array}{ccccccc}
\ve_{\fV, \uw} , \;\; \de_{\fV, (\uu,\uv)} \\ 
y_{\fV,(\um, \uu_F)}        \\
x_{\fV, \uw} , \;\; x_{\fV, (\uu,\uv)} 
\end{array}
  \; \Bigg| \;
\begin{array}{ccccc}
 \uw \in  \fe_\fV,  \;\; (\uu,\uv)  \in \fd_\fV { \- \fd^\lt_\fV} \\
 (\um, \uu_F) { \in \fd^\lt_\fV}  \\    
\uw \in  \II_{3,n}^\star \- \fe_\fV,  \;\; (\uu, \uv) \in \La_\sF^\star \-  \fd_\fV  
\end{array} \right \}.
\end{equation}
and a set of exceptional variables for $\ell$-exceptional divisors
$$ \var_{\fl_\fV}=\{\de_{\fV,(\um, \uu_F)}  
\mid L_F \in \fl_\fV,
\; i.e., \; (\um, \uu_F) \in \fd^\lt_\fV\}$$
such that, $y_{\fV,(\um, \uu_F)} \in \var_{\fV},  (\um, \uu_F) { \in \fd^\lt_\fV} $, are invertible on the chart, which will be restated in (9) below and be prove
accordingly\footnote{We remark here that only  elements 
$(\um, \uu_F)$  in $\fd^\lt_\fV$ may index two variables  
$\de_{\fV,(\um, \uu_F)}$ and  $y_{\fV,(\um, \uu_F)}$: in such a case,
$\de_{\fV,(\um, \uu_F)}$ is $\ell$-exceptional and
$y_{\fV,(\um, \uu_F)}$ is the proper transform of $\de_{\fV',(\um, \uu_F)}$
where $\fV$ lies over $\fV'$, a chart of $\tsR_{({\wp}_{(k\tau)}\fr_\mu\fs_{h-1})}$.
}.

 Furthermore, we also have a set of free variables
\begin{equation}\label{variables-wp/ell-vee} 
\var_\fV^\vee=(\{ y \in \var_\fV \} \- \{y_{(\um,\uu_F)} \mid L_F \in \fl_\fV \})
\sqcup \{ \de_{\fV,(\um, \uu_F)}  
 \mid L_F \in \fl_\fV  \}.
\end{equation}

 We set
$$ \var_\fV^+=\var_\fV \sqcup \var_{\fl_\fV}.$$
Then, all the relations in $\cB_\fV^\gov,  \cB^\frb_\fV,  \{L_{\fV, F} \mid  F \in \sF\}$
are polynomials in $\var_\fV^+$. 





\medskip
{\bf ($\di$) Nine properties on the chart}
\smallskip

 Finally,  on the smooth admissible chart $\fV$, we have
\begin{enumerate}
\item the divisor  $X_{({\wp}_{(k\tau)}\fr_\mu\fs_h), \uw}\cap \fV$ 
 is defined by $(x_{\fV,\uw}=0)$ for every 
$\uw \in \II_{3,n}^\star \- \fe_\fV$;
\item the divisor  $X_{({\wp}_{(k\tau)}\fr_\mu\fs_h), (\uu,\uv)}\cap \fV$ is defined by $(x_{\fV,(\uu,\uv)}=0)$ for every 
$(\uu,\uv) \in \La^\star_\sF\- \fd_\fV$;
\item the divisor  $X_{({\wp}_{(k\tau)}\fr_\mu\fs_h), \uw}$ does not intersect the chart for all $\uw \in \fe_\fV$;
\item the divisor  $X_{({\wp}_{(k\tau)}\fr_\mu\fs_h), (\uu, \uv)}$ does not intersect the chart for all $ (\uu, \uv) \in \fd_\fV$;
\item the $\vp$-exceptional divisor 
$E_{({\wp}_{(k\tau)}\fr_\mu\fs_h), \uw} \;\! \cap  \fV$  labeled by an element $\uw \in \fe_\fV$
is define by  $(\ve_{\fV,  \uw}=0)$ 
for all $ \uw \in \fe_\fV$;
\item the $\vr$-exceptional divisor 
$E_{({\wp}_{(k\tau)}\fr_\mu\fs_h),  (\uu, \uv)}\cap \fV$ labeled by  an element $(\uu, \uv) \in \fd_\fV$
is define by  $(\de_{\fV,  (\uu, \uv)}=0)$ 
for all $ (\uu, \uv) \in \fd_\fV$;  
 \item the $\fl$-exceptional divisor 
$ E_{{  {  \ell_k }  },  L_F}\cap \fV$ labeled by  an element ${ L_F \in} \fl_\fV$
is define by  $({ \de_{\fV, (\um, \uu_F)}} =0)$ 
for all ${ L_F \in} \fl_\fV$;
\item any of the remaining {\rm exceptional} divisors
of $\tsR_{({\wp}_{(k\tau)}\fr_\mu\fs_h)}$ 
other than those that are labelled by some  $\uw \in \fe_\fV$ or $(\uu,\uv) \in \fd_\fV$ 
or $L \in \fl_\fV$ does not intersect the chart.

\medskip
 ($\di e$) In particular, all the aforementioned divisors are smooth.


 \item Assume $({\wp}_{(k\tau)}\fr_\mu\fs_h)=\ell_k$. %

Fix and consider any $j \in [k]$. 
  
Suppose $\fV$ lies over the admissible affine chart $(x_{(\um,\uu_{F_j})} \equiv 1)$ 
  of $\sR$. Then, $ L_{F_j} \notin \fl_\fV$, and  the $ \ell_j$-blowup is trivial. 
  
Suppose $\fV$ lies over the $\vr$-standard chart of 
$\tsR_{\vt_{[j]}}$. Then, $ L_{F_j} \in \fl_\fV$, and 
we can choose  admissible affine charts of
$\tsR_{\ell_k}$ such that they cover $\tsV_{\ell_k}$, and on any such chart $\fV$,
 we can express
\begin{equation} \nonumber 
L_{\fV ,F_j} 
=1+ \sgn(s_{F_j}) { y}_{\fV, (\um, \uu_{F_j})}
\end{equation} 
where ${ y}_{\fV, (\um, \uu_{F_j})} \in \var_\fV$ is  an invertible variable 
 in $\var_\fV$ 
labelled by  $(\um, \uu_{F_j})$.
\end{enumerate}
\end{prop}
\begin{proof} 
 As mentioned and set up earlier (cf. \eqref{the set}), we will  prove by induction on 
 $$(k\tau)\mu h \in \Om$$
where $(11)10$ is the smallest element.

For the initial case when $(k\tau)\mu h =(11)10$,
the scheme is $\tsR_{({\wp}_{(11)}\fr_1\fs_0)}=\tsR_{\vt}$.
In this case, we set $\fl_\fV=\emptyset$.
But then, this proposition is the same as Proposition \ref{meaning-of-var-vtk} with $k=\up$.

For induction, we suppose that all the statements hold over 
$\tsR_{({\wp}_{(k\tau)}\fr_\mu\fs_{h-1})}$
for some (suitable) $(k\tau)\mu (h-1) \in \Om$. 
(Recall that for the largest element of $\Index_{\Phi_k}$, if we add 1 to the index of the step, then, by convention, it corresponds to $\ell_k$.)

We now consider $\tsR_{({\wp}_{(k\tau)}\fr_\mu\fs_h)}$. 
We let 
$$\pi: \tsR_{({\wp}_{(k\tau)}\fr_\mu\fs_h)} \lra 
\tsR_{({\wp}_{(k\tau)}\fr_\mu\fs_{h-1})}$$
be the blowup morphism. We let 
$Z'_{\phi_{(k\tau) \mu h}} \subset \tsR_{({\wp}_{(k\tau)}\fr_\mu\fs_{h-1})}$
be the center of the blowup of $\pi$, this is, the proper transform
of $Z_{\phi_{(k\tau) \mu h}}$ in  $\tsR_{({\wp}_{(k\tau)}\fr_\mu\fs_{h-1})}$.

We have the embedding
$$\xymatrix{
\tsR_{({\wp}_{(k\tau)}\fr_\mu\fs_h)}\ar @{^{(}->}[r]  &
 \tsR_{({\wp}_{(k\tau)}\fr_\mu\fs_{h-1})} \times \PP_{\phi_{(k\tau)\mu h}},  }
$$ where $\PP_{\phi_{(k\tau)\mu h}}$ is the factor projective space. 
We let $\phi'_{(k\tau)\mu h}=\{Y'_0, Y'_1\}$ 
where $Y'_0, Y'_1$ are, respectively,
 the proper transforms in $\tsR_{({\wp}_{(k\tau)}\fr_\mu\fs_{h-1})}$ 
 of the two  divisors of the ${\wp}$-set $\phi_{(k\tau)\mu h}=\{Y^+,Y^-\}$
with $Y^\pm$ being associated with $T^\pm_{(k\tau)}$, or,
 in the case when
$({\wp}_{(k\tau)}\fr_\mu\fs_{h-1})=\wp_k$, 
then $({\wp}_{(k\tau)}\fr_\mu\fs_{h})=\ell_k$ and
 $(Y'_0, Y'_1)=(D_{\wp_k, L_{F_k}}, E_{\wp_k,\vt_k})$.
In addition, we let $[\xi_0, \xi_1]$ 
 be the  homogenous coordinates of $\PP_{\phi_{(k\tau)\mu h}}$
 corresponding to $\{Y'_0, Y'_1\}$.

\medskip
{\bf ($\di$) Smoothness 
of $\tsR_{({\wp}_{(k\tau)}\fr_\mu\fs_h)}$ along $\tsV_{({\wp}_{(k\tau)}\fr_\mu\fs_h)}$
}

The question is local. So, we work on  charts.
We let $\fV$ be any standard affine chart  of $\tsR_{({\wp}_{(k\tau)}\fr_\mu\fs_h)}$ 
lying over a unique standard affine
chart $\fV'$ of  $\tsR_{({\wp}_{(k\tau)}\fr_\mu\fs_{h-1})}$ such that 
$$\fV \subset (\fV' \times (\xi_i \equiv 1)) \cap \tsR_{({\wp}_{(k\tau)}\fr_\mu\fs_h)}, \;
\hbox{ for $i=0$ or 1. }$$
To prove that $\tsR_{({\wp}_{(k\tau)}\fr_\mu\fs_h)}$ is smooth 
 along $\tsV_{({\wp}_{(k\tau)}\fr_\mu\fs_h)}$, we will focus
on the  affine
charts  $\fV$ that meet $\tsV_{({\wp}_{(k\tau)}\fr_\mu\fs_h)}$ in non-empty set,
namely admissible charts, 
and when necessary, we may shrink $\fV$ to achieve smoothness and 
certain desired properties
(for example, making $y_{\fV, (\um, \uu_F)}$ invertible).
As stated in the proposition, these charts are call admissible charts.
Given an admissible chart $\fV'$ of $\tsR_{({\wp}_{(k\tau)}\fr_\mu\fs_{h-1})}$, 
 in the same vein, we only need to focus on the situation when
$ Z'_{\phi_{(k\tau) \mu h}}$ 
meets $\tsV_{({\wp}_{(k\tau)}\fr_\mu\fs_h)}\cap \fV'$ along a nonempty closed subset.

So, we fix and consider any admissible affine chart $\fV'$ of $\tsR_{({\wp}_{(k\tau)}\fr_\mu\fs_{h-1})}$ as discussed above.

By assumption, the admissible affine chart $\fV'$ is smooth and
 comes equipped with
a subset $\fe_{\fV'} \subset \II_{3,n}\- \um$, a subset  $\fd_{\fV'} \subset \La^\star_\sF$,
and admits the following three sets of  variables:

$\bcd$ a set of free variables
\begin{equation}\label{variables-wp/ell'} 
\var_{{\fV'}}:=\left\{ \begin{array}{ccccccc}
\ve_{{\fV'}, \uw} , \;\; \de_{{\fV'}, (\uu,\uv)} \\ 
y_{{\fV'},(\um, \uu_F)}        \\
x_{{\fV'}, \uw} , \;\; x_{{\fV'}, (\uu,\uv)} 
\end{array}
  \; \Bigg| \;
\begin{array}{ccccc}
 \uw \in  \fe_{\fV'},  \;\; (\uu,\uv)  \in \fd_{\fV'} { \- \fd^\lt_{\fV'}} \\
 (\um, \uu_F) { \in \fd^\lt_{\fV'}}  \\    
\uw \in  \II_{3,n}^\star \- \fe_{\fV'},  \;\; (\uu, \uv) \in \La_\sF^\star \-  \fd_{\fV'}  
\end{array} \right \},
\end{equation}
such that $y_{{\fV'},(\um, \uu_F)} $ are invertible on the chart $\fV'$;

$\bcd$ a set of exceptional variables for $\ell$-exceptional divisors
$$ \var_{\fl_{\fV'}}=\{\de_{{\fV'},(\um, \uu_F)}                 
\mid L_F \in \fl_{\fV'},
\; i.e., \; (\um, \uu_F) \in \fd^\lt_{\fV'}\};$$

$\bcd$ and a set of free variables
$$\var_{\fV'}^\vee=(\{ y \in \var_{\fV'} \} \- \{y_{(\um,\uu_F)} \mid L_F \in \fl_{\fV'} \})
\sqcup \{ \de_{{\fV'},(\um, \uu_F)}                        
\mid L_F \in \fl_{\fV'}  \}.$$
All together, they verify the properties (1)-(9) as in the proposition.


First, we consider the case when $({\wp}_{(k\tau)}\fr_\mu\fs_{h}) < \ell_k$.

Then, on the chart $\fV'$, by  assumption, we have
\begin{equation}\label{YY01-wp}
Y'_0 \cap \fV' =(y'_0 =0), \; Y'_1 \cap \fV' =(y'_1 =0), 
 \;\; \hbox{for some $y'_0, y'_1 \in  \var^+_{\fV'}$.} 
 \end{equation} 
 Now, because ${ y}_{\fV, (\um, \uu_{F_j})} \in \var_\fV$ is  
an invertible variable on the admissible affine chart $\fV'$,
we see that  when $({\wp}_{(k\tau)}\fr_\mu\fs_{h}) < \ell_k$,
 we must have $y'_0, y'_1 \in  \var^\vee_{\fV'}$. Then, 
it is immediate that the blowup center in this case
$$ Z'_{\phi_{(k\tau) \mu h}} \cap \fV'= (y'_0= y'_1 =0)$$  
is smooth on the chart $\fV'$.

Next, we consider the case when $({\wp}_{(k\tau)}\fr_\mu\fs_{h}) = \ell_k$,
that is, when $({\wp}_{(k\tau)}\fr_\mu\fs_{h-1}) = \wp_k$.

In this case, the $\ell$-set is 
$(Y'_0, Y'_1)=(D_{\wp_k, L_{F_k}}, E_{\wp_k,\vt_k})$.
By the inductive assumption, we have
$$E_{\wp_k,\vt_k} \cap \fV' = (\de_{\fV', (\um, \uu_{F_k})} =0).$$
The divisor $D_{\wp_k, L_{F_k}} \cap \fV'$ is defined
by the proper transform $L_{\fV', F_k}$. 
 We can assume that $\fV'$
 lies over a unique chart $ \fV_{[0]}$ of $\tsR_{\vt_{[0]}}=\sR$.
Over the chart $ \fV_{[0]}$, we can express
\begin{equation}\label{2 terms ne 0}
L_{\fV_{[0]}, F_k}= \de_{\fV_{[0]}, (\um, \uu_{F_k})} +
\sum_{s \in S_{F_k} \- s_{F_k}} x_{\fV_{[0]}, (\uu_s, \uv_s)}.
\end{equation}
Take any point 
$$ \bz \in \fV \cap \tsV_{({\wp}_{(k\tau)}\fr_\mu\fs_h)}.$$
We let $\bz_0 \in \fV_{[0]}$ be the image point of $\bz$.
If $\de_{\fV_{[0]}, (\um, \uu_{F_k})} (\bz_0) \ne 0$, then 
$\de_{\fV', (\um, \uu_{F_k})} (\bz') \ne 0$ where $\bz' \in \fV'$ is
the image of $\bz$. Therefore, the blowup
does not change the neighborhood of this point, hence for any
point $\bz \in \fV$ lying over $\bz'$, we obtained that $\bz$ is a smooth 
point of $\fV$. Next, we suppose 
$\de_{\fV_{[0]}, (\um, \uu_{F_k})} (\bz_0) = 0$.
Then, there must be at least two terms\footnote{
Here, the homogenous linearized $\pl$ relation plays its decisive role,
which in turn shows the decisive role played by the construction of 
$\sR \subset \bU \times \prod_{\bF \in \sF} \PP_F$.
Also, note that we are forced to focus on admissible charts, as other charts may not even be smooth and, in any case, are irrelevant.}
of $L_{\fV_{[0]}, F_k}$ that
that do not vanish at $\bz_0$. 
At least one of these two terms must be 
a $\vr$-variable (the other might be set to be the constant 1 by de-homogenization).
We may suppose that
$$x_{\fV_{[0]}, (\uu_s, \uv_s)} (\bz_0) \ne 0, \; 
\hbox{for some $s \in S_{F_k} \- s_{F_k}$}.$$
Then,  around a small neighborhood $\fV'' \subset \fV'$
of the point $\bz'$, we obtain that
\begin{equation}\label{2 terms ne 0'}
L_{\fV'', F_k}= \de_{\fV'', (\um, \uu_{F_k})} +
x_{\fV'', (\uu_s, \uv_s)} + \cdots
\end{equation}
and the $\vr$-variable $x_{\fV'', (\uu_s, \uv_s)}$
can not divide any other terms. 
Because  $Z'_{\phi_{(k\tau) \mu h}} \cap \fV''$ is defined by
$$\de_{\fV'', (\um, \uu_{F_k})} =L_{\fV'', F_k}=0,$$
we obtain that the blowup center 
$Z'_{\phi_{(k\tau) \mu h}} \cap \fV''$ is smooth at $\bz'$,
Hence $\bz$ is a smooth in a small affine neighborhood 
$\fV$ lying over $\fV'$.

Therefore, all in all, 
by shrinking $\fV'$ and $\fV$ if necessary (we keep and recycle the same notations),
we obtain that $\fV$ is smooth because
 $$ \pi^{-1}(\fV') \lra \fV' $$ 
 is the blowup of  $ \fV'$ along a smooth center.

This proves that $\tsR_{({\wp}_{(k\tau)}\fr_\mu\fs_h)}$ is smooth
along $\tsV_{({\wp}_{(k\tau)}\fr_\mu\fs_h)}$.

 Therefore, we can cover $\tsV_{({\wp}_{(k\tau)}\fr_\mu\fs_h)}$ by
a set of smooth admissible affine charts $\{\fV\}$ lying over 
a set of smooth admissible affine charts $\{\fV'\}$ of $\tsV_{({\wp}_{(k\tau)}\fr_\mu\fs_{h-1})}$. 
In what follows, by $``$chart$"$ we always mean 
an admissible chart as just described.

We now prove the remaining statements for a fixed smooth admissible affine chart 
$\fV$ of the scheme $\tsR_{({\wp}_{(k\tau)}\fr_\mu\fs_h)}$,
lying over a (unique) fixed smooth admissible affine chart 
$\fV'$ of the scheme $\tsR_{({\wp}_{(k\tau)}\fr_\mu\fs_{h-1})}$.

First, we suppose that the proper transform $Z'_{\phi_{(k\tau) \mu h}}$ 
in $\tsR_{({\wp}_{(k\tau)}\fr_\mu\fs_{h-1})}$
of the $\wp$- or the $\ell_k$-center $Z_{\phi_{(k\tau) \mu h}}$ does not
meet the chart $\fV'$.  
(Here, when $(k\tau) \mu h$ corresponds to $\ell_k$, 
we let $\phi_{(k\tau) \mu h}:=\chi_k$.)
Then, we let $\fV$ inherit all the data from those of $\fV'$, that is,
we set $\fe_{\fV}=\fe_{\fV'}$, $\fd_{\fV}=\fd_{\fV'}$, $\fl_{\fV}=\fl_{\fV'}$, $\var_\fV= \var_{\fV'}$,
$\var_{\fl_{\fV'}}= \var_{\fl_{\fV'}}$, and $\var_\fV^+= \var_{\fV'}^+$:
changing the subindex $``\ \fV' \ "$ for all the variables in $\var_{\fV'}^+$
to $``\ \fV \ "$.
As the $\wp$- or $\ell$-blowup  along the proper transform of $Z_{({\wp}_{(k\tau)}\fr_\mu\fs_h)}$
does not affect the chart $\fV'$, one checks readily that 
the statements of the proposition hold for $\fV$.

Next, we suppose that 
$Z'_{\phi_{(k\tau) \mu h}}$
meets the chart $\fV'$ along a nonempty closed subset. 

Then, the admissible chart  
$\fV$  is a (possibly small)  affine open subset  of  
$$(\fV' \times (\xi_i \equiv 1)) \cap  \tsR_{({\wp}_{(k\tau)}\fr_\mu\fs_h)}$$
of the scheme $ \tsR_{({\wp}_{(k\tau)}\fr_\mu\fs_{h})}$, 
which is, as a closed subscheme of  
$\fV' \times (\xi_i \equiv 1),$  is defined  by
\begin{equation}\label{blowup-relation:wp/ell}
y'_j = y'_i \xi_j, \;  \hbox{ with $j  \in \{0, 1\} \- i$}.             
\end{equation}

There are { six} possibilities for
$Y'_i \cap \fV'$ according to 
the types of the variable $y_i'$.
Based on the first four of such possibilities, 
we set 
\begin{equation}\label{proof:de-fv-vskmu}
\left\{ 
\begin{array}{lccr}
\fe_{\fV}=\fe_{\fV'} \sqcup \uw, \; \fd_{\fV}= \fd_{\fV'}, \; \fl_\fV=\fl_{\fV'},
&  \hbox{if} \; y'_i=x_{\fV', \uw}\; \hbox{for some} \;\; \uw \in \II_{3,n}^\star \- \fe_{\fV'}\\
\fe_{\fV}=\fe_{\fV'}, \; \fd_{\fV}= \fd_{\fV'} ,\; \fl_\fV=\fl_{\fV'}, &\hbox{if} \;    y'_i=\ve_{\fV', \uw} \; \hbox{for some} \;\; \uw \in  \fe_{\fV'}\\
\fd_{\fV}=\fd_{\fV'}\sqcup (\uu,\uv), \; \fe_{\fV}= \fe_{\fV'}, \; \fl_\fV=\fl_{\fV'},& \hbox{if} \; 
y'_i=x_{\fV', (\uu,\uv)}\; \hbox{for some}  \;\; (\uu,\uv) \in  \La_{d,n}^*  \\
\fd_{\fV}=\fd_{\fV'}, \; \fe_{\fV}= \fe_{\fV'}, \; \fl_\fV=\fl_{\fV'},& \hbox{if} \; 
y'_i=\de_{\fV', (\uu,\uv)}\; \hbox{for some}  \;\; (\uu,\uv) \in  \fd_{\fV'} \-\fd^\lt_{\fV'} \\
\end{array} \right.
\end{equation}
For the fifth possibility, we let
$$ \fl_\fV=\fl_{\fV'}, \; \fd_{\fV}=\fd_{\fV'}, \; \fe_{\fV}= \fe_{\fV'} , \;\; \hbox{if} \; 
y'_i=\de_{\fV', (\um,\uu_F)}\; \hbox{for some}  \;\; L_F \in  \fl_{\fV'} .$$
For the sixth possibility, 
we let 
$$\hbox{$\fl_\fV=\fl_{\fV'} \sqcup L_{F_k}, \; \fd_{\fV}=\fd_{\fV'}, \; \fe_{\fV}= \fe_{\fV'} $, \;\;
if $(y_i'=0)$ defines  $D_{\wp_k, L_{F_k}}\cap \fV'$.}$$
(This last case corresponds the case of the $\ell$-blowup with respect to $L_{F_k}$.)

Accordingly, we introduce 
\begin{equation}\label{proof:new-ex-fv-vskmu}
\left\{ 
\begin{array}{lccr}
\ve_{\fV, \uw}=y'_i, \; 
&  \hbox{if} \; y'_i=x_{\fV', \uw}\; \hbox{for some} \;\; \uw \in \fe_\fV \- \fe_{\fV'}\\
\ve_{\fV, \uw}=y'_i  , &\hbox{if} \;    y'_i=\ve_{\fV', \uw} \; \hbox{for some} \;\; \uw \in  \fe_{\fV'}=\fe_{\fV}\\
\de_{\fV, (\uu,\uv)}=y'_i,& \hbox{if} \; 
y'_i=x_{\fV', (\uu,\uv)}\; \hbox{for some}  \;\; (\uu,\uv) \in  \La_{d,n}^*  \\
\de_{\fV, (\uu,\uv)}=y'_i, & \hbox{if} \; 
y'_i=\de_{\fV', (\uu,\uv)}\; \hbox{for some}  \;\; (\uu,\uv) \in  \fd_{\fV'} \- \fd^\lt_{\fV'} \\
\end{array} \right.
\end{equation}
$$\hbox{$\de_{\fV, (\um, \uu_F)}=y'_i,$  \; if $y'_i=\de_{\fV', (\um, \uu_F)}$, 
for some  $L_F \in  \fl_{\fV'}$}$$ 
$$\hbox{$\de_{\fV, (\um, \uu_{F_k})}=y_i'$, \;\; if $(y_i'=0)$ defines  $D_{\wp_k, L_{F_k}}\cap \fV'$.}$$
(The last case corresponds  the $\ell$-blowup with respect to $L_{F_k}$.)

This defines the exceptional variable for the blowup.


To introduce the variable corresponding to $j  \in \{0, 1\} \- i$,
we then set
\begin{equation}\label{proof:var-xi-fv-vskmu}
\left\{ 
\begin{array}{lcr}
x_{\fV,\ua}=\xi_j, &   \hbox{if $y'_j=x_{\fV', \ua} $}\\
\ve_{\fV, \ua}=\xi_j, &  \hbox{if $y'_j= \ve_{\fV', \ua}$} \\
x_{\fV, (\ua, \ub)}= \xi_j, & \;\;\;\; \hbox{if $y'_j= x_{\fV', (\ua, \ub)}$} \\
\de_{\fV, (\ua, \ub)}= \xi_j, & \;\;\;\;\;\;\;\;\;\;\;
\hbox{if $y'_j= \de_{\fV', (\ua, \ub)}$
and $(\ua, \ub) \ne (\um, \uu_{F_k})$}. 
\end{array} \right.
\end{equation}
In addition,  we let $y_{\fV, (\um,\uu_F)}=\xi_j,$ 
if $y'_j= \de_{\fV', (\um,\uu_{F_k})}$ when $L_{F_k} \in \fl_{\fV}$. 
This is in the case $\ell_k$-blowup, and in this case,
$y_{\fV, (\um, \uu_{F_k})}$ is the proper transform 
   of $\de_{\fV', (\um, \uu_{F_k})}$.  
The exceptional divisor of $\ell_k$-blowup
is denoted by $\de_{\fV, (\um, \uu_{F_k})}$, which is already introduced earlier
in \eqref{proof:new-ex-fv-vskmu}.

Thus, we have introduced $y'_i,  \xi_j \in \var^+_\fV$ where $y'_i$, and  $\xi_j$
are endowed with the new names as in \eqref{proof:new-ex-fv-vskmu}, and in 
\eqref{proof:var-xi-fv-vskmu},  respectively.

Next, we define the set 
$\var_\fV \- \{y'_i, \xi_j\}$ 
to consist of the following variables:
\begin{equation}\label{proof:var-fv-vskmu}
\left\{ 
\begin{array}{lccr}
x_{\fV,\uw}=x_{\fV', \uw}, &  \forall \;\; \uw \in \II_{3,n}^\star\- \fe_\fV \;\; \hbox{and $x_{\fV', \uw} \ne y_i',y'_j$}\\
x_{\fV, (\uu, \uv)}=x_{\fV', (\uu, \uv)}, & \;\; \forall \;\; 
(\uu, \uv) \in \La_\sF^\star \- \fd_{\fV} \;\; \hbox{and $x_{\fV', (\uu,\uv)} \ne y_i', y'_j$} \\
\ve_{\fV, \uw}= \ve_{\fV', \uw}, & \forall \;\; \uw \in \fe_\fV \;\; \hbox{and $\ve_{\fV', \uw} \ne y'_i$},  y'_j \\
\de_{\fV, (\uu, \uv)}= \de_{\fV', (\uu, \uv)}, &  \;\; \forall \;\; (\uu, \uv) \in \fd_{\fV} \;\; \hbox{and $\de_{\fV', (\uu, \uv)} \ne y'_i$}, y'_j \\
y_{\fV, (\uu, \uu_F)}= y_{\fV', (\uu, \uu_F)}, &  \;\; \forall \;\; L_F \in \fl_{\fV}. 
\end{array} \right.
\end{equation}
Here, $L_F \in \fl_{\fV}$ implies that $L_F \in \fl_{\fV'}$ as we are considering
$\var_\fV \- \{y'_i, \xi_j\}$.

We let $\var_\fV$ be the set of
 the variables in  \eqref{proof:new-ex-fv-vskmu},
 \eqref{proof:var-xi-fv-vskmu}, and \eqref{proof:var-fv-vskmu}.
Substituting \eqref{blowup-relation:wp/ell},
one sees  that $\var_\fV$ is a set of free variables on the open chart $\fV$.
 This describes \eqref{variables-wp/ell} in the proposition.
 
 We then let $\var_{\fl_\fV}=\{\de_{\fV, (\um,\uu_F) } \mid L_F \in \fl_\fV\},$
 and obtain 
 $$\var_\fV^\vee=(\{ y \in \var_\fV \} \- \{y_{(\um,\uu_F)} \mid L_F \in \fl_\fV \})
\sqcup \{ \de_{\fV,(\um, \uu_F)}  
 \mid L_F \in \fl_\fV  \}.$$
  One sees 
 that $\var_\fV^\vee$ is also a set of free variables on the open chart $\fV$.
This checks \eqref{variables-wp/ell-vee} in the proposition.
 
 We set $\var_\fV^+=\var_\fV \sqcup \var_{\fl_\fV}.$
By substituting \eqref{blowup-relation:wp/ell} and taking proper transforms,
 one sees that all the relations in $\cB_\fV^\gov,  \cB_\fV^{\frb}, L_{\sF, \fV}$
are polynomials in $\var_\fV^+$.

 Now, it remains to verity (1)-(9) of the proposition on the chart $\fV$.

First, consider  the unique new exceptional divisor 
$E_{({\wp}_{(k\tau)}\fr_\mu\fs_h)}$
created by the blowup
$\tsR_{({\wp}_{(k\tau)}\fr_\mu\fs_h)} \lra \tsR_{({\wp}_{(k\tau)}\fr_\mu\fs_{h-1})}. $ 
Then, we have
$$ E_{({\wp}_{(k\tau)}\fr_\mu\fs_h)} \cap \fV = (y'_i=0)$$
where $y'_i$ is renamed as in \eqref{proof:new-ex-fv-vskmu} and 
in the  sentence immediately following it.
This way, the new exceptional divisor $E_{({\wp}_{(k\tau)}\fr_\mu\fs_h)}$
 is labelled on the chart $\fV$.
Further, we have that the proper transform of $Y'_i$ in
 $\tsR_{({\wp}_{(k\tau)}\fr_\mu\fs_h)}$ 
 does not meet the chart $\fV$, and
if $Y'_i$ is an exceptional parameter labeled by some element of 
$\fe_{\fV'} \sqcup \fd_{\fV'} \sqcup \fl_{\fV'}$, 
 then, on the chart $\fV$, its proper transform is no longer labelled by that element.
 This verifies the cases of (3)-(7) {\it whenever the statement therein involves
  the newly created exceptional divisor 
 $E_{({\wp}_{(k\tau)}\fr_\mu\fs_h)}$.}
  
For any of the remaining $\vp$-, $\vr$-, 
and exceptional divisors on $\tsR_{({\wp}_{(k\tau)}\fr_\mu\fs_h)}$,
it is the proper transform of a unique corresponding
 $\vp$-, $\vr$-, and exceptional divisor on $\tsR_{({\wp}_{(k\tau)}\fr_\mu\fs_{h-1})}$. 
 Hence,  by applying   the inductive assumption on $\fV'$ accordingly, 
 we conclude that
every of  (1)-(8) of the proposition hold on $\fV$.
 
 (9)

 Assume ${\wp}_{(k\tau)}\fr_\mu\fs_h=\ell_k$.

Fix and consider any $j \in [k]$.

By the inductive assumption, the statement holds for  $L_{F_j}$
when ${\wp}_{(k\tau)}\fr_\mu\fs_h=\ell_j$.
Now by the construction of $\wp$ or $\ell$-center (which is defined
by the vanishing of a variable, and another variable or 
the proper transform of $L_F$ over any chart
that meets the center),  we conclude that $L_{F_j}$
remains  as stated in the proposition over the admissible chart $\fV$.

We now consider $L_{F_k}$ and prove the statement.

Assume $\fV'$ lies over $(x_{(\um,\uu_{F_k})}\equiv 1)$, then 
$Z_{\chi_k} \cap \fV = \emptyset$ because $E_{\wp_k, \vt_k} \cap \fV = \emptyset$, 
where $Z_{\chi_k}$ is  the $  \ell_k$-center. Hence, the statement holds.

Assume $\fV'$ lies over the $\vr$-standard chart of $\tsR_{\vt_{[k]}}$ (hence,
so does $\fV$). Notice that the variable $\de_{\fV', (\um, \uu_F)}$ never appears
in any relation of the blocks $\fG_{\fV', F_j}$ for all $j \le k$, except $L_{\fV', F_k}$.
Hence, by  Proposition \ref{eq-for-sV-vtk} 
  and the inductive assumption, $L_{\fV', F_k}$ takes the following form
\begin{equation}\label{LFwith-de} 
L_{\fV', F_k}=\sgn(s_{F_k})\de_{\fV', (\um, \uu_F)} 
+ \sum_{s \in  S_{F_k} \- s_{F_k}} \sgn(s) 
\pi_{\fV', \fV_{[0]}}^* x_{\fV', (\uu_s, \uv_s)} .
\end{equation}



Then, one sees from the definition that the $\ell_k$-blowup ideal on the chart $\fV'$ is
$$\langle L_{\fV', F_k}^\star, \; \de_{\fV', (\um, \uu_F)}  \rangle $$
where $L_{\fV', F_k}^\star=\sum_{s \in  S_{F_k} \- s_{F_k}} \sgn(s) 
\pi_{\fV', \fV_{[0]}}^* x_{\fV', (\uu_s, \uv_s)}$.
We let $\PP_{[\xi_0, \xi_1]}$ be the factor projective space of the $ \ell_k$-blowup
such that $(\xi_0, \xi_1)$ corresponds to $(L_{\fV', F_k}^\star, \; \de_{\fV', (\um, \uu_F)})$.

 First, we consider the chart $(\xi_ 0\equiv 1)$.

Then,   $\fV=(\fV' \times (\xi_0 \equiv 1)) \cap \tsR_{\ell_k}$, 
as a closed subscheme of  
$\fV' \times (\xi_0 \equiv 1),$ is defined by
\begin{equation}\label{proof:x0-chart-ell}
\de_{\fV', (\um, \uu_{F_k})} = L_{\fV', F_k}^\star \cdot \xi_1.
\end{equation}
Notice that on this chart, we have
$$E_{\ell_k} \cap \fV =  (L_{\fV', F_k}^\star  =0)$$
where $E_{\ell_k}$ is the exceptional divisor created by the blowup $\tsR_{\ell_k} \to \tsR_{\wp_k}$.
Upon substituting \eqref{proof:x0-chart-ell}, we obtain
$$L_{\fV', {F_k}}=L_{\fV', F_k}^\star (1 +  \sgn(s_{F_k}) \xi_1),$$ 
and 
\begin{equation}\label{for-isom}
L_{\fV, {F_k}}= 1 + \sgn(s_{F_k}) \xi_1 .
\end{equation}
This implies that $\xi_1$ is invertible along $\tsV_{\ell_k}$.
Thus, if necessary, we can shrink the chart $\fV$ such that it still contains
 $\tsV_{\ell_k}$, and assume that $\xi_1$ is invertible on $\fV$.
Now, as $\xi_1$ is the proper transform 
 $y_{\fV, (\um, \uu_{F_k})}$  of $\de_{\fV', (\um, \uu_{F_k})}$,
we can also write
$$L_{\fV, {F_k}}= 1 + \sgn(s_{F_k}) y_{\fV, (\um, \uu_{F_k})} $$ 
with  $y_{\fV, (\um, \uu_{F_k})}$ being invertible on the chart (shrinking
the chart if necessary).
Thus, we obtain the desired form of $L_{\fV, F_k}$ as stated in (9).


 \medskip
Next, we  consider 
 the chart $(\xi_1 \equiv 1)$. 
 
 Then, the chart  $\fV=(\fV' \times (\xi_1 \equiv 1)) \cap \tsR_{\ell_k}$ 
of the scheme $\tsR_{\ell_k}$, as a closed subscheme of  
$\fV' \times (\xi_1\equiv 1),$ is defined by
\begin{equation} \label{proof:wp/ell-t0}
L_{\fV', F_k}^\star = \de_{\fV', (\um, \uu_{F_k})} \xi_0.
\end{equation}
By substitution, we obtain 
$$L_{\fV', F_k}=\de_{\fV', (\um, \uu_{F_k})} (\xi_0 +  \sgn(s_{F_k}) )$$
and
\begin{equation} 
\nonumber
L_{\fV, F_k}=\xi_0 +   \sgn(s_{F_k}).
\end{equation}
Then, this implies that $\xi_0$ is invertible along $\fV \cap \tsV_{\ell_k}$.
Hence, again, if necessary, by shrinking the chart $\fV$ 
such that it still contains $\fV \cap \tsV_{\ell_k}$, 
we can assume that $\xi_0$ is invertible on the chart.
Therefore, we can discard the current chart and switch back to the chart lying
over $(\xi_0 \equiv 1)$.  This sends us back to the previous case
where the statement is proved.


By Corollary \ref{no-(um,uu)} and the above discussion,
 the charts chosen in (9) together cover  the scheme $\tsV_{\ell_k}$.

 This completes the proof.
\end{proof}

\begin{defn}\label{preferred-chart-ell}
An admissible chart of $\tsR_{({\wp}_{(k\tau)}\fr_\mu\fs_h)}$ 
as characterized by
Proposition \ref{meaning-of-var-wp/ell}  (9) 
will be  called a preferred admissible chart. 
\end{defn}
By Proposition \ref{meaning-of-var-wp/ell}  (9),
the set of preferred admissible charts cover 
$\tsV_{({\wp}_{(k\tau)}\fr_\mu\fs_h)}$.


\begin{cor}\label{ell-isom}
We let $$\rho_{\ell_k, \wp_k}: \tsV_{\ell_{k}} \lra \tsV_{\wp_{k}}$$
be the morphism
 induced from the blowup morphism $\rho_{\ell_k, \wp_k}: \tsR_{\ell_{k}} \lra \tsR_{\wp_{k}}$.
 Then, $\rho_{\ell_k, \wp_k}$ is an isomorphism.
\end{cor}
\begin{proof} 
We continue  to use the notation of  the proof of Proposition \ref{meaning-of-var-wp/ell}  (9).
We can cover  $\tsV_{\ell_{k}}$ by preferred admissible affine charts.
Let $\fV$ be any such a chart such that $Z_{\chi_k} \cap \fV' \ne \emptyset$.
Then, by \eqref{for-isom}, we have
$$L_{\fV, {F_k}}= 1 + \sgn(s_{F_k}) \xi_1 .$$
Then, $\xi_1=-\sgn(s_{F_k})$. This implies that 
$$\rho_{\ell_k, \wp_k}^{-1}(\tsV_{\wp_{k}}) =  \tsV_{\wp_{k}} \times [1, -\sgn(s_{F_k})].$$
Hence, the morphisms 
$\rho_{\ell_k, \wp_k}$ is an isomorphism.
\end{proof}

\begin{rem} 
The induced blowup
 $\tsV_{\ell_k} \lra \tsV_{\wp_k}$ is an isomorphism as it is a blowup along
a Cartier divisor, which is also proved in the above
Corollary \ref{ell-isom}. 
However, as  mentioned in \S \ref{tour},
had this blowup  not performed here, then 
the  $``$zero factor$"$ $\de_{\fV, (\um, \uu_{F_k})}$ 
would remain as (the proper transform of)
an exceptional variable directly transformed
from the $\vr$-variable $x_{\fV, (\um, \uu_{F_k})}$
 but would not carry the information of (linearized) $\pl$ relation.  
In addition,  for $F_k$ of rank zero, $L_{F_k}$ could have taken the form
$$L_{\fV, F_k}=\de_{\fV, (\um, \uu_{F_k})} + \hbox{other terms}$$
and during some subsequent $\wp$-blowups (assuming no $\ell$-blowups), 
the term $\de_{\fV, (\um, \uu_{F_k})}$ might acquire some extra exceptional
variables, and 
this would bring the author to an unknown territory.
\end{rem}

Not particularly used, using similar arguments to
some of the above proof, we can show

\begin{prop} Let the notation be as in Proposition \ref{meaning-of-var-wp/ell}. Then, the proper transform of
the divisor $D_{L_F}=(L_F=0)$ in $\tsR_{({\wp}_{(k\tau)}\fr_\mu\fs_h)}$
is smooth for all $\bF \in \sF$.
\end{prop}
\begin{proof}
Fix any $\bF_j \in \sF$ with $j \in [\up]$.
When $j \in [k]$, by Proposition \ref{meaning-of-var-wp/ell} (9),
for any preferred admissible chart $\fV$ of $\tsR_{({\wp}_{(k\tau)}\fr_\mu\fs_h)}$,
we have 
$$L_{\fV, {F_j}}= 1 + \sgn(s_{F_j}) y_{\fV, (\um, \uu_{F_j})} .$$
This clearly implies that the proper transform of
the divisor $D_{L_F}$ in $\tsR_{({\wp}_{(k\tau)}\fr_\mu\fs_h)}$
is smooth.

Now consider the case when $F=F_j$ with $j >k$.
Like above, we consider any admissible chart $\fV$.
Take any point $\bz \in D_{\fV, L_F} \subset \tsR_{({\wp}_{(k\tau)}\fr_\mu\fs_h)}$
lying over $\bz_0 \in \fV_{[0]} \subset \tsR$. Consider $L_{\fV_{[0]}, F}$.
By the homogeneity of the coordinates of $\PP_F$,
there must be two $\vr$-variables\footnote{
This means, a priori, we may make a choice such that $\fV$ never
lies over $(x_{(\um, \uu_F)}=1)$. This might save a case to consider,
but we do not choose to do so.}
 of $\PP_F$ such that they do not
vanish at the point $\bz_0$. Then, 
possibly after shrinking the charts if necessary,
 one of these two $\vr$-variables,
say, $x_{\fV_{[0]},(\uu_s,\uv_s)}$ for some $s \in \S_F$, will remain
to be a variable  $x_{\fV, (\uu_s,\uv_s)}$ in $\var_\fV$, and moreover,
$L_{\fV, F}$ can be expressed as
$$L_{\fV, F}=x_{\fV, (\uu_s,\uv_s)} + \hbox{other terms}$$
such that $x_{\fV, (\uu_s,\uv_s)}$ does not divide any of the other terms.
This proves that $D_{\fV, L_F}$ is smooth at $\bz$, hence is smooth.
\end{proof}

\begin{prop}\footnote{This proposition is not
used anywhere in the current article, but will be applied in the forthcoming
part II of the series.}\label{Taction-wp/ell}
We continue to follow the notation of Proposition \ref{meaning-of-var-wp/ell}.
 We apply induction on 
 $$(k\tau)\mu h \in \Om$$ and assume all the statements in this proposition
hold for $(k\tau)\mu (h-1) \in \Om$.
Consider any admissible smooth affine chart $\fV$ lying over $\fV'$.
Then, every variable of $\var_\fV^+=\var_\fV \sqcup \var_{\fl_\fV}$
comes equipped with a unique $\TT$-weight (a $\TT$-character)
such that with the induced $\TT$-action on $\fV$, 
the morphism $\fV \lra \fV'$ is $\TT$-equivariant
and $\fV \cap \tsV_{({\wp}_{(k\tau)}\fr_\mu\fs_h)}$ is $\TT$-invariant. 
Moreover, 
these $\TT$-actions on various admissible charts are compatible, making
the morphsim 
$\tsV_{({\wp}_{(k\tau)}\fr_\mu\fs_h)} \lra \tsV_{({\wp}_{(k\tau)}\fr_\mu\fs_{h-1})}$ $\TT$-equivariant.
\end{prop}
\begin{proof}
We prove by induction on $(k\tau)\mu h \in \Om$ 
where $(11)10$ is the smallest element.

For the initial case when $(k\tau)\mu h =(11)10$,
the scheme is $\tsR_{({\wp}_{(11)}\fr_1\fs_0)}=\tsR_{\vt}$.
In this case,  this proposition is just Proposition \ref{Taction-vt} with $k=\up$:
 one checks readily that the statement of this proposition holds.

For induction, we suppose that all the statements hold over 
$\tsR_{({\wp}_{(k\tau)}\fr_\mu\fs_{h-1})}$ along 
$\tsV_{({\wp}_{(k\tau)}\fr_\mu\fs_{h-1})}$, that is, over a
set of $\TT$-invariant admissible affine charts of 
$\tsR_{({\wp}_{(k\tau)}\fr_\mu\fs_{h-1})}$ that cover 
$\tsV_{({\wp}_{(k\tau)}\fr_\mu\fs_{h-1})}$,
for some (suitable) $(k\tau)\mu (h-1) \in \Om$.

We now consider any fixed admissible chart
$\fV$ of $\tsR_{({\wp}_{(k\tau)}\fr_\mu\fs_h)}$
lying over a $\TT$-invariant admissible affine chart
$\fV'$ of  $\tsR_{({\wp}_{(k\tau)}\fr_\mu\fs_{h-1})}$.
The blowup morphism
$$\pi: \tsR_{({\wp}_{(k\tau)}\fr_\mu\fs_h)} \lra 
\tsR_{({\wp}_{(k\tau)}\fr_\mu\fs_{h-1})}$$
induces 
$$\fV \lra \fV'.$$

First, we consider the exceptional variables in
\eqref{proof:new-ex-fv-vskmu}.
We then let every variable there to inherit the $\TT$-weight
from that of its corresponding variable over $\fV'$, except the last one,
that is,  when $(y_i'=0)$ defines  $D_{\wp_k, L_{F_k}}\cap \fV'$,
in which case, we assign weight zero to the variable $\de_{\fV, (\um, \uu_{F_k})}$. 

Next, we consider the variables in \eqref{proof:var-xi-fv-vskmu}.
Again, we let every variable there to inherit the $\TT$-weight
from that of its corresponding variable over $\fV'$.
In addition, we assign weight zero to the variable  $y_{\fV, (\um, \uu_{F_k})}$.

Finally, for the variables in
\eqref{proof:var-fv-vskmu}, we simply let every of them to 
inherit the $\TT$-weight
from that of its corresponding variable over $\fV'$.

Then, one sees immediately that the above defines a $\TT$-action on $\fV$,
 makes the morphism $\fV \lra \fV'$ $\TT$-equivariant, and
furthermore, $\fV \cap \tsV_{({\wp}_{(k\tau)}\fr_\mu\fs_h)}$
is $\TT$-invariant.
Moreover, one checks readily that these $\TT$-actions on 
different admissible charts $\fV$ are compatible, 
therefore, we obtain an induced $\TT$-equivariant morphism
$$
\tsV_{({\wp}_{(k\tau)}\fr_\mu\fs_h)}\lra 
 \tsV_{({\wp}_{(k\tau)}\fr_\mu\fs_{h-1})} .
$$ 
\end{proof}

\subsection{Proper transforms of 
defining relations in $(\wp_{(k\tau)}\fr_\mu\fs_h)$ and in $(\ell_k)$}   $\ $

 Consider any fixed $B \in \cB^\gov  \cup \cB^\frb$ and $\bF \in \sF$.
Suppose $B_{\fV'}$ and $L_{\fV', F}$ have been constructed over $\fV'$.
Applying Definition \ref{general-proper-transforms}, we obtain the proper transforms on the chart $\fV$
$$B_{\fV}, \;
\forall \; B \in \cB^\gov  \cup \cB^\frb; \;\; L_{\fV, F}, \; 
\forall \; i\bF \in \sF.$$

\begin{defn}\label{termninatingB-wp} 
{\rm  (cf. Definition \ref{general-termination})} 
Consider any governing binomial relation $B \in \cB^\gov$. 
Let $\fV$ be an admissible affine chart of $\tsR_{({\wp}_{(k\tau)}\fr_\mu\fs_h)}$
 (including $\tsR_{\ell_k}$)
and $\bz \in  \tsV_{({\wp}_{(k\tau)}\fr_\mu\fs_h)} \cap \fV$ be a closed point.
We say that $B$ terminates at $\bz$ if (at least one, hence both of) its two terms 
of $B_{\fV}$ does not vanish at $\bz$.
 We say $B$ terminates on the chart $\fV$ 
 if it terminates at all closed points of $ \tsV_{({\wp}_{(k\tau)}\fr_\mu\fs_h)} \cap \fV$.
 We say $B$ terminates on  $\tsR_{({\wp}_{(k\tau)}\fr_\mu\fs_h)}$ if it terminates on
 all  admissible affine charts $\fV$ of  $\tsR_{({\wp}_{(k\tau)}\fr_\mu\fs_h)}$.
 \end{defn}

In the sequel, for any $B =T^+_B - T^-_B \in \cB^\gov$, we express
$B_\fV= T^+_{\fV, B} - T^-_{\fV, B}$. 
If $B=B_{(k\tau)}$ for some $k \in [\up]$ and $\tau \in [\ft_{F_k}]$, we also write 
$$B_\fV= T^+_{\fV, (k\tau)} - T^-_{\fV, (k\tau)}.$$

Below, we follow the notations of Proposition \ref{meaning-of-var-wp/ell} as well as those in its proof. 

In particular, we have that $\fV$ is a admissible affine chart of $\tsR_{({\wp}_{(k\tau)}\fr_\mu\fs_h)}$
(including $\tsR_{\ell_k}$),
 lying over a admissible affine chart $\fV'$ of $\tsR_{({\wp}_{(k\tau)}\fr_\mu\fs_{h-1})}$. We have that
$\phi'_{(k\tau)\mu h}=\{Y'_0, \;Y'_1\}$  is the proper transforms of $\phi_{(k\tau)\mu h}=\{Y^+, Y^-\}$
 in $\tsR_{({\wp}_{(k\tau)}\fr_\mu\fs_{h-1})}$ with $Y^\pm$ being associated with $T^\pm_{(k\tau)}$
 or $\phi_{(k\tau)\mu h}=\chi_k$ in which case $Y_0'$ is the $\fL$-divisor $D_{\wp, L_{F_k}}$.
 Likewise, $Z'_{\phi_{(k\tau)\mu h}}$ is the proper transforms of the ${\wp}$-center
 $Z_{\phi_{(k\tau)\mu h}}$ or is the $\ell_k$-center $ Z_{\chi_k}$
   in $\tsR_{({\wp}_{(k\tau)}\fr_\mu\fs_{h-1})}$.
 Also, assuming that $Z'_{\phi_{(k\tau)\mu h}} \cap \fV'  \ne \emptyset$,
then, as in \eqref{YY01-wp}, we have
 $$Y'_0 \cap \fV' =(y'_0 =0), \; Y'_1 \cap \fV' =(y'_1 =0), 
 \;\; \hbox{with $y'_0, y'_1 \in \var_{\fV'}^+$}, $$
or in the case of $\ell_k$-blowup, 
$(y'_0 =0)$ defines the (proper transform of)
 the $\fL$-divisor $ D_{\wp, L_{F_k}} \cap \fV'$ and
$y'_1 = \de_{\fV', (\um, \uu_{F_k})}$.
Further, we have $\PP_{\phi_{(k\tau)\mu h}}=\PP_{[\xi_0,\xi_1]}$ with 
 the homogeneous coordinates  $[\xi_0,\xi_1]$ corresponding to $(y'_0,y'_1)$.

\begin{prop}\label{equas-wp/ell-kmuh}
Let the notation be as in Proposition \ref{meaning-of-var-wp/ell} and be as  above.

Let $\fV$ be any preferred admissible chart of 
$\tsR_{(\wp_{(k\tau)}\fr_\mu\fs_h)}$
(cf. Definition \ref{preferred-chart-ell}). Then, the scheme 
$\tsV_{({\wp}_{(k\tau)}\fr_\mu\fs_h)}\cap \fV$, 
as a closed subscheme of the chart $\fV$ 
 is defined by $$\cB_\fV^\gov, \; \cB^\frb_\fV, \; L_{\sF, \fV}.$$

Assume $Z'_{\phi_{(k\tau)\mu h}} \cap \fV'  \ne \emptyset$.
We let $\zeta=\zeta_{\fV, (k\tau)\mu h}$ be the exceptional parameter in $\var_\fV^+$ such that
$$E_{({\wp}_{(k\tau)}\fr_\mu\fs_h)} \cap \fV = (\zeta=0).$$


Then,  the following hold.
\begin{enumerate}
\item Suppose $\fV \subset (\fV' \times (\xi_0 \equiv 1)) \cap \tsR_{(\wp_{(k\tau)}\fr_\mu\fs_h)}$.
We let $y_1 \in \var_\fV^+$  be the proper transform of $y_1'$. 
Then, we have
\begin{itemize}
\item[(1a)]   $T^+_{\fV, (k\tau)}$
 is square-free, $y_1 \nmid T^+_{\fV, (k\tau)}$, 
 and $\deg (T^+_{\fV, (k\tau)}) =\deg (T^+_{\fV', (k\tau)})-1$.
Suppose $\deg_{y_1'} T^-_{\fV', (k\tau)}=b$ for some 
integer $b$, positive by definition, then
we have $\deg_{\zeta} T^-_{\fV, (k\tau)}=b-1$
and $\deg_{y_1} T^-_{\fV, (k\tau)}=b$.
Consequently,
either $T^-_{\fV, (k\tau)}$ is linear in $y_1$ or else $\zeta \mid T^-_{\fV,  (k\tau)}$.
\item[(1b)]  Let $B \in   \cB^\gov$ with $B > B_{(k\tau)}$.   Then,
$T^+_{\fV, B}$ is square-free and $y_1 \nmid T^+_{\fV, B}$.
Suppose $B \in  \cB^\gov_{F_k}$ and  $y_1 \mid T^-_{\fV, B}$,
 then either $T^-_{\fV, B}$ is linear in $y_1$ or $\zeta \mid T^-_{\fV, B}$.
Suppose $B \notin  \cB^\gov_{F_k}$ and  $y_1 \mid T^-_{\fV, B}$, then $\zeta \mid T^-_{\fV, B}$.
\end{itemize} 
\item Suppose $\fV \subset (\fV' \times (\xi_1 \equiv 1)) \cap \tsR_{(\wp_{(k\tau)}\fr_\mu\fs_h)}$.
We let $y_0  \in \var_\fV^+$  be the proper transform of $y_0'$. 
 Then, we have
\begin{itemize}
\item[(2a)]  $T^+_{\fV, (k\tau)}$ is square-free,
 $y_0 \nmid T^-_{\fV, (k\tau)}$, and
 $ \deg (T^-_{\fV, (k\tau)}) =\deg (T^-_{\fV', (k\tau)})-1$.
\item[(2b)] Let $B \in   \cB^\gov$ with $B > B_{(k\tau)}$.  
Then, $T^+_{\fV, B}$ is square-free.
Suppose  $B \in   \cB^\gov_{F_k}$, then  $y_0 \nmid T^-_{\fV, B}$.
Suppose $B \notin  \cB^\gov_{F_k}$ and $y_0 \mid T^-_{\fV, B}$, then $\zeta \mid T^-_{\fV, B}$.
\end{itemize}
\item  $\rho_{(k\tau)}< \infty$. 
Moreover, for every $B \in \cB^\gov$ 
with $B \le B_{(k\tau)}$, we have that $B$ 
terminates on 
$\tsR_{(\wp_{(k\tau)}\fr_{\rho_{(k\tau)}})}=
\tsR_{(\wp_{(k\tau)}\fr_{\rho_{(k\tau)}}\fs_{\si_{(k\tau)\rho_{(k\tau)}}})}$.
In particular, for all $B \in \cB^\gov$ 
with $B \le B_{(k\ft_{F_k})}$, $B$ terminates on 
$\tsR_{\wp_k}$. 
 Consequently, $T^+_{\fV,B}$ is square-free for  all $B \in \cB^\gov$ and for any
admissible chart $\fV$ of $\tsR_{\wp_k}$.

\item 
 Consider any fixed term $T_B$ of any given $B \in \cB^\frb$. 
 We can assume  $y_i'$ turns into $\zeta$ for some $i \in \{0, 1\}$ and
$y_j$ is the proper transform of  $y_j'$ with $j=\{0,1\}\-\{i\}$ 
 Suppose $y_j \mid T_{\fV, B}$, then either $T_{\fV, B}$ is linear in $y_j$ or 
$\zeta \mid T_{\fV, B}$.
\end{enumerate}      
\end{prop}
\begin{proof} 
We continue to follow  the notation in the proof of Proposition \ref{meaning-of-var-wp/ell}.

We prove the  proposition by applying induction on
$(k\tau) \mu h  \in \Om$.


The initial case is 
$(k\tau) \mu h=(11)1 0$ with $\tsR_{({\wp}_{(11)}\fr_1 \fs_0)}=\tsR_\vt$. 
In this case, 
the  statement about defining equations
 of  $\tsR_{(\wp_{(11)}\fr_1 \fs_0)} \cap \fV$ follows from
Proposition \ref{eq-for-sV-vtk} with $k=\up$;
 the remainder statements (1) - (4) are void.

Assume that the proposition holds  for 
$({\wp}_{(k\tau)}\fr_\mu \fs_{h-1})$
with $(k\tau)\mu (h-1)  \in \Om$.

Consider $({\wp}_{(k\tau)}\fr_\mu \fs_h)$. 

Consider any preferred admissible 
 chart $\fV$ of $\tsR_{({\wp}_{(k\tau)}\fr_\mu \fs_{h})}$,
lying over a  preferred admissible chart of $\fV'$ of  $\tsR_{{\wp}_{(k\tau)}\fr_\mu \fs_{h-1})}$.
By assumption, all the desired statements of the proposition hold over the chart $\fV'$.



The  statement of the proposition on the defining equations
 of $\tsV_{(\wp_{(k\tau)}\fr_\mu \fs_{h})} \cap \fV$
follows straightforwardly from the inductive assumption.

For the statements of (1)-(4), we structure our proofs as follows.
Because the $\ell_k$-blowup occurs after all $\wp$-blowups 
in the block $(\fG_k)$ are performed,
we will prove (1), (2), and (3) for $\wp$-blowups in  $(\fG_k)$ first, 
 then we will return to
prove (1) and (2) for the $\ell_k$-blowup.
 (3) is unrelated to $\ell_k$-blowup. 
We prove (4) at the end.

(1)

($1_\wp$) We first consider the case when the blowup is 
a $\wp_k$ blowup 
(not the $\ell_k$-blowup).

 It helps to recall that we have the following blocks of
binomial relations in $(\fG_k)$.
\begin{eqnarray}\label{all gov bi 2}
\hbox{all the governing binomials:} \;\;\;\;\;\;\;\;\;\;\; \;\;\;\;\;\;\;\;\;\;\; \;\;\;\;\;\;\;\; \\
x_{1uv}x_{(12u,13v)} - x_{12u}x_{13v} x_{(123,1uv)}, \; x_{1uv}x_{(13u,12v)}- x_{13u}x_{12v}x_{(123,1uv)}; \nonumber \\
x_{2uv}x_{(12u,23v)} -x_{12u}x_{23v} x_{(123,2uv)}, \;  x_{2uv}x_{(23u,12v)}-x_{23u}x_{12v} x_{(123,2uv)}; \nonumber \\
x_{3uv}x_{(13u,23v)} -x_{13u}x_{23v}x_{(123,3uv)}, \;  x_{3uv}x_{(23u,13v)} -x_{23u}x_{12v}x_{(123,3uv)}; \nonumber\\
 x_{abc}x_{(12a,3bc)}-x_{12a}x_{3bc} x_{(123,abc)},\;
 x_{abc}x_{(13a,2bc)}-x_{13a}x_{2bc} x_{(123,abc)},\nonumber \\
x_{abc}x_{(23a,1bc)} -x_{23a}x_{1bc}x_{(123,abc)} \nonumber
\end{eqnarray}
for all $u<v \in [n]\-[3]$ and $a<b<c \in [n]\- [3]$.
In the first three cases, we have only two binomial relations in the block $(\fG_k)$;
in the last, we have three.

 We may express 
\begin{equation}\label{Bktau here}
B_{(k\tau)}= x_{(\uu_{s_\tau},\uv_{s_\tau})} x_{\uu_k}
 -x_{\uu_{s_\tau}} x_{\uv_{s_\tau}} x_{(\um,\uu_k)}
\end{equation}
where  $x_{\uu_k}$ is the leading variable of $\bF_k$, and $s_\tau \in S_{F_k} \- s_{F_k}$ corresponds
to $\tau \in [\ft_{F_k}]$. 

(1a)

First, observe  that the variables $x_{(\uu_{s_\tau},\uv_{s_\tau})}$ 
and $x_{\uu_k}$ do not appear in any relation of 
the block $ (\fG_{F_j})$
with $j <k$. 
Thus, when $\tau=1$, both of the following two
must hold: (1). $T^+_{\fV, (k1)}$ admits at most two 
factors, one is the proper transform 
$x_{\fV, (\uu_{s_\tau},\uv_{s_\tau})}$ of  
$x_{(\uu_{s_\tau},\uv_{s_\tau})}$
and the other is the proper transform 
$x_{\fV, \uu_k}$ of  $x_{\uu_k}$; 
(2). If $(\uu_{s_\tau},\uv_{s_\tau}) \in \fd_\fV$
or $\uu_k \in \fe_\fV$ labels exceptional divisor, then
it does not appear in $T^+_{\fV, (k1)}$. 
Because the chart $\fV$ is over $(\xi_0 \equiv 1)$, 
the condition of case (2) is met,  hence the case (2) must occur.
This implies that $T^+_{\fV, (k1)}$ is square-free ,
$y_1 \nmid T^+_{\fV, (k1)}$,
as well as $\deg T^+_{\fV, (k1)}=\deg T^+_{\fV', (k1)}-1$.
Now, suppose $\deg_{y_1'} T^-_{\fV', (k\tau)}=b$ for some 
integer $b$,  then we can express
\begin{equation}\label{use here}
B_{\fV', (k1)} = m_0 y_0' - m_1 (y_1')^b
\end{equation}
for some monomials $m_0$ and $m_1$ such that $y_0' \nmid m_0$.
Then, after plugging 
$$y_1'= y_0' \xi_1= \zeta y_1'$$
 into $B_{\fV, (k1)}$ of \eqref{use here},
we obtain
$$B_{\fV, (k1)} = m_0  - m_1 \zeta^{b-1} (y_1)^b.$$
Hence, we have $\deg_{\zeta} T^-_{\fV, (k\tau)}=b-1$
and $\deg_{y_1} T^-_{\fV, (k\tau)}=b$.
Consequently,
either $T^-_{\fV, (k\tau)}$ is either
 linear in $y_1$ or else $\zeta \mid T^-_{\fV,  (k\tau)}$.

For a general $\tau \in [\ft_{F_k}]$, 
From \eqref{Bktau here}, 
$T^+_{(k\tau)}= x_{(\uu_{s_\tau},\uv_{s_\tau})} x_{\uu_k}$,
we see that in $T^+_{\fV',(k\tau)}$, 
only $x_{\uu_k}$ can possibly 
 acquire some exceptional variables
from the previous blowups in the same block $(\fG_k)$ with respect to
$B_{(k\tau')}$ for some $\tau' < \tau$.

 Hence, prior to  performing
$\wp$-blowups with respect to the block $\cB^\gov_{F_k}$,  we have
\begin{equation}\label{B-fV''}
B_{\fV'', (k\tau)}: x_{\fV'', (\uu_{s_\tau},\uu_{s_\tau})} x_{\fV'', \uu_F}
 -a_{\fV'',\tau}, \; \tau \in [\ft_{F_k}]
\end{equation}
 where $a_{\fV'',\tau}$ 
are some monomials and $\fV''$ is an admissible chart.
Then, after we perform some
$\wp$-blowups with respect to $B_{(k\tau')}$ with $\tau' < \tau$,
 we have
 \begin{equation}\label{B-fV'}
B_{\fV', (k\tau)}: x_{\fV', (\uu_{s_\tau},\uu_{s_\tau})} (\prod \ve ) x_{\fV', \uu_F} -a_{\fV',\tau}, \; \tau \in [\ft_{F_k}]
\end{equation}
 where $(\prod \ve )$, which can be equal to 1, or otherwise,
is a square-free product of exceptional variables acquired by  $x_{\fV', \uu_F}$, and
$\fV'$ is an admissible chart.
Recall the convention that 
 we set $ x_{\fV', \uu_F} \equiv 1$ if $\uu_F$ labels 
the exceptional variable $\ve_{\fV', \uu_F}$ on the chart.

By the inductive assumption, $T^+_{\fV', (k\tau)}$ is square-free.
Then, as it is impossible to have $y_1' \mid T^+_{\fV', (k\tau)}$,
we obtain that $T^+_{\fV, (k\tau)}$ is square-free,
$y_1 \nmid T^+_{\fV, (k1)}$,
as well as $\deg T^+_{\fV, (k1)}=\deg T^+_{\fV', (k1)}-1$.
The remainder statement can be proved by the parallel calculation
as in the case above when $\tau=1$.

(1b). Let $B > B_{(k\tau)}$. 

First, we suppose $B \in \cB^\gov_{F_k}$.

In this case, we  we can write $B=B_{(k\tau')}$ with $\tau' \in [\ft_{F_k}]$ and $\tau' > \tau$, and then can express 
$$B=B_{(k\tau')}= x_{(\uu_{s_{\tau'}},\uv_{s_{\tau'}})} x_{\uu_{F_k}} -
x_{\uu_{s_{\tau'}}} x_{\uv_{s_{\tau'}}} x_{(\um,\uu_{F_k})}.$$
We have $T^+_B=x_{(\uu_{s_{\tau'}},\uv_{s_{\tau'}})} x_{\uu_{F_k}}$,  and it
{\it retains} this form prior to the $\wp$-blowups with respect to binomials of $\cB_{F_k}$
because $x_{\uu_{F_k}}$ and $x_{(\uu_{s_{\tau'}},\uv_{s_{\tau'}})}$ do not appear in any 
relation $\fG_{F_j}$ with $j <k$. 
(Recall here the convention: $x_{\fV, \uu} = 1$ if  $\uu \in \fe_{\fV}$;
 $x_{\fV, (\uu,\uv)} = 1$ if  $(\uu,\uv) \in \fd_{\fV}$.)

Starting the $\wp$-blowups with respect to the first binomial relation 
$B_{(k1)}$ of $\cB_{F_k}$, 
$T^+_B$ can only acquire exceptional parameters through the 
leading variable $x_{\uu_{F_k}}$.  
 From here, one sees directly  that $T^+_{\fV, (k\tau')}$ is square-free.


Now, suppose $y_1 \mid T^-_{\fV, B}$.
 We can assume $\deg_{y'_1}  (T^-_{\fV', B})=b$ for some 
integer $b>0$. 
Since  $T^+_B$ is square-free, 
using the calculation parallel to the one used in (1a)
around \eqref{use here},
we can obtain the following two possibilities: 

\noindent 
$\bcd$ $\deg_{y_1}  (T^-_{\fV, B})=b$ and $\deg_{\zeta}  (T^-_{\fV, B})=b-1$, if $y_0' \mid  T^+_{\fV, B}$.

\noindent
$\bcd$ $\deg_{y_1}  (T^-_{\fV, B})=b$ and $\deg_{\zeta}  (T^-_{\fV, B})=b$, if $y_0' \nmid  T^+_{\fV, B}$.

Hence,  either $T^-_{\fV, B}$ is linear in $y_1$ when $b=1$ in either case, 
 or else,   $\zeta \mid T^-_{\fV, B}$ when $b>1$ in the first case or in  any situation of  the second case.

Next, we treat the case when $B \notin \cB^\gov_{F_k}$.

We  we can write $B=B_{(k'\tau')}$ with $\tau' \in [\ft_{F_{k'}}]$
 and $k' > k$.
We can express 
$$B=B_{(k'\tau')}= x_{(\uu_{s_{\tau'}},\uv_{s_{\tau'}})} x_{\uu_{F_{k'}}} -
x_{\uu_{s_{\tau'}}} x_{\uv_{s_{\tau'}}} x_{(\um,\uu_{F_{k'}})}.$$
Since $x_{(\uu_{s_{\tau'}},\uv_{s_{\tau'}})}$ and $x_{\uu_{F_{k'}}}$ do not appear
in any relation in $\fG_{F_j}$ with $j <k'$ and $k <k'$, we see that
$T^+_B=x_{(\uu_{s_{\tau'}},\uv_{s_{\tau'}})} x_{\uu_{F_{k'}}}$
{\it retains} this form under the current blowup, in particular, $T^+_{\fV, B}$ is square-free and $y_1 \nmid T^+_{\fV, B}$. Indeed, for the similar reason,
 $y_0' \nmid T^+_{\fV', B}$.  
Furthermore, if $y_1 \mid T^-_{\fV, B}$, 
then again, because $y_0'$ cannot appear in  $T^+_{\fV', B}$, we obtain
$\zeta \mid T^-_{\fV, B}$.

This proves ($1_\wp$).

(2)

($2_\wp$) We continue to consider the case when the blowup is $\wp_k$ blowup (not the $\ell_k$-blowup).  

(2a).  The reason and calculation used in this part are similar to those used in the
proof of (1).

The proof of the fact that the plus-term 
$T^+_{\fV, (k\tau)}$ is square-free, is again based on that
 $x_{(\uu_{s_\tau},\uv_{s_\tau})}$ uniquely appears in $B_{(k\tau)}$
and only the leading term  $x_{\uu_k}$ can possibly acquire exceptional variables
through earlier blowups in the same block $(\fG_k)$, hence,
 is totally analogous to the corresponding part of (1a). 
It is clear that $y_0' \nmid T^-_{\fV', (k\tau)}$, hence 
$y_0 \nmid T^-_{\fV, (k\tau)}$.
The remainder statements also follow from similar
straightforward calculations. We omit the obvious details.

(2b) Let $B > B_{(k\tau)}$.

Suppose $B \in \cB^\gov_{F_k}$.

The fact that $T^+_{\fV, B}$ is square-free, again, follows from the same line of arguments as  in the corresponding part of (1b), 
we avoid repeating same proof.
In addition, one sees that $y_0' \nmid  T^-_{\fV', B}$,
hence $y_0 \nmid  T^-_{\fV, B}$.

Suppose $B \notin \cB^\gov_{F_k}$. 

We  we can write $B=B_{(k'\tau')}$ with $\tau' \in [\ft_{F_{k'}}]$
 and $k' > k$.
We can express 
$$B=B_{(k'\tau')}= x_{(\uu_{s_{\tau'}},\uv_{s_{\tau'}})} x_{\uu_{F_{k'}}} -
x_{\uu_{s_{\tau'}}} x_{\uv_{s_{\tau'}}} x_{(\um,\uu_{F_{k'}})}.$$
Since $x_{(\uu_{s_{\tau'}},\uv_{s_{\tau'}})}$ and $x_{\uu_{F_{k'}}}$ do not appear
in any relation in $\fG_{F_j}$ with $j <k'$ and $k <k'$, we see that
$T^+_B=x_{(\uu_{s_{\tau'}},\uv_{s_{\tau'}})} x_{\uu_{F_{k'}}}$
{\it retains} this form under the current blowup, in particular, $T^+_{\fV, B}$ is square-free and $y_0', y_1' \nmid T^+_{\fV', B}$. 
Therefore, if $y_0 \mid T^-_{\fV, B}$, 
then also $\zeta \mid T^-_{\fV, B}$.

This proves (2b), hence completes the case of ($2_\wp$).

(3) (This statement is exclusively about $\wp$-blowups.)

From $\tsR_{({\wp}_{(k\tau)}\fr_\mu\fs_{h-1})}$ to $\tsR_{({\wp}_{(k\tau)}\fr_\mu\fs_{h})}$,
over any chart $\fV$ of $\tsR_{({\wp}_{(k\tau)}\fr_\mu\fs_{h})}$,
 by (1a) and (2a), we have either
$$\deg (T^+_{\fV, (k\tau)}) =\deg (T^+_{\fV', (k\tau)})-1$$ or
 $$\deg (T^-_{\fV, (k\tau)}) =\deg (T^-_{\fV', (k\tau)})-1.$$
 Hence, after finitely many steps, over any chart $\fV$, either all variables in
 $B_{\fV, (k\tau)}$ are invertible along the proper transform of $\tsV$,
 or else, one of the two terms of $B_{\fV, (k\tau)}$ must  become a constant.


This implies that the process of $\wp$-blowups in ($\wp_{(k\tau)}$)
must terminate after finitely many rounds.
That is, $\rho_{(k\tau)} < \infty$. 

The  statement about termination of 
the governing binomials $B$ with $B \le B_{(k\tau)}$
 follows from 
$\rho_{(k'\tau')} < \infty$ for all $(k'\tau') < (k\tau)$
(by the inductive assumption) and $\rho_{(k\tau)} < \infty$.

Now we prove the last statement about the square-freeness of $T^+_{\fV,B}$ for any $B \in \cB^\gov$ and any
admissible chart $\fV$ of $\tsR_{\wp_k}$.
When $B \in \cB^\gov$ with $B \in \fG_{F_j}$
for some $j <k$, note that $B$ terminates on $\tsR_{\wp_j}$,
hence the square-freeness of $T^+_{\fV,B}$
 follows from the inductive assumption.
When $B \in \fG_{F_j}$ with $j \ge k$, it follows from 
(the proofs of) (1) and (2).

This establishes (3).

\medskip

($1_\ell$) Now we return to (1) to consider the $\ell_k$-blowup.

In this case, we have $y_0'$, defining the $\fL$-divisor $D_{\wp_k, L_{F_k}}$ on the chart
$\fV'$, does not appear in any relation $B$ of $\cB^\frb$.

(1a) By (3) (already proved), all $B_{(k\tau)}$ of $\cB^\gov_{F_k}$ terminate. Thus,
it follows that $T_{\fV, (k\tau)}^+$ is square-free. The remaining statements are void.

(1b) Then, the same line of aruments  applied to $x_{\uu_{F_{k'}}}$ as in (1a) of 
($1_\wp$) can be reused to obtain the desried statement.

($2_\ell$) We now consider (2) for the $\ell_k$-blowup.

Because $\fV$ is a preferred chart (by assumption), the statement is void.

(4) 

 It is worth recalling that we have the following relations in $\cB^\frb$
 that do not admit root-parents in $R_\vr$.
$$ x_{1bc} x_{2b'c'} x_{(13a,2bc)} x_{(23a,1b'c')}- x_{2bc}  x_{1b'c'} x_{(23a,1bc)} x_{(13a,2b'c')}$$
for suitable choices of $a,b,c, b',c'$ which do not belong to $\{1,2,3\}$.

Fix $a,b,c, \bar a, \bar b, \bar c \in [n]$, all being distinct, satisfying $\bar a <a< b<c$ and $\bar a ,b< \bar b<\bar c$:
$$
x_{12b} x_{3ac}x_{(13b,2\bar b \bar c)} x_{(12 \bar a,3 \bar b \bar c)} x_{(13 \bar a,2ac)} 
-x_{13b} x_{2ac} x_{(12b,3\bar b \bar c)}  x_{(13 \bar a,2 \bar b \bar c)}x_{(12 \bar a,3ac)} .  \label{rk1-1}
$$
Here, the only obstruction for the above binomial to admit
 a root-parent in $R_\vr$
 is the inequality $a<b$.

Fix $a,b,c, a', \bar a, \bar b, \bar c \in [n]$, all being distinct, satisfying $\bar a <a < b<c$ and $a', \bar a< \bar b<\bar c$:
$$
 x_{13b}  x_{2ac} x_{(12b, 13a')} x_{(12a', 3\bar b \bar c)} x_{(13 \bar a, 2 \bar b \bar c)} x_{(12 \bar a,3ac)}   
- x_{12b} x_{3ac}  x_{(13b, 12a')} x_{(13a', 2\bar b \bar c)} x_{(12 \bar a, 3 \bar b \bar c)} x_{(13 \bar a,2ac)} .\label{rk0-1}
$$
Here, the only obstruction for the above binomial to admit
 a root-parent in $R_\vr$ is the inequality $a<b$.

The key point that is observed and will be employed below
is the heavy $``$imbalance$"$ of the $\wp$-centers:
$x_{\uu_{F_k}}$ plays too strongly a leading and 
dominant role, let alone any $x_{(\uu_s, \uv_s)}$ that simply uniquely appears in its own governing binomial relation.

($4_\wp$)  We first consider the case when the blowup 
is $\wp_k$ blowup (not the $\ell_k$-blowup).

As in \eqref{B-fV''}, 
prior to  performing
$\wp$-blowups with respect to the block $\cB^\gov_{F_k}$,  we have
 $$B_{\fV'', (k\tau)}: x_{\fV'', (\uu_{s_\tau},\uu_{s_\tau})} x_{\fV'', \uu_F} -a_{\fV'',\tau}, \; \tau \in [\ft_{F_k}]$$
 where $a_{\fV'', \tau}$ are some monomials and $\fV''$ is an admissible chart.
Then, after we perform some
$\wp$-blowups with respect to $B_{(k\tau')}$ with $\tau' < \tau$,
 as in \eqref{B-fV'},
we have
 $$B_{\fV', (k\tau)}: x_{\fV', (\uu_{s_\tau},\uu_{s_\tau})} (\prod \ve ) x_{\fV', \uu_F} -a_{\fV',}\tau, \; \tau \in [\ft_{F_k}]$$
 where $(\prod \ve )$, which can be equal to 1, or otherwise,
is a square-free product of exceptional variables acquired by  $x_{\fV', \uu_F}$, and $\fV'$ is an admissible chart.

Now fix and consider any $B \in\cB^\frb$. 
We can write $B_{\fV'}=T_{\fV',0} -T_{\fV',1}$.
 
 As before, we let $(y_0', y_1')$ be the $\wp$-set on the chart $\fV'$, prior to performing
the blowup $\tsR_{({\wp}_{(k\tau)}\fr_\mu \fs_{h})} \lra 
\tsR_{({\wp}_{(k\tau)}\fr_\mu \fs_{h-1})}$, inducing
 $\fV \lra \fV'$. W.l.o.g., we assume $y_0'$ is associated with 
$T_{\fV', (k\tau)}^+= x_{\fV', (\uu_{s_\tau},\uu_{s_\tau})} (\prod \ve ) x_{\fV', \uu_F}$.

\noindent
$\bcd$ Suppose neither of $y_0'$ and  $y_1'$ appear in $B_{\fV'}=T_{\fV',0} -T_{\fV',1}$.
Then there is nothing to prove.

\noindent
$\bcd$ Suppose both of $y_0'$ and  $y_1'$ appear in only one term of
$B_{\fV'}$, say, $T_{\fV',0}$. Then, we clearly  obtain 
$y_j , \zeta \mid T_{\fV',0}$ as stated in (4).

\noindent
$\bcd$ Suppose $y_0'$ and  $y_1'$ appear in both terms of
$B_{\fV'}$. W.l.o.g., we may assume $y_0' \mid T_{\fV',0}$, 
$y_1' \mid T_{\fV',1}$.
Then under this situation, by
 Lemma \ref{strong sq free} and the comments in the beginning of
($4_\wp$), we conclude that $y_0'$ is linear in $T_{\fV',0}$, 
that is, its exponent/multiplicity in $T_{\fV',0}$ equals 1.
We may assume that $\deg_{y_1'} T_{\fV',1} = b \ge 1$.
Then, over the chart $(\xi_0\equiv 1)$, we have $y_j$ as stated in (4),
equals $y_1$, 
can only divide $T_{\fV,1}$, and by the same calculation used around \eqref{use here},
it is linear in $T_{\fV,1}$ if $b=1$, or otherwise
$\zeta \mid T_{\fV,1}$.
Next, we consider  the chart $(\xi_1\equiv 1)$. Then, again by the same calculation used around \eqref{use here}, we obtain $y_j$ as state  in (4), equals $y_0$, 
can only divide $T_{\fV,0}$ and is linear in $T_{\fV,0}$.

This proves ($4_\wp$).

    \medskip
($4_\ell$)  We now move on to consider the case when the blowup is  $\ell_k$-blowup.
    
    In this case, note that $y_0'$, locally defining $D_{\wp_k, L_{F_k}}$,
    does not appear in $B \in \cB^\frb$. 
    Also, since $\fV$ is a  preferred chart over $(\xi_0 \equiv 1)$, 
$y_j$ as stated in (4), equals $y_{\fV, (\um, \uu_{F_k)}}$
(see Proposition \ref{meaning-of-var-wp/ell} (9) and its proof).
Therefore, $y_j$ cannot divide any term of $T_{\fV, B}$.
(Note that in this case, we have $\zeta=\de_{\fV, (\um, \uu_{F_k)}}$
    with $L_{F_k} \in \fl_\fV$, 
and  $\zeta$ may appear in  $T_{\fV, B}$.)

    This proves (4).

All in all, by induction,    Proposition \ref{equas-wp/ell-kmuh} is proved.  
\end{proof}

The following is  a restatement of some statements of the preceding proposition.
As the idea will be crucially applied in 
constructing various birational transforms of $\Ga$-schemes, 
we isolate and state it as a corollary.

\begin{cor}\label{linear-or-vanish} Let the setup and notation
be as in Proposition \ref{equas-wp/ell-kmuh}. 
As in Proposition \ref{equas-wp/ell-kmuh} (4),
we can assume  $y_i'$ turns into the exceptional variable
$\zeta$ for some $i \in \{0, 1\}$ and
$y_j$ is the proper transform of  $y_j'$ with $j=\{0,1\}\-\{i\}$.
Then for any binomial $B \in \cB^\gov \sqcup  \cB^{\frb}$
and any tern $T_{\fV,B}$ of $B_\fV$,
if $y_j \mid T_{\fV, B}$, then either $T_{\fV, B}$ is linear in $y_j$,
 or else, $\zeta \mid T_{\fV, B}$.
\end{cor}
\begin{proof} If $B \in \cB^\gov$, it
follows immediately from the combinations of 
Proposition \ref{equas-wp/ell-kmuh} (1) and (2). If
$B \in  \cB^{\frb}$, it is 
the same statement of Proposition \ref{equas-wp/ell-kmuh}  (4).
\end{proof}

\section{$\Ga$-schemes and Their  Birational Transforms}\label{Gamma-schemes}

\subsection{$\Ga$-schemes} $\ $

Here, we return to the initial affine chart $\bU \subset \PP(\wedge^3 E)$.

\begin{defn}\label{ZGa} 
Let $\Ga$ be an arbitrary subset of $\var_{\bU}=\{x_\uu \mid \uu \in \II_{3,n} \-\um \}$. 
We let $I_\Ga$ be the ideal of $\kk[x_\uu]_{\uu \in \II_{3,n}\-\um}$ generated by all the elements 
$x_{\uu}$ in $\Ga$, 
 and, 
 $$I_{\wp,\Ga}=\langle x_\uu, \; \bF \mid x_\uu \in \Ga, \; \bF \in \sF \rangle$$
  be the ideal of $\kk[x_\uu]_{\uu \in \II_{3,n}\-\um}$ generated by $I_\Ga$ together with 
all the de-homogenized $\um$-primary $\pl$ relations of $\Gr^{3, E}$. We let
 $Z_\Ga$ $(\subset \Gr^{3, E} \cap \bU)$ be the closed subscheme of 
 the affine space $\bU$ defined by the ideal $I_{\wp,\Ga}$.
The subscheme $Z_\Ga$ is called  the $\Ga$-scheme of $\bU$. 
Note that  $Z_\Ga \ne \emptyset$ since $0 \in Z_\Ga$.
\end{defn}

(Thus, a $\Ga$-scheme is  an intersection of certain Schubert divisors with the chart $\bU$.
 But, in this article, we do not investigate $\Ga$-schemes in any {\it Schubert} way.)
 

Take $\Gamma =\emptyset$. Then, $I_{\wp,\emptyset}$ is the ideal generated by 
all the de-homogenized $\um$-primary $\pl$ relations.
Thus,  $Z_\emptyset =\bU \cap \Gr^{3, E}$.  

Let $\Ga$ be any fixed subset of $\var_{\bU}$.
We let  $\bU_\Ga$ be the coordinate  subspace of $\bU$ defined by $I_\Ga$.
That is,
$$\bU_\Ga=\{(x_{\uu} =0)_{ x_\uu \in \Ga}\} \subset \bU.$$
This is a coordinate subspace of dimension 
${n \choose 3}-1 - |\Ga|$ where $|\Ga|$ is the cardinality of $\Ga$. Then,
$Z_\Ga$ is the scheme-theoretic intersection of $\Gr^{3, E}$ with
the coordinate  subspace $\bU_\Ga$.
For any $\um$-primary $\pl$ equation $\bF \in \sF$, we let $\bF|_\Ga$ be the induced 
polynomial obtained from  the de-homogeneous polynomial $\bF$
by setting $x_{\uu} =0$  for all $x_\uu \in \Gamma$. 
Then, $\bF|_\Ga$ becomes a polynomial on the affine subspace $\bU_\Ga$.
We point out that $\bF|_\Ga$ can be identically zero on $\bU_\Ga$.

\begin{defn}\label{rel-irrel}
 Let $\Ga$ be any fixed subset of $\var_{\bU}$.
 Let ($\bF$) $F$ be any fixed (de-homogenized)  $\um$-primary $\pl$ relation. 
 We say ($\bF$) $F$  is $\Ga$-irrelevant if  every term of 
 $\bF$ belongs to  the ideal $I_\Ga$.   Otherwise, we say ($\bF$) $F$ is $\Ga$-relevant.
 We let $\sFgr$ be the set of all $\Ga$-relevant de-homogenized $\um$-primary $\pl$ relations.
 We let $\sFgir$ be the set of all $\Ga$-irrelevant  de-homogenized $\um$-primary $\pl$ relations.
\end{defn}
As we will take $\um=(123)$, we will often drop the subindex $\um$, for example, we
will write $\sF^{\rm rel}_\Ga$ for $\sFgr$,  $\sF^{\rm irr}_\Ga$ for $\sFgir$, etc.

If  $\bF$ is $\Ga$-irrelevant,  then $\bF|_\Ga$ is identically zero along $\bU_\Ga$. 
Indeed,  $\bF$ is $\Ga$-irrelevant if and only if every term of $\bF$ contains a member of $\Ga$.
 The sufficiency direction is clear. To see the necessary direction,
 we suppose a term $x_\uu x_\uv \in I_\Ga$, then as $I_\Ga$ is prime 
 (the coordinate  subspace $\rU_{\um, \Ga}$ is integral), we have
 $x_\uu$ or  $x_\uv \in \Ga$.

\subsection{$\sF$-transforms of  $\Ga$-schemes in $\sV_{[k]}$}
\label{subsection:wp-transform-sfk}   $\ $





In what follows, we keep  notation of Proposition \ref{equas-fV[k]}. 

Recall that for any $\bF \in \sF$,
$\La_F=\{(\uu_s, \uv_s) \mid s \in S_F\}.$


\begin{lemma}\label{wp-transform-sVk-Ga} 
Fix any  subset $\Ga$ of $\bU$.  Assume that $Z_\Ga$ is integral. 

Fix and consider any $k \in [\up]$. Then, we have the following:
\begin{itemize}
\item there exists a closed subscheme $Z_{\sF_{  [k]},\Ga}$ of $\sV_{[k]}$
with an induced morphism 
 $$Z_{\sF_{  [k]},\Ga}
 \to Z_\Ga;$$
\item $Z_{\sF_{  [k]},\Ga}$ comes equipped with an irreducible component  
$Z^\dagger_{\sF_{  [k]},\Ga}$ with the induced morphism 
$Z^\dagger_{\sF_{  [k]},\Ga} 
 \to Z_\Ga$;
 \item  for any standard chart $\fV$ of $\sR_{[k]}$ such that
$Z_{[k],\Ga} \cap \fV \ne \emptyset$, there  exists a subset, 
possibly empty,
$$ \tGa^\zero_{\fV} \; \subset \;  \var_\fV. $$
\end{itemize}

Further,  consider any given standard chart $\fV$ of $\sR_{[k]}$ with
$Z_{[k],\Ga} \cap \fV \ne \emptyset$. Then,
the following hold.
\begin{enumerate} 
\item 
The scheme $Z_{[k],\Ga} \cap \fV$, as a closed subscheme of the chart $\fV$,
is defined by the following relations
\begin{eqnarray} 
\;\;\;\;\; y , \; \; \; y \in  \tGa^\zero_\fV ,   \label{Ga-rel-sVk=0} \\
\cB^\frb_{\fV, [k]},    \nonumber \\ 
B_{\fV,(s,t)}: \;\;\;  x_{\fV, (\uu_s, \uv_s)}x_{\fV,\uu_t} x_{\fV,\uv_t} - x_{\fV, (\uu_t,\uv_t)}  
 x_{\fV,\uu_s} x_{\fV,\uv_s},  \;\; s, t \in S_{F_i}, 
  \;  i \in [k], \nonumber\\
L_{\fV, F_i}: \;\; \sum_{s \in S_{F_i}} \sgn (s) x_{\fV, (\uu_s,\uv_s)}, \; \;  
 i \in [k], \nonumber \\
\bF_{\fV,j}: \;\; \sum_{s \in S_{F_j}} \sgn (s) x_{\fV, \uu_s}x_{\fV,\uv_s}, \; \; k < j\le \up.  \nonumber
\end{eqnarray} 
 Further, we take $\tGa^\zero_\fV \subset \var_\fV$
 to be the maximal subset (under inclusion)
among all those subsets that satisfy the above.
\item The induced morphism  $Z^\dagger_{\sF_{  [k]},\Ga}  
\to Z_\Ga$ is birational. 
\item For any variable $y=x_{\fV, \uu}$ or $y=x_{\fV, (\uu,\uv)} \in \var_\fV$,
$Z^\dagger_{[k],\Ga} \cap \fV \subset (y=0)$ 
if and only if $Z_{[k],\Ga} \cap \fV \subset (y=0)$.
(We remark here that this property is not used within this lemma, but 
will be used as the initial case of Lemma \ref{vt-transform-k}.)
\end{enumerate}
\end{lemma}
\begin{proof}
We prove the statement by induction on $k$ with $k \in \{0\} \cup [\up]$.

 When $k=0$,  we have $\sR_{\sF_{[0]}}:=\bU$, $\sV_{\sF_{[0]}}:=\bU \cap \Gr^{3, E}$.
 There exists a unique chart $\fV=\bU$. In this case,  we set 
 $$Z_{\sF_{[0]},\Ga}=Z^\dagger_{\sF_{[0]},\Ga}:=Z_\Ga$$
 Further, we let
 $$ \tGa^\zero_\fV=\Ga.$$ 
Then,  the statement holds trivially.

Inductively, we suppose that Lemma \ref{wp-transform-sVk-Ga}
 holds for $\sV_{\sF_{[k-1]}} \subset \sR_{\sF_{[k-1]}} $.

 We now consider $\sV_{[k]} \subset \sR_{[k]}$.
 
Recall that we have the natural birational morphsim
$$\rho_{[k]}: \sV_{[k]} \lra \sV_{[k-1]},$$
induced from the forgetful map  $\sR_{[k]} \lra \sR_{[k-1]}$.

First, we suppose $F_k$ is $\Ga$-relevant.

In this case, we set 
\begin{equation}\label{construction-ZkLa} \La^\zero_{F_k, \Ga} 
:=\{ x_{(\uu,\uv)}\in \La_{F_k} \mid \hbox{$x_\uu$ or $x_\uv \in \Ga$}\}.
\end{equation}
(Here, recall the convention of \eqref{uv=vu}: $x_{(\uu,\uv)}= x_{(\uv,\uu)}$.) 

We then let $\rho_{[k]}^{-1}(Z_{\sF_{ [k-1]},\Ga})$ be the scheme-theoretic pre-image and define
$Z_{\sF_{[ k]},\Ga}$ to be
  the scheme-theoretic intersection
\begin{equation}\label{construction-ZkGa}
Z_{\sF_{[ k]},\Ga}=\rho_{[k]}^{-1}(Z_{\sF_{ [k-1]},\Ga}) \cap (x_{(\uu,\uv)} =0 \mid (\uu,\uv) \in  \La^\zero_{F_k, \Ga}),
\end{equation}

Next, because $F_k$ is $\Ga$-relevant
 and $Z^\dagger_{\sF_{ [k-1]},\Ga}$ is birational to $Z_\Ga$,
one checks that  $Z^\dagger_{\sF_{ [k-1]},\Ga}$ 
is not contained in the exceptional locus of
the birational morphism $\rho_{[k]}$. 
Thus, there exists  a Zariski open subset $Z^{\dagger\circ}_{\sF_{ [k-1]},\Ga}$
 of $Z^\dagger_{\sF_{ [k-1]},\Ga}$ such that 
 $$\rho_{[k]}^{-1}(Z^{\dagger\circ}_{\sF_{ [k-1]},\Ga})  \lra Z^{\dagger\circ}_{\sF_{ [k-1]},\Ga}$$
is an isomorphism.

We claim 
\begin{equation}\label{inclusion-dagger}
 \rho_{[k]}^{-1}(Z^{\dagger\circ}_{\sF_{ [k-1]},\Ga})   \subset Z_{\sF_{[ k]},\Ga}=
\rho_{[k]}^{-1}(Z_{\sF_{ [k-1]},\Ga}) \cap (x_{(\uu,\uv)}=0 \mid (\uu,\uv) \in  \La^\zero_{F_k, \Ga}).
\end{equation}
To see this, note that since $\bF_k$ is $\Ga$-relevant, 
there exists a term $x_{\uu_s}x_{\uv_s}$ of $\bF_k$ 
for some $s \in S_{F_k}$ such that it does not vanish
generically along $Z^{\dagger}_{\sF_{ [k-1]},\Ga}$ (which is birational to $Z_\Ga$).
Then, we consider the binomial relation of $\sV_{[k]}$ in $\sR_{[k]}$
  \begin{equation}\label{Buv-s}
  x_{(\uu, \uv)}x_{\uu_s} x_{\uv_s} - x_{\uu} x_{\uv} x_{(\uu_s,\uv_s)},
  \end{equation}
  for any $(\uu, \uv) \in \La_{F_k}$.
 It follows that $x_{(\uu, \uv)}$  vanishes identically 
along $\rho_{[k]}^{-1}(Z^{\dagger\circ}_{\sF_{ [k-1]},\Ga}) \cong Z^{\dagger\circ}_{\sF_{ [k-1]},\Ga}$ 
 if $x_\uu$ or $x_\uv \in \Ga$.  Hence, \eqref{inclusion-dagger} holds.

 We then let $Z^\dagger_{\sF_{ [k]},\Ga}$ be the closure of 
 $\rho_{[k]}^{-1}(Z^{\dagger\circ}_{\sF_{ [k-1]},\Ga})$ in $Z_{\sF_{ [k]},\Ga}$.
  Since $Z^\dagger_{\sF_{ [k]},\Ga}$  is closed in $Z_{\sF_{ [k]},\Ga}$
 and contains the Zariski  open subset $\rho_{[k]}^{-1}(Z^{\dagger\circ}_{\sF_{ [k-1]},\Ga})$
 of $Z_{[k],\Ga}$, it is an irreducible component of  $Z_{\sF_{ [k]},\Ga}$.

Further, consider any standard chart $\fV$ of $\sR_{[k]}$,
lying over a unique standard chart $\fV'$ of $\sR_{\sF_{[k-1]}}$,
such that $Z_{[k],\Ga}\cap \fV \ne \emptyset$.
We set 
\begin{equation}\label{zero-one-fV-sFk} 
\tGa^\zero_\fV= \tGa^\zero_{\fV'} \sqcup 
\{ x_{(\uu,\uv)}\in \La_{F_k} \mid \hbox{$x_\uu$ or $x_\uv \in \Ga$}\}.  
\end{equation}

We are now ready to 
prove Lemma \ref{wp-transform-sVk-Ga} (1), (2) and (3) in the case of $\sR_{[k]}$.

(1).  Note that scheme-theoretically, we have
$$ \rho_{[k]}^{-1}(Z_{\sF_{ [k-1]},\Ga}) \cap \fV=
 \pi_{[k], \sF_{[k-1]}}^{-1}(Z_{\sF_{ [k-1]},\Ga}) \cap \sV_{[k]} \cap \fV$$
 where $\pi_{[k],  \sF_{[k-1]}}: \sR_{[k]} \to \sR_{\sF_{[k-1]}}$ is the projection.
 We can apply Lemma \ref{wp-transform-sVk-Ga} (1)  in the case of $\sR_{\sF_{[k-1]}}$
 to  $Z_{\sF_{ [k-1]},\Ga}$ and $\pi_{[k], \sF_{[k-1]}}^{-1}(Z_{\sF_{ [k-1]},\Ga})$, 
 apply Proposition \ref{equas-fV[k]}  to $\sV_{[k]} \cap \fV$, and 
 use the construction \eqref{construction-ZkGa} of $Z_{\sF_{[ k]},\Ga}$
 (cf. \eqref{construction-ZkLa} and \eqref{zero-one-fV-sFk}), 
 we then obtain that $Z_{\sF_{[ k]},\Ga} \cap \fV$, as a closed subscheme of $\fV$, is defined by 
$$ y, \;\; y \in  \tGa^\zero_\fV;\;\; \cB^{\frb}_{[k]}; $$
$$  B_{\fV, (s,t)},\;\; s, t \in S_{F_i} \; \hbox{ with all $i \in [k]$}$$
 $$  L_{\fV,F_i}, \; i \in [k]; \;  \bF_{\fV,j}, \; k <j\le \up.$$ 
 Then, the above implies Lemma \ref{wp-transform-sVk-Ga} (1) in the case of $\sR_{[k]}$.

(2). By construction, we have that the composition
$\tZ^\dagger_{[k],\Ga} \to \tZ^\dagger_{\sF_{[k-1]},\Ga} \to Z_\Ga$
is birational. 
This proves Lemma \ref{wp-transform-sVk-Ga} (2) in the case of $\sR_{[k]}$.

  (3). It suffices to prove that if $Z^\dagger_{\sF_{[ k]},\Ga}\cap \fV \subset (y=0)$,
 then $Z_{\sF_{[ k]},\Ga} \cap \fV \subset (y=0).$
 
 If $y=x_{\fV,\uu}$ ($=x_\uu$, cf. the proof of Proposition \ref{meaning-of-var-p-k=0}), 
 then 
$x_\uu \in \Ga$ because $Z^\dagger_{\sF_{[ k]},\Ga}$ is birational to $Z_\Ga$.
Therefore, $Z_{\sF_{[ k]},\Ga} \cap \fV \subset (y=0)$ by \eqref{Ga-rel-sVk=0},
which holds by (the just proved) Lemma \ref{wp-transform-sVk-Ga} (1) for  $\sR_{[k]}$.

Now assume $y=x_{\fV, (\uu, \uv)}$.
Here, $x_{\fV, (\uu, \uv)}$ is the de-homogenization of $x_{(\uu,\uv)}$
(cf. the proof of Proposition \ref{meaning-of-var-p-k=0}). Below,
upon setting $x_{(\uu_{s_{F_i,o}}, \uv_{s_{F_i,o}})} \equiv 1$ for all $i \in [k]$ (cf. Definition \ref{fv-k=0}),
 we can write $x_{\fV, (\uu, \uv)}=x_{\fV', (\uu, \uv)}=x_{(\uu,\uv)}$.

 Suppose $(\uu, \uv) \in \La_{F_i}$ with $i \in [k-1]$.
By taking the images of $Z^\dagger_{\sF_{[ k]},\Ga}\cap \fV \subset (y=0)$
under $\rho_{[k]}$, we obtain 
$Z^\dagger_{\sF_{[ k-1]},\Ga} \cap \fV' \subset (x_{(\uu, \uv)}=0)$.
Hence, we have
$Z_{\sF_{[k-1]},\Ga} \cap \fV' \subset (x_{(\uu, \uv)}=0)$ by 
Lemma \ref{wp-transform-sVk-Ga} (3)  for 
$ \sR_{\sF_{[k-1]}}$.
Therefore,  $x_{(\uu, \uv)} \in \tGa^\zero_{\fV'}$ by the maximality of the subset $ \tGa^\zero_{\fV'}$.
Then, by  \eqref{zero-one-fV-sFk},
$Z_{\sF_{[ k]},\Ga} \cap \fV \subset   (x_{(\uu, \uv)}=0)$.

Now suppose $(\uu, \uv) \in \La_{F_k}$.
 Consider the relations  $$x_{\fV,\uu} x_{\fV,\uv} -x_{\fV, (\uu, \uv)}x_{\uu_{s_{F_k,o}}}x_{\uv_{s_{F_k,o}}} .$$ 
 Here, we have used $x_{(\uu_{s_{F_k,o}}, \uv_{s_{F_k,o}})} \equiv 1$. 
  Then, we have  $x_{\fV,\uu} x_{\fV,\uv}$ vanishes identically along $Z^\dagger_{\sF_{[ k]},\Ga}$,
  hence,  so does
  one of $x_{\fV, \uu}$ and $x_{\fV, \uv}$, that is, $x_\uu$ or $x_\uv \in \Ga$,
  since $Z^\dagger_{\sF_{[ k]},\Ga}$ (birational to $Z_\Ga$) is integral.\footnote{Here, the assumption that $Z_\Ga$ is integral is used.}
In either case, it implies that 
  $Z_{\sF_{[ k]},\Ga} \subset (x_{(\uu, \uv)}=0)$ by \eqref{construction-ZkLa} and  
  \eqref{construction-ZkGa}. 

This proves the lemma when $F_k$ is $\Ga$-relevant.

\smallskip

 Next, we suppose $F_k$ is $\Ga$-irrelevant. 
 
In this case, we have that
$$ ( \rho_{[k]}^{-1}(Z_{\sF_{  [k-1]},\Ga}) ) / (Z_{\sF_{[k-1]},\Ga})$$ 
is defined by the set of
equations of $L_{F_k}$ and $\cB^{\frb}_{ [k]}$, all regarded as
relations in $\vr$-variables of $\PP_{F_k}$. All these relations are  
 linear in $\vr$-variables of $F_k$, 
according to  Lemma \ref{strong sq free} (1).


 Putting together, we call $\{L_{F_k}, \cB^\frb_{ [k]}\}$ a linear system 
  in $\vr$-variables of $F_k$. 
 
 We can let 
 $\La^{\rm det}_{F_k, \Ga}$ be the subset of $\La_{F_k}$ such that 
  the minor corresponding to  variables
 $$\{x_{\fV,(\uu,\uv)} \mid (\uu,\uv) \in \La^{\rm det}_{F_k,\Ga}\}$$ 
 achieves the maximal rank of the linear system $ \{L_{F_k}, \cB^{\frb}_{[k]}|\}$,
  regarded as relations in $\vr$-variables of $F_k$,
 at any point of some fixed Zariski open subset $Z_{\sF_{[k-1]},\Ga}^{\dagger\circ}$ of 
 $Z^\dagger_{\sF_{[k-1]},\Ga}$. 

 We then set and plug 
 \begin{equation}\label{a-ne-0}
  x_{(\uu,\uv)} = 0, \; \forall \; (\uu,\uv)  \notin \La^{\rm det}_{F_k,\Ga}
 \end{equation}
   into the linear system $\{L_{F_k}, \cB^{\frb}_{ [k]}\}$
   to obtain an induced linear system of full rank 
   over $Z_{\sF_{[k-1]},\Ga}^{\dagger\circ}$. This induced linear system can be solved
  over the Zariski open subset $Z^{\dagger\circ}_{\sF_{[k-1]},\Ga}$ such that
  all variables $$\{x_{(\uu,\uv)} \mid (\uu,\uv) \in \La^{\rm det}_{F_k, \Ga}\}$$ 
 are explicitly  determined by the coefficients of the induced linear system.
    
     We then  let 
   \begin{equation}\label{det=0}
   \La^\zero_{F_k, \Ga} \subset \La_{F_k}
   \end{equation}
   be the subset consisting of    $ (\uu,\uv) \notin \La^{\rm det}_{F_k ,\Ga}$ 
and $(\uu,\uv) \in \La^{\rm det}_{F_k ,\Ga}$ 
   such that $x_{(\uu,\uv)} \equiv 0$ over
   $Z^{\dagger\circ}_{\sF_{[k-1]},\Ga}$.
    Observe here that we immediately obtain that  for any $(\uu,\uv) \in \La_{F_k}$,
\begin{equation}\label{who-in-dagger}
\hbox{$x_{(\uu,\uv)}$  vanishes identically over $Z^{\dagger\circ}_{\sF_{[k-1]},\Ga}$
   if and only if  $(\uu,\uv) \in \La^\zero_{F_k,\Ga}$.}
   \end{equation}

We let $Z_{\sF_{[ k]},\Ga}$ be
  the scheme-theoretic intersection
   \begin{equation}\label{ZkGa-irr}
   \rho_{[k]}^{-1}(Z_{\sF_{ [k-1]},\Ga}) \cap ( x_{(\uu,\uv)}=0 , \; (\uu,\uv) \in  \La^\zero_{F_k,\Ga}) . \end{equation}

Now, 
fix and consider any standard chart $\fV$ of $\sR_{[k]}$, lying over a
  standard chart $\fV'$  of $\sR_{\sF_{[k-1]}}$ with $\tZ_{[k],\Ga}\cap \fV \ne \emptyset$, 
  equivalently, $Z_{\sF_{[k-1]}} \cap \fV' \ne \emptyset$. We set 
\begin{eqnarray}\label{zero-one-fV-sFk-irr}
\tGa^\zero_\fV= \tGa^\zero_{\fV'} \sqcup    \{ x_{\fV, (\uu,\uv)} \mid 
(\uu, \uv) \in \La^\zero_{F_k,\Ga}\}. 
\end{eqnarray}

We are now ready to 
prove Lemma \ref{wp-transform-sVk-Ga} (1), (2) and (3) in the case of $\sR_{[k]}$.

  Similar to the proof of Lemma \ref{wp-transform-sVk-Ga} (1) for the previous case when
  $\bF_k$ is $\Ga$-relevant, 
  by Lemma \ref{wp-transform-sVk-Ga} (1)  in the case of $\sR_{\sF_{[k-1]}}$
  applied to $Z_{\sF_{ [k-1]},\Ga}$ and $\rho_{[k]}^{-1}(Z_{\sF_{ [k-1]},\Ga})$,
 applying Proposition \ref{equas-fV[k]} to $\sV_{F_{[k]}} \cap \fV$, and using 
 \eqref{ZkGa-irr} and \eqref{zero-one-fV-sFk-irr}, 
 we obtain that $Z_{\sF_{[ k]},\Ga} \cap \fV$, as a closed subscheme of $\fV$, is defined by 
$$y, \;\; y \in  \tGa^\zero_\fV; \;\; \cB^{\frb}_{[k]}; $$
$$  B_{\fV, (s,t)},\;\; s, t \in S_{F_i} \; \hbox{ with all $i \in [k]$}$$
$$ L_{\fV,F_i}, \;  \; i \in [k]; \; 
 \bF_{\fV,j}, \; k <j\le \up.$$ 
 Then, the above implies that Lemma \ref{wp-transform-sVk-Ga} (1) holds  on $\sR_{[k]}$.



  Next, by construction, the induced morphism
  $$ \rho_{[k]}^{-1} (Z^{\dagger\circ}_{[k],\Ga}) 
  \cap (x_{(\uu,\uv)}=0 , \; (\uu,\uv) \in  \La^\zero_{F_k,\Ga}) 
   \lra Z^{\dagger\circ}_{\sF_{[k-1]}}$$
  is an isomorphism. We let 
$Z^\dagger_{\sF_{ [k]},\Ga}$ be the closure of 
$$ \rho_{[k]}^{-1} (Z^{\dagger\circ}_{[k],\Ga}) 
  \cap (x_{(\uu,\uv)}=0 , \; (\uu,\uv) \in  \La^\zero_{F_k,\Ga} ) $$ in $Z_{\sF_{[ k]},\Ga}$.
  Then, it is closed in $Z_{\sF_{[ k]},\Ga}$ and contains 
  an open subset of $Z_{\sF_{[ k]},\Ga}$, hence, is an irreducible
  component of $Z_{\sF_{[ k]},\Ga}$.  It follows that the composition
  $$Z^\dagger_{\sF_{ [k]},\Ga}\to Z^{\dagger}_{\sF_{[k-1]}} \to Z_\Ga$$  is birational.
  This proves    Lemma \ref{wp-transform-sVk-Ga} (2) on $\sR_{[k]}$.
  

Finally,  we are to prove  Lemma \ref{wp-transform-sVk-Ga} (3) on $\sR_{[k]}$.
Suppose $Z^\dagger_{\sF_{ [k]},\Ga} \cap \fV \subset (y=0)$ for some $y \in \var_\fV$.
If $y=x_{\fV, \uu}$ or $y=x_{\fV, (\uu, \uv)}$ with  $(\uu, \uv) \in \La_{F_i}$ with $i \in [k-1]$,
then the identical  proof in the previous case carries over here without changes.
We now suppose $Z^\dagger_{\sF_{ [k]},\Ga}\cap \fV \subset (x_{\fV,(\uu,\uv)}=0)$
with $(\uu, \uv) \in \La_{F_k}$, then by \eqref{who-in-dagger},
$(\uu,\uv) \in \La^\zero_{F_k,\Ga}$. Thus,  by \eqref{ZkGa-irr}, 
$Z_{\sF_{ [k]},\Ga} \subset (x_{(\uu,\uv)}=0)$.
This proves    Lemma \ref{wp-transform-sVk-Ga} (3) on $\sR_{[k]}$.

By induction, Lemma \ref{wp-transform-sVk-Ga}  is proved.
\end{proof}

We call $Z_{[k],\Ga}$ the $\sF$-transform of $Z_\Ga$
in $\sV_{[k]}$ for any $k\in [\up]$.

\subsection{$\vt$-transforms of  $\Ga$-schemes in  $\tsV_{\vt_{[k]}}$}
\label{subsection:wp-transform-vtk}  $\ $

We now construct the $\vt$-transform of $Z_\Ga$ in $\tsV_{\vt_{[k]}}
\subset \tsR_{\vt_{[k]}}$.  The meaning of the notation $\cB^\ngv_{\fV, >k}$
below can be found in Definition \ref{<,>}.

\begin{lemma}\label{vt-transform-k} 
 Fix any subset $\Ga$ of $\bU$.  Assume that $Z_\Ga$ is integral. 

Fix any $k \in [\up]$.

Then, we have the following:
\begin{itemize}
\item  there exists a closed subscheme $\tZ_{\vt_{[k]},\Ga}$ of
$\tsV_{\vt_{[k]}}$ with an induced morphism  
$$\tZ_{\vt_{[k]},\Ga} 
\to Z_\Ga; $$
\item   $\tZ_{\vt_{[k]},\Ga}$ comes equipped with an irreducible component  
$\tZ^\dagger_{\vt_{[k]},\Ga}$ with the induced morphism 
$\tZ^\dagger_{\vt_{[k]},\Ga}  
\to Z_\Ga$;
\item  for any preferred admissible chart $\fV$ of $\tsR_{\vt_{[k]}}$ such that
$\tZ_{\vt_{[k]},\Ga} \cap \fV \ne \emptyset$, there are two subsets, possibly empty,
$$ \tGa^\zero_{\fV} \; \subset \;  \var_\fV, \;\;\;
\tGa^\one_{\fV} \; \subset \;  \var_\fV.$$ 
\end{itemize}

Further,  consider any given preferred admissible chart $\fV$ of $\tsR_{\vt_{[k]}}$ with
$\tZ_{\vt_{[k]},\Ga} \cap \fV \ne \emptyset$. Then,  the following hold:
\begin{enumerate}
\item the scheme $\tZ_{\vt_{[k]},\Ga} \cap \fV$,
 as a closed subscheme of the chart $\fV$,
is defined by the following relations
\begin{eqnarray} 
\;\;\;\;\; y , \; \; \; y \in  \tGa^\zero_\fV , \label{Ga-rel-wp-ktauh-00}\\
\;\;\;  y -1, \; \; \; y \in  \tGa^\one_\fV,  \nonumber \\
\cB_\fV^\gov, \; \cB^\ngv_{\fV, >k}, \;  \cB^\frb_\fV, \; L_{\fV,\sF};\nonumber
\end{eqnarray}
 further, we take $\tGa^\zero_\fV \subset \var_\fV$
 to be the maximal subset (under inclusion)
among all those subsets that satisfy the above;
\item the induced morphism $\tZ^\dagger_{ \vt_{[k]},\Ga} \to Z_\Ga$ is birational;
\item for any variable $y \in \var_\fV$, $\tZ^\dagger_{\vt_{[k]},\Ga} \cap \fV \subset (y=0)$ if and only 
if  $\tZ_{\vt_{[k]},\Ga} \cap \fV \subset (y=0)$. Consequently, 
 $\tZ^\dagger_{\vt_{[k]},\Ga} \cap \fV \subset \tZ_{\vt_{[k+1]}} \cap \fV$ if and only 
if  $\tZ_{\vt_{[k]},\Ga} \cap \fV \subset \tZ_{\vt_{[k+1]}}\cap \fV$ 
where $\tZ_{\vt_{[k+1]}}$ is the proper transform of 
 the $\vt$-center $Z_{\vt_{[k+1]}}$ in  $\tsR_{\vt_{[k]}}$. 
\end{enumerate}
\end{lemma}  
\begin{proof} We prove by induction on 
$k \in \{0\} \cup [\up]$. 


The initial case is $k=0$. In this case, we have
 $$\tsR_{\vt_{[0]}}:=\sR_{\sF}, \;\;
 \tsV_{\vt_{[0]}}:=\sV_{\sF}, \;\;
  \tZ_{\vt_{[0]},\Ga}:=Z_{\sF_{ [\up]},\Ga}, \;\;
  \tZ^\dagger_{\vt_{[0]},\Ga}:=Z^\dagger_{\sF_{ [\up]},\Ga}.$$
Then, in this case,  Lemma \ref{vt-transform-k}
 is Lemma  \ref{wp-transform-sVk-Ga} for $k=\up$, where
  we  set $\tGa^\one_\fV=\emptyset$.

We now suppose that Lemma \ref{vt-transform-k} holds
over $\tsR_{\vt_{[k-1]}}$ for some $k\in[\up]$.

We then consider the case of  $\tsR_{\vt_{[k]}}$.

Suppose that $\tZ_{\vt_{[k-1]}, \Ga}$,  or equivalently $\tZ^\dagger_{\vt_{[k-1]},\Ga}$,
by Lemma \ref{vt-transform-k} (3) in the case of $\tsR_{\vt_{[k-1]}}$,
 is not contained in 
$Z'_{\vt_{[k]}}$ where $Z'_{\vt_{[k]}}$ is the proper transform in $\tsR_{\vt_{[k-1]}}$
 of the $\vt$-center $Z_{\vt_{[k]}}$ (of $\tsR_{\vt_{[0]}}$).
We then let $\tZ_{\vt_{[k]},\Ga}$ (respectively, $\tZ^\dagger_{\vt_{[k]},\Ga}$)
 be the proper transform of $\tZ_{\vt_{[k-1]},\Ga}$ (respectively,
  $\tZ^\dagger_{\vt_{[k-1]},\Ga}$) in $\tsV_{\vt_{[k]}}$.
As $\tZ^\dagger_{\vt_{[k]},\Ga}$ is closed in $\tZ_{\vt_{[k]},\Ga}$
and contains a Zariski open subset of $\tZ_{\vt_{[k]},\Ga}$, it is an irreducible
component of $\tZ_{\vt_{[k]},\Ga}$.

Further, consider any standard chart $\fV$ of $\tsR_{\vt_{[k]}}$,
lying over a unique standard chart $\fV'$ of $\tsR_{\vt_{[k-1]}}$,
such that $\tZ_{\vt_{[k]},\Ga}\cap \fV \ne \emptyset$.
We set 
$$\tGa^\zero_\fV=\{y_\fV \mid y_\fV  \hbox{ is the proper transform of some $y_{\fV'} \in \tGa^\zero_{\fV'}$}\};$$
$$\tGa^\one_\fV=\{y_\fV \mid y_\fV  \hbox{ is the proper transform of some $y_{\fV'} \in \tGa^\one_{\fV'}$}\}.$$

We now prove Lemma \ref{vt-transform-k} (1), (2) and (3) in the case of $\tsR_{\vt_{[k]}}$.

 We can apply
 Lemma \ref{vt-transform-k} (1) in the case of $\tsR_{\vt_{[k-1]}}$
 to $\tZ_{\vt_{[k-1]},\Ga}$ to obtain the defining equations of $\tZ_{\vt_{[k-1]},\Ga}\cap \fV'$
as stated in the lemma; we note here that these equations include $\cB^\ngv_{\fV', \ge k}$.
We then  take the proper transforms of these equations in $\fV'$ to obtain the corresponding
equations in $\fV$,  and then apply (the proof of) 
 Proposition \ref{eq-for-sV-vtk} to reduce $\cB^\ngv_{\fV, \ge k}$ to $\cB^\ngv_{\fV, > k}$.
 Because $\tZ_{\vt_{[k]}\Ga}$  is the proper transform of $\tZ_{\vt_{[k-1]},\Ga}$,
this implies Lemma \ref{vt-transform-k} (1)  in the case of $\tsR_{\vt_{[k]}}$.

By construction, we have that the composition
$\tZ^\dagger_{\vt_{[k]},\Ga} \to \tZ^\dagger_{\vt_{[k-1]},\Ga} \to Z_\Ga$
is birational. 
This proves
Lemma \ref{vt-transform-k} (2) in the case of $\tsR_{\vt_{[k]}}$.

To show Lemma \ref{vt-transform-k} (3) in $\tsR_{\vt_{[k]}}$, 
we fix any $y \in \var_\fV$. It suffices to show that
if $\tZ^\dagger_{\vt_{[k]},\Ga}\cap \fV
\subset (y=0)$, then $\tZ_{\vt_{[k]},\Ga}\cap \fV
\subset (y=0)$.  By construction, $y \ne \zeta_{\fV,\vt_{[k]}}$, the exceptional variable
in $\var_\fV$ corresponding to the $\vt$-center $Z_{\vt_{[k]}}$. Hence, $y$ is the proper transform
of some $y' \in \var_{\fV'}$. Then, by taking the images of $\tZ^\dagger_{\vt_{[k]},\Ga}\cap \fV
\subset (y=0)$ under the morphism $\rho_{\vt_{[k]}}: \tsV_{\vt_{[k]}} \to \tsV_{\vt_{[k-1]}}$ 
(which is  induced from the blowup morphism $\pi_{\vt_{[k]}}: \tsR_{\vt_{[k]}} \to \tsR_{\vt_{[k-1]}}$), 
we obtain
$\tZ^\dagger_{\vt_{[k-1]},\Ga}\cap \fV'
\subset (y'=0)$,  hence, $\tZ_{\vt_{[k-1]},\Ga}\cap \fV'
\subset (y'=0)$ by the inductive assumption.
 Then, as $\tZ_{\vt_{[k]},\Ga}$ is the proper transform of $\tZ_{\vt_{[k-1]},\Ga}$,
 we obtain $\tZ_{\vt_{[k]},\Ga}\cap \fV \subset (y=0)$.
 
 The last statement Lemma \ref{vt-transform-k} (3)  follows from the above because
 $\tZ_{\vt_{[k+1]}}\cap \fV=(y_0=y_1=0)$ for some $y_0, y_1 \in \var_\fV$.

\smallskip

We now suppose that $\tZ_{\vt_{[k-1]},\Ga}$, or equivalently $\tZ^\dagger_{\vt_{[k-1]},\Ga}$,
 by  Lemma \ref{vt-transform-k} (3) in $\tsR_{\vt_{[k-1]}}$,
 is contained in the proper transform $Z'_{\vt_{[k]}}$ 
 of the $\vt$-center $Z_{\vt_{[k]}}$. 

Consider  any standard chart $\fV$ of $\tsR_{\vt_{[k]}}$,
lying over a unique standard chart $\fV'$ of $\tsR_{\vt_{[k-1]}}$,
such that $\tZ_{\vt_{[k]},\Ga}\cap \fV \ne \emptyset$.

We let $\vt'_{[k]}$
 be the proper transform in  the chart $\fV'$ of the $\vt$-set ${\vt_{[k]}}$.
Then, $\vt'_{[k]}$ consists of two  variables 
 $$\vt'_{[k]}=\{y'_0,y'_1\} \subset \var_{\fV'}.$$

 We let $\PP^1_{[\xi_0,\xi_1]}$ be the factor projective space for the
$\vt$-blowup $\tsR_{\vt_{[k]}} \to \tsR_{\vt_{[k-1]}}$ with $[\xi_0,\xi_1]$ corresponding to $(y'_0,y'_1)$.
Without loss of generality,  we can assume that the open chart $\fV$ is given by 
  $$(\fV' \times (\xi_0 \equiv 1) ) \cap \tsR_{\vt_{[k]}} \subset \fV' \times \PP^1_{[\xi_0,\xi_1]}.$$

 We let $\zeta_\fV:=\zeta_{\fV, \vt_{[k]}} \in \var_\fV$ be such that 
$E_{\vt_{[k]}} \cap \fV=(\zeta_\fV =0)$ where $E_{\vt_{[k]}}$ is the exceptional divisor
of the blowup $\tsR_{\vt_{[k]}} \to \tsR_{\vt_{[k-1]}}$. 
Note here that according to the proof of Proposition \ref{meaning-of-var-vtk},
the variable $y'_0$ 
corresponds to (or turns into) the exceptional $\zeta_\fV$ on the chart $\fV$.
 We then let $y_1 (=\xi_1) \in \var_\fV$ be 
 the proper transform of $y'_1 \in \var_{\fV'}$ on the chart $\fV$.

  
In addition, we  observe that 
 $$\vt'_{[k]}=\{y'_0,  y'_1\} \subset \tGa^\zero_{\fV'}$$ 
because $\tZ_{\vt_{[k]},\Ga}$ is contained in the proper transform
$Z'_{\vt_{[k]}}$ of the $\vt$-center $Z_{\vt_{[k]}}$.

We set, 
\begin{equation}\label{Gazero-ktah-contained-in-bar-vtk} 
\overline{\Ga}^\zero_\fV= \{\zeta_\fV, \; y_\fV \mid y_\fV  
\hbox{ is the proper transform of some $ y_{\fV'} \in \tGa^\zero_{\fV'} \- \vt'_{[k]}$}\},
\end{equation}
\begin{equation}\label{Gaone-ktah-contained-in-bar-vtk}
 \overline{\Ga}^\one_{\fV}=\{\ y_\fV \mid y_\fV  
\hbox{ is the proper transform of some $y_{\fV'} \in \tGa^\one_{\fV'}$} \}.
\end{equation}

Consider the scheme-theoretic pre-image  
$\rho_{\vt_{[k]}}^{-1}(\tZ_{\vt_{[k-1]},\Ga})$
where $\rho_{\vt_{[k]}}: \tsV_{\vt_{[k]}} \to \tsV_{\vt_{[k-1]}}$ is  induced 
from the blowup morphism $\pi_{\vt_{[k]}}: \tsR_{\vt_{[k]}} \to \tsR_{\vt_{[k-1]}}$.

Note that  scheme-theoretically, we have,
$$\rho_{\vt_{[k]}}^{-1}(\tZ_{\vt_{[k-1]},\Ga})  \cap \fV=
\pi_{\vt_{[k]}}^{-1}(\tZ_{\vt_{[k-1]},\Ga}) \cap \tsV_{\vt_{[k]}} \cap \fV .$$
Applying Lemma \ref{vt-transform-k} (1) in $\tsR_{\vt_{[k-1]}}$ to $\tZ_{\vt_{[k-1]},\Ga}$ 
and $\rho_{\vt_{[k]}}^{-1}(\tZ_{\vt_{[k-1]},\Ga})$, 
and applying Proposition \ref{eq-for-sV-vtk} to 
$\tsV_{\vt_{[k]}} \cap \fV$, we obtain that 
$\rho_{\vt_{[k]}}^{-1}(\tZ_{\vt_{[k-1]},\Ga})  \cap \fV$,
as a closed subscheme of $\fV$,  is defined by 
\begin{equation}\label{vt-pre-image-defined-by}
y_\fV \in \overline{\Ga}^\zero_\fV; \;\;  y_\fV-1, \; y_\fV \in \overline{\Ga}^\one_{\fV};\;\;
\cB^\gov_\fV; \;\; \cB^\ngv_{\fV, >k}; \;\; \cB^\frb_{\fV}; \;\; 
L_{\sF, \fV}.
\end{equation} 
(Observe here that $\zeta_\fV \in  \overline{\Ga}^\zero_\fV$.)

Thus, by setting $y_\fV=0$ for all $y_\fV \in  \overline{\Ga}^\zero_\fV$ and
 $y_\fV=1$ for all $y_\fV \in \overline{\Ga}^\one_{\fV}$ in 
 $\cB^\gov_\fV, \cB^\ngv_{\fV, >k}, \cB^\frb_{\fV}, L_{\sF, \fV}$ of  the above,
 we obtain 
 \begin{equation}\label{vt-lin-xi-pre}
 \tilde{\cB}^\gov_\fV, \tilde\cB^\ngv_{\fV,>k}, \tilde\cB^\frb_{\fV}, \tilde{L}_{\fV,\sF}.
 \end{equation}
 Note that  for any $\bF \in \sF$, if $L_{\fV, F}$ contains $y_1$,
 then it contains $\zeta_\fV$, hence $\tilde L_{\fV, F}$ does not contain $y_1$.
 We keep those equations of \eqref{vt-lin-xi-pre} that  contain the variable
 $y_1$ and obtain
 \begin{equation}\label{lin-xi-vtk}
 \hat{\cB}^\gov_\fV, \hat\cB^\ngv_{\fV,>k}, \hat\cB^\frb_{\fV}, 
 \end{equation}
 viewed as a  system of equations in $y_1$.
 By Proposition \ref{eq-for-sV-vtk} (the last two statements), 
one sees that \eqref{lin-xi-vtk} is  a  {\it linear} system of equations in $y_1$.
Furthermore, we have that 
 $$(\rho_{\vt_{[k]}}^{-1}(\tZ_{\vt_{[k-1]},\Ga})\cap \fV)/(\tZ_{\vt_{[k-1]},\Ga}\cap \fV')$$
 is defined by the linear system \eqref{lin-xi-vtk}.

There are the following two cases for \eqref{lin-xi-vtk}:

$(\star a)$  the rank of the linear system  $\eqref{lin-xi-vtk}$ equals one 
over general points of  $\tZ^\dagger_{\vt_{[k-1]},\Ga}$.

$(\star b)$  the rank of the linear system  $\eqref{lin-xi-vtk}$ equals zero
at general points of  $\tZ^\dagger_{\vt_{[k-1]},\Ga}$, hence at all points
of $\tZ^\dagger_{\vt_{[k-1]},\Ga}$.

\smallskip\noindent
{\it Proof of Lemma \ref{vt-transform-k} for $\tsR_{\vt_{[k]}}$ under the condition $(\star a)$.}
\smallskip

By the condition $(\star a)$,
 there exists a Zariski open subset $\tZ^{\dagger\circ}_{\vt_{[k-1]},\Ga}$ 
of $\tZ^\dagger_{\vt_{[k-1]},\Ga}$ such that the rank of the linear system
 \eqref{lin-xi-vtk} equals one at any point of  $\tZ^{\dagger\circ}_{\vt_{[k-1]},\Ga}$. 
By solving $y_1$ from
 the linear system \eqref{lin-xi-vtk} over $\tZ^{\dagger\circ}_{\vt_{[k-1]},\Ga}$, we obtain 
that the induced morphism
$$
\rho_{\vt_{[k]}}^{-1}(\tZ^{\dagger\circ}_{\vt_{[k-1]},\Ga}) 
 \lra \tZ^\circ_{\vt_{[k-1]},\Ga}$$ is an isomorphism. 

First, we suppose  $y_1$ is identically zero along 
$\rho_{\vt_{[k]}}^{-1}(\tZ^{\dagger\circ}_{\vt_{[k-1]},\Ga})$.
We then set,  
\begin{equation}\label{Gazero-ktah-contained-in-a-vtk} 
\tGa^\zero_\fV=\{y_1 \} \cup \overline{\Ga}^\zero_\fV
\end{equation}
where $\overline{\Ga}^\zero_\fV$ is as in \eqref{Gazero-ktah-contained-in-bar-vtk}.
In this case, we let
\begin{equation}\label{ZktahGa-with-xi-vtk}
\tZ_{\vt_{[k]},\Ga}=\rho_{\vt_{[k]}}^{-1}(\tZ_{\vt_{[k-1]},\Ga})\cap D_{y_1}
\end{equation}
scheme-theoretically, where $D_{y_1}$ is the closure of $(y_1=0)$ in $\tsR_{\vt_{[k]}}$.
Note that  $D_{y_1}$  is simply the proper transform of one of 
 the two divisors in the $\wp$-set. We let
  $D_{y_0}$ be the proper transforms
of the other divisor in the $\wp$-set.
They $D_{y_1}$ and $D_{y_0}$ are disjoint. 
Hence, only one of them can contain 
$\rho_{\vt_{[k]}}^{-1}(\tZ^{\dagger\circ}_{\vt_{[k-1]},\Ga})$.
This implies that  $D_{y_1}$ comes with
$\rho_{\vt_{[k]}}^{-1}(\tZ^{\dagger\circ}_{\vt_{[k-1]},\Ga})$ in such a case,
and does not depend on the choice of the chart $\fV$.

Next, suppose  $y_1$ is not identically zero along 
$\rho_{\vt_{[k]}}^{-1}(\tZ^{\dagger\circ}_{\vt_{[k-1]},\Ga})$.
We then set,  
\begin{equation}\label{Gazero-ktah-contained-in-a-vtk'} 
\tGa^\zero_\fV= \overline{\Ga}^\zero_\fV
\end{equation}
where $\overline{\Ga}^\zero_\fV$ is as in \eqref{Gazero-ktah-contained-in-bar-vtk}.
In this case, we let
\begin{equation}\label{ZktahGa-without-xi-vtk}
\tZ_{\vt_{[k]},\Ga}=\rho_{\vt_{[k]}}^{-1}(\tZ_{\vt_{[k-1]},\Ga}).
\end{equation}
We always set (under the condition $(\star a)$)
\begin{equation}\label{Gaone-ktah-contained-in-a-vtk} \tGa^\one_{\fV}= \overline{\Ga}^\one_{\fV}
\end{equation}
where $\overline{\Ga}^\one_{\fV}$ is as in  \eqref{Gaone-ktah-contained-in-bar-vtk}.

In each case, by construction, we have 
$$\rho_{\vt_{[k]}}^{-1}(\tZ^{\dagger\circ}_{\vt_{[k-1]},\Ga})  
\subset \tZ_{\vt_{[k]}\Ga},$$ and we let $\tZ^\dagger_{\vt_{[k]},\Ga}$ be the closure of 
$\rho_{(\vt_{[k]}}^{-1}(\tZ^{\dagger\circ}_{\vt_{[k-1]},\Ga}) $ in $\tZ_{\vt_{[k]},\Ga}$.
It is an irreducible component of $\tZ_{\vt_{[k]},\Ga}$ because
$\tZ^\dagger_{\vt_{[k]},\Ga}$ is closed in $\tZ_{\vt_{[k]},\Ga}$ and contains
the Zariski open subset $\rho_{\vt_{[k]}}^{-1}(\tZ^{\dagger\circ}_{\vt_{[k-1]},\Ga}) $
of $\tZ_{\vt_{[k]},\Ga}$.
Then, we obtain that the composition
 $$\tZ^\dagger_{\vt_{[k]},\Ga} \lra \tZ^\dagger_{\vt_{[k-1]},\Ga}
 \lra Z_\Ga$$ is birational.
 This proves Lemma \ref{vt-transform-k} (2)  over $\tsR_{\vt_{[k]}}$.


In each case of the above (i.e., \eqref{Gazero-ktah-contained-in-a-vtk} and
\eqref{Gazero-ktah-contained-in-a-vtk'}), 
  by the paragraph of \eqref{vt-pre-image-defined-by},
one sees that $\tZ_{\vt_{[k]},\Ga}\cap \fV$,
 as a closed subscheme of $\fV$, is defined by the equations as stated in the Lemma.  
 This proves Lemma \ref{vt-transform-k} (1)  over $\tsR_{\vt_{[k]}}$.

It remains to prove Lemma \ref{vt-transform-k} (3) over $\tsR_{\vt_{[k]}}$.

Fix any $y \in \var_\fV$, it suffices to show that
if $\tZ^\dagger_{\vt_{[k]},\Ga}\cap \fV
\subset (y=0)$, then $\tZ_{(\vt_{[k]},\Ga}\cap \fV
\subset (y=0)$.  
If $y \ne \zeta_\fV,  y_1$, then $y$ is the proper transform of some variable $y' \in \var_{\fV'}$.
Hence, by taking the images under $\rho_{\vt_{[k]}}$, we have $\tZ^\dagger_{\vt_{[k-1]},\Ga}\cap \fV'
\subset (y'=0)$; by Lemma \ref{vt-transform-k} (3) in $\tsR_{\vt_{[k-1]}}$, we obtain
$\tZ_{\vt_{[k-1]},\Ga}\cap \fV'
\subset (y'=0)$,  thus $y' \in \tGa^\zero_{\fV'}$ by the maximality of
the subset $\tGa^\zero_{\fV'}$. Therefore, 
$\tZ_{\vt_{[k]},\Ga}\cap \fV \subset ( y=0)$,  
by (the already-proved) Lemma \ref{vt-transform-k} (1) for $\tsR_{\vt_{[k]}}$
(cf. \eqref{Gazero-ktah-contained-in-bar-vtk} and \eqref{Gazero-ktah-contained-in-a-vtk} 
or \eqref{Gazero-ktah-contained-in-a-vtk'}). 
Next, suppose $y = y_1$ (if it occurs). Then, by construction, 
$\tZ_{\vt_{[k]},\Ga}\cap \fV \subset (y=0)$. 
Finally, we let $y =\zeta_\fV$. 
 Again, by construction, $\tZ_{\vt_{[k]},\Ga}\cap \fV \subset ( \zeta_\fV=0)$. 

As earlier, the last statement Lemma \ref{vt-transform-k} (3)  follows from the above.

\smallskip\noindent
{\it  Proof of Lemma \ref{vt-transform-k} over $\tsR_{\vt_{[k]}}$ under the condition $(\star b)$.}
\smallskip

Under the condition $(\star b)$, we have that $$ 
\rho_{\vt_{[k]}}^{-1}(\tZ^\dagger_{\vt_{[k-1]},\Ga})
 \lra \tZ^\dagger_{\vt_{[k-1]},\Ga}$$
can be canonically identified with the trivial $\PP_{[\xi_0,\xi_1]}$-bundle:
$$\rho_{\vt_{[k]}}^{-1}(\tZ^\dagger_{(\vt_{[k-1]},\Ga}) 
= \tZ^\dagger_{\vt_{[k-1]},\Ga}
\times \PP_{[\xi_0,\xi_1]}.$$
In this case, we define 
$$\tZ_{\vt_{[k]},\Ga}= \rho_{\vt_{[k]}}^{-1}(\tZ_{\vt_{[k-1]},\Ga})\cap ((\xi_0,\xi_1)= (1,1)),$$
 $$\tZ^\dagger_{\vt_{[k]},\Ga}= 
 \rho_{\vt_{[k]}}^{-1}(\tZ^\dagger_{\vt_{[k-1]},\Ga})
 \cap ((\xi_0,\xi_1)= (1,1)),$$
 both scheme-theoretically.
 The induced morphism $\tZ^\dagger_{\vt_{[k]},\Ga} \lra \tZ^\dagger_{\vt_{[k-1]},\Ga}$
 is an isomorphism. 
 Again, one sees that $\tZ^\dagger_{\vt_{[k]},\Ga}$ is 
 an irreducible component of $\tZ_{\vt_{[k]},\Ga}$. Therefore,
 $$\tZ^\dagger_{\vt_{[k]},\Ga} \lra
 \tZ^\dagger_{\vt_{[k-1]},\Ga} \lra  Z_\Ga$$ is  birational.
 This proves Lemma \ref{wp/ell-transform-ktauh} (2) over $\tsR_{\vt_{[k]}}$.

Further, under the condition $(\star b)$, we set 
 \begin{eqnarray}
 \tGa^\zero_\fV= \overline{\Ga}^\zero_\fV,  \;\;\;\;
\tGa^\one_\fV=\{y_1\}\cup \overline{\Ga}^\one_{\fV}. \nonumber \end{eqnarray}
Then, again, by the paragraph of \eqref{vt-pre-image-defined-by},
one sees that $\tZ_{\vt_{[k]},\Ga}\cap \fV$,
 as a closed subscheme of $\fV$, is defined by the equations as stated in the Lemma. 
 This proves Lemma \ref{vt-transform-k} (1) over $\tsR_{\vt_{[k]}}$.

 It remains to prove Lemma \ref{vt-transform-k} (3) in over $\tsR_{\vt_{[k]}}$.

Fix any $y \in \var_\fV$, it suffices to show that
if $\tZ^\dagger_{\vt_{[k]}, \Ga}\cap \fV
\subset (y=0)$, then $\tZ_{\vt_{[k]},\Ga}\cap \fV
\subset (y=0)$.  By construction, $y \ne y_1$.
Then, the corresponding proof of Lemma \ref{vt-transform-k} (3) for $\tsR_{\vt_{[k]}}$
under the condition $(\star a)$ goes through here without change.
 The last statement Lemma \ref{vt-transform-k} (3)  follows from the above.
This proves Lemma \ref{vt-transform-k} (3) in $\tsR_{\vt_{[k]}}$
under the condition $(\star b)$.

This completes the proof of Lemma \ref{vt-transform-k}.
\end{proof}

 We call $\tZ_{\vt_{[k]},\Ga}$ the $\vt$-transform of $Z_\Ga$ 
 in $\tsV_{\vt_{[k]}}$ for any $k \in [\up]$.

  We need the final case of Lemma 
\ref{vt-transform-k}. We set
$$\tZ_{\vt, \Ga}:=\tZ_{\vt_{[\up]},\Ga},
 \;\; \tZ^\dagger_{\vt, \Ga}:=\tZ^\dagger_{\vt_{[\up]},\Ga}.$$

\subsection{$\wp$- and $\ell$-transforms of  
$\Ga$-schemes in   $\tsV_{({\wp}_{(k\tau)}\fr_\mu\fs_h)}$ and in $\tsV_{\ell_k}$}
\label{subsection:wp-transform-vsk}  $\ $

We now construct the ${\wp}$-transform of $Z_\Ga$ in $\tsV_{({\wp}_{(k\tau)}\fr_\mu\fs_h)}
\subset \tsR_{({\wp}_{(k\tau)}\fr_\mu\fs_h)}$. 
Here, as in Proposition \ref{meaning-of-var-wp/ell},  
we assume that the last of
$\tsR_{({\wp}_{(k\tau)}\fr_\mu\fs_h)}$ within
the block $(\fG_k)$ is $\tsR_{\ell_k}$.


\begin{lemma}\label{wp/ell-transform-ktauh} 
 Fix any subset $\Ga$ of $\bU$.  Assume that $Z_\Ga$ is integral. 

Consider $(k\tau) \mu h \in  \Index_{\Phi_k} \sqcup \{\ell_k\}$.

Then, we have the following:
\begin{itemize}
\item  there exists a closed subscheme $\tZ_{({\wp}_{(k\tau)}\fr_\mu\fs_h),\Ga}$ of
$\tsV_{({\wp}_{(k\tau)}\fr_\mu\fs_h)}$ with an induced morphism  
$\tZ_{({\wp}_{(k\tau)}\fr_\mu\fs_h),\Ga} 
\to Z_\Ga$;
\item   $\tZ_{({\wp}_{(k\tau)}\fr_\mu\fs_h),\Ga}$ comes equipped with an irreducible component  
$\tZ^\dagger_{({\wp}_{(k\tau)}\fr_\mu\fs_h),\Ga}$ with the induced morphism 
$\tZ^\dagger_{({\wp}_{(k\tau)}\fr_\mu\fs_h),\Ga}  
\to Z_\Ga$;
\item  for any admissible affine chart $\fV$ of $\tsR_{({\wp}_{(k\tau)}\fr_\mu\fs_h)}$ such that
$\tZ_{({\wp}_{(k\tau)}\fr_\mu\fs_h),\Ga} \cap \fV \ne \emptyset$, there come equipped with
two subsets, possibly empty,
$$ \tGa^\zero_{\fV} \; \subset \; \var^\vee_\fV, \;\;\;
\tGa^\one_{\fV} \; \subset \;  \var_\fV.$$ 
\end{itemize}

Further,  consider any given chart $\fV$ of $\tsR_{({\wp}_{(k\tau)}\fr_\mu\fs_h)}$ with
$\tZ_{({\wp}_{(k\tau)}\fr_\mu\fs_h),\Ga} \cap \fV \ne \emptyset$. 
Then,  the following hold:
\begin{enumerate}
\item the scheme $\tZ_{({\wp}_{(k\tau)}\fr_\mu\fs_h),\Ga} \cap \fV$,
 as a closed subscheme of the chart $\fV$,
is defined by the following relations
\begin{eqnarray} 
\;\;\;\;\; y , \; \; \; y \in  \tGa^\zero_\fV , \label{Ga-rel-wp-ktauh}\\
\;\;\;  y -1, \; \; \; y \in  \tGa^\one_\fV,  \nonumber \\
\cB_\fV^\gov, \; \cB^\frb_\fV, \; L_{\sF, \fV}; \nonumber
\end{eqnarray}
 further, we take $\tGa^\zero_\fV \subset \var_\fV$
 to be the maximal subset (under inclusion)
among all those subsets that satisfy the above.
\item the induced morphism $\tZ^\dagger_{({\wp}_{(k\tau)}\fr_\mu\fs_h),\Ga} \to Z_\Ga$ is birational; 
\item for any variable $y \in \var_\fV$, $\tZ^\dagger_{({\wp}_{(k\tau)}\fr_\mu\fs_h),\Ga} \cap \fV \subset (y=0)$ if and only 
if  $\tZ_{({\wp}_{(k\tau)}\fr_\mu\fs_h),\Ga} \cap \fV \subset (y=0)$. Consequently, 
 $\tZ^\dagger_{({\wp}_{(k\tau)}\fr_\mu\fs_h),\Ga}\cap \fV \subset \tZ_{\phi_{(k\tau)\mu(h+1)}}\cap \fV$ if and only 
if  $\tZ_{({\wp}_{(k\tau)}\fr_\mu\fs_h),\Ga} \cap \fV \subset \tZ_{\phi_{(k\tau)\mu(h+1)}}\cap \fV$ 
where $\tZ_{\phi_{(k\tau)\mu(h+1)}}$ is the proper transform of the $\wp$-center 
$Z_{\phi_{(k\tau)\mu(h+1)}}$ in  $\tsR_{({\wp}_{(k\tau)}\fr_\mu\fs_h)}$. 
 \end{enumerate}
\end{lemma}  
\begin{proof} We prove by induction on 
 $(k\tau) \mu h \in  \{(11)10)\} \sqcup \Index_{\Phi_k} \sqcup \{\ell_k\}$ (cf. \eqref{indexing-Phi}).


The initial case is $(11)10$. In this case, we have
 $$\tsR_{({\wp}_{(11)}\fr_1\fs_0)}:=\tsR_{\vt}, \;\;
 \tsV_{({\wp}_{(11)}\fr_1\fs_0)}:=\tsV_{\vt} , \;\;
  \tZ_{({\wp}_{(11)}\fr_1\fs_0),\Ga}:=\tZ_{\vt,\Ga}, \;\;
   \tZ^\dagger_{({\wp}_{(11)}\fr_1\fs_0),\Ga}:=\tZ^\dagger_{\vt,\Ga}.$$
Then, in this case,  Lemma \ref{wp/ell-transform-ktauh}
 is Lemma \ref{vt-transform-k} for $k=\up$.

We now suppose that Lemma \ref{wp/ell-transform-ktauh} holds for $(k\tau)\mu(h-1)$
for some $(k\tau)\mu h \in \Index_{\Phi_k}$.

 \medskip
 
 \centerline{$\bullet$ $\wp$-transforms}
 
 \medskip

We treat exclusively $\wp$-transforms of $Z_\Ga$ first, 
in other words, we assume that 
$$\hbox{$\tsR_{({\wp}_{(k\tau)}\fr_\mu\fs_{h-1})}\ne \tsR_{\wp_k}$ and
 $\tsR_{({\wp}_{(k\tau)}\fr_\mu\fs_h)} \ne \tsR_{\ell_k}$.}$$
We treat $\ell$-transforms of $Z_\Ga$ in the end.

We then consider the case of $(k\tau)\mu h$.

We let 
\begin{equation}\label{rho-wpktaumuh}
 \rho_{(\wp_{(k\tau)}\fr_\mu\fs_h)}: \tsV_{(\wp_{(k\tau)}\fr_\mu\fs_h)} \lra
\tsV_{(\wp_{(k\tau)}\fr_\mu\fs_{h-1})}\end{equation}
be the morphism induced from
$ \pi_{(\wp_{(k\tau)}\fr_\mu\fs_h)}: \tsR_{(\wp_{(k\tau)}\fr_\mu\fs_h)} \to
\tsR_{(\wp_{(k\tau)}\fr_\mu\fs_{h-1})}$.

Suppose that $\tZ_{({\wp}_{(k\tau)}\fr_\mu\fs_{h-1}),\Ga}$,  equivalently 
$\tZ^\dagger_{({\wp}_{(k\tau)}\fr_\mu\fs_{h-1}),\Ga}$,
by Lemma \ref{wp/ell-transform-ktauh}  (3) in $({\wp}_{(k\tau)}\fr_\mu\fs_{(h-1)})$,
 is not contained in 
$Z'_{\phi_{(k\tau)\mu h}}$ where $Z'_{\phi_{(k\tau)\mu h}}$ is the proper transform 
in $\tsR_{({\wp}_{(k\tau)}\fr_\mu\fs_{h-1})}$
 of the ${\wp}$-center $Z_{\phi_{(k\tau)\mu h}}$(of $\tsR_{({\wp}_{(k\tau)}\fr_{\mu-1})}$).

We then let $\tZ_{({\wp}_{(k\tau)}\fr_\mu\fs_h),\Ga}$,
 respectively $\tZ^\dagger_{({\wp}_{(k\tau)}\fr_\mu\fs_h),\Ga}$, be the
proper transform of $\tZ_{({\wp}_{(k\tau)}\fr_\mu\fs_{h-1}),\Ga}$,
 respectively $\tZ^\dagger_{({\wp}_{(k\tau)}\fr_\mu\fs_{h-1}),\Ga}$,
in $\tsV_{({\wp}_{(k\tau)}\fr_\mu\fs_h)}$.
As $\tZ^\dagger_{({\wp}_{(k\tau)}\fr_\mu\fs_h),\Ga}$ is closed in $\tZ_{({\wp}_{(k\tau)}\fr_\mu\fs_h),\Ga}$
and contains a Zariski open subset of $\tZ_{({\wp}_{(k\tau)}\fr_\mu\fs_h),\Ga}$, it is an irreducible
component of $\tZ_{({\wp}_{(k\tau)}\fr_\mu\fs_h),\Ga}$.

Further, consider any admissible affine chart $\fV$ of $\tsR_{({\wp}_{(k\tau)}\fr_\mu\fs_h)}$,
lying over a unique admissible affine chart $\fV'$ of $\tsR_{({\wp}_{(k\tau)}\fr_\mu\fs_{h-1})}$,
such that $\tZ_{({\wp}_{(k\tau)}\fr_\mu\fs_h),\Ga}\cap \fV \ne \emptyset$.
We set 
$$\tGa^\zero_\fV=\{y_\fV \mid y_\fV  
\hbox{ is the proper transform of some $y_{\fV'} \in \tGa^\zero_{\fV'}$}\};$$
$$\tGa^\one_\fV=\{y_\fV \mid y_\fV  
\hbox{ is the proper transform of some $y_{\fV'} \in \tGa^\one_{\fV'}$}\}.$$

We now prove  Lemma \ref{wp/ell-transform-ktauh}  (1), (2) and (3) in $({\wp}_{(k\tau)}\fr_\mu\fs_h)$.

Lemma \ref{wp/ell-transform-ktauh}  (1) in $({\wp}_{(k\tau)}\fr_\mu\fs_h)$ 
follows  from Lemma \ref{wp/ell-transform-ktauh} (1) in $({\wp}_{(k\tau)}\fr_\mu\fs_{h-1})$
because $\tZ_{({\wp}_{(k\tau)}\fr_\mu\fs_h),\Ga}$  is the proper transform of 
$\tZ_{({\wp}_{(k\tau)}\fr_\mu\fs_{h-1}),\Ga}$. 

By construction, we have that 
$\tZ^\dagger_{({\wp}_{(k\tau)}\fr_\mu\fs_h),\Ga} \to \tZ^\dagger_{({\wp}_{(k\tau)}\fr_\mu\fs_{h-1}),\Ga} \to Z_\Ga$
is birational.   
This proves Lemma \ref{wp/ell-transform-ktauh}  (2) in $({\wp}_{(k\tau)}\fr_\mu\fs_h)$.

To show Lemma \ref{wp/ell-transform-ktauh}  (3) in $({\wp}_{(k\tau)}\fr_\mu\fs_h)$, 
we fix any $y \in \var_\fV$. It suffices to show that
if $\tZ^\dagger_{({\wp}_{(k\tau)}\fr_\mu\fs_h),\Ga}\cap \fV
\subset (y=0)$, then $\tZ_{({\wp}_{(k\tau)}\fr_\mu\fs_h),\Ga}\cap \fV
\subset (y=0)$.  By construction, $y \ne \zeta_{\fV,(k\tau)\mu h}$, the exceptional variable
in $\var_\fV$ corresponding to the $\wp$-set $\phi_{(k\tau)\mu h}$. Hence, $y$ is the proper transform
of some $y' \in \var_{\fV'}$. Then, by taking the images under the morphism
$\rho_{(\wp_{(k\tau)}\fr_\mu\fs_h)}$ of \eqref{rho-wpktaumuh}, we obtain
$\tZ^\dagger_{({\wp}_{(k\tau)}\fr_\mu\fs_{h-1}),\Ga}\cap \fV'
\subset (y'=0)$,  hence, $\tZ_{({\wp}_{(k\tau)}\fr_\mu\fs_{h-1}),\Ga}\cap \fV'
\subset (y'=0)$ by the inductive assumption.
 Then,  as $\tZ_{({\wp}_{(k\tau)}\fr_\mu\fs_{h}),\Ga}$ 
 is the proper transform of $\tZ_{({\wp}_{(k\tau)}\fr_\mu\fs_{h-1}),\Ga}$,
 we obtain $\tZ_{({\wp}_{(k\tau)}\fs_{h}),\Ga}\cap \fV \subset (y=0)$.
  The last statement  of Lemma \ref{wp/ell-transform-ktauh}  (3)  follows from the above because
 $\tZ_{\phi_{(k\tau)\mu(h+1)}}\cap \fV=(y_0=y_1=0)$ for some $y_0, y_1 \in \var_\fV$.

\smallskip

We now suppose that $\tZ_{({\wp}_{(k\tau)}\fr_\mu\fs_{h-1}),\Ga}$, or equivalently 
$\tZ^\dagger_{({\wp}_{(k\tau)}\fr_\mu\fs_{h-1}),\Ga}$,
 by  Lemma \ref{wp/ell-transform-ktauh}  (3) in $({\wp}_{(k\tau)}\fr_\mu\fs_{(h-1)})$,
 is contained in the proper transform  $Z'_{\phi_{(k\tau)\mu h}}$ 
 of $Z_{\phi_{(k\tau)\mu h}}$. 

Consider  any admissible affine chart $\fV$ of $\tsR_{({\wp}_{(k\tau)}\fr_\mu\fs_h)}$,
lying over a unique admissible affine chart $\fV'$ of $\tsR_{({\wp}_{(k\tau)}\fr_\mu\fs_{h-1})}$,
such that $\tZ_{({\wp}_{(k\tau)}\fr_\mu\fs_{h-1}),\Ga}\cap \fV' \ne \emptyset$.

We let $\phi'_{(k\tau)\mu h}$
 be the proper transform in  the chart $\fV'$ of the $\wp$-set $\phi_{(k\tau)\mu h}$. Then,  
 $\phi'_{(k\tau)\mu h}$ consists of two  variables such that
 $$\psi'_{(k\tau)\mu h}=\{y'_0,  y'_1\} \subset  \var_{\fV'}^\vee.$$ 
In addition, we let $\zeta_\fV \in \var_\fV$  be 
 such that  $E_{({\wp}_{(k\tau)}\fr_\mu\fs_{h})} \cap \fV=(\zeta_\fV =0)$ 
 where $E_{({\wp}_{(k\tau)}\fr_\mu\fs_{h})}$
 is the exceptional divisor
of the blowup $\tsR_{({\wp}_{(k\tau)}\fr_\mu\fs_{h})} \to \tsR_{({\wp}_{(k\tau)}\fr_\mu\fs_{h-1})}$. 
  Without loss of generality, we may assume that $y'_0$ 
corresponds to the exceptional variable $\zeta_\fV$ on the chart $\fV$.
 We then let $y_1 \in \var_\fV$ be 
 the proper transform of $y'_1$. 

Now, we observe that 
 $$\phi'_{(k\tau)\mu h}\subset \tGa^\zero_{\fV'}$$
because $\tZ_{({\wp}_{(k\tau)}\fr_\mu\fs_{h-1}),\Ga}$ is contained in the proper transform
$Z'_{\phi_{(k\tau)\mu h}}$.

We set, 
\begin{equation}\label{Gazero-wp-contained-in-bar} 
\overline{\Ga}^\zero_\fV= \{\zeta_\fV, \; y_\fV \mid y_\fV  
\hbox{ is the proper transform of some $ y_{\fV'} \in \tGa^\zero_{\fV'} \-\phi'_{(k\tau)\mu h}$}\},
\end{equation}
\begin{equation}\label{Gaone-wp-contained-in-bar}
 \overline{\Ga}^\one_{\fV}=\{\ y_\fV \mid y_\fV  
\hbox{ is the proper transform of some $y_{\fV'} \in \tGa^\one_{\fV'}$} \}.
\end{equation}

Consider the scheme-theoretic pre-image  
$\rho_{({\wp}_{(k\tau)}\fr_\mu\fs_h)}^{-1}(\tZ_{({\wp}_{(k\tau)}\fr_\mu\fs_{h-1}),\Ga})$.

 Note that scheme-theoretically, we have
$$\rho_{({\wp}_{(k\tau)}\fr_\mu\fs_h)}^{-1}(\tZ_{({\wp}_{(k\tau)}\fr_\mu\fs_{h-1}),\Ga}) \cap \fV=
\pi_{({\wp}_{(k\tau)}\fr_\mu\fs_h)}^{-1}(\tZ_{({\wp}_{(k\tau)}\fr_\mu\fs_{h-1}),\Ga})
 \cap \tsV_{({\wp}_{(k\tau)}\fr_\mu\fs_h)} \cap \fV .$$
Applying Lemma \ref{wp/ell-transform-ktauh} (1)  in $({\wp}_{(k\tau)}\fr_\mu\fs_{h-1})$ to 
$\tZ_{({\wp}_{(k\tau)}\fr_\mu\fs_{h-1}),\Ga}$
and $\rho_{({\wp}_{(k\tau)}\fr_\mu\fs_h)}^{-1}(\tZ_{({\wp}_{(k\tau)}\fr_\mu\fs_{h-1}),\Ga}) $, 
and applying Proposition \ref{equas-wp/ell-kmuh} to 
$\tsV_{({\wp}_{(k\tau)}\fr_\mu\fs_h)} \cap \fV $, we obtain that  the pre-image 
$\rho_{({\wp}_{(k\tau)}\fr_\mu\fs_h)}^{-1}(\tZ_{({\wp}_{(k\tau)}\fr_\mu\fs_{h-1}),\Ga}) \cap \fV$,
as a closed subscheme of $\fV$,  is defined by 
\begin{equation}\label{wp-pre-image-defined-by}
y_\fV \in \overline{\Ga}^\zero_\fV; \;\;  y_\fV-1, \; y_\fV \in \overline{\Ga}^\one_{\fV};\;\;
\cB^\gov_\fV; \;\; \cB^\frb_{\fV}; \;\; L_{\fV,  \sF}.
\end{equation} 
(Observe here that $\zeta_\fV \in  \overline{\Ga}^\zero_\fV$.)

Thus, by setting 
$$\hbox{ $y_\fV=0$ for all $y_\fV \in \overline{\Ga}^\zero_\fV$ and
 $y_\fV=1$ for all $y_\fV \in \overline{\Ga}^\one_{\fV}$} $$ in $\cB^\gov_\fV, 
 \cB^\frb_{\fV}, L_{\fV,  \sF}$ of the above,
 we obtain 
 \begin{equation}\label{vs-lin-xi-pre} \tilde{\cB}^\gov_\fV, 
  \tilde\cB^\frb_{\fV}, \tilde{L}_{\fV,   \sF}.
 \end{equation}
 Note that for any $\bF \in  \sF$,
  if a term of $L_{\fV, F}$ contains $y_1 \in \var_\fV$, 
   then it contains $\zeta_\fV y_1$, hence $\tilde L_{\fV, F}$ does not contain $y_1$.
  We keep those equations of \eqref{vs-lin-xi-pre} such that they contain the variable $y_1 \in \var_\fV$ 
   and obtain
 \begin{equation}\label{vs-lin-xi}
 \hat{\cB}^\gov_\fV, 
 \hat\cB^\frb_{\fV}, 
 \end{equation}
 viewed as a  system of equations in $y_1$.
Then, by Proposition \ref{equas-wp/ell-kmuh}  (1), 2), and (4)
(see also Corollary \ref{linear-or-vanish}),
\eqref{vs-lin-xi} is  a  {\it linear} system of equations in $y_1$.
(We point out that $y_1$ here can correspond to either $y_0$ or $y_1$ 
as in Proposition \ref{equas-wp/ell-kmuh}.)
Furthermore, one sees that 
 $$(\rho_{({\wp}_{(k\tau)}\fr_\mu\fs_h)}^{-1}(\tZ_{({\wp}_{(k\tau)}\fr_\mu\fs_{h-1}),\Ga})\cap \fV)/
 (\tZ_{({\wp}_{(k\tau)}\fr_\mu\fs_{h-1}),\Ga}\cap \fV')$$
 is defined by the linear system \eqref{vs-lin-xi}.

There are the following two cases for \eqref{vs-lin-xi}:

$(\star a)$  the ranks of the linear system  $\eqref{vs-lin-xi}$ equal one
at  general points of  $\tZ^\dagger_{({\wp}_{(k\tau)}\fr_\mu\fs_{h-1}),\Ga}$.

$(\star b)$  the ranks of the linear system  $\eqref{vs-lin-xi}$ equal zero
at general points of  $\tZ^\dagger_{({\wp}_{(k\tau)}\fr_\mu\fs_{h-1}),\Ga}$, hence at all points
of  $\tZ^\dagger_{({\wp}_{(k\tau)}\fr_\mu\fs_{h-1}),\Ga}$.

\smallskip\noindent
{\it 
Proof of Lemma  \ref{wp/ell-transform-ktauh}   in $({\wp}_{(k\tau)}\fr_\mu \fs_h)$ under the condition $(\star a)$.}
\smallskip

By the condition $(\star a)$,
 there exists a Zariski open subset $\tZ^{\dagger\circ}_{({\wp}_{(k\tau)}\fr_\mu\fs_{h-1}),\Ga}$ 
of $\tZ^\dagger_{({\wp}_{(k\tau)}\fr_\mu\fs_{h-1}),\Ga}$ such that the rank of the linear system
  $\eqref{vs-lin-xi}$ equals one at any point of  $\tZ^{\dagger\circ}_{({\wp}_{(k\tau)}\fr_\mu\fs_{h-1}),\Ga}$. 
By solving  $y_1$ from
 the linear system \eqref{vs-lin-xi} over $\tZ^{\dagger\circ}_{({\wp}_{(k\tau)}\fr_\mu\fs_{h-1}),\Ga}$, we obtain 
that the induced morphism
$$
\rho_{({\wp}_{(k\tau)}\fr_\mu\fs_h)}^{-1}(\tZ^{\dagger\circ}_{({\wp}_{(k\tau)}\fr_\mu\fs_{h-1}),\Ga})  \lra \tZ^\circ_{({\wp}_{(k\tau)}\fr_\mu\fs_{h-1}),\Ga}$$ is an isomorphism. 

Suppose  $y_1$ is identically zero along 
$\rho_{({\wp}_{(k\tau)}\fr_\mu\fs_h)}^{-1}(\tZ^{\dagger\circ}_{({\wp}_{(k\tau)}\fr_\mu\fs_{h-1}),\Ga})$.
We then set,  
\begin{equation}\label{Gazero-ktah-contained-in-a} 
\tGa^\zero_\fV=\{y_1 \} \cup \overline{\Ga}^\zero_\fV
\end{equation}
where $\overline{\Ga}^\zero_\fV$ is as in \eqref{Gazero-wp-contained-in-bar}.
In this case, we let
\begin{equation}\label{ZktahGa-with-xi}
\tZ_{({\wp}_{(k\tau)}\fr_\mu\fs_{h}),\Ga}=\rho_{({\wp}_{(k\tau)}\fr_\mu\fs_h)}^{-1}(\tZ_{({\wp}_{(k\tau)}\fr_\mu\fs_{h-1}),\Ga})\cap D_{y_1}
\end{equation}
scheme-theoretically, where $D_{y_1}$
 is the closure of $(y_1=0)$ in $\tsR_{({\wp}_{(k\tau)}\fr_\mu\fs_h)}$.
Similar to the argument as in the proof Lemma \ref{vt-transform-k},
 $D_{y_1}$ does not depend on the choice of the chart $\fV$.
More precisely, note that  $D_{y_1}$  is  the proper transform of one of 
 the two divisors in the $\wp$-set. We let
  $D_{y_0}$ be the proper transforms
of the other divisor in the $\wp$-set.
They $D_{y_1}$ and $D_{y_0}$ are disjoint. 
Hence, only one of them can contain 
$\rho_{({\wp}_{(k\tau)}\fr_\mu\fs_h)}^{-1}(\tZ^{\dagger\circ}_{({\wp}_{(k\tau)}\fr_\mu\fs_{h-1}),\Ga})$.
This implies that  $D_{y_1}$ comes with
$\rho_{({\wp}_{(k\tau)}\fr_\mu\fs_h)}^{-1}(\tZ^{\dagger\circ}_{({\wp}_{(k\tau)}\fr_\mu\fs_{h-1}),\Ga})$ in such a case,
hence does not depend on the choice of the chart $\fV$.

Suppose $y_1$ is not identically zero along 
$\rho_{({\wp}_{(k\tau)}\fr_\mu\fs_h)}^{-1}(\tZ^{\dagger\circ}_{({\wp}_{(k\tau)}\fr_\mu\fs_{h-1}),\Ga})$.
We then set,  
\begin{equation}\label{Gazero-ktah-contained-in-a-2} 
\tGa^\zero_\fV= \overline{\Ga}^\zero_\fV
\end{equation}
where $\overline{\Ga}^\zero_\fV$ is as in \eqref{Gazero-wp-contained-in-bar} .
In this case, we let
\begin{equation}\label{ZktahGa-without-xi}
\tZ_{({\wp}_{(k\tau)}\fr_\mu\fs_{h}),\Ga}=\rho_{({\wp}_{(k\tau)}\fr_\mu\fs_h)}^{-1}(\tZ_{({\wp}_{(k\tau)}\fr_\mu\fs_{h-1}),\Ga}).
\end{equation}
We always set (under the condition $(\star a)$)
\begin{equation}\label{Gaone-ktah-contained-in-a} \tGa^\one_{\fV}= \overline{\Ga}^\one_{\fV}
\end{equation}
where $\overline{\Ga}^\one_{\fV}$ is as in  \eqref{Gaone-wp-contained-in-bar}.

In each case, we have 
$$\rho_{({\wp}_{(k\tau)}\fr_\mu\fs_h)}^{-1}(\tZ^{\dagger\circ}_{({\wp}_{(k\tau)}\fr_\mu\fs_{h-1}),\Ga})  
\subset \tZ_{({\wp}_{(k\tau)}\fr_\mu\fs_h),\Ga},$$ and we let $\tZ^\dagger_{({\wp}_{(k\tau)}\fr_\mu\fs_h),\Ga}$ be the closure of 
$\rho_{({\wp}_{(k\tau)}\fr_\mu\fs_h)}^{-1}(\tZ^{\dagger\circ}_{({\wp}_{(k\tau)}\fr_\mu\fs_{h-1}),\Ga}) $ in $\tZ_{({\wp}_{(k\tau)}\fr_\mu\fs_h),\Ga}$.
It is an irreducible component of $\tZ_{({\wp}_{(k\tau)}\fr_\mu\fs_h),\Ga}$ because
$\tZ^\dagger_{({\wp}_{(k\tau)}\fr_\mu\fs_h),\Ga}$ is closed in $\tZ_{({\wp}_{(k\tau)}\fr_\mu\fs_h),\Ga}$ and contains
the Zariski open subset $\rho_{({\wp}_{(k\tau)}\fr_\mu\fs_h)}^{-1}(\tZ^{\dagger\circ}_{({\wp}_{(k\tau)}\fr_\mu\fs_{h-1}),\Ga}) $
of $\tZ_{({\wp}_{(k\tau)}\fr_\mu\fs_h),\Ga}$.
Then, it follows that the composition
 $$\tZ^\dagger_{({\wp}_{(k\tau)}\fr_\mu\fs_h),\Ga} \lra \tZ^\dagger_{({\wp}_{(k\tau)}\fr_\mu\fs_{h-1}),\Ga}
 \lra Z_\Ga$$ is birational.  
 This proves Lemma \ref{wp/ell-transform-ktauh} (2) in $({\wp}_{(k\tau)}\fs_h)$.


In each case of \eqref{Gazero-ktah-contained-in-a} and \eqref{Gazero-ktah-contained-in-a-2}, 
by the paragraph of \eqref{wp-pre-image-defined-by},
we have that $\tZ_{({\wp}_{(k\tau)}\fr_\mu\fs_h),\Ga}\cap \fV$
 as a closed subscheme of $\fV$ is defined by the equations as stated in the Lemma. 
 This proves Lemma \ref{wp/ell-transform-ktauh} (1) in $({\wp}_{(k\tau)}\fr_\mu\fs_h)$.

It remains to prove Lemma \ref{wp/ell-transform-ktauh} (3) in $({\wp}_{(k\tau)}\fr_\mu\fs_h)$.

Fix any $y \in \var_\fV$, it suffices to show that
if $\tZ^\dagger_{({\wp}_{(k\tau)}\fr_\mu\fs_{h}),\Ga}\cap \fV
\subset (y=0)$, then $\tZ_{({\wp}_{(k\tau)}\fr_\mu\fs_{h}),\Ga}\cap \fV
\subset (y=0)$.  
If $y \ne \zeta_\fV,  y_1$, then $y$ is the proper transform of some variable $y' \in \var_{\fV'}$.
Hence, by taking the images under $\rho_{({\wp}_{(k\tau)}\fr_\mu\fs_h)}$,
 we obtain $\tZ^\dagger_{({\wp}_{(k\tau)}\fr_\mu\fs_{h-1}),\Ga}\cap \fV'
\subset (y'=0)$, and then,  by Lemma \ref{wp/ell-transform-ktauh} (3) in (${\wp}_{(k\tau)}\fr_\mu\fs_{h-1}$),
$\tZ_{({\wp}_{(k\tau)}\fr_\mu\fs_{h-1}),\Ga}\cap \fV'
\subset (y'=0)$,  thus $y' \in \tGa^\zero_{\fV'}$ by the maximality of $\tGa^\zero_{\fV'}$. 
 Therefore, 
$\tZ_{({\wp}_{(k\tau)}\fr_\mu\fs_{h}),\Ga}\cap \fV \subset ( y=0)$,  
by (the already-proved) Lemma \ref{wp/ell-transform-ktauh} (1) in $(\wp_{(k\tau)}\fr_\mu\fs_h)$.
Next, suppose $y = y_1$ (if it occurs). Then, by construction,
 $\tZ_{({\wp}_{(k\tau)}\fs_{h}),\Ga}\cap \fV \subset (y=0)$. 
Finally, we let $y =\zeta_\fV$. 
 Again, by construction, $\tZ_{({\wp}_{(k\tau)}\fr_\mu\fs_{h}),\Ga}\cap \fV \subset (\zeta_\fV=0)$. 
As in the previous case, the last statement  of Lemma \ref{wp/ell-transform-ktauh}  (3)  follows from the above.

\smallskip\noindent
{\it  Proof of Lemma \ref{wp/ell-transform-ktauh} in $({\wp}_{(k\tau)}\fr_\mu\fs_h)$
 under the condition $(\star b)$.}
\smallskip

Under the condition $(\star b)$, we have that $$ 
\rho_{({\wp}_{(k\tau)}\fr_\mu\fs_h)}^{-1}(\tZ^\dagger_{({\wp}_{(k\tau)}\fr_\mu\fs_{h-1}),\Ga})
 \lra \tZ^\dagger_{({\wp}_{(k\tau)}\fr_\mu\fs_{h-1}),\Ga}$$
can be canonically identified with the trivial $\PP_{[\xi_0,\xi_1]}$-bundle:
$$\rho_{({\wp}_{(k\tau)}\fr_\mu\fs_h)}^{-1}(\tZ^\dagger_{({\wp}_{(k\tau)}\fr_\mu\fs_{h-1}),\Ga}) 
= \tZ^\dagger_{({\wp}_{(k\tau)}\fr_\mu\fs_{h-1}),\Ga}
\times \PP_{[\xi_0,\xi_1]}.$$
In this case, we let $\bp=[1,1] \in \PP_{[\xi_0,\xi_1]}$, and define 
$$\tZ_{({\wp}_{(k\tau)}\fr_\mu\fs_h),\Ga} := 
\rho_{({\wp}_{(k\tau)}\fr_\mu\fs_h)}^{-1}(\tZ_{({\wp}_{(k\tau)}\fr_\mu\fs_{h-1}),\Ga}) 
\times_{\tZ_{({\wp}_{(k\tau)}\fr_\mu\fs_{h-1}),\Ga}} \bp$$
 $$\tZ^\dagger_{({\wp}_{(k\tau)}\fr_\mu\fs_h),\Ga}:= 
 \rho_{({\wp}_{(k\tau)}\fr_\mu\fs_h)}^{-1}(\tZ^\dagger_{({\wp}_{(k\tau)}\fr_\mu\fs_{h-1}),\Ga})
 \times_{\tZ^\dagger_{({\wp}_{(k\tau)}\fr_\mu\fs_{h-1}),\Ga}} \bp.$$
 
 The induced morphism $\tZ^\dagger_{({\wp}_{(k\tau)}\fr_\mu\fs_h),\Ga} \lra \tZ^\dagger_{({\wp}_{(k\tau)}\fr_\mu\fs_{h-1}),\Ga}$
 is an isomorphism. 
 Again, one sees that $\tZ^\dagger_{({\wp}_{(k\tau)}\fr_\mu\fs_h),\Ga}$ is 
 an irreducible component of $\tZ_{({\wp}_{(k\tau)}\fr_\mu\fs_h),\Ga}$. Therefore,
 $$\tZ^\dagger_{({\wp}_{(k\tau)}\fr_\mu\fs_h),\Ga} \lra
 \tZ^\dagger_{({\wp}_{(k\tau)}\fr_\mu\fs_{h-1}),\Ga} \lra  Z_\Ga$$ is birational.
 This proves Lemma \ref{wp/ell-transform-ktauh} (2) in $({\wp}_{(k\tau)}\fr_\mu\fs_h)$.

Further, under the condition $(\star b)$, we set 
 \begin{eqnarray}
 \tGa^\zero_\fV= \overline{\Ga}^\zero_\fV,  \;\;\;\;
\tGa^\one_\fV=\{y_1\} \cup \overline{\Ga}^\one_{\fV}. \nonumber \end{eqnarray}

Then, by 
the paragraph of \eqref{wp-pre-image-defined-by},
we have that $\tZ_{({\wp}_{(k\tau)}\fr_\mu\fs_h),\Ga}\cap \fV$,
 as a closed subscheme of $\fV$, is defined by the equations as stated in the Lemma. 
 This proves Lemma \ref{wp/ell-transform-ktauh} (1) in $({\wp}_{(k\tau)}\fr_\mu\fs_h)$.

 It remains to prove Lemma \ref{wp/ell-transform-ktauh} (3) in $({\wp}_{(k\tau)}\fr_\mu\fs_h)$.

Fix any $y \in \var_\fV$, it suffices to show that
if $\tZ^\dagger_{({\wp}_{(k\tau)}\fs_{h}),\Ga}\cap \fV
\subset (y=0)$, then $\tZ_{({\wp}_{(k\tau)}\fs_{h}),\Ga}\cap \fV
\subset (y=0)$.  By construction, $y \ne y_1$.
Then, the corresponding proof of Lemma \ref{wp/ell-transform-ktauh} (3) in $({\wp}_{(k\tau)}\fr_\mu\fs_h)$
under the condition $(\star a)$ goes through here without change. As earlier, 
the last statement  of Lemma \ref{wp/ell-transform-ktauh}  (3)  follows from the above. 
This proves Lemma \ref{wp/ell-transform-ktauh} (3) in $({\wp}_{(k\tau)}\fr_\mu\fs_h)$
 under the condition $(\star b)$.
 
 \medskip
 
 \centerline{$\bullet$ $\ell$-transform}
 
 \medskip
 
 Now, we assume that 
 $\tsR_{({\wp}_{(k\tau)}\fr_\mu\fs_{h-1})}=\tsR_{\wp_k}$ and
 $\tsR_{({\wp}_{(k\tau)}\fr_\mu\fs_h)}=\tsR_{\ell_k}$.
 By Corollary \ref{ell-isom},
 $\rho_{\ell_k, \wp_k}: \tsR_{\ell_k} \to \tsR_{\wp_k}$ is an isomorphism.
 In this case, we let
 $$ \tZ_{\ell_k, \Ga} =\rho_{\ell_k, \wp_k}^{-1}(\tZ_{\wp_k, \Ga}),$$
 $$ \tZ^\dagger_{\ell_k, \Ga} =\rho_{\ell_k, \wp_k}^{-1}(\tZ^\dagger_{\wp_k, \Ga}).$$

 
 
 


 Consider any admissible affine chart $\fV$ of $\tsR_{\ell_k}$,
lying over a unique admissible affine chart $\fV'$ of $\tsR_{\wp_k}$,
such that $\tZ_{\wp_k,\Ga}\cap \fV' \ne \emptyset$.

First, we suppose that $\tZ_{{ \wp_k},\Ga} \cap \fV'$,  
or equivalently $\tZ^\dagger_{{ \wp_k},\Ga} \cap \fV'$, 
by Lemma \ref{wp/ell-transform-ktauh} (3) in $(\wp_k)$,
 is not contained in   the $\ell$-center $Z_{\chi_k} \cap \fV'$.

We set 
$$\tGa^\zero_\fV=\{y_\fV \mid y_\fV  \hbox{ is the proper transform of some $y_{\fV'} \in \tGa^\zero_{\fV'}$}\};$$
$$\tGa^\one_\fV=\{y_\fV \mid y_\fV  \hbox{ is the proper transform of some $y_{\fV'} \in \tGa^\one_{\fV'}$}\}.$$

Then, Lemma  \ref{wp/ell-transform-ktauh} (1), (2) and (3) follow
from the same  proofs  for the corresponding
cases of $\wp$-blowups.
 We avoid repetation.

 We now suppose that $\tZ_{{\wp_k},\Ga}\cap \fV'$, or equivalently
 $\tZ^\dagger_{{ \wp_k},\Ga} \cap \fV'$,
 by  Lemma \ref{wp/ell-transform-ktauh} (3) in $(\wp_k)$,
 is contained in  $Z_{\chi_k} \cap \fV'$ 
 
This case corresponds to the precious case under the condition $(\star a)$
where $y_1$ there corresponds to $y_{\fV, (\um, \uu_{F_k})}$ here,  and
$y_{\fV, (\um, \uu_{F_k})}$ is not
identically zero along $\tZ_{{ \ell_k},\Ga}$.
So, we follow the proof in that case.

Thus, we set, 
\begin{equation}\label{Gazero-ell-contained-in-bar} 
{\tGa}^\zero_\fV= \{\zeta_\fV, \; y_\fV \mid y_\fV  
\hbox{ is the proper transform of some $ y_{\fV'} \in \tGa^\zero_{\fV'} \-\chi_k$}\},
\end{equation}
\begin{equation}\label{Gaone-ell-contained-in-bar}
{\Ga}^\one_{\fV}= \{\ y_\fV \mid y_\fV  
\hbox{ is the proper transform of some $y_{\fV'} \in \tGa^\one_{\fV'}$} \}.
\end{equation}


Then, following the correponding proofs for the precious case 
{\it under the condition $(\star a)$
where $y_1$ is not identically zero,} 
Lemma  \ref{wp/ell-transform-ktauh} (1), (2) and (3) follow. 
But, we need to point out that here, 
 $ \zeta_\fV=\de_{\fV, (\um,\uu_{F_k})}$ is not a  variable in $\var_\fV$,
 but a free variable in $\var_\fV^\vee$.

Putting all together, this completes the proof of Lemma \ref{wp/ell-transform-ktauh}.    
\end{proof}

 We call $\tZ_{({\wp}_{(k\tau)}\fr_\mu\fs_{h}),\Ga}$ the $\wp$-transform of $Z_\Ga$ 
 in $\tsV_{({\wp}_{(k\tau)}\fr_\mu\fs_{h})}$ for any $(k\tau)\mu h \in \Index_{\Phi_k}$.
 We call $\tZ_{\ell_k}$ he $\wp$-transform of $Z_\Ga$ 
 in $\tsV_{\ell_k}$.

   We need the final case of Lemma 
\ref{wp/ell-transform-ktauh}. We set
$$\tZ_{\ell, \Ga}:=\tZ_{\ell_\up,\Ga}, \;\; \tZ^\dagger_{\ell, \Ga}
:=\tZ^\dagger_{ \ell_\up,\Ga}.$$

\begin{cor} \label{ell-transform-up} 
Fix any admissible smooth affine chart $\fV$ of $\tsR_\ell$ as described in
Lemma \ref{wp/ell-transform-ktauh}
such that $\tZ_{\ell,\Ga} \cap \fV \ne \emptyset.$
Then, $\tZ_{\ell,\Ga} \cap \fV$,  as a closed subscheme of $\fV$,
is defined by the following relations
\begin{eqnarray} 
\;\;\;\;\; y , \; \; \forall \;\;  y \in   \tGa^\zero_\fV { \subset \var^\vee_\fV}; 
\;\;\;  y -1, \; \; \forall \;\;  y \in  \tGa^\one_\fV; \nonumber \\
\cB_\fV^\gov, \; 
\cB^\frb_\fV, \; L_{\sF, \fV}.
\end{eqnarray}
Furthermore, the induced morphism $\tZ^\dagger_{\ell,\Ga} \to Z_\Ga$ is birational.
\end{cor}

\section{The Main Theorem on the Final Scheme $\tsV_\ell$}\label{main-statement}

{\it 
 Let $p$ be an arbitrarily fixed prime number.
 Let $\mathbb F$ be either $\QQ$ or a finite field with $p$ elements.
In this entire section, 
every scheme  is defined over $\ZZ$, consequently,
 is defined over $\mathbb F$, and is considered as a scheme over 
 the perfect field $\mathbb F$.
}

\smallskip

\subsection{Preparing for applying the Jacobian criterion} $\ $

  Take any $(k\tau \mu h) \in \Index_{\Phi_k}$ (cf. \ref{indexing-Phi}). 
  Consider  the $\wp$-blowup at $(k\tau \mu h)$
  $$\pi: \tsR_{(\wp_{(k\tau)}\fr_\mu\fs_h)} \lra \tsR_{(\wp_{(k\tau)}\fr_\mu\fs_{h-1})}.$$
   Fix any admissible smooth affine
 chart $\fV$ of  $\tsR_{(\wp_{(k\tau)}\fr_\mu\fs_h)} $ 
   and we let $\fV'$ be the (unique) admissible smooth affine
 chart of $\tsR_{(\wp_{(k\tau)}\fr_\mu\fs_{h-1})}$ 
   such that $\fV$ lies over $\fV'$. 
    Then, we can assume that the induced morphism $\pi^{-1}(\fV') \lra \fV'$
   corresponds to the ideal of the form
   $$\langle y_0', y_1' \rangle,$$
   where $y_0'$ is a variable in $T^+_{\fV', (k\tau)}$, 
   $y_1'$ is a variable in $T^-_{\fV', (k\tau)}$, and $B_{\fV',(k\tau)}=T^+_{\fV', (k\tau)}-T^-_{\fV', (k\tau)}.$
  We let $\PP_{[\xi_0,\xi_1]}$ be the corresponding factor projective space
  such that $(\xi_0,\xi_1)$ corresponds to  $(y_0', y_1')$.
  
  Observe here that the set of variables in $T^\pm_{\fV', (k\tau)}$ corresponds a subset of
  divisors associated to $T^\pm_{(k\tau)}$, hence possesses a naturally induced total order.
  
  Given any point $\bz$ on a chart, we say a variable is blowup-relevant at $\bz$
  if it can appear in the local blowup ideal $\langle y_0', y_1' \rangle$ as above 
  such that the corresponding  blowup center contains $\bz$. 
For example, a variable is not blowup-relevant at $\bz$ if it does not vanish at $\bz$.
   
   \begin{lemma}\label{keepLargest} 
      Let $\langle y_0', y_1' \rangle$ be the local blowup ideal as in the above such that 
   $y_0'$ is the largest blowup-relevant variable at the point $\bz$
    among all the variables of $T^+_{\fV', (k\tau)}$. Then,
    the chart $\fV$ containing the point $\bz$ can be chosen to lie over $(\xi_1 \equiv 1)$ so that
    the proper transform $y_0$ of $y_0'$ belongs to the  term $T^+_{\fV, (k\tau)}$.
   \end{lemma}
   \begin{proof} Using the notation before the statement of the lemma, 
   we can write,
   $$B_{\fV'} = a' y_0' - b' y_1'$$ 
   where $a'$ and $b'$ are some monomials.
   
   Suppose $\fV$  lies over $(\xi_0 \equiv 1)$. Then, by taking proper transforms, we obtain
    $$B_{\fV} = a  - \xi_1 b $$ 
    where $a$ and $b$ are some monomials.
    Because  $y_0'$ is the largest variable in the term $T^+_{\fV, (k\tau)}$, 
    by the order of the $\wp$-blowups, we have that $B_{\fV}$ 
    terminates. Hence, $\xi_1$ is invertible along
    $\tsV_{(\wp_{(k\tau)}\fr_\mu\fs_h)} \cap \fV$. Thus, 
    by shrinking the chart if necessary, we can switch to the chart 
    $(\xi_1 \equiv 1)$.
    
    Now, let $\fV$  lie over $(\xi_1 \equiv 1)$. Then, we obtain
    $$B_{\fV} = a \xi_0 -  b, $$ 
     where $a$ and $b$ are some monomials, and
   $\xi_0=y_0$ is the proper transform of $y_0'$. Hence, the statement follows.
   \end{proof}

We aim to show that the scheme $\tsV_\ell$ is smooth. The question is local.
 In the sequel, we will focus on a fixed closed point $\bz \in \tsV_\ell$ throughout.

Fix any preferred admissible chart of $\tsR_\ell$ containing $\bz$.
We let $\fV_{[0]}$ be the unique admissible affine chart of $\sR$ 
such that $\fV$ lies over $\fV_{[0]}$ and 
we let $\bz_0 \in \sV_{[0]}$ be the image point of $\bz$.




 \begin{defn}\label{pleasant-v} 
 Consider the ordered set of blocks of relations 
\begin{equation}\label{good-eq-list}\fG=\{\fG_{F_1} < \cdots <\fG_{F_\up}\}.
\end{equation}
 Fix and consider any variable $y \in \var_\fV$ that appears in some relation
 of a block $\fG_{F_k}$ of $\fG$ in the above, for some $k\in[\up]$. 
 We say that $y$ is pleasant if   $y$ does not appear in any relation of  any earlier block $\fG_{F_i}$ with $i<k$.
 \end{defn}

The main  statement that we will prove is

  \begin{lemma}\label{max-minor} 
  Fix any closed point $\bz \in \tsV_\ell$. Consider any $\bF \in \sF $.
  Then, there exists a preferred  admissible
  chart $\fV$ of $\tsR_\ell$ 
  containing the point $\bz$ such that the Jacobian $J(\fG_{\fV,F})$ 
admits a maximal minor $J^*(\fG_{\fV,F})$ such that
  it is an invertible square matrix at $\bz$,
   and all the variables that are used to compute $J^*(\fG_{\fV,F})$
 are pleasant with respect to the list
$\fG=\{\fG_{F_1} < \cdots <\fG_{F_\up}\}$.  
  \end{lemma}
  
To prove this lemma,
    consider the governing binomial relations
 $$B_{\fV_{[0]}, s}:  x_{\fV_{[0]},(\uu_{s}, \uv_{s})}x_{\fV_{[0]}, \uu_F}-
 x_{\fV_{[0]}, (\um, \uu_F)}x_{\fV_{[0]},\uu_{s}} x_{\fV_{[0]},\uv_{s}}$$
 for all $s \in S_F \- s_F$.
We let 
$$S_F^{\vr,\ori}=\{ s \in S_F \- s_F \mid  x_{\fV_{[0]},(\uu_{s}, \uv_{s})} (\bz_0) \ne 0\}$$
and 
$$S_F^{\vr,\inc}=\{ s \in S_F \- s_F \mid  x_{\fV_{[0]},(\uu_{s}, \uv_{s})} (\bz_0) = 0\}$$
The two subsets $S_F^{\vr,\ori}$ and $S_F^{\vr,\inc}$ depend on the point $\bz$.

We divide the proof into two parts.

\subsection{Applying the Jacobian criterion: Case $(\alpha)$}
\label{alpha} $\ $

\medskip\noindent  
  {\it Case $(\alpha)$.  First, we assume 
  $x_{\fV_{[0]}, (\um,\uu_F)}(\bz_0)\ne 0$. }
  
  As  $x_{\fV_{[0]}, (\um, \uu_F)}(\bz_0)\ne 0$, 
  we can assume that $\fV$ lies over $(x_{(\um, \uu_F)} \equiv 1)$.
 
 We can write $$S_F^{\vr,\ori}=\{s_1, \cdots, s_\ell\} \;\; \hbox{and}\;\; S_F^{\vr,\inc}=\{t_1, \cdots, t_q\}$$
for some integers $l$ and $q$ such that $l + q= |S_F|-1$.

Then, on the chart $\fV_{[0]}$, the set $\cB^{\vr,\ori}_F$ consists of the following relations
    \begin{equation}\label{cB-vr-ori-alpha}  
    B_{\fV_{[0]}, s_i}: \;\;  x_{\fV_{[0]},(\uu_{s_i}, \uv_{s_i})}x_{\fV_{[0]}, \uu_F}-
  x_{\fV_{[0]},\uu_{s_i}} x_{\fV_{[0]},\uv_{s_i}} , \;\; 1 \le i \le l.
\end{equation}
By the relation $L_{\fV_{[0]}, F}(\bz_0)=0$, we see that there must exist $s \in S_F\-s_F$ such that
$x_{\fV_{[0]},(\uu_{s}, \uv_{s})}(\bz_0) \ne 0$, thus,
  the set $\cB^{\vr,\ori}_F$ must not be empty. Hence, $l>0$. This observation 
  will be used later.

  The set $\cB^{\vr, \inc}_F$ consists of the following relations
    \begin{equation}\label{cB-vr-inc-alpha}  
    B_{\fV_{[0]}, t_i}: \;\;  x_{\fV_{[0]},(\uu_{t_i}, \uv_{t_i})}x_{\fV_{[0]}, \uu_F}-
  x_{\fV_{[0]},\uu_{t_i}} x_{\fV_{[0]},\uv_{t_i}}, \;\; 1 \le i \le q
\end{equation}
for some integer $q \ge 0$ with $q=0$ when $\cB^{\vr, \inc}_F =\emptyset$.

We treat the relations of \eqref{cB-vr-ori-alpha} first.

First observe that during any of $\wp$-blowups, if a variable $y$ acquires 
an exceptional parameter $\ve$, then we have $\ve <y$ by Definition \ref{order-phi}.

Consider $B_{s_i}$ for any fixed $i \in [l]$.

We suppose $B_s \in \cB^\gov_{F}$ is the smallest binomial relation.

Assume that $\tsR_{(\wp_{(k\tau)}\fr_\mu\fs_h)} \lra \tsR_{(\wp_{(k\tau)}\fr_\mu\fs_{h-1})}$ 
is the last (non-trivial) $\wp$-blowup that makes $B_{s}$
terminate in $\tsR_{(\wp_{(k\tau)}\fr_\mu\fs_h)}$ for some $(k\tau)\mu h \in \Index_{\Phi_k}$
(cf. \eqref{indexing-Phi}). We let $\hs=(\wp_{(k\tau)}\fr_\mu\fs_h)$
and $\hs'=(k\tau)\fr_\mu\fs_{h-1})$.
Then,  over some chart $\fV_{\hs'}$ of $\tsR_{\hs'}$, the last $\wp$-blowup must correspond
to $(x_{\fV_{\hs'},\uu_F}, y_{\fV_{\hs'}})$ when $x_{\fV_{[0]},(\uu_{s}, \uv_{s})}(\bz_0)\ne0$
or $(x_{\fV_{\hs'},(\uu_{s}, \uv_{s})}, y_{\fV_{\hs'}})$ when $x_{\fV_{[0]},(\uu_{s}, \uv_{s})}(\bz_0)=0$,
where $y_{\fV_{\hs'}}$ is a variable in the minus term of $B_{\fV_{\hs'}, s}$.
We apply Lemma \ref{keepLargest} to $B_{\fV_{\hs'}, s}$.
Then,  either we have the $\vp$-variable $x_{\fV_{\hs},\uu_F}$
terminates and belongs to $B_{\fV_\hs, s}$
(e.g., when $x_{\fV_{[0]},(\uu_{s}, \uv_{s})} (\bz_0) \ne 0$),  or,
 $x_{\fV_{\hs'},\uu_F}$  turns into an exceptional-variable
$\ve_{\fV, \uu_F}$ (e..g, in the case when $x_{\fV_{[0]},(\uu_{s}, \uv_{s})} (\bz_0) = 0$).
Further, $x_{\fV_{\hs},\uu_F}$ or $\ve_{\fV_\hs, \uu_F}$
also  appears in all the remaining binomials that are larger than $B_s$ (in this special
case, it is just all  the remaining binomials since $B_s$ is assumed to be smallest; 
we term it this way so that the same line of arguments can be reused later).

During the $\wp$-blowups with respect to $B_s$, after the variable $x_{\uu_F}$
terminates or becomes exceptional, any further $\wp$-blowup must correspond
 $(x_{\fV_{\hs'},(\uu_{s}, \uv_{s})}, y_{\fV_{\hs'}})$ when $x_{\fV_{[0]},(\uu_{s}, \uv_{s})}(\bz_0)=0$.
 But, such a $\wp$-blowup will not affect {\it the plus term} of $B_{\fV_\hs, t}$ with $t \ne s$. 

In any case, $x_{\fV_{\hs},\uu_F}$ or $\ve_{\fV, \uu_F}$, remains to be second largest variable,
second only the $\vr$-variables
in $B_{\fV_\hs, s_i}$ with $s_i >s$. 
 
 We then move on to the second smallest binomial relation of $\cB^\gov_F$ 
 and repeat all the above arguments, until 
it is the turn to start the process of $\wp$-blowups with respect to $B_{s_i}$.

Then,  because $x_{\fV_{\hs},\uu_F}$ or $\ve_{\fV, \uu_F}$, 
 remains to be the largest blowup-relevant variable in $B_{\fV_\hs, s_i}$
 (since $x_{\fV_{[0]},(\uu_{s_i}, \uv_{s_i})}(\bz_0)\ne0$), we can then
 apply Lemma \ref{keepLargest} to  $B_{\fV_\hs, s_i}$ to obtain that 
 that there exists a chart $\fV$
containing the point $\bz$  such that we have
\begin{equation}\label{cB-ori}  \nonumber 
    B_{\fV, s_i}: \;\; a_ix_{\fV_\vt,(\uu_{s_i}, \uv_{s_i})} y_{\fV,\uu_F} - c_i
\end{equation}
for some monomial $a_i$ and $c_i$,
where $y_{\fV,\uu_F}$ is either the $\vp$-variable $x_{\fV,\uu_F}$
or the proper transform of an exceptional-variable $\ve_{\uu_F}$. 


Hence,  by shrinking the charts if necessary, we conclude that 
 that there exists a chart $\fV$
containing the point $\bz$  such that we have
\begin{equation}\label{cB-vr-ori} 
   B_{\fV, s_i}: \;\; a_i x_{\fV_\vt,(\uu_{s_i}, \uv_{s_i})} y_{\fV,\uu_F} - c_i, \;\; \hbox{for all} \;\; i \in [l]
\end{equation}
where $y_{\fV,\uu_F}$ is either the $\vp$-variable $x_{\fV,\uu_F}$
or the proper transform of an exceptional-variable $\ve_{ \uu_F}$ such that all of these
relations  terminate at $\bz$.

Now consider $B_{\fV_\vt, t_i}$ with $i \in [q]$.


Because $x_{\fV_{[0]},(\uu_{t_i}, \uv_{t_i})}(\bz_0)=0$ and 
$x_{\fV_{[0]},(\uu_{t_i}, \uv_{t_i})}$ is the largest
variable in the plus term of $B_{\fV_{[0]}, t_i}$, 
we can  apply Lemma \ref{keepLargest} directly to  $B_{t_i}$ for all $i \in [q]$ to obtain that 
 there exists a chart $\fV$
containing the point $\bz$  such that we have
$$B_{\fV, t_i}: \;\;  b_i x_{\fV, (\uu_{t_i}, \uv_{t_i})} - d_i , i \in [q].$$
where $b_i$ and $d_i$ are some monomials for all $i \in [q]$.
 
Put all together, shrinking the charts if necessary, we conclude that 
there exists a chart $\fV$ of $\tsR^\circ_\ell$, containing the point $\bz$ such that we have
\begin{eqnarray}\label{cB-vrChart-final}  
   B_{\fV, s_i}: \;\; a_i x_{\fV,(\uu_{s_i}, \uv_{s_i})} y_{\fV,\uu_F} - c_i, \;\; i \in [l]  \\
  B_{\fV, t_i}: \;\;  b_i x_{\fV,(\uu_{t_i}, \uv_{t_i})} - d_i , \;\; i \in [q]. \;\;\;\;\;\;\;\;\nonumber
\end{eqnarray}

Further, because $x_{\fV_{[0]}, (\um,\uu_F)}(\bz_0)\ne 0$,
  the $\ell$-blowups do not affect the (unique) chart $\fV_\vt$ 
  of $\tsR_\vt$ which $\fV$ lies over. Hence, we have
\begin{equation}\label{for-1column}
L_{\fV, F} =\sgn(s_F) + \sum_{i=1}^l \sgn (s_i) x_{\fV, (\uu_{s_i},\uv_{s_i})}
  + \sum_{i=1}^q \sgn (s_i) e_i x_{\fV, (\uu_{t_i},\uv_{t_i})}
  \end{equation}
  where $e_i$ are monomials in exceptional variables such that $e_i(\bz)=0$,
  for all $i \in [q]$.

 {\it As the chart $\fV$ is fixed and is clear from the context, in the sequel, to save space,
we will selectively drop some subindex $``\ \fV \ "$. For instance, we may write
$y_{\uu_F}$ for $y_{\fV,\uu_F}$, $x_{(\uu_{s_1},\uv_{s_1})}$ for $x_{\fV, (\uu_{s_1},\uv_{s_1})}$, etc.
A confusion is unlikely.
}

\smallskip


We  introduce the following maximal minor of the Jacobian 
  $J(\cB^\gov_{\fV,F}|_{\tGa_\fV}, L_{\fV,F}|_{\tGa_\fV})$
$$J^*(\cB^\gov_{\fV,F}|_{\tGa_\fV}, L_{\fV,F}|_{\tGa_\fV})= {{\partial(B_{\fV, s_1}|_{\tGa_\fV} \cdots B_{\fV, s_{\l}}|_{\tGa_\fV}, B_{\fV, t_1}|_{\tGa_\fV} \cdots B_{\fV, t_q}|_{\tGa_\fV},
L_{\fV,F}|_{\tGa_\fV})} \over {{\partial(y_{\uu_F},
x_{(\uu_{s_1},\uv_{s_1})} \cdots x_{(\uu_{s_{\l}},\uv_{s_{\l}})},
x_{(\uu_{t_1},\uv_{t_1})} \cdots x_{(\uu_{t_q},\uv_{t_q})}
 )}}} .$$


Then, one calculates and finds that at the point $\bz$, it is equal to
\begin{eqnarray} \nonumber
{\footnotesize
\left(
\begin{array}{cccccccccc}
a_1x_{(\uu_{s_1}, \uv_{s_1})} & a_1y_{\uu_F}   & \cdots & 0 & 0  & \cdots &0 \\
\vdots \\
a_l x_{(\uu_{s_{\l}}, \uv_{s_{\l}})} & 0 &  \cdots & a_l y_{\uu_F} & 0 &  \cdots & 0\\
* & 0 & \cdots & 0 & b_1 & \cdots & 0 \\
\vdots \\
* & 0 & \cdots & 0 & 0 & \cdots & b_q \\
0 & \sgn (s_1)&  \cdots  & \sgn (s_{\l}) & 0 & \cdots & 0
\end{array}
\right) (\bz).
}
\end{eqnarray}
Recall here that we have $l >0$. 

We can use the last $q$ columns to cancel the entries marked $``* "$ in the first column without
affecting the remaining entries.
  
  Then, multiplying the first column by $-y_{\uu_F}$ ($\ne 0$ at $\bz$), we obtain
\begin{eqnarray} \nonumber
{\footnotesize
\left(
\begin{array}{cccccccccc}
-a_1x_{(\uu_{s_1}, \uv_{s_1})}y_{\uu_F} & a_1y_{\uu_F}   & \cdots & 0 & 0  & \cdots &0 \\
\vdots \\
- a_l x_{(\uu_{s_{\l}}, \uv_{s_{\l}})} y_{\uu_F}& 0 &  \cdots & a_l y_{\uu_F} & 0 &  \cdots & 0\\
0 & 0 & \cdots & 0 & b_1  & \cdots & 0 \\
\vdots \\
0 & 0 & \cdots & 0 & 0 & \cdots & b_q   \\
0 & \sgn (s_1)&  \cdots  & \sgn (s_{\l}) & 0 & \cdots & 0
\end{array}
\right) (\bz).
}
\end{eqnarray}
Multiplying the $(i+1)$-th column by $x_{(\uu_{s_i}, \uv_{s_i})}$ and adding it to the first column 
for all $1\le i\le  {\l}$, 
we obtain
\begin{eqnarray} \nonumber
{\footnotesize
\left(
\begin{array}{cccccccccc}
0 & a_1y_{\uu_F}   & \cdots & 0 & 0  & \cdots &0 \\
\vdots \\
0 & 0 &  \cdots & a_l y_{\uu_F} & 0 &  \cdots & 0\\
0 & 0 & \cdots & 0 & b_1   & \cdots & 0 \\
\vdots \\
0 & 0 & \cdots & 0 & 0 & \cdots & b_q  \\
\sum_{i=1}^l \sgn (s_i) x_{(\uu_{s_i}, \uv_{s_i})}  & \sgn (s_1)&  \cdots  & \sgn (s_{\l}) & 0 & \cdots & 0
\end{array}
\right) (\bz).
}
\end{eqnarray}
But, at the point $\bz$, by \eqref{for-1column}, we have
$$\sum_{i=1}^l \sgn (s_i) x_{(\uu_{s_i}, \uv_{s_i})} (\bz) = - \sgn (s_F) \ne 0.$$
Thus, we conclude that 
$J^*(\cB^\ori_{\fV,F}|_{\tGa_\fV}, L_{\fV,F}|_{\tGa_\fV})$ is a square matrix of full rank at $\bz$,
and one sees that all the variables used to compute it are pleasant with respect to
the list \eqref{good-eq-list}. More precisely, $x_{(\uu_{s_i}, \uv_{s_i})}$ and 
 $x_{(\uu_{t_j}, \uv_{t_j})}$ are pleasant because they uniquely appear in the block
 $\fG_{F}$. The variable $y_{\fV, \uu_F}$ does not appear in the block $\fG_{\fV, F'}$ with $F' <F$,
 because all the relations of $\fG_{F'}$, 
 terminate before $\wp$- and $\ell$-blowups
 with respect to the relations of the block of $\fG_F$ are performed.


 
 \subsection{Applying the Jacobian criterion: Case $(\beta)$}
\label{beta} $\ $

 \medskip\noindent  
  {\it Case $(\beta)$.  Next, we assume 
  $x_{\fV_{[0]}, (\um,\uu_F)}(\bz_0)= 0$. }

  As  $x_{\fV_{[0]}, (\um, \uu_F)}(\bz_0)= 0$, 
we can assume  $\fV$ lies over $(x_{(\uv_{s_0}, \uv_{s_0})} \equiv 1)$ for some 
$s_0 \in S_F \- s_F$.

In this case, we can write 
 $$S_F^{\vr,\ori}=\{s_0, s_1, \cdots, s_\ell\} \;\; \hbox{and}\;\; S_F^{\vr,\inc}=\{t_1, \cdots, t_q\}$$
for some integers $l$ and $q$ such that $l + q= |S_F|-2$.

Then, on the chart $\fV_{[0]}$, the set $\cB^{\vr,\ori}_F$ consists of the following relations
    \begin{eqnarray}\label{cB-vr-ori-beta}  
     B_{\fV_{[0]}, s_0}: \;\;  x_{\fV_{[0]},\uu_F} -
  x_{\fV_{[0]}, (\um, \uu_F)}x_{\fV_{[0]},\uu_{s_0}}  x_{\fV_{[0]},\uv_{s_0}} \;\;\;\;\; \\
 \;\;\;\;   B_{\fV_{[0]}, s_i}: \;\;  x_{\fV_{[0]},(\uu_{s_i}, \uv_{s_i})}x_{\fV_{[0]}, \uu_F}-
x_{\fV_{[0]}, (\um, \uu_F)}  x_{\fV_{[0]},\uu_{s_i}} x_{\fV_{[0]},\uv_{s_i}} , \;\; 1 \le i \le l.
\end{eqnarray}

  The set $\cB^{\vr, \inc}_F$ consists of the following relations
    \begin{equation}\label{cB-vr-inc-beta}  
    B_{\fV_{[0]}, t_i}: \;\;  x_{\fV_{[0]},(\uu_{t_i}, \uv_{t_i})}x_{\fV_{[0]}, \uu_F}-
 x_{\fV_{[0]}, (\um, \uu_F)} x_{\fV_{[0]},\uu_{t_i}} x_{\fV_{[0]},\uv_{t_i}}, \;\; 1 \le i \le q
\end{equation}
for some integer $q \ge 0$ with $q=0$ when $\cB^{\vr, \inc}_F =\emptyset$.


 By Corollary \ref{no-(um,uu)}, we can assume that $\fV$ lies over 
 a preferred chart, that is, in this case, the $\vr$-chart with respect to $F$. 
 Then, by Proposition \ref{eq-for-sV-vtk},  we have
  \begin{eqnarray}\label{cB-vrChart}   
     B_{\fV_\vt, s_0}: \;\;  x_{\fV_\vt,\uu_F} -
 \tilde x_{\fV_\vt,\uu_{s_0}} \tilde x_{\fV_\vt,\uv_{s_0}} \;\;\;\;\; \\
 B_{\fV_\vt, s_i}:  x_{\fV_\vt,(\uu_{s_i}, \uv_{s_i})} x_{\fV_\vt,\uu_F}-
 \tilde x_{\fV_\vt,\uu_{s_i}} \tilde x_{\fV_\vt,\uv_{s_i}} , 1 \le i \le l=|S_F| -2 \nonumber \\
 B_{\fV_\vt, t_i}:  x_{\fV_\vt,(\uu_{t_i}, \uv_{t_i})} x_{\fV_\vt,\uu_F}-
 \tilde x_{\fV_\vt,\uu_{t_i}} \tilde x_{\fV_\vt,\uv_{t_i}} ,  \;\;  i \in [q],
\end{eqnarray}
\begin{equation}\label{vrChart-LF-for-lt-inc}
 L_{\fV_\vt, F} =\sgn(s_F) \de_{\fV_\vt, (\um, \uu_F)}
+ \sum_{s \in S_F\- s_F} \sgn (s) x_{\fV_\vt, (\uu_{s},\uv_{s})}
  \end{equation}
where $\fV_\vt$ is the unique chart of $\tsR_\vt$ that $\fV$ lies over.

We treat the relation $B_{\fV_\vt, s_0}$ first. 

Notice that
$x_{\fV_\vt,\uu_F}$ is the largest variable in the plus-term of $B_{\fV_\vt, s_0}$. 
If $B_{s_0}$ is the smallest in $\cB^\gov_F$, then we can  apply
Lemma \ref{keepLargest} directly to $B_{\fV_\vt, s_0}$
 and conclude that there exists a chart $\fV$ containing $\bz$
such that we have
$$B_{\fV, s_0}: \;\; a_0 x_{\fV,\uu_F} - c_0$$
for some monomial $a_0$ and $c_0$.

Suppose $B_{s_0}$ is not the smallest. Then by the same lines of 
 arguments applied for $B_{\fV, s_i}$ with $i \in [l]$ as in {\it Case $(\alpha)$},
we can  obtain   that there exists a chart $\fV$
containing the point $\bz$  such that we have
\begin{equation}\label{cB-vt-ori-beta}  \nonumber 
     B_{\fV, s_0}: \;\; a_0 y_{\fV,\uu_F} - c_0
\end{equation}
where $y_{\fV,\uu_F}$ is either the $\vp$-variable $x_{\fV,\uu_F}$
or the proper transform of an exceptional-variable $\ve_{\fV', \uu_F}$. 

Now consider $B_{\fV_\vt, s_i}$ with $ i \in [l]\}$.
because $x_{\fV_{[0]},(\uu_{s_i}, \uv_{s_i})}(\bz_0) \ne 0$, one sees that 
there exists a chart $\fV$ containing $\bz$
such that we have
$$B_{\fV, s_i}: \;\; a_i x_{\fV,(\uu_{s_i}, \uv_{s_i})} - c_i, \;\; i \in [l]$$
for some monomial $a_i$ and $c_i$.

Next, consider $B_{\fV_\vt, t_i}$ with $i \in [q]$.

Because $x_{\fV_\vt,(\uu_{t_i}, \uv_{t_i})}$ is the largest blowup-relevant 
variable in the plus term of $B_{\fV_\vt, t_i}$, 
we can  apply Lemma \ref{keepLargest} to  $B_{t_i}$ for all $i \in [q]$ to obtain that 
 there exists a chart $\fV$
containing the point $\bz$  such that we have
$$B_{\fV, t_i}: \;\; b_i x_{\fV,(\uu_{t_i}, \uv_{t_i})} - d_i , i \in [q]$$
for some monomial $b_i$ and $d_i$.

Put all together, shrinking the charts if necessary, we conclude that 
there exists a chart $\fV$ containing the point $\bz$ such that we have
\begin{eqnarray}\label{cB-beta-all}  
B_{\fV, s_0}: \;\; a_0 y_{\fV,\uu_F} - c_0        \;\;\;\;\;\;\;\;\;\;\;\; \;\;\;\;\;\;\;\;\;\;\;\;\;\;\;           \\ 
  B_{\fV, s_i}: \;\; a_i x_{\fV,(\uu_{s_i}, \uv_{s_i})} - c_i, \;\; i \in [l]  \;\;\;\;\;\;\;\;\; \nonumber \\
 B_{\fV, t_i}: \;\;  b_i x_{\fV,(\uu_{t_i}, \uv_{t_i})} - d_i , \;\; i \in [q]. \;\;\;\;\;\;\;\;  \nonumber
\end{eqnarray}




Furthermore, by Proposition \ref{meaning-of-var-wp/ell} (9),  
 we can choose the chart $\fV$ such that
 $$L_{\fV, F}= 1 +\sgn (s_F) y_{\fV, (\um,\uu_F)}$$
  where $y_{\fV, (\um,\uu_F)}$ is the variable for the proper transform of
the divisor $E_{\ell, \vt_k}$ and is pleasant with respect to the list \eqref{good-eq-list}.

Now, we  introduce the following maximal minor of the Jacobian $J(\fG_{\fV,F})$                     
$$J^*(\fG_{\fV,F}|_{\tGa_\fV})= {{\partial(B_{\fV, s_0}|_{\tGa_\fV},
B_{\fV, s_1}|_{\tGa_\fV} \cdots B_{\fV, s_{\l}}|_{\tGa_\fV}), B_{\fV, t_1}|_{\tGa_\fV} \cdots 
B_{\fV, t_q}|_{\tGa_\fV}, L_{\fV, F}|_{\tGa_\fV})}
 \over {{\partial(  y_{ \uu_F},
x_{(\uu_{s_1},\uv_{s_1})} \cdots x_{(\uu_{s_{\l}},\uv_{s_l})}}},
x_{(\uu_{t_1},\uv_{t_1})} \cdots x_{(\uu_{t_q},\uv_{t_q})},
 y_{\fV, (\um,\uu_F)})}. $$

Then, one calculates and finds that at the point $\bz$, it is equal to
\begin{eqnarray} \nonumber
\left(
\begin{array}{cccccccccc}
 a_0 & 0 & \cdots   & 0  & 0   & \cdots & 0 & 0\\
*  &  a_1  & \cdots &   0  & 0 & \cdots & 0 & 0\\
\vdots \\
*  & 0 & \cdots & a_l   & 0   & \cdots & 0 & 0       \\
*  & 0&  \cdots & 0 & b_1 & \cdots  &0     & 0\\
\vdots \\
*  &  0 & \cdots & 0 & 0 & \cdots & b_l         & 0     \\
 0 & * & \cdots & *&  * & \cdots  & * &   \sgn (s_F) 
\end{array}
\right) (\bz).
\end{eqnarray}
Thus, we conclude that 
$J^*(\cB^\gov_{\fV,F}|_{\tGa_\fV})$ 
is a square matrix of full rank at $\bz$, and all the variables that are used to compute it are pleasant.
  
\medskip
{\bf Proof of Lemma \ref{max-minor}.}

\begin{proof} 
It follows immediately by combining the computations in \S \ref{alpha} and \S \ref{beta}.
   \end{proof}

  \subsection{Making conclusion on smoothness} $\ $

   \begin{defn}\label{disjoint-smooth} 
A scheme $X$ is smooth if it is a disjoint union of finitely many connected smooth schemes of
possibly  various dimensions.
\end{defn}

\begin{thm} \label{main-thm}
 Let $\Ga$ be any subset $\var_{\bU}$.
Assume that $Z_\Ga$ is integral.
Let $\tZ_{\ell,\Ga}$ be  the $\ell$-transform of $Z_\Ga$ in $ \tsV_{\ell}$.
Then,  $\tZ_{\ell,\Ga}$ is smooth over  $\Spec \mathbb F$.
Consequently,  $\tZ^\dagger_{\ell,\Ga}$ is smooth over  $\Spec \mathbb F$.
 
 In particular,  when $\Ga=\emptyset$, we obtain that
$\tsV_\ell$ is smooth  over $\Spec \mathbb F$.
\end{thm}
\begin{proof}
Let $\Ga$ be any subset $\var_{\bU}$. Assume that $Z_\Ga$ is integral.

We let $\tZ_{\ell,\Ga}$ be  the $\ell$-transform of $Z_\Ga$ in $ \tsV_{\ell}$.
(As mentioned earlier, $Z_\Ga$ and $\tZ_{\ell,\Ga}$ are considered as $\FF$-schemes.)
Recall that  $\tZ_{\ell,\emptyset}=\tsV_\ell$ when $\Ga=\emptyset$.

Fix any closed point $\bz \in  \tZ_{\ell,\Ga} \subset \tsV_\ell.$
We let $\fV$ be a admissible affine chart containing the point $\bz$ 
as chosen in Lemma \ref{max-minor}. In the sequel, we call 
such a chart a preferred chart for the point $\bz$.

By Corollary \ref{ell-transform-up},
the scheme $\tZ_{\ell, \Ga} \cap \fV$, if nonempty, 
 as a closed subscheme of
the chart $\fV$ of $\tsR_\ell$,  is defined by 
\begin{eqnarray} 
 y, \; \; y \in  \tGa^\zero_\fV; \;\;\;\; y-1, \;\;y \in \tGa^\one_\fV; \;\;\; \label{eq-for-Ga} \\ 
\cB^\gov_\fV, \; \cB^\frb_\fV, \; L_{\sF, \fV}. \;\;\;  \nonumber
 \end{eqnarray}
We let $$\tGa_\fV=\tGa^\zero_\fV \sqcup \tGa^\one_\fV.$$
By setting $y=0$ for all $y \in \tGa^\zero_\fV $ and 
$y=1$ for all $y \in \tGa^\one_\fV $, we obtain a smooth open subset $\fV_\Ga$ of $\fV$:
 $$\fV_\Ga =\{ y =0, \; y \in  \tGa^\zero_\fV;  \;  y=1, \; y \in \tGa^\one_\fV\} \subset \fV.$$
The open susbet $\fV_\Ga$ comes equipped with the set of free variables
 $$\{ y \mid y \in \var_\fV \-  \tGa_\fV\}.$$     
 For any polynomial $f \in \kk[y]_{y \in \var_\fV}$, we let $f|_{\tGa_\fV}$ be obtained from
 $f$ by setting all variables in  $\tGa^\zero_\fV$ to be  0
 and setting  all variables in  $\tGa^\zero_\fV$ to be 1.
 This way,  $f|_{\tGa_\fV}$
 becomes a polynomial over $\fV_\Ga$.
  For any subset $P$ of polynomials over $\fV$, we let
$P|_{\tGa_\fV}=\{f|_{\tGa_\fV} \mid f \in P\}.$ This way, we have
$\cB^\gov_\fV|_{\tGa_\fV}, \cB^\frb_\fV|_{\tGa_\fV}$,  etc.
 
  Then, $\tZ_{\ell, \Ga}\cap \fV$ 
  can be identified with the closed subscheme of $\fV_\Ga$ defined by
 \begin{eqnarray} 
\cB^\gov_\fV|_{\tGa_\fV}, \; \cB^\frb_\fV|_{\tGa_\fV}, \; L_{\sF, \fV}|_{\tGa_\fV}. \;\;\; \label{eq-for-Ga-reduced}
 \end{eqnarray}

Now, we introduce the following maximal minor of the Jacobian $J(\fG_\fV|_{\tGa_\fV})$
 \begin{equation}\label{the-grand-matrix}  
J^*(\fG_\fV|_{\tGa_\fV})=\left(
\begin{array}{cccccccccc}
 J^*(\fG_{\fV,F_1}|_{\tGa_\fV})  & 0 & 0& \cdots & 0  \\
 * & J^*(\fG_{\fV,F_2}|_{\tGa_\fV}) &  0 & \cdots & 0  \\
\vdots &    \\
 * &  * & * & \cdots & J^*(\fG_{\fV, F_\up}|_{\tGa_\fV}) \\
 \end{array}
\right).
\end{equation}
By Lemmas \ref{max-minor},
all the blocks along diagonal are invertible at $\bz$; the entries in the upper right blocks
are due the fact that the variables used to compute the diagonal blocks are all pleasant. 
Therefore,  \eqref{the-grand-matrix}  is a square matrix of full rank at the point $\bz$.
We need to point out here that the terminating variables that we use to compute diagonal blocks
as in Lemma \ref{max-minor} can not belong to $\tGa_\fV^{=1}$ because 
when the variables of $\tGa_\fV^{=1}$ are introduced, the corresponding governing binomial relation
must not terminate by consrtuction; they obviously do not belong to $\tGa_\fV^{=0}$.

\smallskip
 Now, we begin to prove that $\tZ_{\ell,\Ga}$ is smooth.

 First, we consider the case when $\Ga=\emptyset$. In this case, we have
  $Z_\emptyset=\bU \cap \Gr^{3, E}$ and
$\tZ_{\ell,\emptyset }=\tsV_\ell$. 

As earlier, we fix and consider an arbitrary closed point $\bz \in \fV \subset \tsV_\ell$
 where $\fV$ is a preferred admissible affine chart of $\tsR_\ell$.

We let $J:=J(\cB^\frb, \cB^\gov_\fV, L_{\fV,\sF})$ be the full Jacobian
 of all the defining equations of $\tsV_\ell \cap \fV$ in $\fV$.
We let  $J^*:=J^*(\fG_\fV)$ 
be the matrix of \eqref{the-grand-matrix} in the case of $\Ga=\emptyset$.
at the given point $\bz \in \tsV_\ell$.  (The maximal minor $J^*$ depends on the point $\bz$.)
Let $T_\bz (\tsV_\ell)$ be the Zariski tangent space of
$\tsV_\ell$ at $\bz$. Then,
we have
$$\dim T_\bz (\tsV_\ell)= \dim \tsR_\ell- \rk J(\bz)\le \dim \tsR_\ell- \rk J^*(\bz)$$
$$= \dim \bU + |\cB^\gov|  - (|\cB^\gov| + \up)
= \dim \bU  - \up = \dim \tsV_\ell,$$
where  $\dim \tsR_\ell=\dim \bU + |\cB^\gov|$ by \eqref{dim} and
$\rk J^*(\bz) =|\cB^\gov| + \up$ by \eqref{the-grand-matrix}.
Hence, $\dim T_\bz (\tsV_\ell)  = \dim \tsV_\ell,$
thus, $\tsV_\ell$ is smooth at $\bz$. Therefore, $\tsV_\ell$ is smooth.

Consequently,  one sees that on any preferred admissible affine chart $\fV$ of
the  scheme $\tsR_\ell$, all the relations of $\cB^\frb_\fV$
 must lie in the ideal generated by relations of $\cB^\gov_\fV$ and 
$L_{\sF, \fV}$,  
 thus,   can be discarded from the chart $\fV$.

 Now, we return to a general  subset $\Ga$ of $\var_{\bU}$
 as stated in the theorem.
 
Again, we fix and consider any closed point $\bz \in \fV \subset \tZ_{\ell,\Ga}$
 where $\fV$ is a preferred admissible affine chart of $\tsR_\ell$ for the point.

  By the previous paragraph (immediately after proving that $\tsV_\ell$ is smooth),
  over any preferred admissible affine chart $\fV$ of $\tsR_\ell$
  with $\tZ_{\ell,\Ga} \cap \fV \ne \emptyset$, we can discard 
   $\cB^\frb_\fV|_{\tGa_\fV}$
   from the defining equations
 of $\tZ_{\ell,\Ga} \cap \fV$ and focus only on
 the equations of $\cB^\gov_\fV|_{\tGa_\fV}$ and 
 $L_{\fV,  \sF}|_{\tGa_\fV}$.
 In other words, 
  $\tZ_{\ell,\Ga} \cap \fV$, if nonempty, as a closed subcheme of $\fV_\Ga$
  (which depends on both $\Ga$ and the point $\bz$), is defined by the equations in 
 $$\cB^\gov_\fV|_{\tGa_\fV}, \;\; L_{\sF, \fV}|_{\tGa_\fV}.$$

 Then, by  \eqref{the-grand-matrix}, the rank of the full Jacobian of 
 $\cB^\gov_\fV|_{\tGa_\fV}$ and $L_{\sF, \fV}|_{\tGa_\fV}$
 equals to the number of the above defining equations at the closed point $\bz$ 
 of $\tZ_{\ell,\Ga} \cap \fV$.
 Hence, $\tZ_{\ell,\Ga}$ is smooth at $\bz$, thus,  so is $\tZ_{\ell,\Ga}$.

This proves the theorem.
\end{proof}

Let $X$ be an integral scheme.  We say $X$ admits a resolution  if there exists  a smooth
scheme $\tX$ and a projective  birational morphism
from $\tX$ onto $X$.

\begin{thm}\label{cor:main} 
Let $\Ga$ be any subset $\var_{\bU}$.
Assume that $Z_\Ga$ is integral. Then, the morphism
 $\tZ^\dagger_{\ell, \Ga} \to Z_{\Ga}$ can be decomposed as
\begin{equation}\label{decom}
\tZ^\dagger_{\vr, \Ga} \to \cdots 
\to \tZ^\dagger_{\hs,\Ga}  \to \tZ^\dagger_{\hs',\Ga} \to \cdots \to
Z^\dagger_{\sF_{[j]},\Ga}  \to Z^\dagger_{\sF_{ [j-1]},\Ga} \to \cdots \to Z_\Ga
\end{equation}
such that every morphism $\tZ^\dagger_{\hs,\Ga}  \to \tZ^\dagger_{\hs',\Ga}$
in the sequence is $\tZ^\dagger_{\vt_{[k]},\Ga}  \to \tZ^\dagger_{\vt_{ [k-1]},\Ga}$ for some $k \in [\up]$, or
$ \tZ^\dagger_{(\wp_{(k\tau)}\fr_\mu\fs_{h}),\Ga}
 \to \tZ^\dagger_{(\wp_{(k\tau)}\fr_\mu \fs_{h-1}),\Ga}$ for some $(k\tau) \mu h \in \Index_{\Phi_k}$, or
$ \tZ^\dagger_{\ell_k,\Ga} \to \tZ^\dagger_{\wp_k} $ for some $k \in [\up]$.
Further, every morphism in the sequence is surjective, projective, and  birational.  
In particular, $\tZ^\dagger_{\ell, \Ga} \to Z_\Ga$ 
is a resolution if $Z_\Ga$ is singular.
\end{thm}
\begin{proof} The smoothness of $\tZ^\dagger_{\ell, \Ga}$ follows from Theorem \ref{main-thm};
the decomposition of $\tZ^\dagger_{\ell, \Ga} \to Z_{\Ga}$ follows from
Lemmas \ref{wp-transform-sVk-Ga},  \ref{vt-transform-k},
 \ref{wp/ell-transform-ktauh}.
\end{proof}

In Part II, we will apply the above to obtain resolution for singular
affine or projective variety $X$ over a perfect field.

\section{Appendix:
Mn\"ev's universality \`a la Lafforgue}\label{universality} $\ $

For readers' convenience, we  review here
Lafforgue's presentation of \cite{La03} on Mn\"ev's universality theorem.

As before, suppose we have a set of vector spaces, 
$E_1, \cdots, E_n$ such that 
$E_\alpha$ is of dimension 1 (or, a free module of rank 1 over $\ZZ$).
 We let 
$$E_I = \bigoplus_{\alpha \in I} E_\alpha, \;\; \forall \; I \subset [n],$$
$$E:=E_{[n]}=E_1 \oplus \ldots \oplus E_n.$$ 
(Lafforgues \cite{La03} considers the  more general case by allowing $E_\alpha$ to be
of any  finite dimension.)

For any fixed  integer $1\le d <n$, the Grassmannian
$$\Gr^{3, E}=\{ F \hookrightarrow E \mid \dim F=d\}$$
 decomposes into a disjoint union of locally closed strata
$$\Gr^{3, E}_\ud=\{ F \hookrightarrow E \mid \dim (F\cap E_I)=d_I,  \;\; \forall \; I \subset [n] \}$$
indexed by the family  $\ud=(d_I)_{I \subset [n]}$ of nonnegative integers $d_I \in \NN$ verifying

$\bullet$ $d_\emptyset=0, d_{[n]}=d$,

$\bullet$ $d_I +d_J\le d_{I\cup J} + d_{I \cap J}$, for all $I, J \subset [n]$.

The family $\ud$ is called a matroid of rank $d$ on the set $[n]$.
The stratum $\Gr^{3, E}_\ud$ is called a  matroid Schubert cell.

The Grassmannian $\Gr^{3, E}$ comes equipped with the (lattice) polytope
$$\Delta^{d,n} =\{ (x_1, \cdots, x_n) \in {\mathbb R}^n \mid 0 \le x_\alpha\le 1, \;
\forall \; \alpha; \; x_1 +\cdots + x_n = d\}.$$
For any $\ui=(i_1,\cdots,i_3) \in \II_{3,n}$, we let $\bx_\ui = (x_1, \cdots, x_n)$ be defined by
\begin{equation}\label{eta-L-2}
\left\{ 
\begin{array}{lcr}
x_i=1, &  \hbox{if $i \in \ui$,} \\
x_i=0, & \hbox{otherwise}.
\end{array} \right.
\end{equation}
It is known that  $\Delta^{d,n} \cap \NN^n =\{\bx_\ui \mid \ui \in \II_{3,n}\}$ and
it consists of precisely the vertices of the polytope $\Delta^{d,n}$.

Then, the matroid $\ud=(d_I)_{I \subset [n]}$ above defines the 
following subpolytope of $\Delta^{d,n}$
$$\Delta^{d,n}_\ud =\{ (x_1, \cdots, x_n) \in \Delta^{d,n} \mid  \sum_{\alpha \in I} x_\alpha \ge d_I, \; \forall \; I \subset [n]\}.$$
This is called the matroid subpolytope of $\Delta^{d,n}$ corresponding to $\ud$.

Recall that we have a canonical decomposition
$$\wedge^3 E=\bigoplus_{\ui \in \II_{3,n}} E_{i_1}\otimes \cdots \otimes E_{i_3}$$
and it gives rise to the $\pl$ embedding of the Grassmannian
$$\Gr^{3, E} \hookrightarrow \PP(\wedge^3 E)=\{(p_\ui)_{\ui \in \II_{3,n}} \in \GG_m 
\backslash (\wedge^3 E  \- \{0\} )\}.$$

\begin{prop}\label{to-Ga} {\rm (Proposition, p4, \cite{La03})} 
Let $\ud$ be any matroid of rank $d$ on the set $[n]$ as considered above.
Then, in the Grassmannian 
$$\Gr^{3, E} \hookrightarrow \PP(\wedge^3 E)=\{(p_\ui)_{\ui \in \II_{3,n}} \in \GG_m 
\backslash (\wedge^3 E \- \{0\} )\},$$
the  matroid Schubert cell $\Gr^{3, E}_\ud$, as a locally closed subscheme, is defined by
$$p_\ui = 0, \;\;\; \forall \; \bx_\ui \notin \Delta^{d,n}_\ud ,$$ 
$$p_\ui \ne 0, \;\;\; \forall \; \bx_\ui \in \Delta^{d,n}_\ud . $$   
\end{prop}

Let $\ud=(d_I)_{I \subset [n]}$ be a matroid of rank $d$ on the set $[n]$ as above. 
Assume that $\ud_{[n]\setminus \{\alpha\}} =d-1$ for all $1\le \alpha\le n$.
Then, the configuration space $C^{d,n}_\ud$ defined by the matroid $\ud$ is the classifying scheme
of families of $n$ points
$$P_1, \cdots, P_n$$
on the projective space $\PP^{d-1}$ such that for any nonempty subset $I \subset [n]$,
the projective subspace $P_I$ of $\PP^{d-1}$ generated by the points $P_\alpha, \alpha \in I$, is
of dimension 
$$\dim P_I = d-1 -\ud_I.$$

\begin{thm}\label{Mn-La} {\rm (Mn\"ev, Theorem I. 14, \cite{La03})}
Let $X$ be an affine scheme of finite type over $\Spec \ZZ$.
Then, there exists a matroid $\ud$ of rank 3 on
the set $[n]$ such that $\PGL_3$ acts freely on the configuration space $C^{3,n}_\ud$.
Further, there exists a positive integer $r$ and
 an open subset $U \subset X \times \AA^r$ projecting surjectively
onto $X$ such that $U$ is isomorphic to the quotient space
${\underline C}^{3,n}_\ud :=C^{3,n}_\ud/\PGL_3$.
\end{thm}

\begin{thm}\label{GM} {\rm (Gelfand, MacPherson, Theorem I. 11, \cite{La03})}
Let $\ud$ be any matroid of rank $d$ on the set $[n]$ as considered above.
Then, the action of $\PGL_{d-1}$ on $C^{d,n}_\ud$ is free if and only if
$\dim_{\mathbb R} \Delta^{d,n}_\ud =n-1$. Similarly, 
 the action of $\GG_m^n/\GG_m$ on $\Gr^{d,n}_\ud$ is free if and only if
$\dim_{\mathbb R} \Delta^{d,n}_\ud =n-1$. 
Further, when $\dim_{\mathbb R} \Delta^{d,n}_\ud =n-1$,  the quotient
$C^{d,n}_\ud/\PGL_{d-1}$ can be canonically identified with  the quotient
$\Gr^{3, E}_\ud/(\GG_m^n/\GG_m )$.
\end{thm}


By the above correspondence, we have the following equivalent version of Theorem \ref{Mn-La}.

\begin{thm}\label{Mn-La-Gr} {\rm (Mn\"ev, Theorem I. 14, \cite{La03})}
Let $X$ be an affine scheme of finite type over $\Spec \ZZ$.
Then, there exists a matroid $\ud$ of rank 3 on
the set $[n]$ such that $(\GG_m^n/\GG_m )$ acts freely on the  matroid Schubert cell $\Gr^{3,E}_\ud$.
Further, there exists a positive integer $r$ and an open subset $U \subset X \times \AA^r$ projecting 
onto $X$ such that $U$ is isomorphic to the quotient space
$\bGr^{3,E}_\ud:=\Gr^{3,E}_\ud/(\GG_m^n/\GG_m)$.
\end{thm}

Finally, for a given matroid Schubert cell $\Gr^{3,E}_\ud$, we specify its
corresponding $\Ga$ as follows.
Since 
$\Delta^{3,n}_\ud \ne \emptyset$, there exists $\um \in \II_{3,n}$ such that
$\bx_\um \in \Delta^{3,n}_\ud$. (By permutation if necessary, $\um$ can be assumed to be $(123)$.)
We define
\begin{equation}\label{ud=Ga}
\Ga:=\Ga_\ud =\{ \ui \in \II_{3,n} \mid \bx_{\ui} \notin \Delta^{3,n}_\ud \}.
\end{equation}

\bigskip\medskip

Part II will follow.

\end{document}